\documentclass[reqno]{amsart}

\RequirePackage[dvipsnames]{xcolor}
\usepackage{amsfonts, fancyhdr, amsmath, amssymb, hyperref, url, amsthm, subfigure, xypic, tikz-cd, verbatim,lscape, enumitem, capt-of}
\usepackage[all]{xy}
\usepackage{amscd,graphics}

\numberwithin{equation}{section}

\theoremstyle{norm}
\newtheorem{thm}{Theorem}[section]
\newtheorem{lem}[thm]{Lemma}

\newtheorem{notation}[thm]{Notation}
\newtheorem{prop}[thm]{Proposition}
\newtheorem{df}[thm]{Definition}
\newtheorem{conj}[thm]{Conjecture}
\newtheorem{cor}[thm]{Corollary}
\newtheorem{terminology}[thm]{Terminology}
\newtheorem{rem}[thm]{Remark}
\newtheorem{exam}[thm]{Example}
\newtheorem{quest}[thm]{Question}

\newcommand{\CP}{\mathbb{C}P}

\DeclareMathOperator{\Thom}{\textup{Thom}}
\DeclareMathOperator{\Pin}{\textup{Pin(2)}}

\title[Intersection Forms of Spin 4-Manifolds and Mahowald Invariant]{Intersection Forms of Spin 4-Manifolds and the $\Pin$-Equivariant Mahowald Invariant}
\author{Michael J. Hopkins}
\author{Jianfeng Lin}
\author{XiaoLin Danny Shi}
\author{Zhouli Xu}

\begin{document}
%%%%%%%%%%%Abstract

\maketitle
\begin{abstract}
In studying the ``11/8-Conjecture'' on the Geography Problem in 4-dimensional topology, Furuta proposed a question on the existence of Pin(2)-equivariant stable maps between certain representation spheres.  In this paper, we present a complete solution to this problem by analyzing the Pin(2)-equivariant Mahowald invariants.   As a geometric application of our result, we prove a ``10/8+4"-Theorem.

We prove our theorem by analyzing maps between certain finite spectra arising from BPin(2) and various Thom spectra associated with it.  To analyze these maps, we use the technique of cell diagrams, known results on the stable homotopy groups of spheres, and the $j$-based Atiyah--Hirzebruch spectral sequence. 

\begin{comment}
A fundamental problem in 4-dimensional topology is the following geography question: ``which simply connected topological 4-manifolds admit a smooth structure?" After the celebrated work of Kirby--Siebenmann, Freedman, and Donaldson, the last uncharted territory of this geography question is the ``11/8-Conjecture''. This conjecture, proposed by Matsumoto, states that for any smooth spin 4-manifold, the ratio of its second-Betti number and signature is least 11/8.  

Furuta proved the ``10/8+2"-Theorem by studying the existence of certain Pin(2)-equivariant stable maps between representation spheres. In this paper, we present a complete solution to this problem by analyzing the Pin(2)-equivariant Mahowald invariants. In particular, we improve Furuta's result into a ``10/8+4"-Theorem. Furthermore, we show that within the current existing framework, this is the limit.

We discuss the Pin(2)-equivariant Mahowald invariants of powers of certain Euler classes in the $RO(\textup{Pin(2)})$-graded equivariant stable homotopy groups of spheres. Our proof analyzes maps between certain finite spectra arising from $B$Pin(2) and various Thom spectra associated with it. To analyze these maps, we use the technique of cell diagrams, known results on the stable homotopy groups of spheres, and the $j$-based Atiyah--Hirzebruch spectral sequence. 
\end{comment}
\end{abstract}

\tableofcontents

%%%%%%%%%%%%%%%%%%%%%%%%%%
\section{Introduction}
\subsection{The classification problem of simply connected 4-manifolds} A fundamental question in four-dimensional topology is the following: 
\begin{quest}\label{classification of 4-manifolds}
How to classify all closed simply connected topological 4-manifolds?
\end{quest}

%{\color{red} Break into two questions.  Leave the smooth for later. }

To start our discussion, let $N$ be a simply connected topological 4-manifold. There are two important invariants of $N$:
\begin{enumerate}
\item The intersection form $Q_{N}$: this is a symmetric unimodular bilinear form over $\mathbb{Z}$ given by the cup-product 
\begin{eqnarray*}
Q_{N}:H^{2}(N;\mathbb{Z})\times H^{2}(N;\mathbb{Z})&\longrightarrow& \mathbb{Z}, \\
(a,b)&\longmapsto& \langle a\cup b,[N]\rangle.
\end{eqnarray*}
\item The Kirby--Siebenmann invariant $ks(N)$ (defined in \cite{Kirby-Siebenmann1977}): this is an element in $H^4(N; \mathbb{Z}/2) = \mathbb{Z}/2$.% that vanishes if $N$ admits a piecewise linear structure or equivalently, a smooth structure (By the work of \cite{Munkres1960, Munkres1964a, Munkres1964b, Hirsch1963}). {\color{red} Delete PL. Move smooth to later, add reference.}
\end{enumerate}
%In the category of topological manifolds, 
Question \ref{classification of 4-manifolds} was resolved by the following famous work of Freedman: %{\color{red} ask David about was resolved}

\begin{thm}[Freedman \cite{Freedman1982}]\label{Freedman}\hfill
\begin{enumerate}
\item Two closed simply connected topological 4-manifolds are homeomorphic if and only if their intersection forms are isomorphic and their Kirby--Siebenmann invariants are the same. 
\item When the form is not even, any combination of the symmetric unimodular bilinear form and Kirby--Siebenmann invariant can be realized by a closed simply connected topological 4-manifold. 
\item When the form is even, the combination can be realized if and only if the Kirby--Siebenmann invariant is equal to the signature of the form divided by 8 modulo 2.  (Note that the signature of an even form must be divisible by $8$.  See \cite[Section 1.1.3]{Donaldson-Kronheimer1990} for example.)
\end{enumerate}
\end{thm}

Therefore, given two manifolds, one can deduce whether they are homeomorphic or not by computing their intersection forms and Kirby--Siebenmann invariants.  Moreover, Theorem~\ref{Freedman} implies that any symmetric unimodular bilinear form can be realized by exactly two non-homeomorphic closed simply connected topological 4-manifolds if it is non-even, and by exactly one manifold if it is even. \vspace{0.1in}

%{\color{red} Add the question for the smooth category}
We will now move on to the smooth category.

\begin{quest}\label{classification of smooth 4-manifolds}
How to classify all closed simply connected smooth 4-manifolds?
\end{quest}
By the works of Cairns, Whitehead, Munkres, Hirsch, and Kirby--Siebenmann \cite{Cairns, Whitehead, Munkres1960, Munkres1964a, Munkres1964b, Hirsch1963, Kirby-Siebenmann1977},  the Kirby--Siebenmann invariant of any smooth manifold, and in particular, a smooth 4-manifold, is zero.  This fact, combined with Theorem~\ref{Freedman}, shows that two closed simply connected smooth 4-manifolds are homeomorphic if and only if they have isomorphic intersection forms.  Therefore, Question~\ref{classification of smooth 4-manifolds} naturally breaks down into the following two questions:
\begin{quest}\label{geography}
Given a symmetric unimodular bilinear form $Q$, can it be realized as the intersection form of a closed simply connected smooth 4-manifold?
\end{quest}
\begin{quest}\label{botany}
Suppose that the answer to Question~\ref{geography} is yes, then how many non-diffeomorphic 4-manifolds can realize the given form?  
%Given a closed simply connected smooth 4-manifold, how many smooth structures does the underlying topological manifold have? 
\end{quest}

In other words, Question \ref{geography} is asking which closed simply connected topological 4-manifolds admit a smooth structure.  Question \ref{botany} is asking that if they do, how many different smooth structures do they admit.  Topologists often refer Question~\ref{geography} as the ``Geography Problem'' and Question~\ref{botany} as the ``Botany Problem''.  

The main motivation of our work comes from the Geography Problem.  In the past thirty years, starting with Donaldson's groundbreaking work in \cite{Donaldson1983}, significant progress towards the resolution of the Geography Problem has been made.  

Let's divide symmetric unimodular bilinear forms $Q$ over $\mathbb{Z}$ into two categories: the definite ones and the indefinite ones. For definite forms, a complete algebraic classification is still unknown. Nevertheless, Donaldson proved the following seminal theorem.

\begin{thm}[Donaldson's Diagonalizability Theorem \cite{Donaldson1983}]\label{Donaldson theorem A} A definite symmetric unimodular bilinear form $Q$ can be realized as the intersection form of a closed simply connected smooth 4-manifold if and only if $Q$ can be represented by the matrix $I$ or $-I$.
\end{thm}

This gives a complete answer to Question \ref{geography} in the case when $Q$ is definite.

For indefinite forms, a powerful algebraic theorem of Hasse and Minkowski (see \cite{Serre1977}) states that if $Q$ is not even, it must be isomorphic to a diagonal form with entries $\pm 1$, and if $Q$ is even, it must be isomorphic to 
%two indefinite symmetric unimodular bilinear forms over $\mathbb{Z}$ are isomorphic if and only if they have the same rank, signature, and type (even or not even).  In particular, if a form is indefinite and not even, it must be isomorphic to a diagonal form with entires $\pm 1$.  Such a form can always be realized by a connected sum of copies of $\mathbb{C}P^{2}$ and $\overline{\mathbb{C}P^{2}}$. 

%On the other hand, if the form $I$ is even, it is well known that the signature of $I$ must be divisible by $8$ (see \cite[Section 1.1.3]{Donaldson-Kronheimer1990} for example).  This implies that the rank of $I$ is even.  By the Hasse--Minkowski Theorem, 
\begin{equation} \label{equ:even indef}
	k E_{8}\oplus  q\begin{pmatrix}0&1\\1&0\end{pmatrix}
\end{equation}
for some $k\in \mathbb{Z}$ and $q\in \mathbb{N}$ (for negative $k$, $kE_{8}$ denotes the direct sum of $|k|$ copies of $-E_{8}$). %By replacing $I$ with $-I$ if necessary, we may assume that $k\geq 0$.  

When the bilinear form $Q$ is not even, by the theorem of Hasse and Minkowski, $Q$ can always be realized by a connected sum of copies of $\mathbb{C}P^{2}$ and $\overline{\mathbb{C}P^{2}}$.

When the bilinear form $Q$ is even, by Wu's formula \cite{Wu1950}, the closed simply connected 4-manifold $M$ realizing $Q$ must be spin. Furthermore, by Rokhlin's theorem \cite{Rohlin1952}, the integer $k$ in (\ref{equ:even indef}) must be even. By reversing the orientation of $M$, we may assume that $k \geq 0$. 
%Now, suppose that the form $I$ can be realized by a closed simply connected smooth 4-manifold $M$.  By Wu's formula \cite{Wu1950}, the second Stiefel--Whitney class of $M$, $\omega_{2}(M)$, vanishes.  This implies that $M$ is spin.  Applying Rokhlin's theorem \cite{Rohlin1952}, we conclude that $k$ must be even. 

To this end, the following celebrated conjecture of Matsumoto \cite{Matsumoto1982} serves as the last missing piece to this puzzle:  

\begin{conj}[The $\frac{11}{8}$-Conjecture, version 1]\label{11/8 version 1} The form $$2pE_{8}\oplus q\begin{pmatrix}0&1\\1&0\end{pmatrix}$$ can be realized as the intersection form of a closed smooth spin $4$-manifold if and only if $q\geq 3p$.
\end{conj}

\begin{rem}\rm
Note that Conjecture \ref{11/8 version 1} is for general closed smooth spin $4$-manifolds, which are not necessarily simply connected.
\end{rem}

The ``if'' part of Conjecture~\ref{11/8 version 1} is straightforward: if $q\geq 3p$, then the form can be realized by 
$$\mathop{\#}\limits_{p}K3\mathop{\#}\limits_{q-3p}(S^{2}\times S^{2}).$$
Recall that the intersection form of $K3$ and $S^{2}\times S^{2}$ are 
$$2E_{8}\oplus 3\begin{pmatrix}0&1\\1&0\end{pmatrix} \text{ and } \begin{pmatrix}0&1\\1&0\end{pmatrix},$$  
respectively. 
	
The ``only if'' part of Conjecture~\ref{11/8 version 1} can be reformulated as follows:

\begin{conj}[The $\frac{11}{8}$-Conjecture, version 2]\label{11/8 version 2} 
Any closed smooth spin 4-manifold $M$ must satisfy the inequality
$$b_{2}(M)\geq \frac{11}{8}|\operatorname{sign}(M)|,$$
where $b_2(M)$ and $\operatorname{sign}(M)$ are the second Betti number and the signature of $M$, respectively.
\end{conj}

\begin{df}\rm
An even symmetric unimodular bilinear form is \textit{spin realizable} if it can be realized as the intersection form of a closed smooth spin 4-manifold.  
\end{df}

By studying anti-self-dual Yang--Mills equations, Donaldson proved Conjecture~\ref{11/8 version 1} in the case when $p=1$, under the additional assumption that $H_{1}(M;\mathbb{Z})$ has no $2$-torsions \cite{Donnaldso1986,Donaldson1987}.  The condition on $H_{1}(M;\mathbb{Z})$ was later removed by Kronheimer \cite{Kronheimer1994}, who made use of the $\Pin$-symmetries in Seiberg--Witten theory.  Later, Furuta combined Kronheimer's approach with a technique called the ``finite dimensional approximation'' and proved the following significant result:

\begin{thm}[Furuta's $\frac{10}{8}$-Theorem \cite{Furuta2001}]\label{10/8+2} 
For $p\geq 1$, the bilinear form $$2pE_{8}\oplus q\begin{pmatrix}0&1\\1&0\end{pmatrix}$$ is spin realizable only if $q\geq 2p+1$.
\end{thm}

As we will explain in Section~\ref{subsec:SeibergWitten}, Furuta proved Theorem~\ref{10/8+2} by studying a problem in equivariant stable homotopy theory (Question \ref{main question}), which concerns the existence of certain stable $\Pin$-equivariant maps between representation spheres.  The main purpose of this paper is to provide a complete answer to this $\Pin$-equivariant problem.  A consequence of our main theorem (Theorem~\ref{main theorem}) is the following:
\begin{thm}\label{theorem: 10/8+4}
For $p\geq 2$, the bilinear form $$2pE_{8}\oplus q\begin{pmatrix}0&1\\1&0\end{pmatrix}$$ is spin realizable only if
$$ q \geq  \begin{cases}
    2p+2        \quad p\equiv 1,2,5,6 &\pmod 8 \\
    2p+3       \quad p\equiv 3,4,7 &\pmod 8 \\
    2p+4    \quad p\equiv 0&\pmod 8.
  \end{cases}$$
\end{thm}
\begin{cor}\label{corollary: 10/8+4} Any closed simply connected smooth spin $4$-manifold $M$ that is not homeomorphic to $S^{4}$, $S^{2}\times S^{2}$, or $K3$ must satisfy the inequality
\begin{equation}\label{equ: 10/8+4}
b_{2}(M)\geq \frac{10}{8}|\operatorname{sign}(M)|+4.	
\end{equation}
\end{cor}
\begin{proof}
Recall that the rank of $E_{8}$ is 8, and that the signatures of $E_{8}$ and $\left(\begin{smallmatrix}0&1\\1&0\end{smallmatrix}\right)$ are $8$ and $0$, respectively. Therefore, (\ref{equ: 10/8+4}) is equivalent to the inequality
 $$q\geq 2p+2.$$By Theorem~\ref{theorem: 10/8+4}, this is true when $p\geq 2$. By Theorem~\ref{Donaldson theorem A} and Theorem~\ref{10/8+2}, the only exceptional cases are the following: 
 $$(p,q) = (0,0), \, (0,1), \, (1,3).$$
 These cases correspond to $S^{4},S^{2}\times S^{2}$, and $K3$ by Theorem~\ref{Freedman}.
\end{proof}

As we will see in Section~\ref{subsec:MainResult}, Corollary~\ref{corollary: 10/8+4} is the ``limit'' of current methods towards resolving the $\frac{11}{8}$-Conjecture using Bauer--Furuta invariants (see Remark~\ref{rem: limit}). 

%%%%%%%%%%%%

\subsection{Finite dimensional approximation in Seiberg--Witten theory}\label{subsec:SeibergWitten}
In this subsection, we will give a brief summary of Furuta's proof of Theorem~\ref{10/8+2}.

Let $M$ be a smooth spin 4-manifold. By doing surgery along essential loops in $M$ (which does not change its intersection form), we may assume that $b_{1}(M)=0$. The Seiberg--Witten equations (a set of first order nonlinear elliptic differential equations), together with the Coulomb gauge fixing condition, can be combined to produce a nonlinear continuous map 
$$
\widetilde{SW}:H_{1}\rightarrow H_{2}
$$
between two Hilbert spaces $H_{1}$ and $H_{2}$.  Instead of describing the map $\widetilde{SW}$ explicitly, we list three of its key properties:
\begin{enumerate}[label=(\Roman*)]
\item $\widetilde{SW}$ can be decomposed into the sum $L+C$, where $L:H_{1}\rightarrow H_{2}$ is a Fredholm operator and $C$ is a nonlinear map that send any bounded subset of $H_{1}$ to a compact subset of $H_{2}$.

\item There exist constants $R_{0}, \epsilon$ such that 
\begin{equation}\label{bounded property}
0\in \widetilde{SW}^{-1}(B(H_{2},\epsilon))\subset B(H_{1},R_{0}),
\end{equation}
where $B(-,-)$ denotes the closed ball in $H_{i}$ with center $0$ and given radius.  

\item The Lie group 
$$\Pin:=\{e^{i\theta}\}\cup \{je^{i\theta}\}\subset \mathbb{H}$$
acts on both $H_1$ and $H_2$.  Under these actions, the map $\widetilde{SW}$ is a $\Pin$-equivariant map. 
\end{enumerate}
By choosing a finite dimensional subspace $V_{2}$ of $H_2$ that is transverse to the image of $L$ and invariant under the $\Pin$-action, one can define the ``approximated Seiberg--Witten map'' 
$$\widetilde{SW}_{\text{apr}}:=L+\operatorname{pr}_{V_{2}}\circ C:V_{1}\rightarrow V_{2}.$$
Here, $V_{1}:=L^{-1}(V_{2})$ and $\operatorname{pr}_{V_{2}}:H_{2}\rightarrow V_{2}$ is the orthogonal projection.  For $\epsilon>0$, consider the set
$\widetilde{SW}_{\text{apr}}^{-1}(B(V_{2},\epsilon))$.  By property (II) above and elliptic bootstrapping arguments, one can show that whenever $V_{2}$ is large enough, the following 
condition holds 
\begin{equation}\label{approximated bounded property}
\widetilde{SW}_{\text{apr}}(\partial B(V_{1},R_{0}+1))\subset \overline{V_{2}\setminus B(V_{2},\epsilon)}.
\end{equation}

Now, consider the representation spheres
$$S^{V_{1}}=B(V_{1},R_{0}+1)/\partial B(V_{1},R_{0}+1)$$
and 
$$S^{V_{2}}=V_{2}/(\overline{V_{2}\setminus B(V_{2},\epsilon)}).$$
Then by (\ref{approximated bounded property}), the map $\widetilde{SW}_{\text{apr}}$ induces a $\Pin$-equivariant map 
$$\widetilde{SW}^{\vee}_{\text{apr}}: S^{V_{1}}\rightarrow  S^{V_{2}}.$$
%by the formula
%$$
%\widetilde{SW}^{\vee}_{\text{apr}}([x])=\begin{cases}
%    [\widetilde{SW}_{\text{apr}}(x)]    \quad \text{if }\widetilde{SW}_{\text{apr}}(x)\in B(V_{2},\epsilon),  \\
%  [\partial B(V_{2},\epsilon)]     \quad \text{otherwise}.
%  \end{cases}
%$$
%Here, $[-]$ denotes the equivalence class in the quotient space.  
Applying $\Sigma^\infty(-)$, the map $\Sigma^\infty(\widetilde{SW}^{\vee}_{\text{apr}})$ represents an element in $\pi^{\Pin}_{\bigstar}(S^{0})$, the $RO(\Pin)$-graded equivariant  stable homotopy group of spheres. It was proved by Bauer and Furuta \cite{Bauer-Furuta2004} that this element is independent with respect to the choices of auxiliary data (e.g., the Riemann metric and the spaces $V_{1},\ V_{2}$) and is an invariant of the smooth structure on $M$. This invariant is called the Bauer--Furuta invariant and is denoted by $BF(M)$.		

The following theorem is due to Furuta \cite{Furuta2001}. We include a sketch proof for completeness.

\begin{thm}[Furuta \cite{Furuta2001}]\label{lem: SW map is Furuta} \hfill
\begin{enumerate}
\item Suppose $I_{M}=2pE_{8}\oplus q \left(\begin{smallmatrix}0&1\\1&0\end{smallmatrix}\right)$.  Then 
$$BF(M) \in \pi_{p\mathbb{H}- q\widetilde{\mathbb{R}}}^{\Pin}S^0.$$ 
Here, $\mathbb{H}$ is the four-dimensional representation of $\Pin$, with $\Pin$ acting on it via left multiplication, and $\widetilde{\mathbb{R}}$ is a $1$-dimensional representation such that the unit component acts as identity and the other component acts as negative identity.

\item The element $BF(M)$ fits into the commutative diagram 
$$\begin{tikzcd}
S^{p\mathbb{H}} \ar[rd, "BF(M)"] &  \\ 
S^0 \ar[u,"a_{\mathbb{H}}^p"] \ar[r, swap, "a^q_{\widetilde{\mathbb{R}}}"] & S^{q\widetilde{\mathbb{R}}}, 
\end{tikzcd}$$
where $a_\mathbb{H} \in \pi_{-\mathbb{H}}^{\Pin} S^0$ and $a_{\widetilde{\mathbb{R}}} \in \pi_{-\widetilde{\mathbb{R}}}^{\Pin} S^0$ are stable homotopy classes that represents the inclusions 
$S^0 \hookrightarrow S^\mathbb{H}$ and $S^0 \hookrightarrow S^{\widetilde{\mathbb{R}}}$ of fixed points. 
\end{enumerate}
\end{thm}

\begin{proof}[Sketch proof] (1) The $RO(\Pin)$-grading of  $BF(M)$ is $V_{1}-V_{2}$.  This is the index of the operator $L$ and can be computed by the Atiyah--Singer index theorem. \vspace{0.1in}

\noindent (2) By the specific definitions of $H_{i}$, $V_1$ and $V_2$ are direct sums of $\mathbb{H}$ and $\widetilde{\mathbb{R}}$.  Therefore, the $\Pin$-fixed points of $S^{V_1}$ and $S^{V_2}$ are both $0$ and $\infty$.  %the $\Pin$-fixed point-set of $S^{V_{1}}$ is 
%$$\left\{[0], [\partial B(V_{1},R_{0}+2) ]\right\}$$ 
%and the $\Pin$-fixed point-set of $S^{V_{2}}$ is 
%$$\left\{[0],[\partial B(V_{2},\epsilon)]\right\}.$$  
By (\ref{bounded property}) and (\ref{approximated bounded property}), the map $\widetilde{SW}^{\vee}_{\text{apr}}$ sends $0$ to $0$ and $\infty$ to $\infty$.  Therefore, it induces a homotopy equivalence on the $\Pin$-fixed points.  It follows that after applying the suspension functor $\Sigma^\infty(-)$, the map $\Sigma^\infty(\widetilde{SW}^{\vee}_{\text{apr}})$ induces an identity on $\Pin$-geometric fixed points. 
\end{proof}

%\begin{df}\rm {\color{blue} change this to power point}For $p,q>0$, \textit{a Furuta--Mahowald class} of level-$(p,q)$ is an element in
%$$
%\pi^{\Pin}_{p[\mathbb{H}]- q[\widetilde{\mathbb{R}}]}(S^{0})
%$$
%that induces a homotopy equivalence on the $\Pin$-geometric fixed points.
%\end{df}
\begin{df}\rm
For $p \geq 1$, \textit{a Furuta--Mahowald class} of level-$(p,q)$ is a stable map 
$$\gamma: S^{p\mathbb{H}} \longrightarrow S^{q\widetilde{\mathbb{R}}} $$
that fits into the diagram 
$$\begin{tikzcd}
S^{p\mathbb{H}} \ar[rd, "\gamma"] &  \\ 
S^0 \ar[u,"a_{\mathbb{H}}^p"] \ar[r, swap, "a^q_{\widetilde{\mathbb{R}}}"] & S^{q\widetilde{\mathbb{R}}}.
\end{tikzcd}$$
\end{df}
Using equivariant $K$-theory, Furuta proved the following theorem, from which Theorem~\ref{10/8+2} directly follows.
\begin{thm}[Furuta \cite{Furuta2001}]\label{Furuta: necessary condition} A level-$(p,q)$ Furuta--Mahowald class exists \textup{\textbf{only if}} $q\geq 2p+1$.
\end{thm}

\subsection{Main theorem}\label{subsec:MainResult}
At this point, it is natural to ask the following question:
\begin{quest}\label{main question}
What is the necessary and sufficient condition for the existence of a level-$(p,q)$ Furuta--Mahowald class? 
\end{quest}

\begin{rem}\rm \label{rem:Universe}
We would now like to discuss the choice of the universe (i.e. the Pin(2)-representations that one stabilize with respect to when passing from the space level to the spectrum level).  In Furuta's original proof of Theorem 1.16 \cite{Furuta2001}, he used the universe consisting of only the representations $\mathbb{H}$ and $\widetilde{\mathbb{R}}$, because this universe is the most relevant to the geometric problem.  Modified proofs by Manolescu~\cite{Manolescu} and Bryan~\cite{Bryan}, using divisibilities of the $K$-theoretic Euler classes, show that the statement of Theorem 1.16 holds for any universe. 

For Question 1.17, the answer could potentially depend on the choice of the universe.  By works of Schmidt~\cite[Theorem~2.6, Theorem~3.2]{Schmidt2003} and Minami~\cite{Minami}, any Furuta--Maholwald class can be desuspended to the same diagram on the space level as long as $q \geq 2p+1$.  By the discussions in the previous paragraph, the bound $q \geq 2p+1$ in Theorem 1.16 holds for any universe.  Therefore, a level-$(p,q)$ Furuta--Mahowald class in one universe can be desuspended to a space-level map $S^{p\mathbb{H}} \to S^{q\widetilde{\mathbb{R}}}$, and then be further suspended to a level-$(p,q)$ Furuta--Mahowald class in any other universe.  It follows that the answer to Question 1.17 does not depend on the choice of the universe.   
 
Without loss of generality, we always work with the complete universe.   
\end{rem}

One might hope that the answer to Question~\ref{main question} is $q\geq 3p$ because this would directly imply the $\frac{11}{8}$-conjecture (Conjecture~\ref{11/8 version 1}).  Unfortunately, John Jones showed that this is false by exhibiting a counter-example for $p =5$. See \cite{Furuta-Kametani-Matsue-Minami2007} for a more conceptual explaination of why such counter-examples must exist. 

Subsequently, Jones proposed the following conjecture:

\begin{conj}[Jones \cite{Furuta-Kametani-Matsue-Minami2007}]\label{Jones} 
For $p\geq 2$, a level-$(p,q)$ Furuta--Mahowald class exists \textup{\textbf{if and only if}}
$$ q \geq  \begin{cases}
    2p+2        \quad p\equiv 1 &\pmod 4 \\
        2p+2        \quad p\equiv 2 &\pmod 4 \\
    2p+3       \quad  p\equiv 3 &\pmod 4 \\
    2p+4    \quad  p\equiv 4 &\pmod 4.
  \end{cases}
$$
\end{conj}
For the necessary condition, various progress has been made by Stolz \cite{Stolz1989}, Schmidt \cite{Schmidt2003}, and Minami \cite{Minami}.  Before our paper, the best result is given by Furuta--Kametani:
\begin{thm}[Furuta--Kametani \cite{Furuta-Kametani2007}]\label{thm: Furuta-Kametani}  For $p\geq 2$, a level-$(p,q)$ Furuta--Mahowald class exists \textup{\textbf{only if}}
$$ q \geq  \begin{cases}
2p+1 \quad  p\equiv 1&\pmod 4 \\
    2p+2        \quad  p\equiv 2&\pmod 4 \\
        2p+3       \quad  p\equiv 3&\pmod 4 \\
    2p+3       \quad  p\equiv 4&\pmod 4. \\
  \end{cases}
$$
\end{thm}

Much less is known about the sufficient condition for the existence of Furuta--Mahowald classes.  So far, the best result is in Schmidt's thesis \cite{Schmidt2003}, in which Schmidt used $\Pin$-equivariant stable homotopy theory to attack Conjecture~\ref{Jones} for $p \leq 5$.  In particular, Schmidt showed the existence of a level-$(5,12)$ Furuta--Mahowald class.  This is also the first attempt to study this problem by using $\Pin$-equivariant stable homotopy theory. 
%Schmidt's argument used $\Pin$-equivariant stable homotopy theory By computing stable cohomotopy groups of low dimensions. 

In this paper, we completely resolve Question \ref{main question}.  The following theorem is the main result of our paper:
\begin{thm}[The limit is $\frac{10}{8}+4$]\label{main theorem} For $p\geq 2$, a level-$(p,q)$ Furuta--Mahowald class exists \textup{\textbf{if and only if} }
$$ q \geq  \begin{cases}
    2p+2        \quad p\equiv 1&\pmod 8 \\
        2p+2        \quad p\equiv 2&\pmod 8 \\
    2p+3       \quad  p\equiv 3&\pmod 8 \\
        2p+3       \quad  p\equiv 4&\pmod 8 \\
                2p+2        \quad p\equiv 5&\pmod 8 \\
                        2p+2        \quad p\equiv 6&\pmod 8 \\
                                2p+3       \quad  p\equiv 7&\pmod 8 \\
    2p+4    \quad  p\equiv 8 &\pmod 8.
  \end{cases}
$$
\end{thm}

\begin{rem}\rm
The ``only if'' part of Theorem~\ref{main theorem} directly implies Theorem~\ref{theorem: 10/8+4} and Corollary~\ref{corollary: 10/8+4}.  
\end{rem}

\begin{rem}\label{rem: limit}\rm
The ``if'' part of Theorem~\ref{main theorem} implies that without further input from geometry or analysis, the best result one can achieve in proving Conjecture~\ref{11/8 version 2}, using the existence of Furuta--Mahowald classes, is $\frac{10}{8}+4$.  Note that by Remark~\ref{rem:Universe} this ``limit'' does not depend on the choice of the universe.  In order to break this ``limit'' and to further attack the $\frac{11}{8}$-conjecture, more delicate properties of the Seiberg--Witten map have to be studied.  In particular, the Seiberg--Witten map should not be merely treated as a continuous map. 
\end{rem} 

\begin{rem}\rm
Our answer differs from Conjecture~\ref{Jones} when $p\equiv 4\pmod 8$.  Note that in \cite{Schmidt2003}, Schmidt proved that Conjecture~\ref{Jones} is true for $p \leq 5$.  We came to a different conclusion for $p =4$ because there is a minor error in Schmidt's computation (see Remark~\ref{rem:SchmidtError} for more details).  
\end{rem}

%%%%%%%%%%
\subsection{The Pin(2)-equivariant Mahowald invariant}
Let $G$ be a finite group or a compact Lie group and let $RO(G)$ denote its real representation ring.  One can consider $\pi^G_\bigstar S^0$, the $RO(G)$-graded stable homotopy groups of spheres.  Unlike the classical nonequivariant case, there are many non-nilpotent elements in $\pi^G_\bigstar S^0$.  Here are some examples: 

\begin{enumerate}
\item For each prime $p$, the multiplication-by-$p$ map 
$$p:S^0\longrightarrow S^0$$
between spheres with trivial $G$-actions is non-nilpotent.    
\item The geometric fix point functor induces a homomorphism 
$$\Phi^G: \pi_0^G S^0 = [S^0, S^0]^G \longrightarrow [S^0, S^0] = \mathbb{Z}$$
from the Burnside ring of $G$ to $\mathbb{Z}$.  Since $\Phi^G(-)$ preserves smash products, any preimage of the nonequivariant multiplication-by-$p$ map is also a non-nilpotent element in $\pi_0^G S^0$.
 
\item Let $V$ be a real irreducible representation of $G$.  The \emph{Euler class} $a_{V}$ is the stable class in $\pi_{-V}^G S^0$ that represents the inclusion
$$a_V: S^0 \longrightarrow S^V$$
of the fix points.  Since all the powers of $a_V$ induce nonzero maps in equivariant homology, $a_{V}$ is non-nilpotent in $\pi^G_\bigstar S^0$.
\end{enumerate}

\begin{df}\label{df:GMahowaldInvariant}\rm
Suppose that $\alpha$ and $\beta$ are elements in $\pi^G_\bigstar S^0$ with $\beta$ non-nilpotent.  The \emph{$G$-equivariant Mahowald invariant of $\alpha$ with respect to $\beta$} is the following set of elements in $\pi^G_\bigstar S^0$:
$$M_{\beta}^G(\alpha) = \{\gamma\,|\, \alpha = \gamma \beta^k, \ \alpha \ \textup{is not divisible by} \ \beta^{k+1}\}.$$
In other words, an element $\gamma$ belongs to $M_{\beta}^G(\alpha)$ if the left diagram exists and the right diagram does not exist for any class $\gamma' \in \pi_\bigstar^G S^0$.
\begin{displaymath}
	\xymatrix{
	S^{-k|\beta|} \ar[rrdd]^{\gamma} & & & & S^{-(k+1)|\beta|} \ar[rrdd]^{\gamma'} & & \\
	& & & & & & \\
	S^0 \ar[uu]^{\beta^k} \ar[rr]^{\alpha} & & S^{-|\alpha|} & & S^0 \ar[uu]^{\beta^{k+1}} \ar[rr]^{\alpha} & & S^{-|\alpha|}.
	}
\end{displaymath}
\end{df}
\begin{rem}\rm
It is clear from Definition~\ref{df:GMahowaldInvariant} that the $RO(G)$-degree of each of the elements in $M_{\beta}^G(\alpha)$ is $k |\beta| - |\alpha|$.
\end{rem}

Historically, the $G$-equivariant Mahowald invariant has been studied in many cases:

\noindent (1) Let $G = C_2$ be the cyclic group of order 2.  The real representation ring of $C_2$ is
$$RO(C_2) = \mathbb{Z} \oplus \mathbb{Z},$$
generated by the trivial representation 1 and the sign representation $\sigma$.  The classical Borsuk--Ulam theorem in the unstable category is equivalent to the following statement when phrased in terms of the $C_2$-equivariant Mahowald invariant:
\begin{thm}[Borsuk--Ulam]
For all $q \geq 0$, the $RO(C_2)$-degree of $M^{C_2}_{a_{\sigma}}(a_{\sigma}^q)$ is zero.	
\end{thm}

\noindent (2) Let $G = C_2$.  Consider the homomorphism 
$$\Phi^{C_2}: \pi_{n}^{C_2} S^0 = [S^{n}, S^0]^{C_2} \longrightarrow [S^n, S^0] = \pi_n S^0$$
that is induced by the geometric fix point functor.  For any non-equivariant class $\alpha \in \pi_n S^0$, consider all of its preimages under the map $\Phi^{C_2}$ and their corresponding $C_2$-equivariant Mahowald invariants with respect to the Euler class $a_\sigma$.  

Among all the elements in $M^{C_2}_{a_\sigma}\big((\Phi^{C_2})^{-1} \alpha\big)$, pick the element that has the highest degree in its $\sigma$-component.  Then, apply the forgetful functor to the nonequivariant world.  Bruner and Greenlees \cite{BrunerGreenlees} proved that this construction produces the classical Mahowald invariant $M(\alpha)$ of $\alpha$, which has been studied extensively by Mahowald, Ravenel, and Behrens \cite{MahowaldRavenel, Behrens2, Behrens1}.

$$\begin{tikzcd}
S^{n+k\sigma} \ar[rrdd, ""{name=1, above = 0.1em, right = 0.2em}] & & S^{n+k} \ar[rrdd, "M(\alpha)", ""{name=2,below = 0.1em, left = 0.5em}] & & \\
&  & & & \\
S^n \ar[uu, "a_\sigma^k"] \ar[rr, "(\Phi^{C_2})^{-1} \alpha", ""{name=3, below}] &  & S^{0} & & S^0 \\
& & & &\\
S^n  \ar[rr, "\alpha", ""{name=4, above = 0.3em}] & & S^{0} & & 
\ar[rr, Rightarrow, "\text{forget}", from = 1, to = 2]
\ar[dd, Rightarrow, "\Phi^{C_2}", from = 3, to = 4]
\end{tikzcd}$$

\begin{comment}
 \begin{displaymath}
	\xymatrix{
	S^{n+k\sigma} \ar[rrdd] & & S^{n+k} \ar[rrdd]^{M(\alpha)} & & \\
	& \ar@{=>}[rr]^{\textup{forget}} & & & \\
	S^n \ar[uu]^{a_\sigma^k} \ar[rr]^{(\Phi^{C_2})^{-1} \alpha} & \ar@{=>}[dd]^{\Phi^{C_2}} & S^{0} & & S^0 \\
	& & & &\\
	S^n  \ar[rr]_{\alpha} & & S^{0} & & 
	}
\end{displaymath}
\end{comment}

In particular, when $n=0$ and $\alpha$ is a power of 2, Bredon \cite{BredonStableStems, BredonEquivariant} made conjectures about the degrees of the elements in $M^{C_2}_{a_\sigma}\big((\Phi^{C_2})^{-1} 2^q\big)$ for $q \geq 1$.  His conjecture was proved by Landweber \cite{Landweber}, who used equivariant K-theory.  Later, Bruner and Greenlees \cite{BrunerGreenlees} translated Mahowald and Ravenel's work \cite{MahowaldRavenel} and obtained an independent proof of Bredon's conjecture.

\begin{thm}[Landweber \cite{Landweber}, Mahowald--Ravenel \cite{MahowaldRavenel}]
For $q \geq 1$, the set $M(2^q)$ contains the first nonzero element of Adams filtration $q$.  Moreover, the following 4-periodic result holds:
\begin{equation*}
 |M^{C_2}_{a_\sigma}\big((\Phi^{C_2})^{-1} 2^q\big)| = \left\{
 \begin{aligned}
    (8k+1) \sigma & \ \ \ \textup{if} \ q = 4k+1\\
    (8k+2) \sigma & \ \ \ \textup{if} \ q = 4k+2\\
    (8k+3) \sigma & \ \ \ \textup{if} \ q = 4k+3\\
    (8k+7) \sigma & \ \ \ \textup{if} \ q = 4k+4.
      \end{aligned}
 \right.
\end{equation*}
\end{thm}

We would like to mention that Bredon--L\"{o}ffler \cite{BredonEquivariant, BredonStableStems} and Mahowald--Ravenel \cite{MahowaldRavenel} have independently made the following conjecture:
\begin{conj}[Bredon--L\"{o}ffler, Mahowald--Ravenel]
For any non-equivariant class $\alpha$ that is of positive degree, we have the inequality
$$|M(\alpha)| \leq 3|\alpha|.$$
\end{conj}
Jones \cite{Jones} proved that $|M(\alpha)| \geq 2|\alpha|$ for all non-equivariant classes $\alpha$ of positive degrees.  The $C_2$-equivariant formulation of the classical Mahowald invariant gives a simpler proof of Jones's result (see \cite{BrunerGreenlees, Bruner}, for example). 
\vspace{0.1in}

\noindent (3) Let $G = C_4$, the cyclic group of order 4.  The real representation ring of $C_4$ is
$$RO(C_4) = \mathbb{Z} \oplus \mathbb{Z} \oplus \mathbb{Z},$$
generated by the trivial representation 1, the sign representation $\sigma_4$, and the two-dimensional representation $\lambda$ that corresponds to rotation by 90 degrees.  The $C_4$-equivariant Mahowald invariant of powers of $a_{\sigma_4}$ with respect to $a_{2\lambda}$ has been studied by Crabb \cite{Crabb}, Schmidt \cite{Schmidt2003}, and Stolz \cite{Stolz1989}.

\begin{thm}[Crabb \cite{Crabb}, Schmidt \cite{Schmidt2003}, Stolz \cite{Stolz1989}] \label{C4 Minv}
For $q \geq 1$, the following 8-periodic result holds:
\begin{equation*}
 |M^{C_4}_{a_{2\lambda}}(a_{\sigma_4}^q)| + q \sigma_4 = \left\{
 \begin{aligned}
    &8k \lambda & \ \ \ \textup{if} \ q = 8k+1 \ \\
    &8k \lambda & \ \ \ \textup{if} \ q = 8k+2 \ \\
    &(8k+2) \lambda & \ \ \ \textup{if} \ q = 8k+3 \ \\
    &(8k+2) \lambda & \ \ \ \textup{if} \ q = 8k+4 \ \\
    &(8k+2) \lambda & \ \ \ \textup{if} \ q = 8k+5\ \\
    &(8k+4) \lambda & \ \ \ \textup{if} \ q = 8k+6\ \\
    &(8k+4) \lambda & \ \ \ \textup{if} \ q = 8k+7\ \\
    &(8k+4) \lambda & \ \ \ \textup{if} \ q = 8k+8.
      \end{aligned}
 \right.
\end{equation*}
	\end{thm}
Since $C_4$ is a subgroup of $\Pin$, Theorem~\ref{C4 Minv} was used by Minami \cite{Minami} and Schmidt \cite{Schmidt2003} to deduce the existence of Furuta--Mahowald classes.  Crabb \cite{Crabb} also studied the $C_4$-equivariant Mahowald invariant of powers of $a_{\sigma_4}$ with respect to $a_{\lambda}$.

For our case, we are interested in the group $G = \textup{Pin(2)}$ and its irreducible representations $\mathbb{H}$ and $\widetilde{\mathbb{R}}$ (defined in Theorem~\ref{lem: SW map is Furuta}).  By definition, it is clear that a level-$(p,q)$ Furuta--Mahowald class exists if and only if the $\mathbb{H}$-degree of 
$$|M^{\textup{Pin(2)}}_{a_{\mathbb{H}}}(a_{\widetilde{\mathbb{R}}}^q)| + q \widetilde{\mathbb{R}}$$
is greater than or equal to $p$. 

To prove our main theorem (Theorem~\ref{main theorem}), we translate it into a problem of analyzing the Pin(2)-equivariant Mahowald invariants of powers of $a_{\widetilde{\mathbb{R}}}$ with respect to $a_{\mathbb{H}}$.  After this translation, our main theorem is equivalent to the following theorem: 
\begin{thm}\label{thm:mainTheoremMInvariant}
For $q \geq 4$, the following 16-periodic result holds:
\begin{comment}
\begin{equation*}
 |M^{\textup{Pin(2)}}_{a_{\mathbb{H}}}(a_{\widetilde{\mathbb{R}}}^q)| + q \widetilde{\mathbb{R}} = \left\{
 \begin{aligned}
    (8k-1) \mathbb{H} & \ \ \ \textup{if} \ q = 16k+1\\
    (8k-1) \mathbb{H} & \ \ \ \textup{if} \ q = 16k+2\\
    (8k-1) \mathbb{H} & \ \ \ \textup{if} \ q = 16k+3\\
    (8k+1) \mathbb{H} & \ \ \ \textup{if} \ q = 16k+4\\
    (8k+1) \mathbb{H} & \ \ \ \textup{if} \ q = 16k+5\\
    (8k+2) \mathbb{H} & \ \ \ \textup{if} \ q = 16k+6\\
    (8k+2) \mathbb{H} & \ \ \ \textup{if} \ q = 16k+7\\
    (8k+2) \mathbb{H} & \ \ \ \textup{if} \ q = 16k+8\\
    (8k+3) \mathbb{H} & \ \ \ \textup{if} \ q = 16k+9\\
    (8k+3) \mathbb{H} & \ \ \ \textup{if} \ q = 16k+10\\
    (8k+4) \mathbb{H} & \ \ \ \textup{if} \ q = 16k+11\\
    (8k+5) \mathbb{H} & \ \ \ \textup{if} \ q = 16k+12\\
    (8k+5) \mathbb{H} & \ \ \ \textup{if} \ q = 16k+13\\
    (8k+6) \mathbb{H} & \ \ \ \textup{if} \ q = 16k+14\\
    (8k+6) \mathbb{H} & \ \ \ \textup{if} \ q = 16k+15\\
    (8k+6) \mathbb{H} & \ \ \ \textup{if} \ q = 16k+16.
      \end{aligned}
 \right.
\end{equation*}
\end{comment}
\begin{equation*}
 |M^{\textup{Pin(2)}}_{a_{\mathbb{H}}}(a_{\widetilde{\mathbb{R}}}^q)| + q \widetilde{\mathbb{R}} = \left\{
 \begin{array}{ll  |  ll}
    (8k-1) \mathbb{H} &  \textup{if} \ q = 16k+1   \hspace{0.1in}&\hspace{0.1in}    (8k+3) \mathbb{H} & \textup{if} \ q = 16k+9\\
    (8k-1) \mathbb{H} &  \textup{if} \ q = 16k+2    \hspace{0.1in}&\hspace{0.1in}   (8k+3) \mathbb{H} &  \textup{if} \ q = 16k+10\\
    (8k-1) \mathbb{H} &  \textup{if} \ q = 16k+3  \hspace{0.1in}&\hspace{0.1in}  (8k+4) \mathbb{H} &  \textup{if} \ q = 16k+11\\
    (8k+1) \mathbb{H} & \textup{if} \ q = 16k+4   \hspace{0.1in}&\hspace{0.1in} (8k+5) \mathbb{H} &  \textup{if} \ q = 16k+12\\
    (8k+1) \mathbb{H} &  \textup{if} \ q = 16k+5   \hspace{0.1in}&\hspace{0.1in}  (8k+5) \mathbb{H} &  \textup{if} \ q = 16k+13\\
    (8k+2) \mathbb{H} &  \textup{if} \ q = 16k+6   \hspace{0.1in}&\hspace{0.1in} (8k+6) \mathbb{H} &  \textup{if} \ q = 16k+14\\
    (8k+2) \mathbb{H} &  \textup{if} \ q = 16k+7  \hspace{0.1in}&\hspace{0.1in}  (8k+6) \mathbb{H} &  \textup{if} \ q = 16k+15\\
    (8k+2) \mathbb{H} &  \textup{if} \ q = 16k+8  \hspace{0.1in}&\hspace{0.1in}  (8k+6) \mathbb{H} &  \textup{if} \ q = 16k+16.
      \end{array}
 \right.
\end{equation*}
\end{thm}
Note that when $q = 16k+11$, 
$$|M^{\textup{Pin(2)}}_{a_{\mathbb{H}}}(a_{\widetilde{\mathbb{R}}}^q)| + q \widetilde{\mathbb{R}} = (8k+4) \mathbb{H}.$$
If the answer had been $(8k+3) \mathbb{H}$ instead, then Theorem~\ref{thm:mainTheoremMInvariant} would be an 8-periodic result and Jones's conjecture (Conjecture~\ref{Jones}) would be true.  This deviation from Jones conjecture is explained in details in Step 6 of our proof (See Sections~\ref{sec:OutlineofProof} and \ref{sec:Steps5and6}).

\subsection{Summary of techniques}
To resolve Question~\ref{main question}, which is a problem in $\Pin$-equivariant stable homotopy theory, we first translate it into a problem in non-equivariant stable homotopy theory.  More specifically, we consider the sequence of maps
$$X(m) \longrightarrow X(m-1) \longrightarrow \cdots \longrightarrow S^0, $$
%$$\xymatrix{
%X(m) \ar[r] & X(m-1) ar[r] & \cdots \ar[r] & S^0
%}$$
which are maps between certain Thom spectra over $\textup{BPin}(2)$ that are induced by inclusions of (virtual) subbundles.  Given this sequence of maps, our Pin(2)-equivariant problem is equivalent to asking what is the maximal skeleton of each $X(m)$ that maps trivially to $S^0$.  We call the ``vanishing" line that connects these skeletons the \textit{Mahowald line}.  Intuitively, by drawing the cell diagrams for each $X(m)$, we can visualize the Mahowald line in Figure~\ref{fig:OutlineAccurateMahowaldLine}.  See Section~\ref{subsec:equivToNonEquiv} for more details.   

One can also form a Mahowald line for the computation of the classical Mahowald invariants for powers of 2.  The analogous diagram to Figure~\ref{fig:OutlineAccurateMahowaldLine} in the classical case has the cell diagram for $\Sigma RP^{\infty}_{-\infty}$ in each column.  Maps between the columns are the multiplication by 2 maps.  The classical Mahowald line in this case is established by Mahowald--Ravenel by proving a lower bound and an upper bound for the line, and observing that they coincide.  Our proof in the $\Pin$-equivariant case is in the same spirit as Mahowald--Ravenel.  However, as we point out below, it is significantly more complicated and delicate than the classical arguments:

\begin{enumerate} [leftmargin=*]
\item Classically, the lower bound is proved by using a theorem of Toda \cite{Toda}, which states that 16 times the identity maps on certain 8-cell subquotients of $RP^{\infty}$ are zero.  This implies that the Mahowald line rises by at least 8 dimensions every time we move by four columns.  In our situation, the analogue of Toda's result does not hold.  Therefore, our situation requires a more delicate inductive argument that gives us control over several cells above the Mahowald line (this control is not needed in the classical case). 

\item Classically, the upper bound is proved via detection by the real connective $K$-theory $ko$.  In our case, this techniques does not work at $X(8k+3)$, $k \geq 1$, which is the crux of the geometric application of our main theorem (Theorem~\ref{theorem: 10/8+4} and Corollary~\ref{corollary: 10/8+4}). To handle this case, we need a careful study of both the $j$-based and the sphere-based Atiyah--Hirzebruch spectral sequence of $X(8k+3)$. 

\item Classically, the lower bound and the upper bound are proven independently, and they happen to coincide.  In our case, the proofs for the lower bound and the upper bound are not independent.  More precisely, we first establish a rough lower bound in Step 1 (Section~\ref{subsec:Step1}) and a rough upper bound in Step 2 (Section~\ref{subsec:Section2Step2}).  These rough bounds do not coincide, but they do give us some information on the cells that are located in between them (Step 3, Section~\ref{subsec:Section2Step3}).  Using this information, we refine the lower bound and the upper bound step-by-step, while updating information about the undetermined cells until the two bounds finally match each other (Steps 4--7, Sections~\ref{subsec:Section2Step4}--\ref{subsec:Step6}). 
\end{enumerate}

%%%%%%%
\subsection{Summary of contents}
We now turn to give a summary of the paper.  In Section~\ref{sec:OutlineofProof}, we provide an outline-of-proof for our main theorem (Theorem~\ref{main theorem}).  We first reduce the $\Pin$-equivariant statement regarding the existence of a level-$(p,q)$ Furuta--Mahowald class into a non-equivariant statement (Proposition~\ref{prop: equivariant to nonequivariant}).  The non-equivariant statement is determined by the location of the Mahowald line.  Theorem~\ref{thm:exactMahowaldLine} proves the exact location of the Mahowald line, from which our main theorem directly follows.  Our proof of Theorem~\ref{thm:exactMahowaldLine} consists of seven steps, described in Sections~\ref{subsec:Step1}--\ref{subsec:Step6}.  The readers should regard Section~\ref{sec:OutlineofProof} as a roadmap to the rest of the paper, as it contains all the main statements needed to prove Theorem~\ref{thm:exactMahowaldLine}.

%consisting of $X(m)$, which are Thom spectra of certain bundles over $\textup{BPin}(2)$.  Given this cell diagram, we introduce the notion of a Mahowald line (Definitions~\ref{df:MahowaldLine} and ~\ref{df:MahowaldLine2}).  a cell diagram (Figure~\ref{fig:OutlineAccurateMahowaldLine})

In Section~\ref{sec:Preparations}, we define maps between certain subquotients of $X(m)$ that will be useful in the later sections.  In Section~\ref{sec:AttachingMaps}, we prove certain attaching maps in $X(m)$.  Sections~\ref{sec:Step1}--\ref{sec:Steps5and6} prove all the statements that are listed in Sections~\ref{subsec:Step1}--\ref{subsec:Step6}.

This paper has two appendices.  Appendix~\ref{sec:AppendixA} proves the combinatorial statements that are needed for the arguments in Sections~\ref{sec:secondlock} and \ref{sec:Steps5and6}.  Appendix~\ref{sec:AppendixB} recalls the definition of cell diagrams, a tool that we use for illustration purposes throughout the paper.

\subsection{Acknowledgements} 
The authors would like to thank Mark Behrens, Rob Bruner, Simon Donaldson, Houhong Fan, Dan Isaksen, Achim Krause, Peter Kronheimer, Ciprian Manolescu, Haynes Miller, Tom Mrowka, Doug Ravenel, and Guozhen Wang for helpful conversions.  The authors would also like to thank Zilin Jiang and Yufei Zhao for writing a program in the early stages of the project to check our combinatorial results in Appendix A.  The first author was supported by NSF grant DMS-1810917; the second author was supported by NSF grant DMS-1707857; and the fourth author was supported by NSF grant DMS-1810638.

%%%%%%%%%%%%%%%%%%%%%%
%%%%%%%%%%%%%%%%%%%%%%
\newpage
\section{Outline of Proof for Main Theorem}\label{sec:OutlineofProof}

In this section, we give an outline of our proof for Theorem~\ref{main theorem}.
\subsection{Equivariant to nonequivariant reduction}\label{subsec:equivToNonEquiv}
Consider the classifying space $B\Pin=S(\infty\mathbb{H})/\Pin$. There is a universal bundle 
$$
\begin{tikzcd} \Pin  \ar[r,hookrightarrow]&  E\Pin \ar[r]& B\Pin. \end{tikzcd}
$$
We let $\lambda$ be the line bundle associated to the representation $\tilde{\mathbb{R}}$ 
and set 
$$X(m):=\Thom(\textup{BPin}(2),-m\lambda).$$

Alternatively, there is a $C_{2}$-action on the space $BS^{1}=\mathbb{C}P^{\infty}$, given by:
\begin{equation}\label{C2 action}
(z_1, z_2, z_3, z_4, \ldots, z_{2n-1}, z_{2n} ) \longmapsto (-\bar{z}_2, \bar{z}_1, -\bar{z}_4, \bar{z}_3, \ldots, -\bar{z}_{2n}, \bar{z}_{2n-1}).
\end{equation}
The quotient space of $BS^{1}$ with respect to this $C_2$-action is the classifying space $\textup{BPin}(2)$.  
Given this, $\lambda$ can also be defined as the line bundle that is associated to the principal bundle 
$$\begin{tikzcd} C_2 \ar[r,hookrightarrow]&  BS^{1} \ar[r] &\textup{BPin}(2).\end{tikzcd}$$

Note that there is a fiber bundle 
\begin{equation}\label{BpinoverBSU}
\begin{tikzcd}
\mathbb{R}\textup{P}^{2}\ar[r,hookrightarrow] & \textup{BPin}(2) \ar[r]&  \mathbb{H}P^{\infty}.
\end{tikzcd}
\end{equation}
The cellular structure on $\mathbb{H}P^{\infty}$ (one cell in dimension $4k$ for each $k \geq 0$) and $\mathbb{R}\textup{P}^{2}$ (one cell in dimensions $0$,$1$,$2$) induces a cellular structure on $\textup{BPin}(2)$, and hence on $X(m)$.  Given this cellular structure, we use $X(m)^{a}_{b}$ to denote the subquotient of $X(m)$ that contains all cells of dimensions between $a$ and $b$. 

For $m\geq n$, the inclusion $n \lambda \hookrightarrow m \lambda$ of subbundles induces a map
$$i(m,n):X(m)\longrightarrow X(n).$$
Let 
$$c(0):X(0)=\Sigma^{\infty}\textup{BPin}(2)_{+} \longrightarrow S^{0}$$
be the stabilization of the base-point preserving map that sends all of $\textup{BPin}(2)$ to the point in $S^0$ that is not the base-point.  For $m>0$, define the map $c(m)$ to be the composition 
$$X(m)\xrightarrow{i(m,n)} X(0)\xrightarrow{c(0)}S^{0}.$$ 
We will also define the map $c(m)^{k}$ to be the restriction of $c(m)$ to the subcomplex $X(m)^{k}$:
$$c(m)^{k}:X(m)^{k}\longrightarrow S^{0}.$$

\begin{prop}\label{prop: equivariant to nonequivariant}
A level-$(p,q)$ Furuta--Mahowald class exists if and only if the map
$$
c(q)^{4p-2-q}:X(q)^{4p-2-q}\longrightarrow S^{0}
$$
is zero.
\end{prop}

\begin{figure}
\begin{center}
\makebox[\textwidth]{\includegraphics[trim={0cm 4.2cm 0.6cm 7.8cm}, clip, page = 1, scale = 0.7]{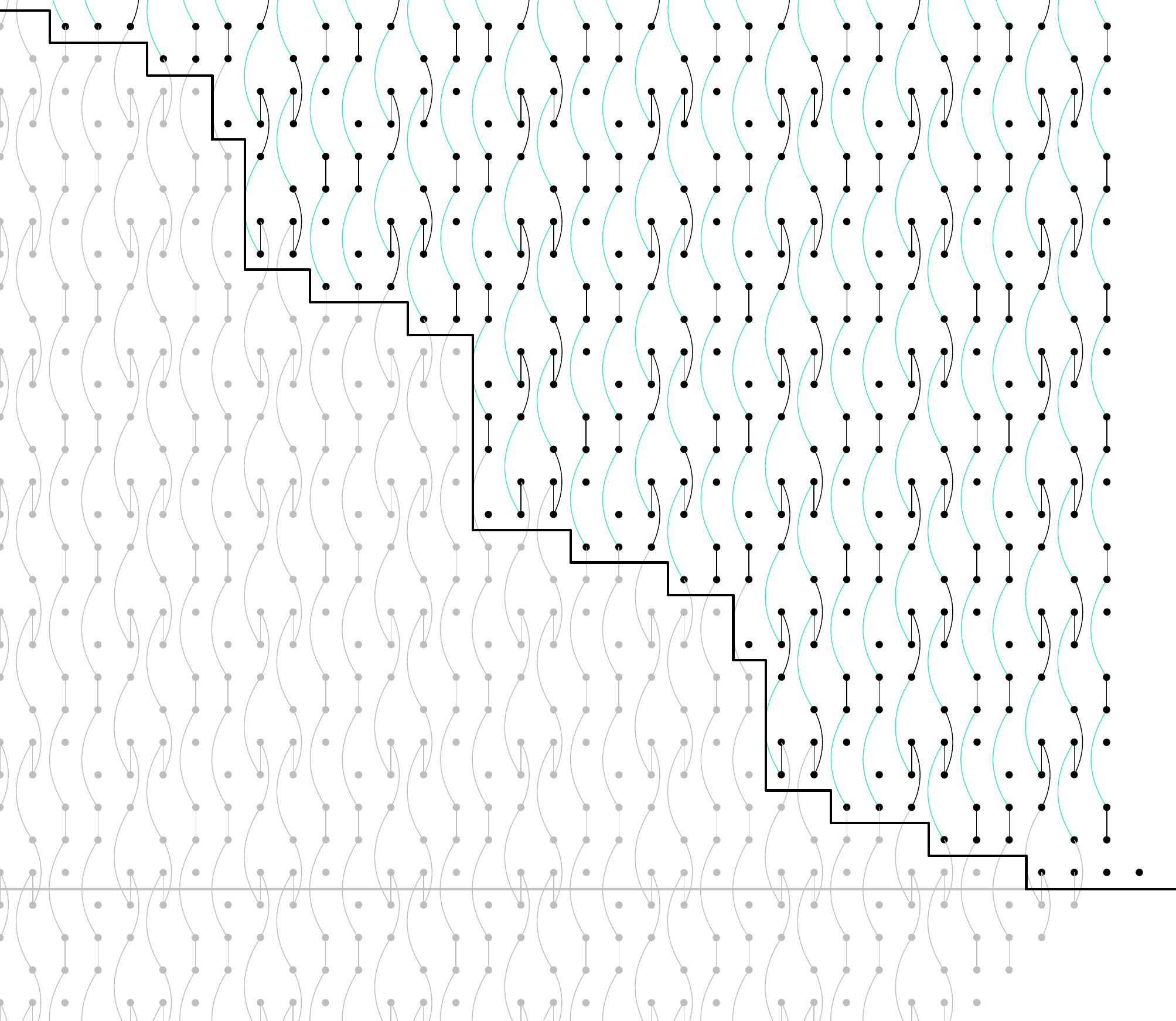}}
\end{center}
\begin{center}
\caption{The Mahowald line.}
\hfill
\label{fig:OutlineAccurateMahowaldLine}
\end{center}
\end{figure}

Motivated by Proposition~\ref{prop: equivariant to nonequivariant}, we make the following definition: 
\begin{df}\rm \label{df:MahowaldLine}
The function $\mathfrak{L}:\mathbb{N}\rightarrow \mathbb{N}$ is defined by setting $\mathfrak{L}(k)$ to be the largest integer such that the map 
$$c(k)^{\mathfrak{L}(k)}:X(k)^{\mathfrak{L}(k)}\longrightarrow S^{0}$$
is null-homotopic.
\end{df}

%\begin{rem}\label{rem:canChangeCellStructure}
%Note that Proposition~\ref{prop: equivariant to nonequivariant} and Definition~\ref{df:MahowaldLine} does not depend on the CW-decomposition of the homotopy type of $X(m)$.  More specifically, if X' is another CW spectrum with a homotopy equivalence $h: X' \to X(m)^k$.  Then Proposition~ and Definition~ are the same with respect to the map 
%X' \to X(m)^k \to S^0.  
%\end{rem}

\begin{df}\rm \label{df:MahowaldLine2}
The function $\mathfrak{L}(k)$ can be visualized by drawing a line over the $\mathfrak{L}(k)$-cell in the cell-diagram of $X(k)$.  When we connect these lines for all $k \geq 0$, the resulting ``staircase'' pattern is called the \textit{Mahowald line}.
\end{df}

In light of Proposition~\ref{prop: equivariant to nonequivariant}, our goal is to find the exact location of the Mahowald line.

\begin{thm}\label{thm:exactMahowaldLine}
The function $\mathfrak{L}(m)$ takes values as follows: 
\begin{eqnarray*}
\mathfrak{L}(0) = \mathfrak{L}(1) = \mathfrak{L}(2) &=& -1, \\
\mathfrak{L}(3) &=& 0,
\end{eqnarray*}
and for all $k \geq 1$, 
\begin{eqnarray*}
\mathfrak{L}(16k+4) &=& 16k, \\
\mathfrak{L}(16k+5) &=& 16k, 
\end{eqnarray*}
\begin{eqnarray*}
\mathfrak{L}(16k+6) &=& 16k+1,\\
\mathfrak{L}(16k+7) &=& 16k+1,\\
\mathfrak{L}(16k+8) &=& 16k+1,\\
\mathfrak{L}(16k+9) &=& 16k+2,\\
\mathfrak{L}(16k+10) &=& 16k+2,\\ 
\mathfrak{L}(16k+11) &=& 16k+6, \\
\mathfrak{L}(16k+12) &=& 16k+8, \\
\mathfrak{L}(16k+13) &=& 16k+8, \\
\mathfrak{L}(16k+14) &=& 16k+9, \\
\mathfrak{L}(16k+15) &=& 16k+9, \\
\mathfrak{L}(16k+16) &=& 16k+9, \\
\mathfrak{L}(16k+17) &=& 16k+10,\\
\mathfrak{L}(16k+18) &=& 16k+10, \\
\mathfrak{L}(16k+19) &=& 16k+10.
\end{eqnarray*}
\end{thm}

Theorem~\ref{thm:exactMahowaldLine} directly implies Theorem~\ref{main theorem}.  Our proof of Theorem~\ref{thm:exactMahowaldLine} consists of seven steps, each giving a new bound on $\mathfrak{L}(k)$ (see Figure~\ref{fig:OutlineStep0}):
\begin{enumerate}
    \item Step 1 proves a lower bound for $\mathfrak{L}(k)$.
    \item Step 2 proves an upper bound for $\mathfrak{L}(k)$.  This upper bound agrees with the lower bound in Step 1 except at $\mathfrak{L}(8k+3)$, $k \geq 1$.
    \item Steps 3--5 prove that $\mathfrak{L}(8k+3)\leq 8k-2$ for all $k \geq1$.
    \item Step 6 proves that $\mathfrak{L}(8k+3)\geq 8k-2$ when $k$ is odd; 
    \item  Step 7 proves that $\mathfrak{L}(8k+3)=8k-6$ when $k$ is even.
\end{enumerate}

\begin{figure}
\begin{center}
\makebox[\textwidth]{\includegraphics[trim={2cm 4.2cm 0.6cm 7.8cm}, clip, page = 1, scale = 0.7]{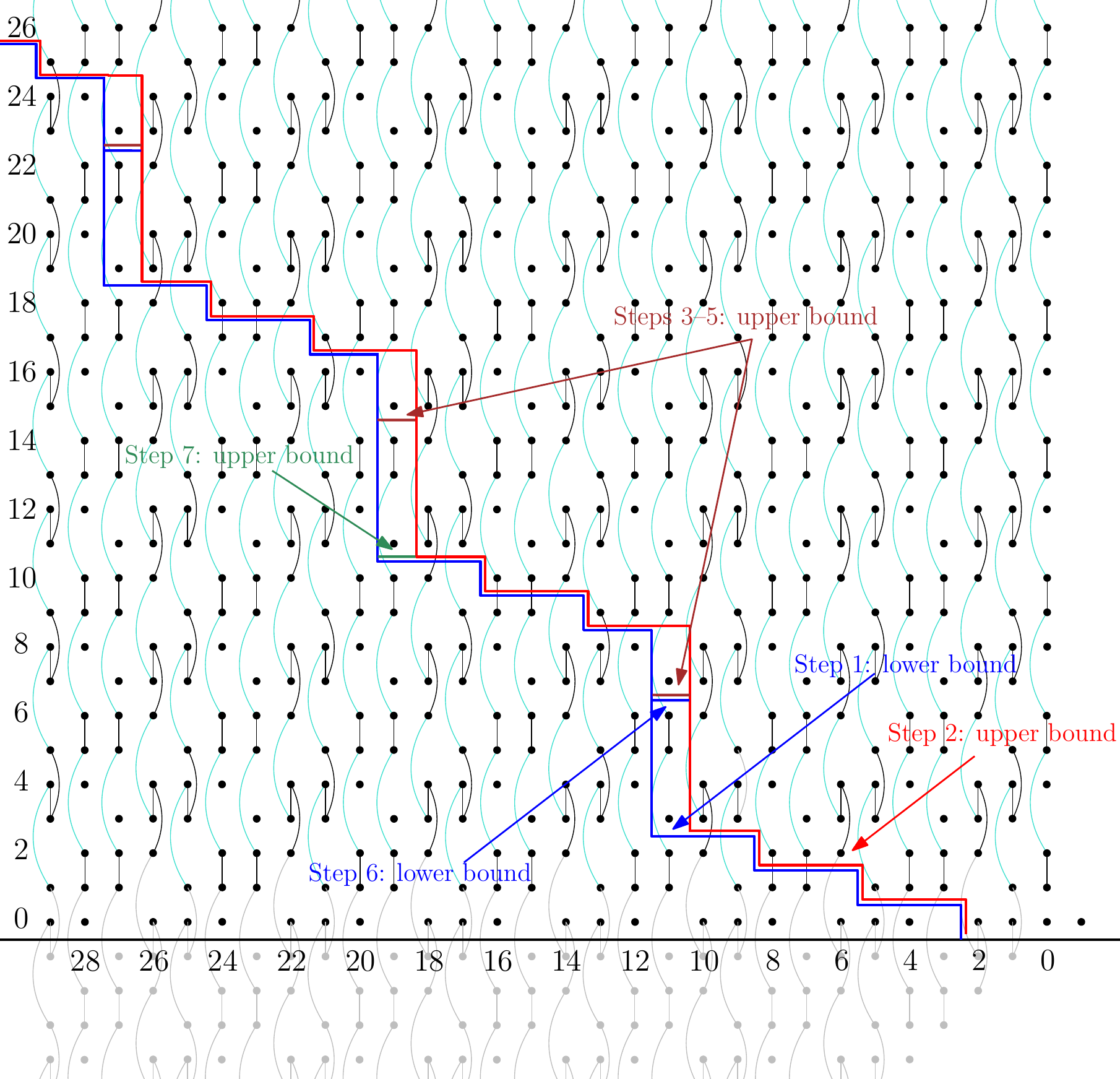}}
\caption{Various bounds for the Mahowald line.} \hfill
\label{fig:OutlineStep0}
\end{center}
\end{figure}

\begin{proof}[Proof of Proposition~\ref{prop: equivariant to nonequivariant}]
Consider the diagram
\begin{equation}\label{diagram Furuta--Mahowald class}
\begin{tikzcd}
S^{p\mathbb{H}} \ar[rd, dashed, "g"] &  \\ 
S^0 \ar[u,"\mathbf{1}"] \ar[r,"\mathbf{2}"] & S^{q\widetilde{\mathbb{R}}} \\ 
S(p \mathbb{H})_+ \ar[u,"\mathbf{3}"] \ar[ru, swap, "\mathbf{4}"]&
\end{tikzcd}
\end{equation}
In the diagram above, $\mathbf{1} = a_\mathbb{H}^p$ and $\mathbf{2}= a_{\widetilde{\mathbb{R}}}^q$.  The left column is the cofiber sequence
$$S(p \mathbb{H})_+ \longrightarrow S^0 \longrightarrow S^{p\mathbb{H}},$$
where $S(p \mathbb{H})$ is the unit sphere of the representation $p \mathbb{H}$.  By our discussion in Section~\ref{subsec:SeibergWitten}, a level-$(p,q)$ Furuta--Mahowald class exists if and only if there exists a map $g$ that makes diagram (\ref{diagram Furuta--Mahowald class}) commute.

Since the first column is a cofiber sequence, $g$ exists if and only if the composition $\mathbf{4}=\mathbf{2}\circ \mathbf{3}$ is null-homotopic.  The Spanier--Whitehead dual of map $\mathbf{2}$ is the map 
$$D\mathbf{2}:S^{-q\widetilde{\mathbb{R}}}\longrightarrow S^{0}.$$
Map $\mathbf{4}$ is null-homotopic if and only if the map 
$$\mathbf{5}:=D\mathbf{2}\wedge \mathbf{3}: S^{-q\widetilde{\mathbb{R}}} \wedge S(p\mathbb{H})_+ \longrightarrow S^0$$ 
is null-homotopic. 

Map $\mathbf{5}$ can be written as the composition 
$$\mathbf{5}: S^{-q\widetilde{\mathbb{R}}}\wedge S(p \mathbb{H})_+\xrightarrow{D\mathbf{2}\wedge \operatorname{id}_{S(p \mathbb{H})_+ }}  S(p \mathbb{H})_+\stackrel{\mathbf{3}}{\longrightarrow} S^{0}.$$ 
Note that $S^{-q\widetilde{\mathbb{R}}}\wedge S(p \mathbb{H})_+$ is $\Pin$-free for all $q \geq 0$ and $\Pin$ acts trivially on $S^{0}$.  Therefore, $\mathbf{5}$ is null-homotopic if and only if the nonequivariant map 
$$(S^{-q\widetilde{\mathbb{R}}}\wedge S(p \mathbb{H})_+)_{h\Pin}\stackrel{\mathbf{7}}{\longrightarrow}  (S(p \mathbb{H})_+)_{h\Pin} \stackrel{\mathbf{8}}{\longrightarrow} S^{0}$$
is null-homotopic (see Theorem~4.5 in \cite{LewisMaySteinberger}).  Here, 
$${(-)_{h \Pin} = ((-) \wedge E\Pin_+) / \Pin}$$
is the homotopy orbit.  The maps $\mathbf{7}$ and $\mathbf{8}$ are induced by $D\mathbf{2}\wedge \operatorname{id}_{S(p \mathbb{H})_+ }$ and $\mathbf{3}$, respectively. 

Note that the restriction of the fiber bundle (\ref{BpinoverBSU}) to $\mathbb{H}P^{p-1}$ gives the bundle 
$$\begin{tikzcd}
\mathbb{R}\textup{P}^{2} \ar[r,hookrightarrow]&  S(p\mathbb{H})/\Pin \ar[r]& \mathbb{H}P^{p-1}.
\end{tikzcd}
$$	
Therefore, the inclusion 
$$
\begin{tikzcd}
 S(p\mathbb{H})/\Pin \ar[r, hookrightarrow]&   S(\infty\mathbb{H})/\Pin=B\Pin
 \end{tikzcd}
$$
is the inclusion of the $(4p-2)$-skeleton. This implies that 
$$(S^{-q\widetilde{\mathbb{R}}}\wedge S(p \mathbb{H})_+)_{h\Pin}=\Thom(B\Pin^{4p-2},-q\lambda)=X(m)^{4p-2-q}.$$
%The short exact sequence 
%$$1\longrightarrow S^{1}\longrightarrow \Pin\longrightarrow C_{2}\longrightarrow 1$$
%implies that there is an equality 
%\begin{eqnarray*}
%(S^{-q\widetilde{\mathbb{R}}}\wedge S(p \mathbb{H})_+)_{h\Pin}&=&\left((S^{-q\widetilde{\mathbb{R}}}\wedge S(p \mathbb{H})_+)_{hS^{1}}\right)_{hC_{2}}\\
%&=&\left(S^{-q\sigma}\wedge \Sigma^{\infty}\mathbb{C}P^{2p-1}_{+}\right)_{hC_{2}}.
%\end{eqnarray*}
%To identify $\left(S^{-q\sigma}\wedge \Sigma^{\infty}\mathbb{C}P^{2p-1}_{+}\right)_{hC_{2}}$, note the following two facts:
%\begin{enumerate}
%\item The cellular structure on $\textup{BPin}(2)$ induces a cellular structure on $BS^{1}$ where the $(4p-2)$-skeleton of $BS^{1}$ is $\mathbb{C}P^{2p-1}$.
%\item The $\Pin$-action on $S(p\mathbb{H})$ induces a $C_{2}$-action on the quotient space $S(p\mathbb{H})/S^{1}=\mathbb{C}P^{2p-1}$.  This is exactly the action (\ref{C2 action}), restricted to $\mathbb{C}P^{2p-1} \subset \mathbb{C}P^\infty$.
%\end{enumerate}
%It follows from these two facts that 
%$$(S^{-q\sigma}\wedge \Sigma^{\infty}\mathbb{C}P^{2p-1}_{+})_{hC_{2}} = X(q)^{4p-2-q}.$$ 
Under this identification, maps $\mathbf{7}$ and $\mathbf{8}$ are equal to $i(m,0)$ and $c(0)$ respectively.  The map $c(q)^{4p-2-q}$ is exactly the composition map $\mathbf{8} \circ \mathbf{7}$, which is null-homotopic if and only if a level-$(p,q)$ Furuta--Mahowald class exists. 
\end{proof}
%%%%%%%%%%%%%%%%%%%

\subsection{The Mahowald line at odd primes}
For each prime $p$, we can localize the map $c(m)^k: X(m)^k \to S^0$ at $p$ to obtain a map 
$$c(m)^k_{(p)}: X(m)^k_{(p)} \longrightarrow S^0_{(p)}.$$
Similar to Definition~\ref{df:MahowaldLine}, we define the function $\mathfrak{L}_{(p)}: \mathbb{N} \to \mathbb{N}$ as follows: $\mathfrak{L}_{(p)}(k)$ is the largest integer such that the map 
$$c(k)^{\mathcal{L}_{(p)}(k)}: X(k)^{\mathcal{L}_{(p)}(k)}_{(p)} \longrightarrow S^0_{(p)}$$
null-homotopic.  It is clear from this definition that for all $k \in \mathbb{N}$, 
$$\mathcal{L}(k) = \min_{p \text{ prime}} \mathcal{L}_{(p)}(k)$$
The line determined by the function $\mathfrak{L}_{(p)}$ called the \textit{$p$-local Mahowald line}.

We show that, at any odd prime $p$, the $p$-local Mahowald line is above the 2-local Mahowald line (see Figures~\ref{fig:OutlineAccurateMahowaldLine} and ~\ref{fig:OddPrimeMahowaldLine}). This will reduce our problem to a 2-primary problem.  After this subsection, we will focus on the case when we localize at the prime $p=2$ for the rest of the paper. 

Recall the fiber bundle
$$\begin{tikzcd}
\mathbb{R}\textup{P}^2 \ar[r,hookrightarrow]& \textup{BPin}(2) \ar[r]& \mathbb{H}\textup{P}^\infty.
\end{tikzcd}$$
As discussed in Section~\ref{subsec:equivToNonEquiv}, the cell structure for $\mathbb{R}\textup{P}^2$ and $\mathbb{H}\textup{P}^\infty$ induce a cell structure for $\textup{BPin}(2)$.  

The standard cell structures for $\mathbb{R}\textup{P}^2$ has one cell in dimensions $0$, $1$, and $2$.  The 2-cell is attached to the 1-cell by $2$, which is invertible when localized at $p$.  Therefore, 
$$H_* (\mathbb{R}\textup{P}^2; \mathbb{Z}_{(p)}) = \left\{\begin{array}{ll} \mathbb{Z}_{(p)} & \text{when } *= 0, \\
0 & \text{otherwise.}\end{array} \right.$$
This implies that when we localize at $p$, there is a cellular structure for $\mathbb{R}\textup{P}^2$ with only one cell in dimension 0, and no cells in other dimensions.  Since the cell structure for $\mathbb{H}\textup{P}^\infty$ has one cell in dimension $4n$ for all $n \geq 0$, the induced cell structure for $\textup{BPin}(2)$ from the fiber bundle above also has one cell in dimension $4n$ for all $n \geq 0$.  

The bundle $2\lambda$ is orientable because its first Stiefel--Whitney class is 0.  There is a Thom-isomorphism 
$$H^{*}(X(2m); \mathbb{Z}_{(p)}) = H^*(\Thom(\textup{BPin}(2), -2m\lambda); \mathbb{Z}_{(p)}) \cong H^{*+2m}(\textup{BPin}(2);\mathbb{Z}_{(p)}).$$
This Thom-isomorphism implies that
$$H_*(X(2m);\mathbb{Z}_{(p)}) = \left\{\begin{array}{ll}\mathbb{Z}_{(p)} & \text{when } * = -2m+4n, \, n \geq0, \\ 
0 & \text{otherwise.} \end{array}\right.$$
It follows that there is a cell structure for $X(2m)_{(p)}$ with one cell in dimension ${(-2m +4n)}$ for all $n \geq 0$.  Note that by the cellular approximation theorem, Proposition~\ref{prop: equivariant to nonequivariant} and Definition~\ref{df:MahowaldLine} do not depend on the cellular structure of $X(m)_{(p)}$.  Therefore, we can use this specific cell structure to deduce a lower bound for the $p$-local Mahowald line (see Figure~\ref{fig:OddPrimeMahowaldLine}).  This lower bound is above the 2-local Mahowald line (shown in {\color{gray} gray}).

\begin{figure}[!h]
\begin{center}
\makebox[\textwidth]{\includegraphics[trim={0.8cm 3.5cm 0.3cm 7.8cm}, clip, page = 1, scale = 0.7]{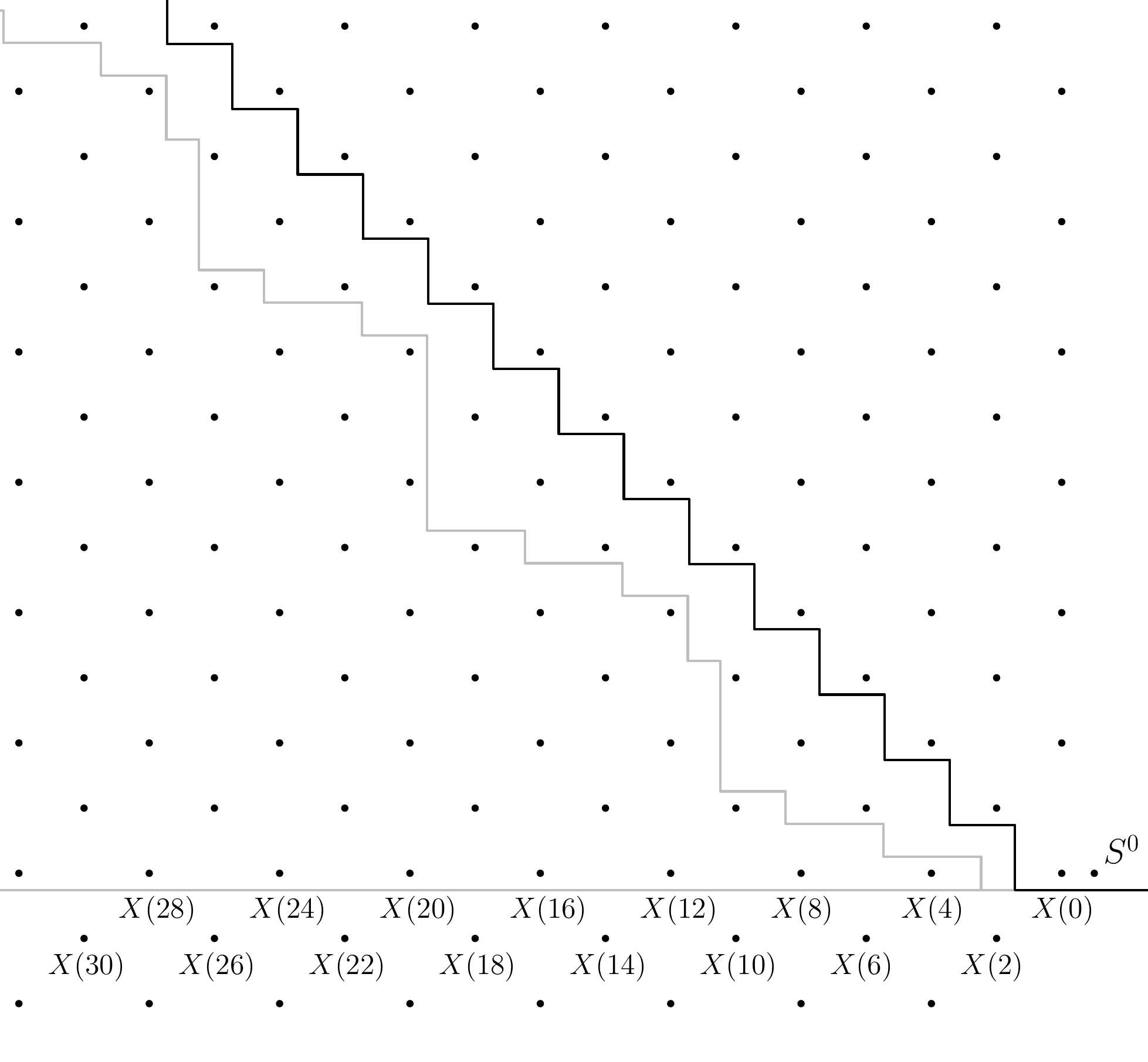}}
\end{center}
\begin{center}
\caption{The lower bound of the $p$-local Mahowald line at $p>2$ (black) is above the 2-local Mahowald line ({\color{gray} gray}).}
\hfill
\label{fig:OddPrimeMahowaldLine}
\end{center}
\end{figure}

%%%%%%%%%%%%%%%%%
\subsection{Step 1: lower bound} \label{subsec:Step1}
From now on, we localize at the prime $p =2$.  In the discussions below, the arrow $\hookrightarrow$ denotes a map that induces an injection on $\textup{H}\mathbb{F}_2$-homology, and the arrow $\twoheadrightarrow$ denotes a map that induces a surjection on $\textup{H}\mathbb{F}_2$-homology (see Defintion~\ref{hf2}). 

\begin{thm}\label{thm: inductive fk}
For every $k\geq 0$, there exist maps
\begin{itemize}
\item $f_{k}: X(8k+4)_{8k+1}^{\infty} \longrightarrow S^{0}$
\item $g_{k}:S^{8k+4} \lhook\joinrel\longrightarrow X(8k+4)_{8k+1}^{\infty}$
\item $a_{k}: S^{8k+4} \longrightarrow X(8k-4)^{8k-4}_{8k-7}$
\item $ b_{k}:X(8k-4)^{8k-4}_{8k-7}\longrightarrow S^{0}$
\end{itemize}
with the following properties (see Figure~\ref{fig:OutlineStep1}):
\begin{enumerate}[label=(\roman*)]
\item The diagram 
\begin{equation}\label{f is quotient}
    \xymatrix{
        X(8k+4) \ar[r]\ar@{->>}[d] & S^{0}  \\
        X(8k+4)_{8k+1}^{\infty} \ar[ur]_{f_{k}} }
\end{equation}
commutes.

\item The map $g_{k}$ induces an isomorphism on $H_{8k+4}(-;\mathbb{F}_{2})$.  In other words, $S^{8k+4}$ is a $\textup{H}\mathbb{F}_2$-subcomplex of $X(8k+4)_{8k+1}^{\infty}$ via the map $g_k$ (see Definition~\ref{hf2}).

\item The following diagram is commutative:
\begin{equation}\label{gf factors throught W}
    \xymatrix{
        S^{8k+4} \ar[d]^{a_{k}} \ar@{^{(}->}[r]^-{g_{k}} & X(8k+4)_{8k+1}^{\infty} \ar[d]^{f_{k}}\\
    X(8k-4)^{8k-4}_{8k-7}\ar[r]^-{b_{k}} & S^{0}.}
\end{equation}
\item  Let $\phi_{k}:S^{8k+1}\rightarrow S^{0}$ be the restriction of $f_{k}$ to the bottom cell of ${X(8k+4)^{\infty}_{8k+1}}$.  Then for $k\geq 1$, the map $\phi_{k}$ satisfies the inductive relation 
$$\phi_{k}- \phi_{k-2} \cdot \chi_{k}\in \langle \phi_{k-1}, 2, \tau_{k} \rangle,$$
where $\tau_{k}\in \{0,8\sigma\}$ in $\pi_7$ and $\chi_{k}$ is some element in $\pi_{16}$. We will show in Lemma~\ref{eta attaching map between columns} that $\phi_0 = \eta$ and we set $\phi_{-1}=0$.
\end{enumerate}
\end{thm}

\begin{figure}
\begin{center}
\makebox[\textwidth]{\includegraphics[trim={2.5cm 4.2cm 0.4cm 6.7cm}, clip, page = 1, scale = 0.7]{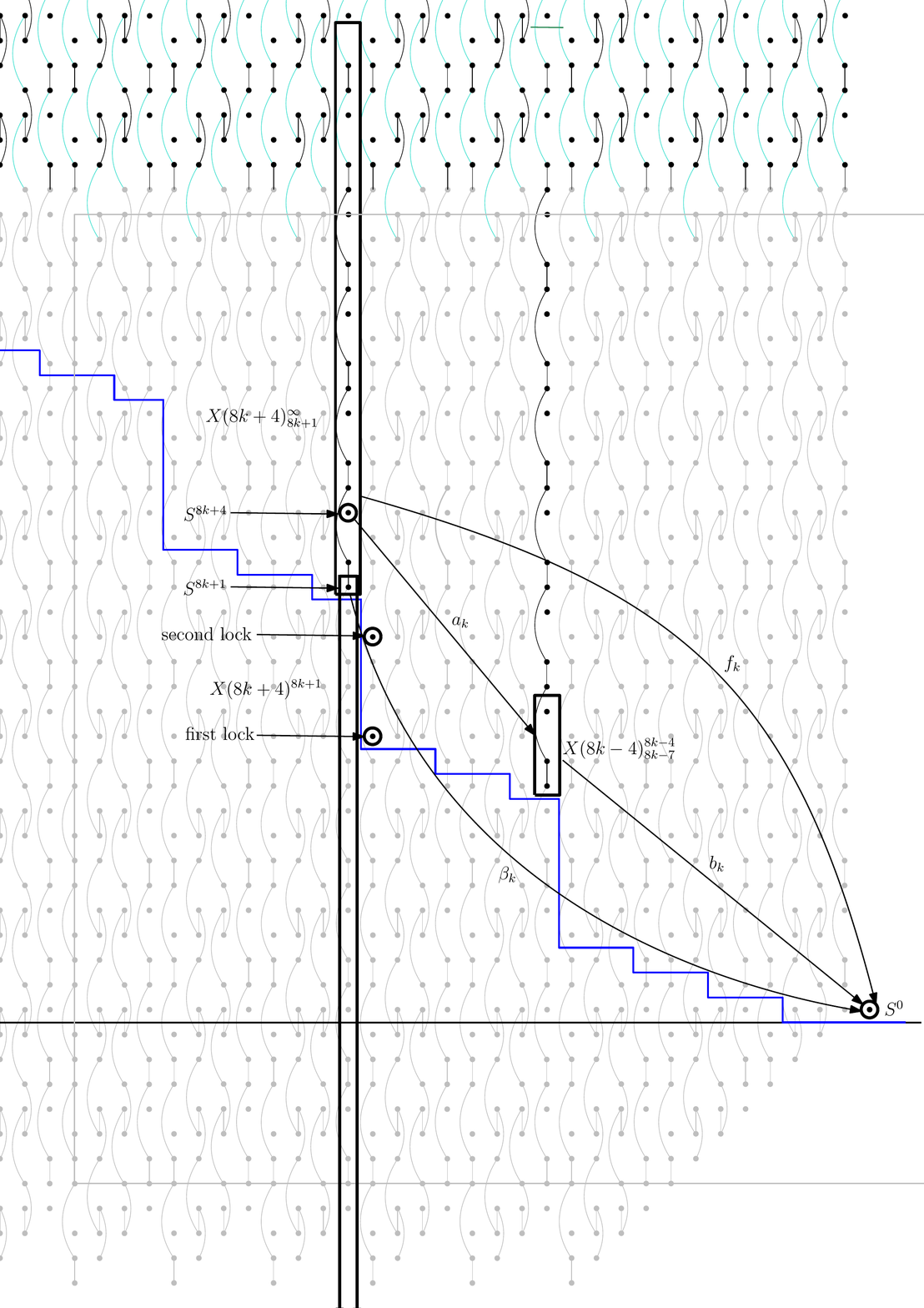}}
\end{center}
\begin{center}
\caption{Constructing $f_k$ and proving a lower bound for the Mahowald line.}
\hfill
\label{fig:OutlineStep1}
\end{center}
\end{figure}

We prove Theorem~\ref{thm: inductive fk} by using cell diagram chasing arguments. 

\begin{rem}\rm
Property \textit{(i)} immediately implies that the map 
$$c(8k+4)^{8k}: X(8k+4)^{8k} \longrightarrow S^0$$ 
is null homotopic, and therefore it is the main property that we desire for $f_k$.  Properties \textit{(ii)} and \textit{(iii)} are added so that we can construct $f_{k}$ inductively from $f_{k-1}$.  Property (iv) is an additional requirement on $f_{k}$ that will be useful in the Step 3.
\end{rem}

\begin{cor}\label{corollary: lower bound} For any $k\geq 0$ and $0\leq m\leq 7$, we have the inequality 
$$\mathfrak{L}(8k+m+4)\geq 8k+\tau(m),$$
where
\begin{equation*}
\tau(m) = \begin{cases}
0 & m=0,1\\
1 & m=2,3,4\\
2 & m=5,6,7.\\
\end{cases}
\end{equation*}
This line is shown in {\color{blue} blue} in Figure~\ref{fig:OutlineStep1}.
\end{cor}

\begin{proof} 
When $m=0$, the claim directly follows from diagram (\ref{f is quotient}).  When ${1 \leq m \leq 7}$, the claim follows from the case when $m=0$ and the following commutative diagram:
\begin{displaymath}
  \hspace{-0.5in} \xymatrix{X(8k+4+m)\ar[r]& X(8k+4+(m-1))\ar[r]& \cdots\ar[r] & X(8k+4) \ar[r]^-{c(8k+4)} & S^{0}\\
    X(8k+4+m)^{8k+\tau(m)}\ar[r]\ar@{^{(}->}[u] & X(8k+4+(m-1))^{8k+\tau(m-1)}\ar@{^{(}->}[u] \ar[r]& \cdots\ar[r]  &X(8k+4)^{8k+4} \ar@{^{(}->}[u] }
\end{displaymath}
\end{proof}

%%%%%%%%%
\subsection{Step 2: upper bound detected by $KO$}\label{subsec:Section2Step2}
Using $\Pin$-equivariant $KO$ theory, we prove the following proposition:
\begin{prop}\label{prop: upper bound detected by ko} For any $k\geq 1$, the composition 
$$X(8k+2)^{8k-4}\xrightarrow{c(8k+2)^{8k-4}}S^{0}\longrightarrow KO$$
is nonzero. 
\end{prop}

Proposition~\ref{prop: upper bound detected by ko} has the following corollary:
\begin{cor}\label{cor: upper bound in (8k+2)}
The map $c(8k+2)^{8k-5}:X(8k+2)^{8k-5}\longrightarrow S^{0}$ is nontrivial.
\end{cor}
\begin{proof}
For the sake of contradiction, suppose that the map $c(8k+2)^{8k-5}$ is trivial.  Then the map 
$$c(8k+2)^{8k-4}: X(8k+2)^{8k-4}\longrightarrow S^{0}$$
will factor through the quotient map ${X(8k+2)^{8k-4}\twoheadrightarrow S^{8k-4}}$ via some map \\ ${f:S^{8k-4}\rightarrow S^{0}}$.  Since no element in $\pi_{8k-4}S^0$ is detected by $KO$, the composition 
$$\begin{tikzcd}
X(8k+2)^{8k-4} \ar[r, twoheadrightarrow]&  S^{8k-4} \ar[r, "f"] & S^{0} \ar[r]& KO
\end{tikzcd}
$$
is trivial.  This is a contradiction to Proposition~\ref{prop: upper bound detected by ko}.  
\end{proof}

\begin{cor}\label{corollary: exact value except 8k+3} The equality 
$$\mathfrak{L}(8k+m+4)= 8k+\tau(m)$$
holds for all $k\geq 0$ and $0\leq m\leq 6$.  Here, $\tau(m)$ is defined as in Corollary~\ref{corollary: lower bound}.
\end{cor}
\begin{proof}
Corollary~\ref{cor: upper bound in (8k+2)} implies that 
$$\mathfrak{L}(8k+6+4)\leq 8k+\tau(6).$$ 
This directly implies that 
$$\mathfrak{L}(8k+m+4)\leq 8k+\tau(m)$$
for all $0\leq m\leq 6$.  The claim follows by combining this inequality with the inequality in Corollary~\ref{corollary: lower bound}.
\end{proof}

%%%%%%%%%
\subsection{Step 3: identifying the map on the first lock as $\{P^{k-1}h_1^{3}\}$} \label{subsec:Section2Step3}
After establishing the lower bound for $\mathfrak{L}(k)$, the $(8k-5)$-cell and the $(8k-1)$-cell in $X(8k+3)$ will play significant roles for the rest of our argument.  We call them the \textbf{``first lock''} and the \textbf{``second lock''}, respectively (see Figure~\ref{fig:OutlineStep1}). 

In this step, we will focus on the first lock.  Combining Theorem~\ref{thm: inductive fk} (iv) with an inductive Toda bracket computation, we prove the following proposition, which will essential in the proof of Proposition \ref{pro: order two between 2 locks} and Proposition \ref{prop:Step6CommDiagram}.
\begin{prop}\label{pro: beta-eta-square}
For all $k, m\geq 0$, we have the relations
$$\phi_{k}\cdot \{P^{m}h_{1}^{2}\}=\{P^{m+k}h_{1}^3\}.$$
\end{prop}

The following corollary is a consequence of Proposition~\ref{pro: beta-eta-square} and Theorem~\ref{thm: inductive fk}~(i):

\begin{cor}\label{cor: ph13 is Mohowald invariant}
For all $k\geq 0$, the diagram
\begin{equation}\label{eq: ph13 is Mohowald invariant}
\begin{tikzcd}
X(8k+3)^{8k-5} \ar[r, twoheadrightarrow] \ar[rd, swap, "c(8k+3)^{8k-5}"] & S^{8k-5}\ar[d, "\{P^{k-1}h_1^{3}\}"]  \\
& S^{0} 
\end{tikzcd}
\end{equation}
commutes.
\end{cor}
Corollary~\ref{cor: ph13 is Mohowald invariant} identifies the map on the first lock as $\{P^{k-1}h_1^{3}\}$.

\subsection{Step 4: A technical lemma for the upper bound} \label{subsec:Section2Step4}
To prove an upper bound for $\mathfrak{L}(k)$, we make use of the spectrum $j''$, which is defined as the fiber of the map
$$ko\xrightarrow{\psi^{3}-1} ko\langle 2\rangle.$$
Here, $ko\langle 2\rangle$ is the 1-connected cover of $ko$.  The following proposition is proved by analyzing the interactions between $j''$ and the spectrum $ko_{\mathbb{Q}/\mathbb{Z}}$. 
    
\begin{prop}\label{pro: only -1 cell matters}
For any $k,m\geq 0$, the map
\begin{equation}\label{eq outline: j'' injective without -1 cell}
j''^{0}(S^{4m+3})\longrightarrow j''^{0}(X(8k+3)_{0}^{4m+3})
\end{equation}
induced by the quotient map $X(8k+3)_{0}^{4m+3}\twoheadrightarrow S^{4m+3}$ is injective.
\end{prop}

\begin{terminology} \label{AHSS}\rm
Let $X$ be a CW spectrum that has at most one cell in each dimension.  Recall that the cohomological $E$-based Atiyah--Hirzebruch spectral sequence for $X$ has the following form: 
$$E_1^{s,t} = \bigoplus_{s \in I } \pi_t E[s] \Longrightarrow E^{s-t}X.$$
Here, $I$ is the indexing set containing the dimensions of the cells of $X$, $s$ is the cellular filtration of $X$.  The degrees for the $d_r$-differentials are as follows:
$$d_r: E_r^{s,t} \longrightarrow E_r^{s+r, t+r-1}. $$

Similarly, the homological $E$-based Atiyah--Hirzebruch spectral sequence for $X$ has the following form: 
$$E_1^{s,t} = \bigoplus_{s \in I } \pi_t E[s] \Longrightarrow E_{s+t}X.$$
Here, $I$ is the indexing set containing the dimensions of the cells of $X$, $s$ is the cellular filtration of $X$.  The degrees for the $d_r$-differentials are as follows:
$$d_r: E_r^{s,t} \longrightarrow E_r^{s-r, t-r+1}. $$
\end{terminology}

Proposition~\ref{pro: only -1 cell matters} can be interpreted as follows: in the $j''$-based cohomological Atiyah--Hirzebruch spectral sequence of $X(8k+3)_{0}^{4m+3}$, any nonzero class of the form 
$$a[4m+3], \,\,\, a\in \pi_{4m+3}j''$$
survives.  Using this, we can further show that in the $j''$-based Atiyah--Hirzebruch spectral sequence of $X(8k+3)^{4m+3}$, a nonzero class 
$$a[4m+3]$$
with $a\in \pi_{4m+3}j''$ can only be killed by a differential of the form 
$$b[-1]\longrightarrow a[4m+3],$$
where $b\in \pi_0 j''=\mathbb{Z}_{(2)}$.  Note that $\pi_{m}j''=0$ for $m\leq -1$, so this implies that a cell of dimension $\leq -2$ cannot support a differential with target $a[4m+3]$.

\subsection{Step 5: the second lock is not passed}

\begin{prop}\label{pro: order two between 2 locks}
There exists a map 
$$t_{k}: X(8k+3)^{8k-1}_{8k-5}\longrightarrow S^{0}$$
 with the following properties (see Figure~\ref{fig:OutlineStep4}):
\begin{enumerate}[label=(\roman*)]
\item The map 
$$c(8k+3)^{8k-1}: X(8k+3)^{8k-1}\longrightarrow S^{0}$$
factors through the quotient map 
$$\begin{tikzcd}
X(8k+3)^{8k-1} \ar[r, twoheadrightarrow]& X(8k+3)^{8k-1}_{8k-5}
\end{tikzcd}$$
via $t_{k}$:
\begin{equation}\label{diag: tau is quotient}
\begin{tikzcd}
X(8k+3)^{8k-1}\ar[rr,"c(8k+3)^{8k-1}"] \ar[d, twoheadrightarrow] && S^{0} \\
X(8k+3)^{8k-1}_{8k-5} \ar[rru, swap, "t_k"] 
\end{tikzcd}
\end{equation}
\item The map $t_{k}$ factors through a quotient map 
$$\begin{tikzcd}
 X(8k+3)^{8k-1}_{8k-5} \ar[r, twoheadrightarrow]& \Sigma^{8k-5}C\nu
 \end{tikzcd}
$$
via a map 
$$
t'_{k}:  \Sigma^{8k-5}C\nu\longrightarrow S^{0}.
$$

\item The restriction of $t'_{k}$ to its bottom cell is the map
$$\{P^{k-1}h_1^{3}\}:S^{8k-5}\longrightarrow S^{0}.$$ 

\item The map $t_{k}$ has order 2 in $j''$.  In other words, the following composition is zero: 
$$\Sigma^{8k-5}C\nu\stackrel{2t'_{k}}{\longrightarrow}S^{0}\longrightarrow j''.$$
\end{enumerate}
\end{prop}

\begin{figure}
\begin{center}
\makebox[\textwidth]{\includegraphics[trim={2.7cm 4.2cm 0.6cm 13cm}, clip, page = 1, scale = 0.7]{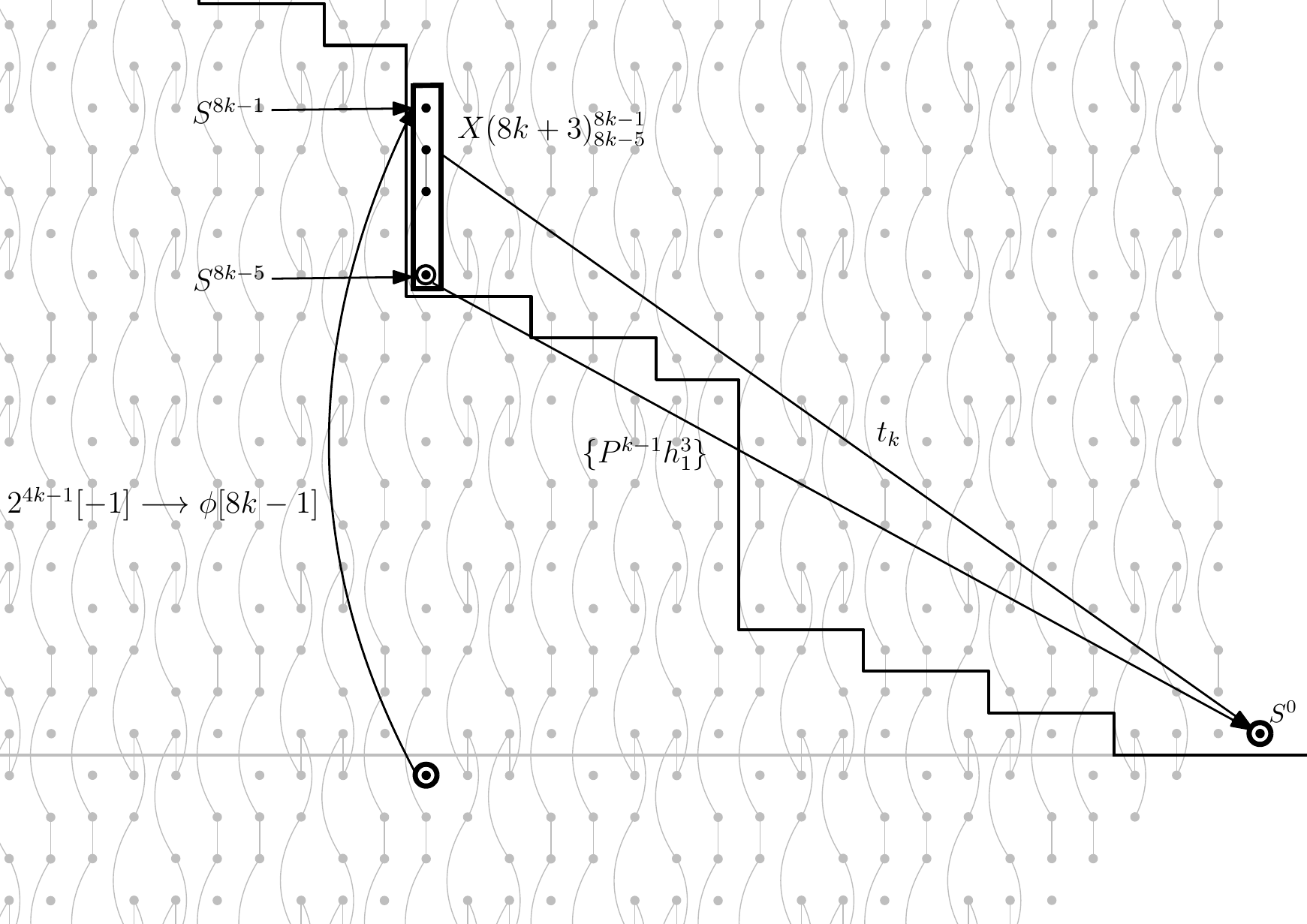}}
\end{center}
\begin{center}
\caption{Proposition~\ref{pro: order two between 2 locks}.}
\hfill
\label{fig:OutlineStep4}
\end{center}
\end{figure}

Properties (i) and (iii) in Proposition~\ref{pro: order two between 2 locks} are direct consequences of diagram (\ref{eq: ph13 is Mohowald invariant}).  Property (ii) and (iv) is established by a local cell diagram chasing argument.

\begin{lem}\label{lem: differential to second lock}
In the $j''$-based Atiyah--Hirzebruch spectral sequence of ${X(8k+3)^{8k-1}}$, there is a differential
\begin{equation}\label{eq: differential to second lock}
2^{4k-1}[-1]\longrightarrow \phi [8k-1],
\end{equation}
where $\phi$ is a nonzero element in $\pi_{8k-1}j''$.
\end{lem}
To prove Lemma~\ref{lem: differential to second lock}, we first construct a map 
$$X(8k+3)_{-1}^{8k-1}\longrightarrow \Sigma^{-8k-3}\mathbb{C}P^{8k+1}_{4k+1}$$ 
that is of degree one on both the top and the bottom cell.  Then, we prove a differential in $\Sigma^{-8k-3}\mathbb{C}P^{8k+1}_{4k+1}$ by computing certain $e$-invariants using the Chern character.  Pulling back this differential to $X(8k+3)_{-1}^{8k-1}$ proves the desired differential.  

\begin{thm}\label{thm: second lock not passed}
The composition map
$$f: X(8k+3)^{8k-1}\xrightarrow{c(8k+3)^{8k-1}}S^{0}\longrightarrow j''$$
is not zero.
\end{thm}
\begin{proof}
For the sake of contradiction, suppose that $f$ is zero.  Consider the composition
$$\begin{tikzcd}
g: X(8k+3)_{-1}^{8k-1} \ar[r, twoheadrightarrow] & X(8k+3)_{8k-5}^{8k-1} \ar[r, "t_k"]& S^0 \ar[r] & j''.
\end{tikzcd} $$
By Proposition~\ref{pro: order two between 2 locks}(i), the map $f$ is the composition in the top row of the following diagram: 
$$\begin{tikzcd}X(8k+3)^{8k-1} \ar[r,twoheadrightarrow] & X(8k+3)_{-1}^{8k-1} \ar[d] \ar[r,"g"] & j''\\
& \Sigma X(8k+3)^{-2} \ar[ru, dashed]. \end{tikzcd} $$
Since the sequence 
$$\begin{tikzcd} X(8k+3)^{8k-1} \ar[r,twoheadrightarrow] & X(8k+3)_{-1}^{8k-1} \ar[r] & \Sigma X(8k+3)^{-2} \end{tikzcd}$$
is a cofiber sequence and $[\Sigma X(8k+3)^{-2}, j''] = 0$ ($j''$ has no negative homotopy groups), the map $g$ is zero.  

Let $\beta \in j''^0(X(8k+3)_0^{8k-1})$ be the pullback of $1 \in j''^0(S^0) = \mathbb{Z}$ under the composition 
$$\begin{tikzcd}X(8k+3)_{0}^{8k-1} \ar[r,twoheadrightarrow]& X(8k+3)_{8k-5}^{8k-1} \ar[r, "t_k"]& S^0. \end{tikzcd} $$
Let $\alpha \in j''^0(X(8k+3)_0^{8k-5})$ be the pullback of $\beta$ under the inclusion 
$$\begin{tikzcd}X(8k+3)_0^{8k-5} \ar[r,hookrightarrow] & X(8k+3)_{0}^{8k-1}. \end{tikzcd}$$
Then the following three facts hold: 
\begin{enumerate}[label=(\roman*)]
\item $2\beta = 0$. 
\item $\beta$ pulls back to $0 \in j''^0(X(8k+3)_{-1}^{8k-1})$ under the map 
$$\begin{tikzcd}X(8k+3)_{-1}^{8k-1} \ar[r,twoheadrightarrow]& X(8k+3)_{0}^{8k-1}. \end{tikzcd}$$
\item $\alpha \neq 0$. 
\end{enumerate}
Fact (i) is true by Proposition~\ref{pro: order two between 2 locks}(iv).  Fact (ii) is true because the map $g$ is zero.  To see that fact (iii) is true, note that by Proposition~\ref{pro: order two between 2 locks}(iii), $\alpha$ can be represented as the map 
$$\begin{tikzcd} X(8k+3)_0^{8k-5} \ar[r,twoheadrightarrow] &S^{8k-5} \ar[rr, "\{P^{k-1}h_1^3 \}"] &&S^0 \ar[r] &j'' \end{tikzcd}$$
Since $\{P^{k-1}h_1^3\}$ is detected by $j''$, the composition 
$$\begin{tikzcd} S^{8k-5} \ar[rr, "\{P^{k-1}h_1^3 \}"] &&S^0 \ar[r] &j'' \end{tikzcd}$$
is nonzero.  Proposition~\ref{pro: only -1 cell matters} then implies that $\alpha \neq 0$.

Consider the following commutative diagram, where the rows are induced from cofiber sequences: 
$$\begin{tikzcd}
a \ar[r,mapsto] & \beta \ar[r,mapsto] & 0 \\
j''^0(S^0)  \ar[d, equal] \ar[r, "\partial"] & j''^0(X(8k+3)_0^{8k-1}) \ar[r] \ar[d] & j''^0(X(8k+3)_{-1}^{8k-1}) \ar[d] \\ 
j''^0(S^0)\ar[r, "\partial'"] & j''^0(X(8k+3)_0^{8k-5}) \ar[r] & j''^0(X(8k+3)_{-1}^{8k-5}). \\
a \ar[r,mapsto] & \alpha \neq 0
\end{tikzcd}$$
By fact (ii), $\beta = \partial(a)$ for some $a \in j''^0(S^0) = \mathbb{Z}_{(2)}$.  By the definition of $\alpha$ and fact (iii), $\partial'(a) = \alpha \neq 0$. 

By Lemma~\ref{lem: differential to second lock}, $\partial (2^{4k-1}) = \gamma$, where $\gamma\in j''^0(X(8k+3)_0^{8k-1})$ is the pullback of a nonzero element $\phi \in j''^0(S^{8k-1})$ under the map 
$$\begin{tikzcd}X(8k+3)_{0}^{8k-1} \ar[r,twoheadrightarrow] & S^{8k-1}. \end{tikzcd}$$
Since $\gamma$ pulls pack to $0 \in j''^0(8k+3)_{0}^{8k-5}$, $\partial'(2^{4k-1}) = 0$.  This implies that 
$$\nu(a) < \nu(2^{4k-1}) = 4k-1$$
(here $\nu(-)$ denotes the 2-adic valuation).  Therefore, 
\begin{eqnarray*}
\gamma &=& \partial (2^{4k-1})  \\
&=& \left(\frac{2^{4k-1}}{2a}\right) \partial (2a) \\ 
&=& \left(\frac{2^{4k-1}}{2a}\right) 2 \beta \\
&=& 0 \hspace{1in} \text{(by fact (i))}. 
\end{eqnarray*}
This is a contradiction because $\gamma \neq 0$ by Proposition~\ref{pro: only -1 cell matters}.  

%{\color{blue} For the sake of contradiction, suppose that this composition is zero.  By Proposition~\ref{pro: order two between 2 locks} (ii), the element $\{P^{k-1}h_1^{3}\}[8k-5]$ must be killed by some differential in the $j''$-based Atiyah--Hirzebruch spectral sequence of $X(8k+3)^{8k-5}$.  By Proposition~\ref{pro: only -1 cell matters}, this differential must be of the form
%$$b[-1]\longrightarrow \{P^{k-1}h_1^{3}\}[8k-5].$$

%Note that in order for differential (\ref{eq: differential to second lock}) to be defined, we must have ${\nu(2b) \leq 4k-1}$ ($\nu(2b)$ denotes the 2-adic valuation of $2b$).  However, Proposition~\ref{pro: order two between 2 locks} (iii) implies that the class $2b[-1]$ is a permanent cycle.  Therefore, the class $2^{4k-1}[-1]$ is also a permanent cycle and the target of differential (\ref{eq: differential to second lock}) must have been killed by a shorter differential.  This contradicts Proposition~\ref{pro: only -1 cell matters}.}
\end{proof}

\begin{cor}\label{cor: second lock not passed}
We have the inequality 
$$\mathfrak{L}(8k+3)\leq 8k-2$$ 
for all $k\geq 0$.
\end{cor}
 
\subsection{Step 6: the first lock is passed when $k$ is odd}\label{subsec:Step5}
%%%%%commented out
\begin{comment}
\begin{prop}\label{pro: first lock for odd}
When $k$ is odd, there is a differential 
\begin{equation}\label{eq: differential that kills first lock}
2^{4k-4}[-1]\longrightarrow \{P^{k-1}h_{1}^{3}\}[8k-5]  
\end{equation}
in the $S^{0}$-based Atiyah--Hirzebruch spectral sequence of $X(8k+3)^{8k-5}$.
\end{prop}
A sketch of the proof for Proposition~\ref{pro: first lock for odd} is as follows: first, we construct a map 
$$X(8k+3)_{-1}^{8k-5}\longrightarrow \Sigma^{-8n-3}\mathbb{C}P^{8n-1}_{4n+1}$$
that is of degree one on both the top and the bottom cell.  Next, we modify this map by replacing the target with a spectrum $\Sigma^{-1}Z(k)$, which is defined as the homotopy fiber of a certain map
$$\Sigma^{-8n-3}\mathbb{C}P^{8n-1}_{4n+1}\longrightarrow S^{8n-7}.$$
Using the connectivity of the 0-connected cover of $T_{4k-3}$, the $(4k-3)$-layer of the Adams tower for $S^{0}$, we prove that there exists a differential of the form
$$2^{4k-4}[-1]\longrightarrow a[8k-5], \,\,\,a\in j''(S^{8k-5})$$
in the $S^{0}$-based Atiyah--Hirzebruch spectral sequence of $\Sigma^{-1}Z(k)$.  By computing the $e$-invariant of the element $a$ using Chern character, we show that $a=\{P^{k-1}h_1^{3}\}$.  

The map 
$$X(8k+3)_{-1}^{8k-5}\longrightarrow \Sigma^{-1}Z(k)$$ 
induces a map from the $S^0$-based Atiyah--Hirzebruch spectral sequence of $\Sigma^{-1}Z(k)$ to that of $X(8k+3)_{-1}^{8k-5}$.  Pushing forward the differential above produces the desired differential.  
\end{comment}
%%%%%%%%%%%
In this step, we will show that when $k$ is odd, $\mathfrak{L}(8k+3) \geq 8k-2$.  To prove this, we first construct a spectrum $\Sigma^{-1}Z(k)$ for any $k$.  This spectrum is defined as the homotopy fiber of a certain map
$$\begin{tikzcd}\Sigma^{-8k-3}\mathbb{C}P^{8k-1}_{4k+1}\ar[r,twoheadrightarrow]& S^{8k-7}.\end{tikzcd}$$
The spectrum $\Sigma^{-1}Z(k)$ has bottom cell in dimension $(-1)$ and top cell in dimension $(8k-5)$.  
\begin{prop}\label{prop:Step6CommDiagram}
There exists a map 
$$\rho: X(8k+3)_{-1}^{8k-2} \longrightarrow \Sigma^{-1}Z(k) $$
such that the following diagram commutes: 
\begin{equation} \label{diagram:Z(k)diagram}
\begin{tikzcd}
X(8k+3)^{8k-2} \ar[d, twoheadrightarrow] \ar[rrrdd,"c(8k+3)^{8k-2}"]\\ 
X(8k+3)_{-1}^{8k-2} \ar[d, "\rho"] && \\ 
\Sigma^{-1}Z(k) \ar[r, twoheadrightarrow] &S^{8k-5} \ar[rr,swap, "\{P^{k-1}h_1^3\}"] && S^0 
\end{tikzcd}
\end{equation}
\end{prop}
%Moreover, there exists a map 
%$$f: X(8k+3)_{-1}^{8k-2} \longrightarrow \Sigma^{-1}Z(k)$$
%that makes the following diagram commute: 
%\begin{equation}\label{diagram:Z(k)diagram}
%\begin{tikzcd}
%X(8k+3)_{-1}^{8k-2} \ar[r, twoheadrightarrow] \ar[d, "f"] & X(8k+3)_{8k-5}^{8k-2} \ar[d] \ar[rd] &\\
%\Sigma^{-1}Z(k) \ar[r, twoheadrightarrow] & S^{8k-5} \ar[r] & S^0.
%\end{tikzcd}
%\end{equation}

\begin{prop}\label{pro: first lock for odd}
When $k$ is odd, the composition
$$\begin{tikzcd}\Sigma^{-1}Z(k) \ar[r, twoheadrightarrow]& S^{8k-5} \ar[rr,"\{P^{k-1}h_{1}^{3}\}"]&& S^0\end{tikzcd}$$
is zero. 
\end{prop}
Proposition~\ref{pro: first lock for odd} is proven by considering $T_{4k-3}$, the $(4k-3)$-layer of the Adams tower for $S^{0}$.  Using the connectivity of the 0-connected cover of $T_{4k-3}$, we prove that there exists a differential of the form
$$2^{4k-4}[-1]\longrightarrow a[8k-5], \,\,\,a\in \pi_{8k-5} $$
in the $S^{0}$-based Atiyah--Hirzebruch spectral sequence of $\Sigma^{-1}Z(k)$.  Moreover, $a$ is in the image of $j$.  By computing the $e$-invariant of the element $a$ using Chern character, we show that $a=\{P^{k-1}h_1^{3}\}$.  

It follows from Proposition~\ref{pro: first lock for odd} that the map
$$\begin{tikzcd}X(8k+3)^{8k-2} \ar[rr,"c(8k+3)^{8k-2}"]&&S^0\end{tikzcd}$$
is also zero by the commutativity of diagram~(\ref{diagram:Z(k)diagram}).  

\begin{cor}\label{cor: first lock passed for k odd}
When $k$ is odd, we have the inequality 
$$\mathfrak{L}(8k+3)\geq 8k-2.$$ 
\end{cor}
%\begin{proof}
%By Corollary~\ref{cor: ph13 is Mohowald invariant} and Proposition~\ref{pro: first lock for odd}, the map 
%$$
%c(8k+3)^{8k-5}:X(8k+3)^{8k-5}\longrightarrow S^{0}
%$$
%is null-homotopic.
%\end{proof}

\subsection{Step 7: the first lock is not passed when $k$ is even} \label{subsec:Step6}

\begin{prop}\label{pro: first lock for even}
When $k$ is even, the class 
$$2^{4k-4-\nu(k)}[-1]$$
is a permanent cycle in the $j''$-based Atiyah--Hirzebruch spectral sequence of ${X(8k+3)^{8k-5}}$.
\end{prop}
The proof of Proposition~\ref{pro: first lock for even} is sketched as follows: first, by restricting the map $\rho$ in Proposition~\ref{prop:Step6CommDiagram} to the $(8k-5)$-skeleton, we obtain a map 
$$X(8k+3)_{-1}^{8k-5} \longrightarrow \Sigma^{-1}Z(k),$$ 
where $Z(k)$ is constructed in Section~\ref{subsec:Step5}.  Then, we establish a permanent cycle 
$$2^{4k-4-\nu(k)}[-1]$$
in the $j''$-based Atiyah--Hirzebruch spectral sequence for $\Sigma^{-1}Z(k)$ when $k$ is even via Chern character computations.  This permanent cycle is then used to prove the desired permanent cycle.

\begin{thm} 
When $k$ is even, the composition map
\begin{equation}\label{eq: first lock in j}
X(8k+3)^{8k-5}\xrightarrow{c(8k+3)^{8k-5}}S^{0}\longrightarrow j''
\end{equation}
is not null.
\end{thm}
\begin{proof}
By Corollary~\ref{cor: ph13 is Mohowald invariant}, one can rewrite (\ref{eq: first lock in j}) as the composition
\begin{equation}\label{eq: first lock in j 2} 
\begin{tikzcd}
X(8k+3)^{8k-5} \ar[r, twoheadrightarrow]&  S^{8k-5} \ar[rr, "\{P^{k-1}h_1^{3}\}"] && j''.
\end{tikzcd}
\end{equation}
For the sake of contradiction, suppose that (\ref{eq: first lock in j 2}) is null-homotopic.  By Proposition~\ref{pro: only -1 cell matters}, there must exist a differential of the form
\begin{equation} \label{equation:b[-1]differential}
b[-1]\longrightarrow \{P^{k-1}h_1^{3}\}[8k-5]
\end{equation}
for some $b\in \mathbb{Z}_{(2)}$.  

Recall that in Lemma~\ref{lem: differential to second lock}, we established the differential
$$2^{4k-1}[-1]\longrightarrow \phi[8k-1]$$
for some nonzero element $\phi \in \pi_{8k-1}j''$.  %Note also that 
%$$j'^{0}(X(8k+3)^{8k-1}_{8k-4})=j'^{0}(S^{8k-1}).$$
This, combined with differential~(\ref{equation:b[-1]differential}), shows that there exists a differential 
\begin{equation}\label{eq: diff kills generator}
2b[-1]\longrightarrow \gamma [8k-1].
\end{equation}
Furthermore, the elements $\phi$ and $\gamma\cdot \frac{2^{4k-1}}{2b}$, when considered as elements in \\${j''^{0}(X(8k+3)^{8k-1}_{0})}$, are equal.  Since 
$$\nu\left(\frac{2^{4k-1}}{2b}\right)\geq 4k-1-(1+4k-5-\nu(k))=3+\nu(k)$$
and $\pi_{8k-1}j''=\mathbb{Z}/(2^{4+\nu(k)})$, $\gamma$ must be the generator of  $\pi_{8k-1}j''$.  

Consider the exact sequence 
$$j''^0(S^{8k-1}) = j''^0(X(8k+3)_{8k-4}^{8k-1})\longrightarrow j''^{0}(X(8k+3)^{8k-1})\longrightarrow j''^{0}(X(8k+3)^{8k-5})$$
that is induced from the cofiber sequence 
$$X(8k+3)^{8k-5} \longrightarrow X(8k+3)^{8k-1} \longrightarrow X(8k+3)_{8k-4}^{8k-1}.$$
Differential~(\ref{eq: diff kills generator}) implies that the map 
$$j''^0(S^{8k-1}) = j''^0(X(8k+3)_{8k-4}^{8k-1})\longrightarrow j''^{0}(X(8k+3)^{8k-1})$$ 
is zero.  Therefore, the map 
$$j''^{0}(X(8k+3)^{8k-1})\longrightarrow j''^{0}(X(8k+3)^{8k-5})$$
is injective.  However, our induction hypothesis states that the composition map 
$$\begin{tikzcd}X(8k+3)^{8k-5} \ar[r,hookrightarrow]& X(8k+3)^{8k-1} \ar[rr, "c(8k+3)^{8k-1}"] && S^{0} \ar[r] & j''
 \end{tikzcd}$$
is zero.  The injection above will imply that the composition map 
$$X(8k+3)^{8k-1}\xrightarrow{c(8k+3)^{8k-1}}S^{0}\longrightarrow j''$$
is also zero.  This contradicts Theorem~\ref{thm: second lock not passed}.\end{proof}

\begin{cor}\label{cor: first lock not passed when $n$ even}
When $k$ is even, we have the equality 
$$\mathfrak{L}(8k+3)=8k-5.$$
\end{cor}

%%%%
In light of Proposition~\ref{prop: equivariant to nonequivariant}, our main theorem (Theorem~\ref{main theorem}) follows directly from the various bounds that we have established for the Mahowald line (see Figure~\ref{fig:OutlineStep0}).

\begin{comment}
\begin{enumerate}
    \item Corollary~\ref{corollary: lower bound} implies that a level-$(p,q)$ Furuta--Mahowald class exists when
    $$ q \geq  \begin{cases}
    2p+2        \quad p\equiv 1,2&\pmod 4 \\
    2p+3       \quad  p\equiv 3&\pmod 4 \\
    2p+4    \quad  p\equiv 0&\pmod 4.
  \end{cases}
$$
    \item Corollary~\ref{cor: second lock not passed} implies that a level-$(p,2p+1)$ Furuta--Mahowald class does not exist when $4\mid p$. 
    \item Corollary~\ref{cor: first lock passed for k odd} implies that a level-$(p,p+3)$ Furuta--Mahowald class exists when $p\equiv 4 \pmod 8$. 
    \item Corollary~\ref{cor: first lock not passed when $n$ even} implies that a level-$(p,p+3)$ Furuta--Mahowald class does not exist when $p\equiv 0 \pmod 8$.
\end{enumerate}
Combining these results with Furuta-Kametani's bound (\ref{eq:Furuta-Kametani}), we finish the proof of the main theorem.
\end{comment}

%%%%%%%%%%%%%%%%%%%%%
%%%%%%%%%%%%%%%%%%%%%%
%%%%%%%%%%%%%%%%%%%%%%
%%%%%%%%%%%%%%%%%%%%%%%%
%%%%%%%%%%%%%%%%%%%%%%%%
%%%%%%%%%%%%%%%%%%%%%%%%%
\newpage
\section{Preliminaries}\label{sec:Preparations}
In this section, we set up some preliminaries that will be useful in the later sections.  In Section~\ref{subsec:Maps between subquoteints}, we define maps between certain subquotients of $X(m)$.  In Section~\ref{subsec:TrasnferMaps}, we discuss the transfer map. 
 
\subsection{Maps between subquoteints} \label{subsec:Maps between subquoteints}
\begin{df}\rm
Let $m$, $n$, and $l$ be integers with $m>n\geq 0$.  The function $h(n,m,l)\in \mathbb{Z}$ is inductively defined as follows (see Figure~\ref{fig:MapsBetweenQuotients}):
\begin{itemize}
    \item $h(n,n-1,l)=\begin{cases}
    l-1       \quad &\text{if }l+n\equiv 0,3 \pmod 4,  \\
  l     \quad &\text{otherwise. }
  \end{cases}$
  \item $h(m,n, l)=h(m-1,n, h(n,n-1,l))$ when $m-n\geq 2$.
\end{itemize}
We also set $h(m,n,\infty)=\infty$.  
\end{df}

Intuitively, the integer $h(m, n,l)$ can be described as follows: start with the $l$-cell in $X(m)$ and walk to the right (towards $X(n)$), moving down one cell every time we encounter an empty cell.  The cell we reach at $X(n)$ is $h(m, n, l)$.  

\begin{figure}
\begin{center}
\makebox[\textwidth]{\includegraphics[trim={2.5cm 4.2cm 0.4cm 6.7cm}, clip, page = 1, scale = 0.7]{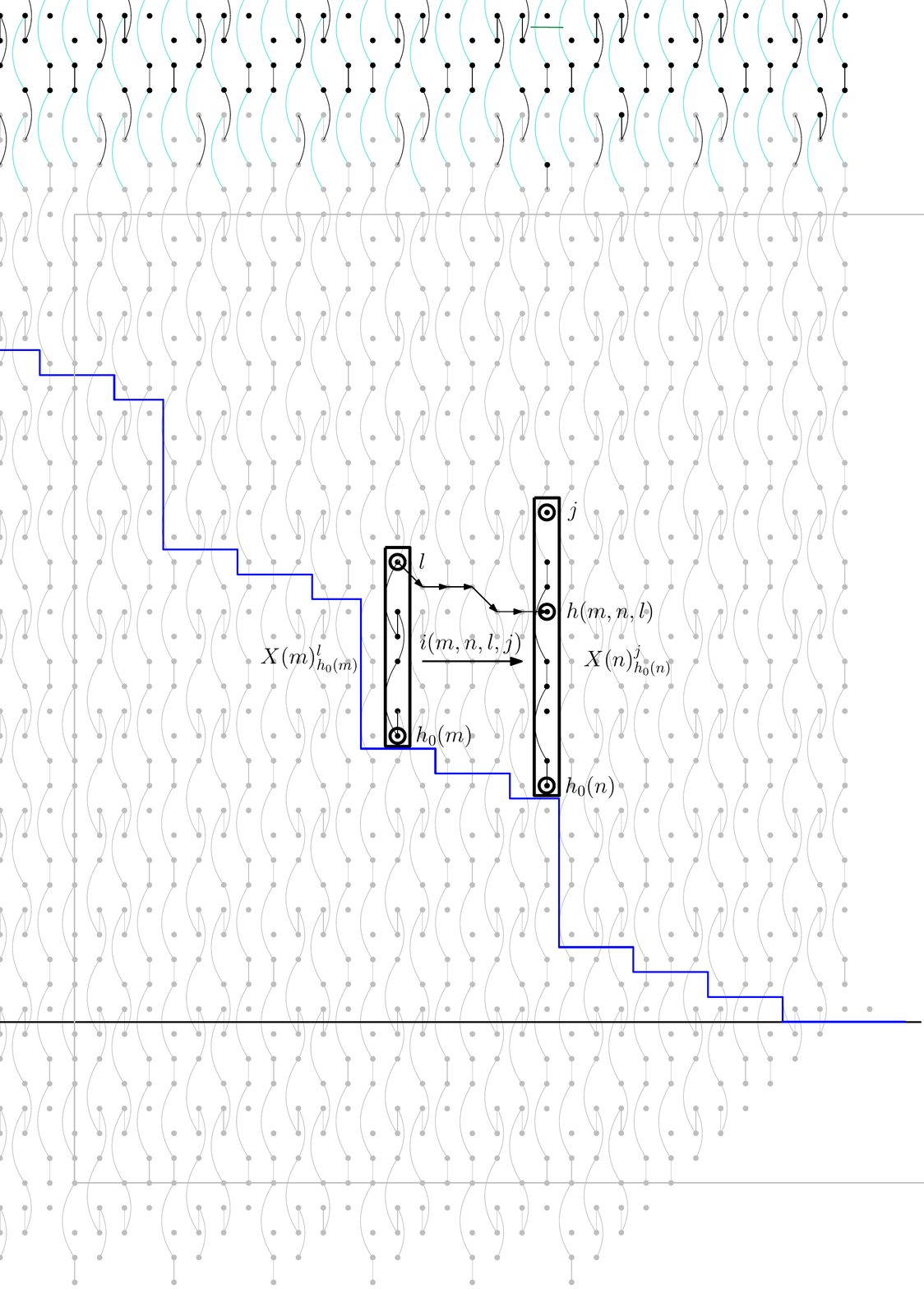}}
\captionof{figure}{Maps between subquotients.}
\hfill
\label{fig:MapsBetweenQuotients}
\end{center}
\end{figure}

\begin{df}\rm
For $k\geq 0$ and $0\leq m\leq 7$, define 
$$h_{0}(4+8k+m)=8k+\tau(m) +1,$$ 
where the function $\tau(m)$ is defined as in Corollary $\ref{corollary: lower bound}$.  In other words, the \\${h_0(4+8k+m)}$-cell of $X(4+8k+m)$ is the first cell that is above the lower bound line proved in Section~\ref{subsec:Step1} (the {\color{blue} blue} line in Figure~\ref{fig:MapsBetweenQuotients}). 
\end{df}

\begin{prop}
Let $m$, $n$, $l$, $j$ be integers such that the following conditions hold: 
\begin{enumerate}[label=(\alph*)]
    \item $m=8k+4+a$ and $n=8k+4+b$, where $k\geq 0$ and $a,b \in \{0,\ldots,7\}$;
    \item $m > n$;
    \item $l \geq h_{0}(m)$;
    \item  $j\geq h(m,n,l)$.
\end{enumerate}
Then there exists a map
$$
i(m,n,l,j):X(m)^{l}_{h_{0}(m)}\longrightarrow X(n)^{j}_{h_{0}(n)}.
$$
Furthermore, the maps $i(m, n, l, j)$ are compatible with each other in the sense that the following three properties hold: 
\begin{enumerate}
\item (Compatibility with respect to quotient).  The following diagram commutes for all $m > n$:
\begin{equation*}
\xymatrix{X(m)^{\infty}_{h_{0}(m)}\ar[rr]^{i(m,n,\infty,\infty)}&& X(n)^{\infty}_{h_{0}(n)}\\
X(m)\ar@{->>}[u] \ar[rr]^{i(m,n)}& &X(n).\ar@{->>}[u] }
\end{equation*}
\item (Compatibility with respect to inclusion).  If $(m, n, l', j')$ is another tuple satisfying the conditions above with $l'\leq l$ and $j'\leq j$, then the following diagram commutes:
\begin{equation}\label{diagram:subquotient compatible2}
\xymatrix{X(m)^{l}_{h_{0}(m)}\ar[rr]^{i(m,n,l,j)}&& X(n)^{j}_{h_{0}(n)}\\
X(m)^{l'}_{h_{0}(m)}\ar@{^{(}->}[u] \ar[rr]^{i(m,n,l',j')}& &X(n)^{j'}_{h_{0}(n)}.\ar@{^{(}->}[u] }    
\end{equation}
\item (Compatibility with respect to composition).  If $(m,n,l,j)$ and $(n,p,j,q)$ are two tuples satisfying the conditions of the proposition, then
$$
i(m,p,l,q)= i(n,p,j,q)\circ i(m,n,l,j).
$$ 
\end{enumerate}
\end{prop}

To avoid clustering the notations in the later sections, we will simply use the special arrow
$$X(m)^{l}_{h_{0}(m)}\rightharpoonup X(n)^{j}_{h_{0}(n)}$$ 
to denote the map $i(m,n,l,j)$ when the context is clear.  

\begin{proof}
We will construct the maps $i(m, n, l, j)$ in four steps, increasing the level of generality at each step.  \\

\noindent \textbf{Step 1:} {\boldmath$m=n+1,\, l=j=\infty.$}  By our definition of $h_0(-)$ and the cellular approximation theorem, there is always a map
$$X(n+1)^{h_{0}(n+1)-1} \longrightarrow X(n)^{h_{0}(n)-1}.$$ 
Furthermore, this map makes the bottom square of the diagram
\begin{equation*}
\begin{tikzcd}
X(n+1)_{h_{0}(n+1)}\ar[rr,dashed,"i(n+1{,} n{,}\infty{,}\infty)"] &&X(n)_{h_{0}(n)}\\
X(n+1)\ar[rr,"i(n+1{,}n)"]\ar[u,two heads] &&X(n)\ar[u,two heads]\\
X(n+1)^{h_{0}(n+1)-1}\ar[rr]\ar[u,hook] &&X(n)^{h_{0}(n)-1}\ar[u,hook].
\end{tikzcd}
\end{equation*}
commute.  Since both columns are cofiber sequences, there is an induced map 
$$i(n+1,n,\infty,\infty): X(n+1)_{h_{0}(n+1)}\rightarrow X(n)_{h_{0}(n)}$$ 
between the cofibers making the whole diagram commute.  The top square of the commutative diagram above implies that property (1) holds for $m=n+1$. \\

\noindent \textbf{Step 2:} {\boldmath$m=n+1,\ j=h(n+1,n,l)$.}  Note that by the definition of $h(n+1,n,l)$, 
$$X(n)_{h_{0}(n)}^{l}=X(n)_{h_{0}(n)}^{h(n+1,n,l)}.$$ 
We define the map $i(n+1,n,l,h(n+1,n,l))$ to be the map
$$X(n+1)_{h_0(n+1)}^l \longrightarrow X(n)_{h_0(n)}^l = X(n)_{h_0(n)}^{h(n+1,n,l)}.$$
The map $i(n+1,n,l,h(n+1,n,l))$ fits into the following commutative diagram:
\begin{equation}\label{diagram:subquotient compatible 1}
\begin{tikzcd}
X(n+1)^{\infty}_{h_{0}(n+1)}\ar[rrr,"i(n+1{,}n{,}\infty{,}\infty)"]&&&X(n)^{\infty}_{h_{0}(n)}\\
X(n+1)^{l}_{h_{0}(n+1)}\ar[rrr,"i(n+1{,}n{,}l{,}h(n+1{,}n{,}l))"]\ar[u,hook] &&& X(n)^{h(n+1,n,l)}_{h_{0}(n)}\ar[u,hook]. 
\end{tikzcd}
\end{equation}

\noindent \textbf{Step 3:} {\boldmath$m=n+1$.}  We define the map $i(n+1,n,l,j)$ to be the composition
$$
X(n+1)_{h_{0}(n+1)}^{l}\xrightarrow{i(n+1,n,l,h(n+1,n,l))}X(n)_{h_{0}(n)}^{h(n+1,n,l)}\hookrightarrow X(n)_{h_{0}(n)}^{j}.
$$

We now prove that property (2) holds when $m=n+1$.  The case when $l=\infty$ is directly implied by diagram~(\ref{diagram:subquotient compatible 1}).  

Suppose that $l<\infty$.  Consider the two compositions 
$$
\mathbf{1}: X(n+1)^{l'}_{h_{0}(m)}\hookrightarrow X(n+1)^{l}_{h_{0}(m)}\xrightarrow{i(n+1,n,l,j)}X(n)^{j}_{h_{0}(n)} 
$$
and
$$
\mathbf{2}: X(n+1)^{l'}_{h_{0}(n+1)}\xrightarrow{i(n+1,n,l',j')}X(n)^{j'}_{h_{0}(n)} \hookrightarrow X(n)^{j}_{h_{0}(n)}
$$
in diagram~(\ref{diagram:subquotient compatible2}).  We want to show that these two compositions are equal.  After post-composing with the inclusion map
$$
X(n)^{j}_{h_{0}(n)}\hookrightarrow X(n)^{\infty}_{h_{0}(n)},
$$
the maps $\mathbf{1}$ and $\mathbf{2}$ are homotopic to each other (this is because we have already verified Property (2) when $\ell = \infty$). %that they are both homotopic to the composition
%$$X(m)^{l'}_{h_{0}(m)}\xrightarrow{i(m,n,l',j')}X(n)^{j'}_{h_{0}(n)} \hookrightarrow X(n)^{\infty}_{h_{0}(n)}.)$$

Consider the cofiber sequence 
$$
\Sigma^{-1} X(n)^{\infty}_{j+1}\longrightarrow X(n)^{j}_{h_{0}(n)}\hookrightarrow X(n)^{\infty}_{h_{0}(n)},
$$
Since the difference $\mathbf{1}- \mathbf{2}$ is null after post-composing with the map 
$${X(n)^{j}_{h_{0}(n)}\hookrightarrow X(n)^{\infty}_{h_{0}(n)}},$$
it factors through the fiber via a certain map 
$${\mathbf{3}: X(n+1)^{l'}_{h_{0}(n+1)}\rightarrow \Sigma^{-1} X(n)^{\infty}_{j+1}}:$$ 
$$\begin{tikzcd}
X(n+1)^{l'}_{h_{0}(n+1)} \ar[rd, "\mathbf{1}- \mathbf{2}"] \ar[d, dashed, "\mathbf{3}"]&&\\ 
\Sigma^{-1} X(n)^{\infty}_{j+1}\ar[r] &X(n)^{j}_{h_{0}(n)} \ar[r, hookrightarrow] &X(n)^{\infty}_{h_{0}(n)}.
\end{tikzcd}$$

If the left vertical arrow in diagram~(\ref{diagram:subquotient compatible2}) is the identity map, then diagram~(\ref{diagram:subquotient compatible2}) commutes by definition.  Otherwise, it is straightforward to check that the dimension of the top cell of $X(n+1)^{l'}_{h_{0}(n+1)}$ is less than the dimension of the bottom cell in $\Sigma^{-1} X(n)^{\infty}_{j+1}$.  Therefore, the map $\mathbf{3}$ is zero by the cellular approximation theorem.  This implies $\mathbf{1}=\mathbf{2}$ and that property (2) holds when $m = n+1$.  \\

\noindent \textbf{Step 4: General {\boldmath $m,n,l,j$}.} Choose a sequence $l_m, l_{m-1}, \ldots, l_n$ such that 
\begin{enumerate}
\item $l_{m}=l,\ l_{n}=j$.
\item $l_{s}\geq h(s+1,s,l_{s+1})$ for all $m-1 \geq s \geq n$.
\end{enumerate}
We define the map $i(m,n,l,j)$ to be the composition
$$\prod^{m}_{r=n+1}i(r,r-1,l_{r},l_{r-1}) = i(n+1, n, l_{n+1}, l_n)\circ \cdots \circ i(m, m-1, l_m, l_{m-1}). $$
Note that by our discussion in step 3, this composition does not depend on the choice of the sequence $(l_m, l_{m-1}, \ldots, l_n)$.  Property (3) holds immediately by definition.  Properties (1) and (2) hold by our discussions in steps 1 and 3, respectively. 
\end{proof}

%%%%%%%%%%%%%
\subsection{Transfer maps} \label{subsec:TrasnferMaps}
\begin{prop}\label{prop:TransferMaps}
There is a cofiber sequence 
\begin{equation}\label{cofiber of adjacent columns}
    X(m+1)\xrightarrow{i(m+1,m)}X(m)\xrightarrow{s_m} \Sigma^{-m}\mathbb{C}P^{\infty}_{+}
\end{equation}
\end{prop}
\begin{proof}
The map $i(m+1,m)$ can be rewritten as the map 
$$
\left(S(\infty \mathbb{H})_{+}\wedge S^{-(m+1)\widetilde{\mathbb{R}}}\wedge S^{0}\right)_{h\Pin}\longrightarrow 
\left(S(\infty \mathbb{H})_{+}\wedge S^{-(m+1)\widetilde{\mathbb{R}}}\wedge S^{\widetilde{\mathbb{R}}}\right)_{h\Pin},
$$
which is induced by the map $a_{\widetilde{\mathbb{R}}}:S^{0}\rightarrow S^{\widetilde{\mathbb{R}}}$.  The cofiber sequence 
$$
S^{0}\xrightarrow{a_{\widetilde{\mathbb{R}}}} S^{\widetilde{\mathbb{R}}}\longrightarrow \Sigma ({C_2}_+)
$$
produces the cofiber sequence
$$
 X(m+1)\xrightarrow{i(m+1,m)}X(m)\xrightarrow{s_m}\left(S(\infty \mathbb{H})_{+}\wedge S^{-(m+1)\widetilde{\mathbb{R}}}\wedge\Sigma (C_{2})_{+} \right)_{h\Pin}.
$$
Note that
\begin{eqnarray*}
\left(S(\infty \mathbb{H})_{+}\wedge S^{-(m+1)\widetilde{\mathbb{R}}}\wedge\Sigma (C_{2})_{+} \right)_{h\Pin}&=&\left(\left(S(\infty \mathbb{H})_{+}\wedge S^{-(m+1)\widetilde{\mathbb{R}}}\wedge\Sigma (C_{2})_{+} \right)_{hS^{1}}\right)_{hC_{2}}\\
&=&\left(\mathbb{C}P^{\infty}_{+}\wedge S^{-(m+1)\sigma}\wedge \Sigma(C_{2})_{+} \right)_{hC_{2}}\\
&=&\mathbb{C}P^{\infty}_{+}\wedge S^{-(m+1)}\wedge S^{1}\\
&=&\Sigma^{-m}\mathbb{C}P^{\infty}_{+}.
\end{eqnarray*}
This establishes the cofiber sequence~(\ref{cofiber of adjacent columns}).
\end{proof}

Let $V$ denote the rank-3 bundle over $BSU(2)=\mathbb{H}P^{\infty}$ that is associated to the adjoint representation of $SU(2)$ on its Lie algebra $\mathfrak{su}(2)$. 

\begin{comment}
For a general fiber bundle $F\hookrightarrow E\xrightarrow{p} B$ with fiber $F$ a manifold, the tangent space of $F$ forms a vector bundle over $E$.  This vector bundle is called the vertical tangent bundle and is denoted by $T_{\text{vert}}E$.  Given any vector bundle $W$ over $B$, there is a stable transfer map 
$$\text{Tr}:\Thom(B,W)\longrightarrow \Thom(E,p^{*}W-T_{\text{vert}}E).$$ 
This map is constructed as follows: when $B$ is compact, we can embed $E$ into the bundle $\underline{\mathbb{R}}^{m}\oplus W$ ($\underline{\mathbb{R}}^{m}$ is the trivial bundle over $B$ with dimension $m\gg0$). By collapsing a tubular neighborhood of this embedding to a single point (namely, the Pontryagin--Thom construction), we produce a map
\begin{equation}\label{transfer map 1}
\Thom(B,\underline{\mathbb{R}}^{m}\oplus W) \longrightarrow \Thom(E,N).
\end{equation}
Here, $N$ is the normal bundle of $F$ in the fiber of $\underline{\mathbb{R}}^{m}\oplus W$. 

Note that $$T_{\text{vert}}E\oplus N=\underline{\mathbb{R}}^{m}\oplus p^{*}W.$$ 
The map $\text{Tr}$ is induced by (\ref{transfer map 1}). When $B$ is noncompact, one can apply this construction on compact subsets and take the limit.
\end{comment}

Given a Lie group $G$ with a closed subgroup $H$, there is a fiber bundle 
$$\begin{tikzcd} G/H \ar[r, hookrightarrow]&  BH \ar[r, "p"]& BG. \end{tikzcd}$$ 
Let $V_{H}$ (resp. $V_{G}$) be the vector bundle over $BH$ (resp. $BG$) associated to the adjoint representation on the Lie algebra.  There is a well-known transfer map 
$$\text{Tr}:\Thom(BG,V_{G})\rightarrow \Thom(BH,V_{H})$$
that has been studied by Becker--Gottlieb~\cite{BeckerGottlieb}, Becker--Schultz~\cite{BeckerSchultz}, and Bauer~\cite{BauerFramedManifolds}.  
Now, set 
\begin{eqnarray*}
G&=&SU(2), \\
H&=&\Pin,\\
V_{G}&=&V,\\
V_{H}&=&\lambda.
\end{eqnarray*} 
(Recall that $\lambda$, as defined in Section~\ref{subsec:equivToNonEquiv}, is the line bundle that is associated to the principal bundle $C_2 \hookrightarrow BS^1 \to \textup{BPin(2)}$.)  We obtain a transfer map  
\begin{equation*}
\textup{Tr}:\Thom(\mathbb{H}P^{\infty},V)\longrightarrow X(-1).
\end{equation*}

\begin{prop}\label{transfer}
The transfer map 
\begin{equation}\label{eq: transfer}
\textup{Tr}:\Thom(\mathbb{H}P^{\infty},V)\longrightarrow X(-1)
\end{equation}
induces an isomorphism on $(H\mathbb{F}_{2})_{4n+3}$ for all $n$.
\end{prop}
\begin{proof}
Consider the pull back of $\text{Tr}$ under the inclusion map $\text{pt}\hookrightarrow \mathbb{H}P^{\infty}$. We obtain the following commutative diagram:
$$\begin{tikzcd}
\Thom(\mathbb{H}P^{\infty},V) \ar[r, "\text{Tr}"]& X(-1)\\ 
S^{3}\ar[u,"\mathbf{1}"] \ar[r,"\mathbf{2}"] & \Thom(\mathbb{R}P^{2},\lambda|_{\mathbb{R}P^{2}})\ar[u,"\mathbf{3}"].
\end{tikzcd}$$
Note that $(H\mathbb{F}_{2})_{3}$ of all the spectra in the diagram above are $\mathbb{F}_{2}$. 

Since map $\mathbf{3}$ is induced by the inclusion of fiber of the bundle 
$$\begin{tikzcd} \mathbb{R}P^{2} \ar[r,hookrightarrow]& B\Pin \ar[r ]& \mathbb{H}P^{\infty} \end{tikzcd}$$
and the Serre spectral sequence for this bundle collapses, map $\mathbf{3}$ induces an isomorphism on $(H\mathbb{F}_{2})_{3}$.  Moreover, map $\mathbf{2}$ is the Pontryagin--Thom collapsing map, and it induces an isomorphism on $(H\mathbb{F}_{2})_{3}$.  It follows from this that $\text{Tr}$ must induce an isomorphism on $(H\mathbb{F}_{2})_{3}$. 
  
To prove that $\text{Tr}$ induces an isomorphism on $(H\mathbb{F}_{2})_{4n+3}$ for any $n$, note that both $H_{*}(\Thom(\mathbb{H}P^{\infty},V);\mathbb{F}_{2})$ and $H_{*}(X(1);\mathbb{F}_{2})$ are modules over $H^{*}(\mathbb{H}P^{\infty};\mathbb{F}_{2})$. Moreover, the induced map $\text{Tr}_*$ on $\mathbb{F}_2$-homology preserves this module structure.  Therefore, the statement is reduced to proving an isomorphism for the case $n=0$, which we have just proved.
\end{proof}

We equip $\Thom(\mathbb{H}P^{\infty},V)$ with the cell structure that has one cell in dimension $4n+3$ for each $n\geq 0$.

\begin{lem}\label{lem:HPTruncationTwist}
$\Thom(\mathbb{H}P^{\infty},V)^{4n+7}_{4n+3}$ is homotopy equivalent to $\Sigma^{4n+3}C(2+n) \nu$.
\end{lem}
\begin{proof}
Let $U$ denote $\Thom(\mathbb{H}P^{\infty},V)$.  We have the following equivalences: 
\begin{eqnarray*}
U^{4n-1}&=&\Thom(\mathbb{H}P^{n-1},V|_{\mathbb{H}P^{n-1}}),\\
 U^{4n+7}&=&\Thom(\mathbb{H}P^{n+1},V|_{\mathbb{H}P^{n+1}}).
 \end{eqnarray*}
Also, 
$$\mathbb{H}P^{n+1}_{n}=\Thom(\mathbb{H}P^{1},nH),$$
where $H$ is the tautological bundle over $\mathbb{H}P^{1}$.  These equivalences imply that 
$$U^{4n+7}_{4n+3}=\Thom(\mathbb{H}P^{1},nH\oplus V).$$ 

Note the following general fact: given a vector bundle $E$ over $S^{4}$, the attaching map in $\Thom(S^{4},E)$ is given by $\frac{p_{1}(E)}{2}\cdot \nu\in \pi_{3}$.  This fact can be proven by analyzing $\Thom(\mathbb{H}P^{1},H)$, which corresponds to the generator $\nu$ of $\pi_{3}$.

We will now compute $p_{1}(nH\oplus V)$.  By restricting the representations of $SU(2)$ to the subgroup $S^{1}$, we deduce that under the map 
$BS^{1}\rightarrow BSU(2)$, the bundle $V$ pulls back to $r(L^2)+1$ and the bundle $H$ pulls back to $r(L  + L^{-1})$ ($L$ is the tautological bundle over $\mathbb{C}P^\infty$).  Therefore, 
$$
p_{1}(V)=p_{1}(r(L^2))=c^{2}_{1}(L^2)-2c_{2}(L^2)=4
$$
and 
$$
p_{1}(H)=p_{1}(r(L  + L^{-1}))=c^{2}_{1}(L  + L^{-1})-2c_{2}(L  + L^{-1})=2.
$$
It follows that $p_{1}(nH\oplus V)=4+2n$.  This completes the proof.
\end{proof}

%%%%%%%%%%%%%%%%%%%%%%%%
%%%%%%%%%%%%%%%%%%%%%%%%%
\newpage
\section{Attaching maps in $X(m)$}\label{sec:AttachingMaps}

\subsection{$\textup{H}\mathbb{F}_2$-subquotients}

We recall the following definition and lemma from \cite{WangXu}:

\begin{df} \rm \label{hf2}
Let $A$, $B$, $C$ and $D$ be CW spectra, $i$ and $q$ be maps
\begin{displaymath}
    \xymatrix{
A \ar@{^{(}->}[r]^-i & B, & B \ar@{->>}[r]^-q & C.
    }
\end{displaymath}
We say that $(A, i)$ is an \textit{$\textup{H}\mathbb{F}_2$-subcomplex} of $B$ if the map $i$ induces an injection on mod 2 homology.  An $\textup{H}\mathbb{F}_2$-subcomplex is denoted by a hooked arrow as above.  Similarly, we say that $(C, q)$ is an \textit{$\textup{H}\mathbb{F}_2$-quotient complex} of $B$ if the map $q$ induces a surjection on mod 2 homology.  An $\textup{H}\mathbb{F}_2$-quotient complex is denoted by a double-headed arrow as above.  When the maps involved are clear in the context, we may ignore the maps $i$ and $q$ and just say that $A$ is an $\textup{H}\mathbb{F}_2$-subcomplex of $B$, and $C$ is an $\textup{H}\mathbb{F}_2$-quotient complex of $B$.

Furthermore, $D$ is an \textit{$\textup{H}\mathbb{F}_2$-subquotient} of $B$ if $D$ is either an $\textup{H}\mathbb{F}_2$-subcomplex of an $\textup{H}\mathbb{F}_2$-quotient complex of $B$ or an $\textup{H}\mathbb{F}_2$-quotient complex of an $\textup{H}\mathbb{F}_2$-subcomplex of $B$.
\end{df}

Note that from Definition~\ref{hf2}, $\textup{H}\mathbb{F}_2$-subcomplexes and $\textup{H}\mathbb{F}_2$-quotient complexes are \emph{not} necessarily subcomplexes and quotient complexes on the point-set level.  Our definitions should be thought of as in the homological or homotopical sense.  A motivating example to illustrate this is the following: the top cell of the spectrum $\mathbb{R}\textup{P}_1^3$ splits off, so there is a map from $S^3$ to $\mathbb{R}\textup{P}_1^3$ that induces an injection on mod 2 homology.  Therefore $S^3$ is an $\textup{H}\mathbb{F}_2$-subcomplex of $\mathbb{R}\textup{P}_1^3$ in our sense.  However, on the point-set level, the image of the attaching map is not a point and so $S^3$ is not a subcomplex of $\mathbb{R}\textup{P}_1^3$ in the classical sense. 

It follows directly from Definition~\ref{hf2} that if $(A, i)$ is an $\textup{H}\mathbb{F}_2$-subcomplex of $B$, then the cofiber of $i$ is an $\textup{H}\mathbb{F}_2$-quotient complex of $B$.  We will often denote this quotient complex as $B/A$.  Dually, if $(C, q)$ is an $\textup{H}\mathbb{F}_2$-quotient complex of $B$, then the fiber of $q$ is an $\textup{H}\mathbb{F}_2$-subcomplex of $B$.

%\begin{rem}
%Later in this paper, when we talk about	a certain cell or a certain skeleton of a CW spectrum $A$, we mean an $\textup{H}\mathbb{F}_2$-subquotient of $A$. By cellular approximation theorem and the fact that $\mathbb{F}_2$ is a field, this $\textup{H}\mathbb{F}_2$-subquotient of $A$ is unique up to homotopy equivalence.
%\end{rem}

The following lemma is useful in constructing $\textup{H}\mathbb{F}_2$-subquotients.

\begin{lem} \label{pbpfhf2}
Suppose that $(A, i)$ is an $\textup{H}\mathbb{F}_2$-subcomplex of $B$.  Let $C$ be the cofiber of $i$ and let $(D, j)$ be an $\textup{H}\mathbb{F}_2$-subcomplex of $C$. Define $E$ to be the homotopy pullback of $D$ along $B\rightarrow C$. Then $E$ is an $\textup{H}\mathbb{F}_2$-subcomplex of $B$. Moreover, $A$ is an $\textup{H}\mathbb{F}_2$-subcomplex of $E$ with quotient $D$.

Dually, suppose $(C, q)$ is an $\textup{H}\mathbb{F}_2$-quotient complex of $B$. Let $A$ be the fiber of $q$. let $(F, p)$ be an $\textup{H}\mathbb{F}_2$-quotient complex of $A$. Define $G$ to be the homotopy pushout of $F$ along $A \rightarrow B$. We have that $G$ is an $\textup{H}\mathbb{F}_2$-quotient complex of $B$. Moreover, $C$ is an $\textup{H}\mathbb{F}_2$-quotient complex of $G$ with fiber $F$.
\end{lem}

Lemma~\ref{pbpfhf2} follows from the short exact sequences of homology induced by the following commutative diagrams of cofiber sequences and diagram chasing.
$$\begin{tikzcd}
A \ar[r, hookrightarrow] \ar[d, equal] & E \ar[r, twoheadrightarrow] \ar[d, hookrightarrow]  & D \ar[d, hookrightarrow, "j"] \\
A \ar[r, hookrightarrow, "i"] & B \ar[r, twoheadrightarrow] & C
\end{tikzcd}$$

$$\begin{tikzcd}
A \ar[r, hookrightarrow] \ar[d, twoheadrightarrow, "p"] & B \ar[r, twoheadrightarrow, "q"] \ar[d, twoheadrightarrow] & C \ar[d, equal] \\ 
F \ar[r, hookrightarrow] & G \ar[r, twoheadrightarrow] & C
\end{tikzcd}$$

\begin{comment}
\begin{displaymath}
    \xymatrix{
A \ar@{^{(}->}[r] \ar@{=}[d] & E \ar@{->>}[r] \ar@{^{(}->}[d] & D \ar@{^{(}->}[d]^-j\\
A \ar@{^{(}->}[r]^-i & B \ar@{->>}[r] & C \\
A \ar@{^{(}->}[r] \ar@{->>}[d]^-p & B \ar@{->>}[r]^-q \ar@{->>}[d] & C \ar@{=}[d] \\
F \ar@{^{(}->}[r] & G \ar@{->>}[r] & C
    }
\end{displaymath}
\end{comment}

\begin{df} \rm
For any element $\alpha$ in the stable homotopy groups of spheres, we say that there is an \textit{$\alpha$-attaching map} from dimension $n$ to dimension $n+|\alpha|+1$ in a CW spectrum $Z$ if $\Sigma^{n}C\alpha$ is an $\textup{H}\mathbb{F}_2$-subquotient of $Z$.  Here, $|\alpha|$ is the degree of $\alpha$ and $C\alpha$ is the cofiber of $\alpha$.
\end{df}

\begin{lem} \label{2eta}
	Suppose that $Z$ is a CW spectrum, with only one cell in dimension $n$. Then the following claims hold: 
	\begin{enumerate}
		\item There is a 2-attaching map from dimension $n$ to dimension $n+1$ in $Z$ if and only if the map
		$$Sq^1: H^n(Z; \mathbb{F}_2) \longrightarrow H^{n+1}(Z; \mathbb{F}_2)$$
		is nonzero.
		\item There is an $\eta$-attaching map from dimension $n$ to dimension $n+2$ in $Z$ if and only if the map
		$$Sq^2: H^n(Z; \mathbb{F}_2) \longrightarrow H^{n+2}(Z; \mathbb{F}_2)$$
		is nonzero.
	\end{enumerate} 
\end{lem}

\begin{proof}
	This follows from naturality and the fact that $Sq^1\neq 0$ in $H^*(C2; \mathbb{F}_2)$ and $Sq^2\neq 0$ in $H^*(C\eta; \mathbb{F}_2)$. 
\end{proof}

%%%%%%%%%%%%%%
\subsection{The $2$ and $\eta$-attaching maps in $X(m)$}
%All homology are understood to be in $\mathbb{F}_2$-coefficients.   
Recall that 
$$X(m) = \textup{Thom}(\textup{BPin}(2), -m\lambda).$$

\begin{prop} \label{xmhomology}
The mod 2 homology of $X(m)$ is as follows: 
\begin{itemize}
	\item For $m \equiv 0 \pmod{4}$, 
	\begin{equation*}
 H_j X(m) = \left\{
 \begin{array}{ll}
    \mathbb{F}_2 &j\equiv 0,1,2 \pmod{4},\\
    0 &j\equiv 3 \pmod{4}. 
      \end{array}
 \right.
\end{equation*}

\item For $m \equiv 1\pmod{4}$,
	\begin{equation*}
 H_j X(m) = \left\{
 \begin{array}{ll}
    \mathbb{F}_2 &j\equiv 0,1,3 \pmod{4},\\
    0 &j\equiv 2 \pmod{4}. 
      \end{array}
 \right.
\end{equation*}

\item For $m \equiv 2\pmod{4}$,
	\begin{equation*}
 H_j X(m) = \left\{
 \begin{array}{ll}
    \mathbb{F}_2 &j\equiv 0,2,3 \pmod{4},\\
    0 &j\equiv 1 \pmod{4}. 
      \end{array}
 \right.
\end{equation*}

\item For $m \equiv 3\pmod{4}$,
	\begin{equation*}
 H_j X(m) = \left\{
 \begin{array}{ll}
    \mathbb{F}_2 &j\equiv 1,2,3 \pmod{4},\\
    0 &j\equiv 0 \pmod{4}. 
      \end{array}
 \right.
\end{equation*}
\end{itemize}
\end{prop}

\begin{proof}
When $m=0$, $X(0) = \textup{BPin}(2)$, which is a bundle over $\mathbb{H}\textup{P}^\infty$ with fiber $\mathbb{R}\textup{P}^2$.  The corresponding Serre spectral sequence collapses at the $E_2$-page, from which we obtain a computation for $H_*X(0)$.  

The homologies for all the other $X(m)$'s follow from the homology of $X(0)$ and the Thom isomorphism. 
\end{proof}

Recall from Proposition~\ref{prop:TransferMaps} that there is a cofiber sequence
\begin{equation} \label{diagram:cofiberSequenceTransfer}
	\xymatrix{
	X(m+1) \ar[rr]^-{i(m+1,m)} & & X(m) \ar[rr]^-{s_m} & & \Sigma^{-m}\mathbb{C}P^\infty
	}
\end{equation}
for every $m \geq 0$.
\begin{lem} \label{cpx}
	The induced homomorphisms ${i(m+1,m)}_*$ and ${s_m}_*$ on mod 2 homologies can be described as follows:
	\begin{enumerate}
		\item The map 
		$${i(m+1,m)}_*:H_j X(m+1) \longrightarrow H_j X(m)$$ 
		is an isomorphism if and only if
		\begin{itemize}
		\item $m \equiv 0 \pmod{4}$ and $j \equiv 0,1\pmod{4}$;
		\item $m \equiv 1\pmod{4}$ and $j \equiv 0,3\pmod{4}$;
		\item $m \equiv 2\pmod{4}$ and $j \equiv 2,3\pmod{4}$;
		\item $m \equiv 3\pmod{4}$ and $j \equiv 1,2\pmod{4}$.
	\end{itemize}
	In other words, ${i(m+1,m)}_*$ is an isomorphism when both the domain and the codomain are nonzero.
	\item The map 
		$${s_m}_*:H_j X(m) \longrightarrow H_j (\Sigma^{-m}\mathbb{C}P^\infty)$$ 
		is an isomorphism if and only if 
		\begin{itemize}
		\item $m \equiv 0\pmod{4}$ and $j \equiv 2\pmod{4}$;
		\item $m \equiv 1\pmod{4}$ and $j \equiv 1\pmod{4}$;
		\item $m \equiv 2\pmod{4}$ and $j \equiv 0\pmod{4}$;
		\item $m \equiv 3\pmod{4}$ and $j \equiv 3\pmod{4}$.
	\end{itemize}
	\end{enumerate}
\end{lem}

Intuitively, part $(2)$ of Lemma~\ref{cpx} is saying that for the cells in $\Sigma^{-m}\mathbb{C}P^\infty$, the ones in dimensions $4k+2-m$ come from $X(m)$, and the ones in dimensions $4k-m$ come from $\Sigma X(m+1)$.

\begin{proof}
	The proofs for both part $(1)$ and $(2)$ follow from the associated long exact sequences on mod 2 homology groups from the cofiber sequence~(\ref{diagram:cofiberSequenceTransfer}). 
\end{proof}

\begin{prop} \label{xmhomologysq}
In the mod 2 homology of $X(m)$,
\begin{enumerate}[leftmargin = *]
	\item $$Sq^1:H^j X(m)\longrightarrow H^{j+1} X(m)$$
	is nonzero if and only if 
	\begin{itemize}
		\item $m \equiv 0 \pmod{4}$ and $j \equiv 1 \pmod{4}$;
		\item $m \equiv 1 \pmod{4}$ and $j \equiv 3\pmod{4}$;
		\item $m \equiv 2\pmod{4}$ and $j \equiv 3\pmod{4}$;
		\item $m \equiv 3\pmod{4}$ and $j \equiv 1\pmod{4}$.
	\end{itemize}
	\item $$Sq^2:H^j X(m)\longrightarrow H^{j+2} X(m)$$
	is nonzero if and only if 
	\begin{itemize}
		\item $m \equiv 1\pmod{4}$ and $j \equiv 3\pmod{4}$;
		\item $m \equiv 2\pmod{4}$ and $j \equiv 2\pmod{4}$.
	\end{itemize}
\end{enumerate}
\end{prop}

\begin{proof}
Recall that $\textup{BPin}(2)$ is a bundle over $\mathbb{H}\textup{P}^\infty$ with fiber $\mathbb{R}\textup{P}^2$.  The existence of the $Sq^1$'s and the $Sq^2$'s in $H^*X(0) = H^*\textup{BPin}(2)$ follows from the collapse of the Serre spectral sequence.  More precisely, 
$$H^*\textup{BPin}(2) = \mathbb{F}_2[q, v]/ (q^3 = 0)$$
where $|q| = 1$ and $|v| = 4$.  If we denote $Sq = \sum_{i \geq 0} Sq^i$ to be the total Steenrod squaring operation, then 
\begin{eqnarray*}
Sq(1) &=& 1 + \sum_{i \geq 3} Sq^i(1), \\ 
Sq(q) &=& q + q^2 + \sum_{i \geq 3} Sq^i(q), \\
Sq(q^2) &=& q^2 + \sum_{i \geq 3} Sq^i(q^2), \\
Sq(v) &=& v + \sum_{i \geq 3} Sq^i(v). \\
\end{eqnarray*}

To deduce the $Sq^1$'s and $Sq^2$'s in $X(m)$ when $m \geq 1$, note that by the Thom isomorphism, 
$$H^*X(m) = H^{*+m}X(0) \cdot \Phi_{-m\lambda}.$$
Here, $\Phi_{-m\lambda} \in H^{-m}X(m)$ is the Thom class associated with the virtual bundle $-m \lambda$.  For any $\alpha \in H^{*+m}X(0)$, 
\begin{eqnarray}
Sq(\alpha \cdot \Phi_{-m\lambda}) &=& Sq(\alpha) \cdot Sq(\Phi_{-m\lambda}) \nonumber \\
&=& Sq(\alpha) \cdot w(-m\lambda) \cdot \Phi_{-m\lambda}, \label{eq:StiefelWhitneySq1Sq2}
\end{eqnarray}
where $w(-)$ denotes the total Stiefel--Whitney class.  Since 
$$1 = w(0) = w(\lambda \oplus -\lambda) = w(\lambda) w(-\lambda)$$
and $w(\lambda) = 1 + q$, we have that
$$w(-m\lambda) = w(-\lambda)^m = \frac{1}{(1+q)^m} = (1 + q + q^2)^m.$$
Substituting this into equation~(\ref{eq:StiefelWhitneySq1Sq2}) and letting $\alpha$ take values from elements in $H^*X(0)$ produce all the $Sq^1$'s and $Sq^2$'s in $X(m)$.

\begin{comment}
For $Sq^1$'s, the existence of nonzero $Sq^1$'s follow from these in $H^* \mathbb{R}\textup{P}_{-m}^{-m+2}$. The non-existence of nonzero $Sq^1$'s follow from the Adem relation
	$$Sq^1Sq^1=0,$$
	and dimension reasons.
	
For $Sq^2$'s, the existence of nonzero $Sq^2$'s follow from these in $H^* \mathbb{R}\textup{P}_{-m}^{-m+2}$. The same reason shows the non-existence of nonzero $Sq^2$'s for the cases 
\begin{itemize}
\item	$m \equiv 0$ mod $4, \ j \equiv 0 $ mod 4,
\item $m \equiv 3$ mod $4, \ j \equiv 1 $ mod 4.
\end{itemize}
The other 4 cases are exactly the 4 cases in Lemma~\ref{cpx} that dually
 $${s_m}^*:H^j (\Sigma^{-m}\mathbb{C}P^\infty) \longrightarrow H^j X(m) $$ 
		is an isomorphism. The following commutative diagram shows that the $Sq^2$'s are zero.
\begin{displaymath}
	\xymatrix{
	H^j (\Sigma^{-m}\mathbb{C}P^\infty) \ar[rr]^{{s_m}^* \neq 0} \ar[dd]^{Sq^2 \neq 0} & &  H^j X(m) \ar[dd]^{Sq^2 = 0}\\
	& & \\
	H^{j+2} (\Sigma^{-m}\mathbb{C}P^\infty) \ar[rr]^{{s_m}^* = 0} & & H^{j+2} X(m)
	}
\end{displaymath}
\end{comment}		
\end{proof}

\begin{figure}
\begin{center}
\makebox[\textwidth]{\includegraphics[trim={0.8cm 2.6cm 0.6cm 10.2cm}, clip, page = 1, scale = 0.7]{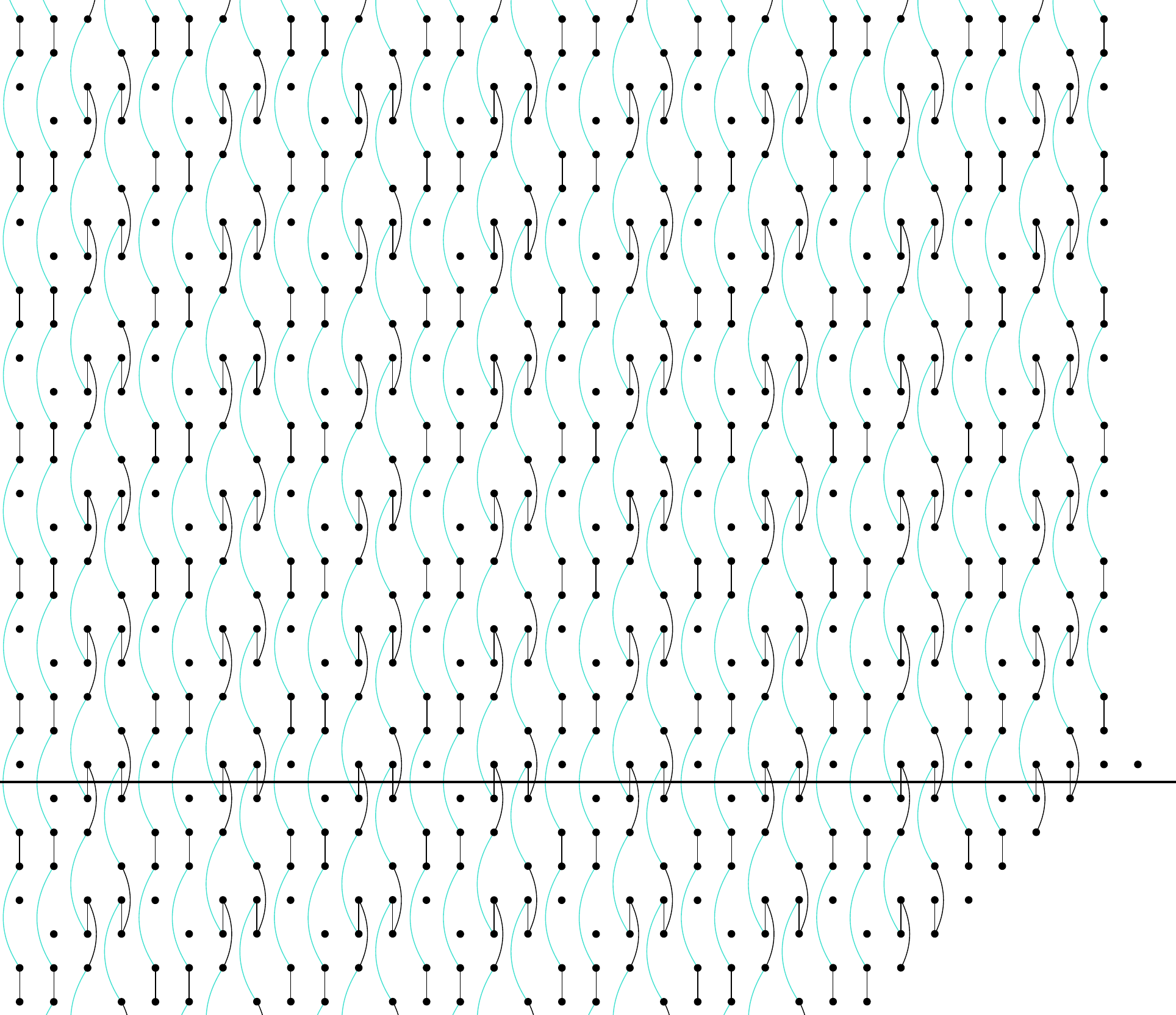}}
\end{center}
\begin{center}
\caption{Some attaching maps in $X(m)$.}
\hfill
\label{fig:Section4Diagram1}
\end{center}
\end{figure}

%We have the following 2, $\eta$ and $\eta^2$-attaching maps in $X(m)$:
\begin{cor} \label{2 and eta attaching maps within columns}
There are 2 and $\eta$-attaching maps in $X(m)$ if and only if they are marked in Figure~\ref{fig:Section4Diagram1}.
\end{cor}

\begin{proof}
	The 2 and $\eta$-attaching maps follow from Lemma~\ref{2eta} and Proposition~\ref{xmhomologysq}.
\end{proof}

\begin{lem} \label{eta attaching map between columns}
Suppose that $m$ and $j$ satisfy one of following conditions:
\begin{itemize}
		\item $m \equiv 0 \pmod{4}$ and $j \equiv 2\pmod{4}$;
		\item $m \equiv 1\pmod{4}$ and $j \equiv 1\pmod{4}$;
		\item $m \equiv 2\pmod{4}$ and $j \equiv 0\pmod{4}$;
		\item $m \equiv 3\pmod{4}$ and $j \equiv 3\pmod{4}$.
	\end{itemize}
	Then the map 
	\begin{displaymath}
		\xymatrix{
		S^{j+1} \ar@{=}[r] & X(m+1)^{j+1}_j \ar[rr] & & X(m)^{j+1}_j \ar@{=}[r] & S^j
		}
	\end{displaymath}
	is $\eta$.
	
\end{lem}

\begin{proof}
	By Lemma~\ref{cpx}, the cofiber the map is 
	$$(\Sigma^{-m}\mathbb{C}P)_j^{j+2}.$$
	Since there is a nonzero $Sq^2$ in its cohomology, this cofiber is indeed $\Sigma^j C\eta$.
\end{proof}

\subsection{$\eta^2$-attaching maps in $X(m)$}

\begin{prop} \label{eta square attaching maps within columns}
	There is an $\eta^2$-attaching map in $X(m)$ from dimension $j$ to dimension $(j+3)$ if and only if it is one of the following four cases (see Figure~\ref{fig:Section4Diagram1}): 
	\begin{itemize}
		\item $m \equiv 0 \pmod{4}$ and $j \equiv 2\pmod{4}$;
		\item $m \equiv 1\pmod{4}$ and $j \equiv 1\pmod{4}$;
		\item $m \equiv 2\pmod{4}$ and $j \equiv 3\pmod{4}$;
		\item $m \equiv 3\pmod{4}$ and $j \equiv 2\pmod{4}$.
	\end{itemize} 
\end{prop}

\begin{proof}
	For dimension reasons, there are eight cases of possible $\eta^2$-attaching maps in total.  We need to show that of these eight cases, four cases have $\eta^2$-attaching maps and four cases don't. Recall that $\pi_2 = \mathbb{Z}/2$, generated by $\eta^2$.

\begin{enumerate}[leftmargin = *]
		\item	
\textbf{Case 1:} $m \equiv 1 \pmod{4}$ and $j \equiv 1\pmod{4}$. Consider the map
$$X(m+1)_j^{j+3} \longrightarrow X(m)_j^{j+3}.$$
By Corollary~\ref{2 and eta attaching maps within columns}, the cells in dimension $j+2$ are not attached to the lower skeletons of $X(m+1)_j^{j+3}$ and $X(m)_j^{j+3}$. Therefore, they are $\textup{H}\mathbb{F}_2$-subcomplexes. Taking cofibers, we have the following commutative diagram:
\begin{displaymath}
		\xymatrix{
	\Sigma^{j+1} C\eta	\ar@{=}[r] & X(m+1)^{j+3}_j/S^{j+2} \ar@{-->}[rr] & & X(m)^{j+3}_j/S^{j+2} \ar@{=}[r] & \Sigma^{j} C\alpha \\
		 & X(m+1)^{j+3}_j \ar[rr] \ar@{->>}[u] & & X(m)^{j+3}_j \ar@{->>}[u] & \\ 
		 & S^{j+2} \ar[rr]^-{id} \ar@{^{(}->}[u] & & S^{j+2} \ar@{^{(}->}[u] &
		}
	\end{displaymath}
	Since $X(m)^{j+3}_j/S^{j+2}$ is a 2 cell complex, it must be the cofiber of a class $\alpha \in \pi_2$ in the stable homotopy groups of spheres.

\begin{displaymath}
		\xymatrix{
		*+<10pt>[o][F-]{_{j+3}} \ar@{-}[d]^{2} \ar@{-}@/_1pc/[dd]_{\eta} \ar[r]^{1} & *+<10pt>[o][F-]{_{j+3}} \ar@{-}[d]_{2} \ar@{-}@/^1pc/[ddd]^{\alpha} & & & *+<10pt>[o][F-]{_{j+3}} \ar@{-}@/_1pc/[dd]_{\eta} \ar[r]^{1} & *+<10pt>[o][F-]{_{j+3}} \ar@{-}@/^1pc/[ddd]_{\alpha}\\
		*+<10pt>[o][F-]{_{j+2}} \ar[r]^{1} & *+<10pt>[o][F-]{_{j+2}} & & & & \\
		*+<10pt>[o][F-]{_{j+1}} \ar[rd]_{\eta} & & & & *+<10pt>[o][F-]{_{j+1}} \ar[rd]_{\eta} & \\
		& *+<15pt>[o][F-]{{j}} & & & & *+<15pt>[o][F-]{{j}} \\
		X(m+1)^{j+3}_j \ar[r] & X(m)^{j+3}_j & & & X(m+1)^{j+3}_j/S^{j+2} \ar@{-->}[r] & X(m)^{j+3}_j/S^{j+2}
		}
	\end{displaymath}
	It is clear that we must have $\alpha = \eta^2$.  If it is not, then $X(m)^{j+3}_j/S^{j+2}$ would split as $S^j \vee S^{j+3}$, and we would have a map $$\Sigma^{j+1} C\eta \longrightarrow S^j$$ 
	whose restriction to the bottom cell is $\eta$ by Lemma~\ref{eta attaching map between columns}.  This is not possible. 
	\item \textbf{Case 2:} $m \equiv 2 \pmod{4}$ and $j \equiv 3\pmod{4}$.  Consider the map
	$$X(m)_j^{j+3} \longrightarrow X(m-1)_j^{j+3}.$$
	From the 2 and $\eta$-attaching maps in Corollary~\ref{2 and eta attaching maps within columns}, this map is the Spanier--Whitehead dual (up to suspension) of the map
$$X(m+1)_j^{j+3} \longrightarrow X(m)_j^{j+3}$$
in the case when $m \equiv 1\pmod{4}$ and $j \equiv 1\pmod{4}$. Therefore, we must have the $\eta^2$-attaching map.
\begin{displaymath}
		\xymatrix{
		*+<10pt>[o][F-]{_{j+3}} \ar[rrd]_{\eta}   \ar@{-}@/_1pc/[ddd]_{\eta^2}  &  & \\
		 & & *+<10pt>[o][F-]{_{j+2}} \ar@{-}@/^1pc/[dd]^{\eta}\\
		*+<10pt>[o][F-]{_{j+1}} \ar@{-}[d]^{2} \ar[rr]^{1} & & *+<10pt>[o][F-]{_{j+1}} \ar@{-}[d]_{2}  \\
		*+<15pt>[o][F-]{{j}} \ar[rr]^{1} & & *+<15pt>[o][F-]{{j}}  \\
		X(m)_j^{j+3} \ar[rr] & & X(m-1)_j^{j+3} 		}
	\end{displaymath}
	\item \textbf{Case 3:} $m \equiv 3\pmod{4}$ and $j \equiv 2\pmod{4}$.  The proof is similar to the case when $m \equiv 1\pmod{4}$ and $j \equiv 1\pmod{4}$. Consider the map
	$$X(m)_j^{j+3} \longrightarrow X(m-1)_j^{j+3}.$$
	By Corollary~\ref{2 and eta attaching maps within columns}, the cells in dimension $j+1$ are not attached to the lower skeletons of $X(m)_j^{j+3}$ and $X(m-1)_j^{j+3}$.  Therefore, they are $\textup{H}\mathbb{F}_2$-subcomplexes. Taking the cofibers, we have the following commutative diagram:
\begin{displaymath}
		\xymatrix{
	\Sigma^{j} C\phi	\ar@{=}[r] & X(m)^{j+3}_j/S^{j+1} \ar@{-->}[rr] & & X(m-1)^{j+3}_j/S^{j+1} \ar@{=}[r] &  \Sigma^{j} C\eta \\
		 & X(m)^{j+3}_j \ar[rr] \ar@{->>}[u] & & X(m-1)^{j+3}_j \ar@{->>}[u] & \\ 
		 & S^{j+1} \ar[rr]^-{id} \ar@{^{(}->}[u] & & S^{j+1} \ar@{^{(}->}[u] &
		}
	\end{displaymath}
	Since $X(m)^{j+3}_j/S^{j+1}$ is a 2 cell complex, it must be the cofiber of a class $\phi \in \pi_2$ in the stable homotopy groups of spheres.
\begin{displaymath}
		\xymatrix{
		*+<10pt>[o][F-]{_{j+3}} \ar[rd]_{\eta} \ar@{-}@/_1pc/[ddd]_{\phi}    &  & & & *+<10pt>[o][F-]{_{j+3}}  \ar[rd]_{\eta} \ar@{-}@/_1pc/[ddd]_{\phi}     &  \\
		  & *+<10pt>[o][F-]{_{j+2}} \ar@{-}[d]_{2} \ar@{-}@/^1pc/[dd]^{\eta} & & & & *+<10pt>[o][F-]{_{j+2}} \ar@{-}@/^1pc/[dd]^{\eta} \\
		*+<10pt>[o][F-]{_{j+1}} \ar[r]^{1}  & *+<10pt>[o][F-]{_{j+1}} & &  &  & \\
		*+<15pt>[o][F-]{{j}} \ar[r]^{1} & *+<15pt>[o][F-]{{j}} & & & *+<15pt>[o][F-]{{j}} \ar[r]^{1} & *+<15pt>[o][F-]{{j}} \\
		X(m)^{j+3}_j \ar[r] & X(m-1)^{j+3}_j & & & X(m)^{j+3}_j/S^{j+1} \ar@{-->}[r] & X(m-1)^{j+3}_j/S^{j+1}
		}
	\end{displaymath}
	It is clear that we must have $\phi = \eta^2$.  If it is not, then $X(m)^{j+3}_j/S^{j+1}$ would split as $S^j \vee S^{j+3}$, and we would have a map $$S^{j+3}\longrightarrow  \Sigma^{j} C\eta.$$ 
	  By Lemma~\ref{eta attaching map between columns}, post-composing this map with the quotient map $\Sigma^{j} C\eta \twoheadrightarrow S^{j+2}$ would give $\eta$, which is not possible. 
\item \textbf{Case 4:} $m \equiv 0\pmod{4}$ and $j \equiv 2\pmod{4}$.  Consider the map
	$$X(m+1)_j^{j+3} \longrightarrow X(m)_j^{j+3}.$$
	From the 2 and $\eta$-attaching maps in Corollary~\ref{2 and eta attaching maps within columns}, this is the Spanier--Whitehead dual (up to suspension) of the map
$$X(m)_j^{j+3} \longrightarrow X(m-1)_j^{j+3}$$
in the case when $m \equiv 3 \pmod{4}$ and $j \equiv 2\pmod{4}$. Therefore, we must have the $\eta^2$-attaching map.  Alternatively, one may also prove this $\eta^2$-attaching map by considering the map
$$X(m)_j^{j+3} \longrightarrow X(m-1)_j^{j+3}.$$
\end{enumerate}
\noindent Now, we will show that in the other four cases, there do not exist $\eta^2$-attaching maps. 
\begin{enumerate}[leftmargin=*]
\item	\textbf{Case 1:} $m \equiv 3\pmod{4}$ and $j \equiv 3\pmod{4}$. Consider the map
$$X(m+1)_j^{j+3} \longrightarrow X(m)_j^{j+3}.$$
By Corollary~\ref{2 and eta attaching maps within columns}, the cells in dimension $j+2$ are not attached to the lower skeletons of $X(m+1)_j^{j+3}$ and $X(m)_j^{j+3}$.  Therefore, they are $\textup{H}\mathbb{F}_2$-subcomplexes. Taking the cofibers, we have the following commutative diagram:
\begin{displaymath}
		\xymatrix{
	S^{j+1} \vee S^{j+3}	\ar@{=}[r] & X(m+1)^{j+3}_j/S^{j+2} \ar@{-->}[rr] & & X(m)^{j+3}_j/S^{j+2} \ar@{=}[r] & \Sigma^{j} C\alpha' \\
		 & X(m+1)^{j+3}_j \ar[rr] \ar@{->>}[u] & & X(m)^{j+3}_j \ar@{->>}[u] & \\ 
		 & S^{j+2} \ar[rr]^-{id} \ar@{^{(}->}[u] & & S^{j+2} \ar@{^{(}->}[u] &
		}
	\end{displaymath}
	Since $X(m)^{j+3}_j/S^{j+2}$ is a 2 cell complex, it must be the cofiber of a class $\alpha' \in \pi_2$ in the stable homotopy groups of spheres.
\begin{displaymath}
		\xymatrix{
		*+<10pt>[o][F-]{_{j+3}} \ar@{-}[d]^{2}  \ar[r]^{1} & *+<10pt>[o][F-]{_{j+3}} \ar@{-}[d]_{2} \ar@{-}@/^1pc/[ddd]^{\alpha'} & & & *+<10pt>[o][F-]{_{j+3}} \ar[r]^{1} & *+<10pt>[o][F-]{_{j+3}} \ar@{-}@/^1pc/[ddd]_{\alpha'}\\
		*+<10pt>[o][F-]{_{j+2}} \ar[r]^{1} & *+<10pt>[o][F-]{_{j+2}} & & & & \\
		*+<10pt>[o][F-]{_{j+1}} \ar[rd]_{\eta} & & & & *+<10pt>[o][F-]{_{j+1}} \ar[rd]_{\eta} & \\
		& *+<15pt>[o][F-]{{j}} & & & & *+<15pt>[o][F-]{{j}} \\
		X(m+1)^{j+3}_j \ar[r] & X(m)^{j+3}_j & & & X(m+1)^{j+3}_j/S^{j+2} \ar@{-->}[r] & X(m)^{j+3}_j/S^{j+2}
		}
	\end{displaymath}
	It is clear that we must have $\alpha' = 0$.   Otherwise, we would have $\alpha' =\eta^2$ and there would be a map $$S^{j+3} \longrightarrow \Sigma^{j} C\eta^2.$$ 
	Post-composing this map with the quotient map $\Sigma^{j} C\eta^2 \twoheadrightarrow S^{j+3}$ gives us the identity map. This is not possible. 
	\item \textbf{Case 2:} $m \equiv 0\pmod{4}$ and $j \equiv 1\pmod{4}$.  Consider the map
	$$X(m)_j^{j+3} \longrightarrow X(m-1)_j^{j+3}.$$
	From the 2 and $\eta$-attaching maps in Corollary~\ref{2 and eta attaching maps within columns}, this is the Spanier--Whitehead dual (up to suspension) of the map
$$X(m+1)_j^{j+3} \longrightarrow X(m)_j^{j+3}$$
in the case $m \equiv 3 \pmod{4}$ and $j \equiv 3\pmod{4}$. Therefore, there cannot be an $\eta^2$-attaching map.

	\item \textbf{Case 3:} $m \equiv 1\pmod{4}$ and $j \equiv 0\pmod{4}$. Consider the map
$$X(m)_j^{j+3} \longrightarrow X(m-1)_j^{j+3}.$$
By Corollary~\ref{2 and eta attaching maps within columns}, the cells in dimension $j+1$ are not attached to the lower skeletons of $X(m)_j^{j+3}$ and $X(m-1)_j^{j+3}$.  Therefore, they are $\textup{H}\mathbb{F}_2$-subcomplexes. Taking cofibers, we have the following commutative diagram: 
\begin{displaymath}
		\xymatrix{
	\Sigma^{j} C\phi'	\ar@{=}[r] & X(m)^{j+3}_j/S^{j+1} \ar@{-->}[rr] & & X(m-1)^{j+3}_j/S^{j+1} \ar@{=}[r] & S^j \vee S^{j+1}  \\
		 & X(m)^{j+3}_j \ar[rr] \ar@{->>}[u] & & X(m-1)^{j+3}_j \ar@{->>}[u] & \\ 
		 & S^{j+1} \ar[rr]^-{id} \ar@{^{(}->}[u] & & S^{j+1} \ar@{^{(}->}[u] &
		}
	\end{displaymath}
	Since $X(m)^{j+3}_j/S^{j+1}$ is a 2 cell complex, it must be the cofiber of a class $\phi' \in \pi_2$ in the stable homotopy groups of spheres.
\begin{displaymath}
		\xymatrix{
		*+<10pt>[o][F-]{_{j+3}} \ar[rd]_{\eta} \ar@{-}@/_1pc/[ddd]_{\phi'}    &  & & & *+<10pt>[o][F-]{_{j+3}}  \ar[rd]_{\eta} \ar@{-}@/_1pc/[ddd]_{\phi'}     &  \\
		  & *+<10pt>[o][F-]{_{j+2}} \ar@{-}[d]_{2}  & & & & *+<10pt>[o][F-]{_{j+2}}  \\
		*+<10pt>[o][F-]{_{j+1}} \ar[r]^{1}  & *+<10pt>[o][F-]{_{j+1}} & &  &  & \\
		*+<15pt>[o][F-]{{j}} \ar[r]^{1} & *+<15pt>[o][F-]{{j}} & & & *+<15pt>[o][F-]{{j}} \ar[r]^{1} & *+<15pt>[o][F-]{{j}} \\
		X(m)^{j+3}_j \ar[r] & X(m-1)^{j+3}_j & & & X(m)^{j+3}_j/S^{j+1} \ar@{-->}[r] & X(m-1)^{j+3}_j/S^{j+1}
		}
	\end{displaymath}
	It is clear that we must have $\phi' = 0$. Otherwise, if $\phi' = \eta^2$, we would have a map $$ \Sigma^{j} C\eta^2 \longrightarrow S^j$$ 
	  whose restriction on the bottom cell is the identity. This is not possible. 
\item \textbf{Case 4:} $m \equiv 2\pmod{4}$ and $j \equiv 0\pmod{4}$. Consider the map
	$$X(m+1)_j^{j+3} \longrightarrow X(m)_j^{j+3}.$$
	From the 2 and $\eta$-attaching maps in Corollary~\ref{2 and eta attaching maps within columns}, this is the Spanier--Whitehead dual (up to suspension) of the map
$$X(m)_j^{j+3} \longrightarrow X(m-1)_j^{j+3}$$
in the case when $m \equiv 1\pmod{4}$ and $j \equiv 0\pmod{4}$.  Therefore, there cannot be an $\eta^2$-attaching map.
\end{enumerate}
\end{proof}

%%%%%%%%%%%%%%%%%%%%%%
\subsection{Periodicity in $X(m)$}
\begin{prop}\label{prop:periodicityX(m)}
For any $m,n,k\geq 0$, there is an equivalence
$$X(m)^{4n+6-m}_{4n-m}\simeq\Sigma^{4k}X(m+4k)^{4n+6-m-4k}_{4n-m-4k}.$$
\end{prop}
\begin{proof}
Given any two $G$-representations $U$ and $V$, there is a cofiber sequence 
$$S(U)_+ \longrightarrow S(U\oplus V)_+ \longrightarrow S(V)_+ \wedge S^U.$$
Let $U = n\mathbb{H}$ and $V = \infty \mathbb{H}$.  The cofiber sequence 
$$
S(n\mathbb{H})_{+}\longrightarrow S(\infty \mathbb{H})_{+}\longrightarrow S(\infty \mathbb{H})_+\wedge S^{n\mathbb{H}}
$$
produces the cofiber sequence
$$
\left(S(n\mathbb{H})_{+}\wedge S^{-m\widetilde{\mathbb{R}}}\right)_{h\Pin}\longrightarrow \left(S(\infty \mathbb{H})_{+}\wedge S^{-m\widetilde{\mathbb{R}}}\right)_{h\Pin} \longrightarrow \left(S(\infty \mathbb{H})_+\wedge S^{n\mathbb{H}-m\widetilde{\mathbb{R}}}\right)_{h\Pin}.
$$
This cofiber sequence can be rewritten as 
$$\begin{tikzcd}
X(m)^{4n-m-1} \ar[r, hookrightarrow]&  X(m) \ar[r]& \Thom(B\Pin,nH-m\lambda).
\end{tikzcd}
$$
Here, $H$ and $\lambda$ denote the bundles over $B\Pin$ that are associated to the representations $\mathbb{H}$ and $\widetilde{\mathbb{R}}$, respectively.  From this, we deduce that
$$
X(m)_{4n-m}=\Thom(B\Pin,nH-m\lambda).
$$

Let $B\Pin^{6}$ be the  $6$-skeleton of $B\Pin$.  We have the equality
$$
X(m)_{4n-m}^{4n-m+6}=\Thom(B\Pin^{6},(nH-m\lambda)|_{B\Pin^{6}}).
$$
To finish the proof, it suffices to show that the bundle $4\lambda|_{B\Pin^{6}}$ is stably trivial.  Note that since $\omega_{1}(4\lambda)=\omega_{2}(4\lambda)=0$, this bundle is spin and can be classified by a stable map 
$$f:B\Pin^{6}\rightarrow B\text{Spin}.$$ 
Moreover, since $p_{1}(4\lambda)=4p_{1}(\lambda)=0$, $f$ can be further be lifted to $B\text{String}$.  It follows that $f = 0$ because $B\text{String}$ is $7$-connected.
\end{proof}

%%%%%%%%%%%%%%%%%%%%%%%
\subsection{Some $\textup{H}\mathbb{F}_2$-subquotients of $X(m)$}
In this subsection, we define and discuss some $\textup{H}\mathbb{F}_2$-subquotients of $X(m)$.

We start with the 3 cell complex $X(8k+4)_{8k+1}^{8k+4}$ and the 4 cell complex ${X(8k+3)_{8k-5}^{8k-1}}$.

\begin{lem} \label{the 3 cell complex wk}
	The 3 cell complex $X(8k+4)_{8k+1}^{8k+4}$ splits:
	$$X(8k+4)_{8k+1}^{8k+4} \simeq S^{8k+4} \vee \Sigma^{8k+1} C2.$$
\end{lem}

\begin{proof}
	By Corollary~\ref{2 and eta attaching maps within columns} and Proposition~\ref{eta square attaching maps within columns}, there are no $\eta$ and $\eta^2$-attaching maps in $X(8k+4)_{8k+1}^{8k+4}$.  The claim then follows from the fact that $\pi_1 = \mathbb{Z}/2$ and $\pi_2 = \mathbb{Z}/2$ are generated by $\eta$ and $\eta^2$ respectively.
\end{proof}

\begin{lem} \label{the 4 cell complex bw 12 locks}
	The $4$-cell complex $X(8k+3)_{8k-5}^{8k-1}$ splits: 
	$$X(8k+3)_{8k-5}^{8k-1} \simeq \Sigma^{8k-5} C\nu \vee \Sigma^{8k-3} C2.$$
	%as a wedge of $\Sigma^{8k-5} C\nu$ and $\Sigma^{8k-3} C2$ up to homotopy.
\end{lem}

\begin{displaymath}
    \xymatrix{
 *+<10pt>[o][F-]{_{8k-1}}  \ar@{-}@/_2pc/[dddd]_{\nu} \\
    *+<10pt>[o][F-]{_{8k-2}} \ar@{-}[d]^{2} \\
    *+<10pt>[o][F-]{_{8k-3}} \\
    \\
    *+<10pt>[o][F-]{_{8k-5}} }
\end{displaymath}

\begin{proof}
Consider the $(8k-2)$-skeleton of $X(8k+3)_{8k-5}^{8k-1}$, which is the 3 cell complex $X(8k+3)_{8k-5}^{8k-2}$. By Corollary~\ref{2 and eta attaching maps within columns} and Proposition~\ref{eta square attaching maps within columns}, there are no $\eta$ and $\eta^2$-attaching maps in $X(8k+3)_{8k-5}^{8k-2}$. Since $\pi_1= \mathbb{Z}/2$ and $\pi_2 = \mathbb{Z}/2$ are generated by $\eta$ and $\eta^2$ respectively, we have the following equivalence: 
$$X(8k+3)_{8k-5}^{8k-2} \simeq S^{8k-5} \vee \Sigma^{8k-3}C2.$$
This gives $\Sigma^{8k-3}C2$ as an $\textup{H}\mathbb{F}_2$-subcomplex of $X(8k+3)_{8k-5}^{8k-2}$, and, therefore, as an $\textup{H}\mathbb{F}_2$-subcomplex of $X(8k+3)_{8k-5}^{8k-1}$.

Now consider the attaching map 
$$S^{8k-2} \longrightarrow X(8k+3)_{8k-3}^{8k-2}$$
whose cofiber is $X(8k+3)_{8k-3}^{8k-1}$. By Corollary~\ref{2 and eta attaching maps within columns}, the cell in dimension $8k-1$ is not attached to the cell in dimension $8k-2$ by $2$.  It is also not attached to the cell in dimension $8k-3$ by $\eta$.  Therefore, it is null homotopic and we have the following homotopy equivalence:
$$X(8k+3)_{8k-3}^{8k-1} \simeq  \Sigma^{8k-3}C2  \vee S^{8k-1}.$$ 
This gives $S^{8k-1}$ as an $\textup{H}\mathbb{F}_2$-subcomplex of $X(8k+3)_{8k-3}^{8k-1}$. 

By Lemma~\ref{pbpfhf2}, we can pullback $S^{8k-1}$ along the quotient map 
$$\xymatrix{X(8k+3)_{8k-5}^{8k-1} \ar@{-->}[r] & X(8k+3)_{8k-3}^{8k-1} }$$ 
and obtain a 2 cell complex as an $\textup{H}\mathbb{F}_2$-subcomplex of $X(8k+3)_{8k-5}^{8k-1}$.
\begin{displaymath}
    \xymatrix{
S^{8k-5} \ar@{^{(}->}[r] \ar@{=}[d] & \Sigma^{8k-5} C\nu \ar@{->>}[r] \ar@{^{(}->}[d] & S^{8k-1} \ar@{^{(}->}[d]\\
S^{8k-5} \ar@{^{(}->}[r] & X(8k+3)_{8k-5}^{8k-1} \ar@{->>}[r] & X(8k+3)_{8k-3}^{8k-1}
    }
\end{displaymath}

We claim that this 2 cell complex must be $\Sigma^{8k-5} C\nu$. In fact, consider the map
$$X(8k+3)_{8k-5}^{8k-1} \longrightarrow (\Sigma^{-8k-3}\mathbb{C}P)_{8k-5}^{8k-1}$$
induced by the map $X(8k+3) \rightarrow \Sigma^{-8k-3}\mathbb{C}P^\infty$. Since there is a nontrivial $Sq^4$ on $H^{8k-5} (\Sigma^{-8k-3}\mathbb{C}P)_{8k-5}^{8k-1}$, we must have a nontrivial $Sq^4$ on $H^{8k-5} X(8k+3)_{8k-5}^{8k-1}$ and the 2 cell complex. This produces the $\nu$-attaching map. Therefore, $\Sigma^{8k-5} C\nu$ is an $\textup{H}\mathbb{F}_2$-subcomplex of $X(8k+3)_{8k-5}^{8k-1}$.

In summary, we have shown that both $\Sigma^{8k-3}C2$ and $\Sigma^{8k-5} C\nu$ are $\textup{H}\mathbb{F}_2$-subcomplexes of ${X(8k+3)_{8k-5}^{8k-1}}$. Their wedge gives an isomorphism on mod 2 homology and is therefore a homotopy equivalence. This completes the proof of the lemma. 
\end{proof}

\begin{prop} \label{prop ek}
There exists a 4 cell complex $E(k)$ that is an $\textup{H}\mathbb{F}_2$-subcomplex of $X(8k+4)_{8k-4}^{8k+4}$.  It has cells in dimensions $8k-4$, $8k-3$, $8k$ and $8k+4$.  
\end{prop}

\begin{proof}
First, by Corollary~\ref{2 and eta attaching maps within columns}, the cells in dimensions $8k-2$ and $8k$ are not attached by $\eta$ in $X(8k+4)$.  Therefore, there is an equivalence
	$$X(8k+4)_{8k-2}^{8k} \simeq S^{8k} \vee S^{8k-2}.$$
	In particular, we have $S^{8k-2}$ as an $\textup{H}\mathbb{F}_2$-quotient complex of $X(8k+4)_{8k-2}^{8k}$ and $X(8k+4)_{8k-4}^{8k}$, and $S^{8k}$ as an $\textup{H}\mathbb{F}_2$-subcomplex of $X(8k+4)_{8k-2}^{8k}$ and ${X(8k+4)_{8k-2}^{8k+2}}$. 
	
	Define $F(k)$ to be the fiber of the following composition: 
	\begin{displaymath}
		\xymatrix{
		X(8k+4)_{8k-4}^{8k} \ar@{->>}[r] & X(8k+4)_{8k-2}^{8k}  \ar@{->>}[r] & S^{8k-2}.
		}
	\end{displaymath}
Then $F(k)$ is a 3 cell complex with cells in dimensions $8k-4, \ 8k-3$ and $8k$.  This 3 cell complex is an $\textup{H}\mathbb{F}_2$-subcomplex of $X(8k+4)_{8k-4}^{8k}$ and $X(8k+4)_{8k-4}^{8k+4}$. It is clear that we have the following commutative diagram in the homotopy category:
\begin{displaymath}
		\xymatrix{
		X(8k+4)_{8k-4}^{8k-3} \ar@{=}[r] \ar@{^{(}->}[d] & X(8k+4)_{8k-4}^{8k-3} \ar@{^{(}->}[d] & \\
		F(k) \ar@{->>}[d] \ar@{^{(}->}[r] & X(8k+4)_{8k-4}^{8k+4} \ar@{->>}[r] \ar@{->>}[d] & X(8k+4)_{8k-4}^{8k+4}/F(k) \ar@{=}[d] \\
		S^{8k} \ar@{^{(}->}[r] &  X(8k+4)_{8k-2}^{8k+4} \ar@{->>}[r] & X(8k+4)_{8k-2}^{8k+4}/S^{8k}
		}
	\end{displaymath}	
Therefore, we can identify the 4 cell complex
$$X(8k+4)_{8k-4}^{8k+4}/F(k) = {X(8k+4)_{8k-2}^{8k+4}/S^{8k}}.$$

Now, we claim that the top cell of $X(8k+4)_{8k-4}^{8k+4}/F(k)$ splits off. In fact, consider the attaching map 
	$$S^{8k+3} \longrightarrow X(8k+4)_{8k-2}^{8k+2}/S^{8k},$$
	whose cofiber is $X(8k+4)_{8k-2}^{8k+4}/S^{8k}$. We will show that this attaching map is null-homotopic. Consider the $E_1$-page of the Atiyah--Hirzebruch spectral sequence of the 3 cell complex $X(8k+4)_{8k-2}^{8k+2}/S^{8k}$ that converges to its $(8k+3)$-homotopy groups:
	$$\pi_{8k+3}S^{8k+2} \oplus \pi_{8k+3}S^{8k+1} \oplus \pi_{8k+3}S^{8k-2} = \pi_1 \oplus \pi_2 \oplus \pi_5 = \mathbb{Z}/2 \oplus \mathbb{Z}/2.$$
The right hand side is generated by 
$$\eta[8k+2] \in \pi_{8k+3}S^{8k+2} \text{ and } \eta^2[8k+1] \in \pi_{8k+3}S^{8k+1}.$$ 
By Corollary~\ref{2 and eta attaching maps within columns} and Proposition~\ref{eta square attaching maps within columns}, there are no $\eta$ and $\eta^2$-attaching maps in $X(8k+4)_{8k-2}^{8k+4}/S^{8k}$.  This proves our claim. 

Therefore, we have a splitting
	$$X(8k+4)_{8k-2}^{8k+4}/S^{8k} \simeq S^{8k+4} \vee X(8k+4)_{8k-2}^{8k+2}/S^{8k}.$$
	In particular, this splitting exhibits $S^{8k+4}$ as an $\textup{H}\mathbb{F}_2$-subcomplex of 
	$$X(8k+4)_{8k-2}^{8k+4}/S^{8k} = X(8k+4)_{8k-4}^{8k+4}/F(k).$$
	
Lastly, we pullback $S^{8k+4}$ along the quotient map $$\xymatrix{X(8k+4)_{8k-4}^{8k+4} \ar@{->>}[r] & X(8k+4)_{8k-4}^{8k+4}/F(k):}$$
	 \begin{displaymath}
		\xymatrix{
		F(k) \ar@{=}[d] \ar@{^{(}->}[r] & E(k) \ar@{->>}[r] \ar@{->>}[d] & S^{8k+4}  \ar@{^{(}->}[d] \\
		F(k) \ar@{^{(}->}[r] &  X(8k+4)_{8k-4}^{8k+4} \ar@{->>}[r] & X(8k+4)_{8k-4}^{8k+4}/F(k).
		}
	\end{displaymath}
	By Lemma~\ref{pbpfhf2}, $E(k)$ is an $\textup{H}\mathbb{F}_2$-subcomplex of $X(8k+4)_{8k-4}^{8k+4}$ with cells in dimensions $8k-4, \ 8k-3, \ 8k$ and $8k+4$.  This concludes the proof of the proposition. 
\end{proof}

%Following Proposition~\ref{prop ek} and its proof, we make the following definitions.

\begin{df}\rm
	Define $E(k)$ to be the 4 cell complex in Proposition~\ref{prop ek}.  Define $F(k)$ to be the $8k$-skeleton of $E(k)$. Define
	$$G(k) := X(8k+4)_{8k-4}^{\infty}/F(k)$$
	and $G(k)^{8k+1}$ to be its $(8k+1)$-skeleton. 
\end{df}

It is clear from Proposition~\ref{eta square attaching maps within columns} that 
$$G(k)^{8k+1} = \Sigma^{8k-2} C\eta^2.$$

\begin{prop} \label{prop:DefinitionY(k)}
There is a 2 cell complex $Y(k)$ with cells in dimensions $8k-4$ and $8k-8$, such that it is an $\textup{H}\mathbb{F}_2$-quotient complex of ${X(8k+4)^{8k-2}_{8k-8}}$.	
\end{prop}

\begin{proof}
It suffices to show that ${X(8k+4)^{8k-2}_{8k-8}}$ has an $\textup{H}\mathbb{F}_2$-subcomplex $W$ with cells in dimensions $8k-7, 8k-6, 8k-3$ and $8k-2$.

Firstly, by Corollary~\ref{2 and eta attaching maps within columns}, we know that $\Sigma^{8k-7}C2$ is an $\textup{H}\mathbb{F}_2$-subcomplex of ${X(8k+4)^{8k-2}_{8k-8}}$. Secondly, 	by Corollary~\ref{2 and eta attaching maps within columns} and the fact that $\pi_4 =0$ and $\pi_5=0$, we know that $\Sigma^{8k-3}C2$ is an $\textup{H}\mathbb{F}_2$-subcomplex of ${{X(8k+4)^{8k-2}_{8k-8}}/\Sigma^{8k-7}C2}$. Therefore, by Lemma~\ref{pbpfhf2}, we have the following diagram and in particular we may define $W$.
$$\xymatrix{
\Sigma^{8k-7}C2 \ar@{^{(}->}[r] \ar@{=}[d] & W \ar@{->>}[r] \ar@{^{(}->}[d] & \Sigma^{8k-3}C2 \ar@{^{(}->}[d] \\
\Sigma^{8k-7}C2 \ar@{^{(}->}[r] & X(8k+4)^{8k-2}_{8k-8} \ar@{->>}[r] \ar@{->>}[d] & X(8k+4)^{8k-2}_{8k-8}/\Sigma^{8k-7}C2 \ar@{->>}[d] \\
& Y(k) \ar@{=}[r] & Y(k).
}$$
We then complete the proof by defining $Y(k)$ to be the cofiber of the map
$$\xymatrix{W \ar@{^{(}->}[r] & X(8k+4)^{8k-2}_{8k-8}.}$$ 
\end{proof}

%%%%%%%%%%%%%%%%%%%%%%%%%%%%%%
%%%%%%%%%%%%%%%%%%%%%%%%%%%%%%
%%%%%%%%%%%%%%%%%%%%%%%%%%%%%%
\newpage
\section{Step 1: Proof of Theorem~\ref{thm: inductive fk}}\label{sec:Step1}

In this section, we present the proof of Theorem~\ref{thm: inductive fk}, which states that:
For every $k\geq 0$, there exist maps
\begin{itemize}
\item $f_{k}: X(8k+4)_{8k+1}^{\infty} \longrightarrow S^{0}$
\item $g_{k}:S^{8k+4} \lhook\joinrel\longrightarrow X(8k+4)_{8k+1}^{\infty}$
\item $a_{k}: S^{8k+4} \longrightarrow X(8k-4)^{8k-4}_{8k-7}$
\item $ b_{k}:X(8k-4)^{8k-4}_{8k-7}\longrightarrow S^{0}$
\end{itemize}
with the following properties:
\begin{enumerate}[label=(\roman*)]
\item The diagram 
\begin{equation}%\label{f is quotient}
    \xymatrix{
        X(8k+4) \ar[r]\ar@{->>}[d] & S^{0}  \\
        X(8k+4)_{8k+1}^{\infty} \ar[ur]_{f_{k}} }
\end{equation}
commutes.

\item The map $g_{k}$ induces an isomorphism on $H_{8k+4}(-;\mathbb{F}_{2})$.  In other words, $S^{8k+4}$ is a $\textup{H}\mathbb{F}_2$-subcomplex of $X(8k+4)_{8k+1}^{\infty}$ via the map $g_k$. %(see Definition~\ref{hf2}).

\item The following diagram is commutative:
\begin{equation}%\label{gf factors throught W}
    \xymatrix{
        S^{8k+4} \ar[d]^{a_{k}} \ar@{^{(}->}[r]^-{g_{k}} & X(8k+4)_{8k+1}^{\infty} \ar[d]^{f_{k}}\\
    X(8k-4)^{8k-4}_{8k-7}\ar[r]^-{b_{k}} & S^{0}.}
\end{equation}
\item  Let $\phi_{k}:S^{8k+1}\rightarrow S^{0}$ be the restriction of $f_{k}$ to the bottom cell of ${X(8k+4)^{\infty}_{8k+1}}$.  Then for $k\geq 1$, the map $\phi_{k}$ satisfies the inductive relation 
$$\phi_{k}- \phi_{k-2}\cdot  \chi_{k}\in \langle \phi_{k-1},2,\tau_{k}\rangle,$$
where $\tau_{k}\in \{0,8\sigma\}$ in $\pi_7$ and $\chi_{k}$ is some element in $\pi_{16}$. Note that by Lemma~\ref{eta attaching map between columns} that $\phi_0 = \eta$ and we set $\phi_{-1}=0$. 
\end{enumerate}

\subsection{An outline of the proof}

\begin{figure}
\begin{center}
\makebox[\textwidth]{\includegraphics[trim={2.5cm 4.2cm 0.3cm 7.8cm}, clip, page = 1, scale = 0.7]{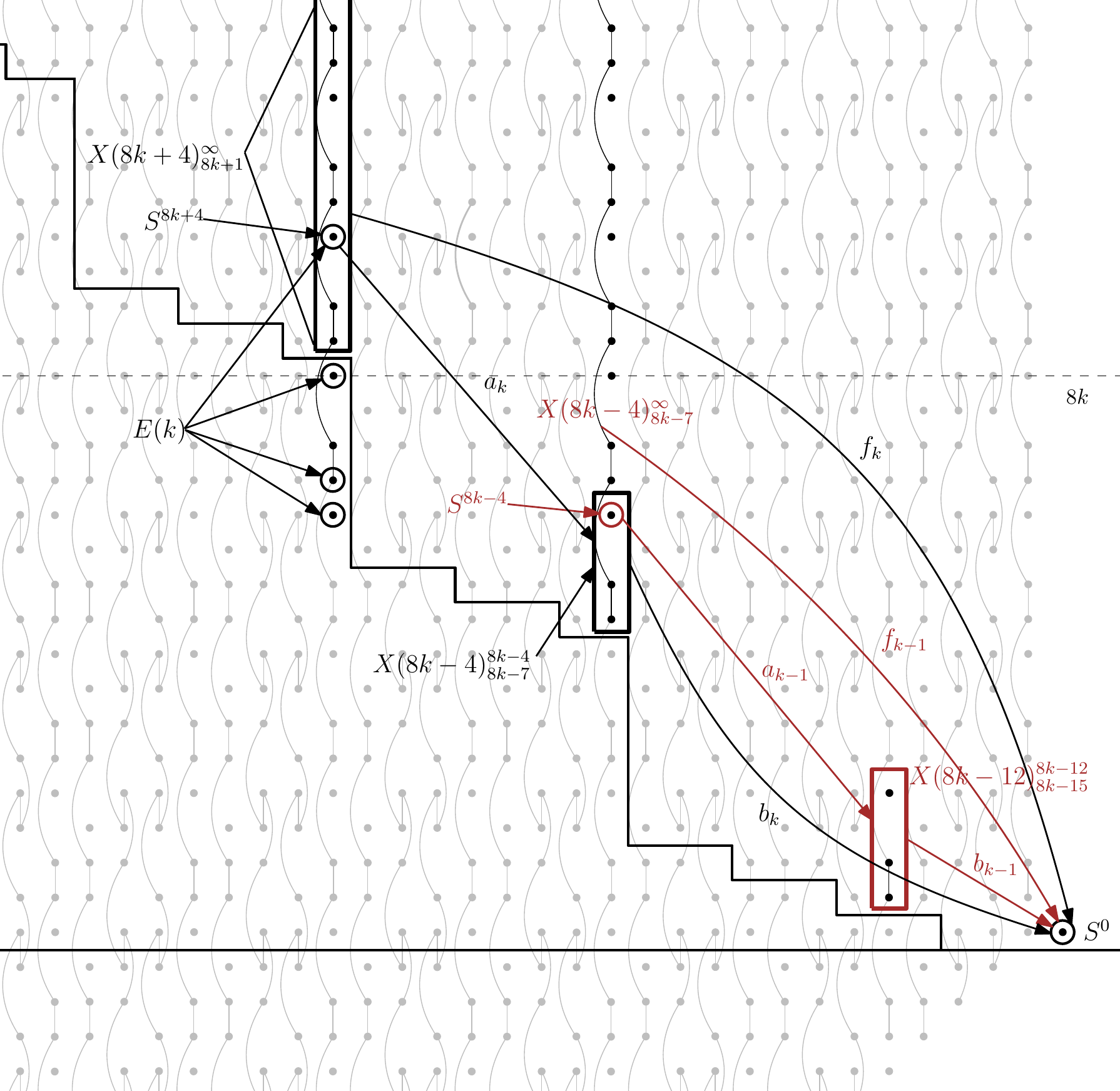}}
\end{center}
\begin{center}
\caption{Step 1 main picture.}
\hfill
\label{fig:Step1MainPic}
\end{center}
\end{figure}

In this subsection, we list the main steps of our proof of Theorem~\ref{thm: inductive fk}. The intuition is explained later in Remark~\ref{rem intuition}. 

We need to show the existence of 4 families of maps 
$$f_k, \ g_k, \ a_k, \ \text{and} \ b_k $$
for all $k \geq 0$, that satisfy two commutative diagrams, namely the ones in (i) and (iii) of Theorem~\ref{thm: inductive fk}, a property for $g_k$, namely (ii) of Theorem~\ref{thm: inductive fk} and a property for $f_{k}$, namely (iv) of Theorem~\ref{thm: inductive fk}.

The strategy of our proof can be summarized as the following. We first prove the existence of the maps $a_k$ for all $k \geq 0$, and then construct the maps $g_k$ for all $k \geq 0$. We check that $g_k$ satisfies property (ii) in Theorem~\ref{thm: inductive fk}. This is \textbf{Step~1.1} and \textbf{Step~1.2} of our proof. 

In the rest of the proof, we show inductively the existence of the maps $f_k$ and $b_k$, and that the two diagrams in (i) and (iii) of Theorem~\ref{thm: inductive fk} commute. 

We first define $b_0$ to be the zero map and show the existence of $f_0$. We check that the two diagrams in (i) and (ii) of Theorem~\ref{thm: inductive fk} commute. This is \textbf{Step~1.3} that gives the starting case $k=0$.

Next, we assume the maps $f_{k-1}$ and $b_{k-1}$ exist and the two diagrams in (i) and (ii) of Theorem~\ref{thm: inductive fk} commute for the 4 maps $(f_{k-1}, \ g_{k-1}, \ a_{k-1}, b_{k-1})$. We define the map $b_k$ and show the existence of $f_k$, using information in the induction. Note that there are choices for $f_k$. This is \textbf{Step~1.4}.

Then, we check that the two diagrams in (i) and (ii) of Theorem~\ref{thm: inductive fk} commute for the 4 maps $(f_{k}, \ g_{k}, \ a_{k}, \ b_{k})$, for all choices of $f_k$. This is \textbf{Step~1.5}.

Finally, in \textbf{Step~1.6}, we prove that there exists one choice of $f_k$, such that it satisfies an inductive relation between the restriction of $f_{k},\ f_{k-1},\ f_{k-2}$ to the bottom cell of their domains. For this choice of $f_k$, this establishes property (iv) and finishes the proof.

More precisely, the details of Steps~1.1-1.6 are stated as the following.
\begin{enumerate}

\item \textbf{\underline{Step 1.1}}: We establish the existence of the maps $a_k$ for all $k \geq 0$.

\begin{prop}\label{proposition: construct a} 
For every $k\geq 0$, there exists a map $c_{k}$ that fits into the following commutative diagram
\begin{equation}\label{diagram: construct a}
    \xymatrix{
        E(k) \ar@{^{(}->}[r] \ar@{->>}[d] & X(8k+4)_{8k-4}^{\infty} \ar@^{->}[dr]& \\
S^{8k+4} \ar[drr]_{c_{k}}   &     & X(8k-3)_{8k-7}^{\infty}   \\
& & X(8k-3)_{8k-7}^{8k-4} \ar@{^{(}->}[u]
}
\end{equation}
\end{prop}

\begin{figure}
\begin{center}
\makebox[\textwidth]{\includegraphics[trim={2.5cm 4.2cm 0.3cm 7.8cm}, clip, page = 1, scale = 0.7]{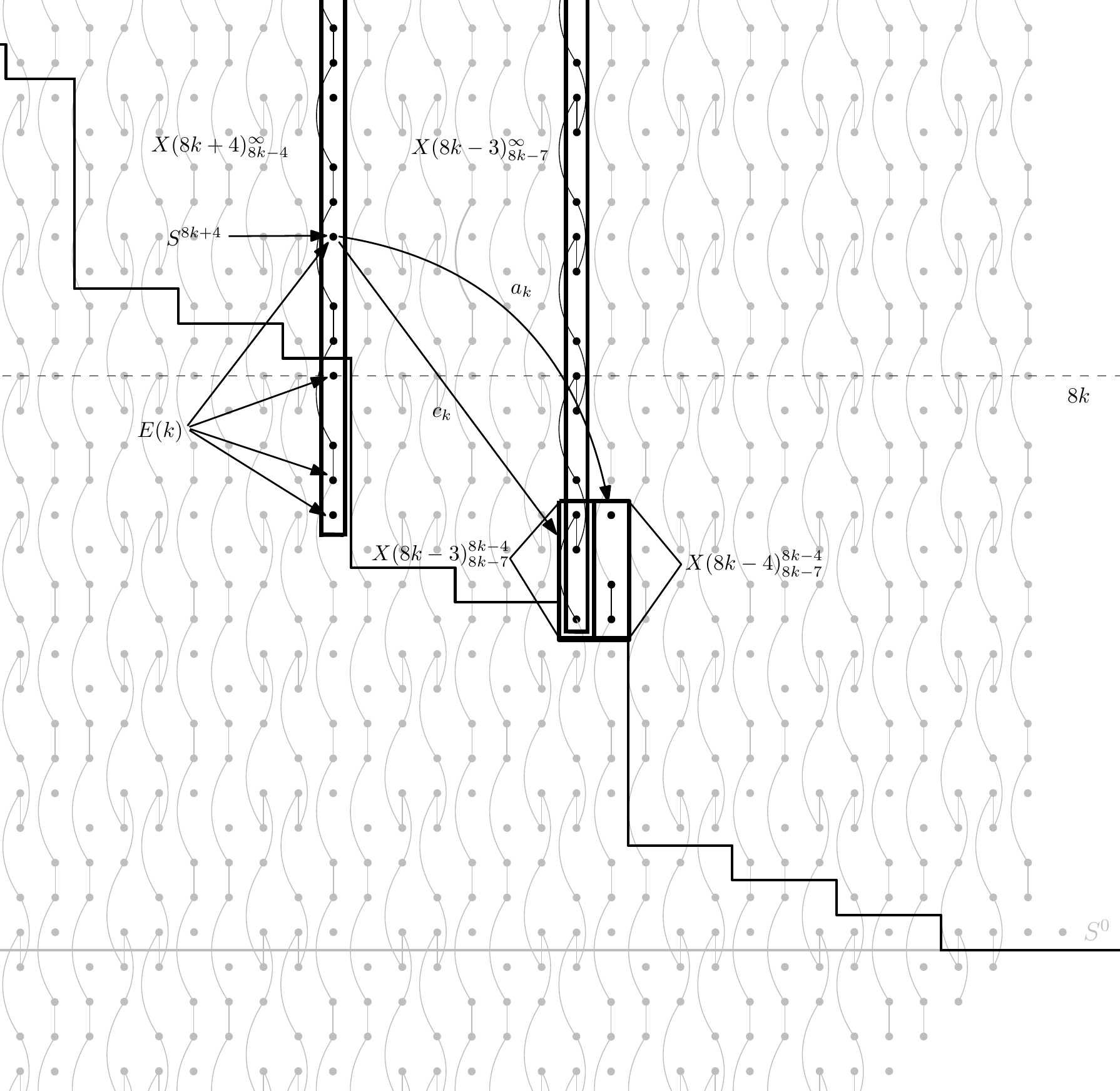}}
\end{center}
\begin{center}
\caption{Step 1.1 picture.}
\hfill
\label{fig:Step1point1Pic}
\end{center}
\end{figure}

The proof of Proposition~\ref{proposition: construct a} is an extensive and careful study of the cell structures of the columns between $8k+4$ and $8k-3$ and in dimensions between $8k+4$ and $8k-7$. It involves the computation of stable stems $\pi_s$ in the range $s \leq 11$. We define $a_{k}$ as the composition
$$
\xymatrix{
 S^{8k+4} \ar[r]^-{c_{k}} & X(8k-3)^{8k-4}_{8k-7} \ar@^{->}[r] & X(8k-4)^{8k-4}_{8k-7}.
 }
$$

\item \textbf{\underline{Step 1.2}}: Using Proposition~\ref{proposition: construct a} and the homotopy extension property, which is stated as Lemma~\ref{homotopy extension} in Subsection \ref{subsec: hep}, we show the existence of two maps $u_k$ and $v_k$ in the following Proposition~\ref{proposition: construction of u,v}.
 
\begin{prop}\label{proposition: construction of u,v} For every $k\geq 0$, there exist maps $u_{k}, \ v_{k}$ that fit into the following commutative diagram:
\begin{equation} \label{diagram: construction of u,v}
\xymatrix{
 E(k) \ar@{^{(}->}[r] \ar@{->>}[d]& X(8k+4)_{8k-4}^\infty \ar@^{->}[dr] \ar@{->>}[d]& \\ 
S^{8k+4}\ar@{^{(}->}[r]^{u_{k}}
\ar[drr]_{c_{k}} & G(k) \ar[r]^-{v_{k}} & X(8k-3)_{8k-7}^\infty \\
& & X(8k-3)^{8k-4}_{8k-7}\ar@{^{(}->}[u]
}
\end{equation}

Moreover, the map $u_k$ induces an isomorphism on $H_{8k+4}(-;\mathbb{F}_{2})$. In other words, $(S^{8k+4}, u_{k})$ is an $\textup{H}\mathbb{F}_2$-subcomplex of $G(k)$.
\end{prop}

We define the map $g_k$ as the following composite
\begin{displaymath}
	\xymatrix{
	S^{8k+4} \ar@{^{(}->}[r]^{u_{k}} & G(k) \ar@{->>}[r] & G(k)_{8k+1}^\infty = X(8k+4)_{8k+1}^\infty
	}
\end{displaymath}
Note here we use the octahedron axiom to identify $G(k)_{8k+1}^\infty$ with ${X(8k+4)_{8k+1}^\infty}$. It then follows from Proposition~\ref{proposition: construction of u,v} that the map $g_k$ induces an isomorphism on $H_{8k+4}(-;\mathbb{F}_{2})$, which establishes  property (ii) in Theorem~\ref{thm: inductive fk}.\\

\item \textbf{\underline{Step 1.3}}:
We define 
$$b_0: X(-4)^{-4}_{-7} \longrightarrow S^0$$
to be the zero map. Note that the 3 cells of $X(-4)^{-4}_{-7}$ are in dimensions $-4, -6, -7$, so this is the only choice. Since $\pi_4 = 0$, the following diagram (iii) in Theorem~\ref{thm: inductive fk} for $k=0$ commutes regardless of the construction of $f_0$. 
\begin{displaymath}
    \xymatrix{
        S^{4} \ar[d]^{a_{0}} \ar@{^{(}->}[r]^-{g_{0}} & X(4)_{1}^{\infty} \ar[d]^{f_{0}}\\
    X(-4)^{-4}_{-7}\ar[r]^-{b_{0}} & S^{0} 
    }
\end{displaymath}
For the existence of the map $f_0$, if suffices to show the following composite is zero.
\begin{displaymath}
    \xymatrix{
        X(4)^{0} \ar@{^{(}->}[r] & X(4) \ar[r] & S^{0} 
    }
\end{displaymath}
This is a special case. This gives the following commutative diagram (i) in Theorem~\ref{thm: inductive fk} for $k=0$.
\begin{displaymath}
    \xymatrix{
    X(4)^{0} \ar@{^{(}->}[d] \ar[rd]^{=0} & \\
        X(4) \ar[r] \ar@{->>}[d] & S^{0}  \\
        X(4)_{1}^{\infty} \ar@{-->}[ur]_{f_{0}} & 
        }
\end{displaymath}
This gives the starting case $k=0$ of our inductive argument.\\

\item \textbf{\underline{Step 1.4}}:
For $k\geq 1$, we assume the maps $f_{k-1}$ and $b_{k-1}$ exist, the two diagrams in (i) and (iii) of Theorem~\ref{thm: inductive fk} commute for the 4 maps $(f_{k-1}, \ g_{k-1}, \ a_{k-1}, b_{k-1})$, and $f_{k-1}$ satisfies property (iv) in Theorem~\ref{thm: inductive fk}.

We define the map $b_k$ to be the composite 
\begin{displaymath}
    \xymatrix{
X(8k-4)^{8k-4}_{8k-7} \ar@{^{(}->}[r] & X(8k-4)_{8k-7}^\infty \ar[r]^-{f_{k-1}} & S^{0}.
}
\end{displaymath}
Using the commutative diagram (\ref{gf factors throught W}) in (iii) of Theorem~\ref{thm: inductive fk} for the case $k-1$, we have the following proposition:

\begin{prop}\label{proposition: construct f} 
The following composite is zero. 
\begin{equation}\label{bottom in G killed}
\xymatrix{
S^{8k-2} \ar@{^{(}->}[r] & G(k) \ar[r]^-{v_{k}} & X(8k-3)_{8k-7}^\infty \ar@^{->}[r]& X(8k-4)_{8k-7}^\infty \ar[r]^-{f_{k-1}} & S^{0} 
}
\end{equation} 
\end{prop}
Note that the first map is the inclusion of the bottom cell of $G(k)$, and that the map $v_k$ is established in Step~1.2 before the induction.

As a result, there exist maps 
$$
f_{k}:X(8k+4)_{8k+1}^\infty = G(k)_{8k+1}^\infty = G(k)/S^{8k-2}  \longrightarrow S^{0}
$$ 
that fit into the following commutative diagram:
\begin{equation}\label{diagram: construction of f}
\xymatrix{
 S^{8k-2} \ar@{^{(}->}[d] \ar@/^1pc/[drrrr]^{=0} & & & \\
 G(k) \ar[r]^-{v_{k}} \ar@{->>}[d]& X(8k-3)_{8k-7}^\infty \ar@^{->}[r] & X(8k-4)_{8k-7}^\infty\ar[rr]_-{f_{k-1}} & & S^{0}\\
  G(k)/S^{8k-2} \ar@{=}[d] & & & \\ 
X(8k+4)_{8k+1}^\infty \ar@{-->}@/_1pc/[uurrrr]_-{f_{k}} & & & }
\end{equation}
Note that there are many choices of $f_k$ that makes the above diagram (\ref{diagram: construction of f}) commute.\\

\item \textbf{\underline{Step 1.5}}: In this step, we prove the following proposition.

\begin{prop}\label{f g a b satisfy requirements}
For \textbf{any choice of $f_{k}$ in Step~1.4}, the two diagrams (\ref{f is quotient}) and (\ref{gf factors throught W}) in (i) and (iii) of Theorem~\ref{thm: inductive fk} commute for the 4 maps $(f_{k}, \ g_{k}, \ a_{k}, \ b_{k})$.
\end{prop}
The proof is a straightforward cell diagram chasing argument.\\

\item \textbf{\underline{Step 1.6}}: In this step, we prove the following proposition.

\begin{prop}\label{prop induct beta} 
Let $\phi_{m}:S^{8k+1}\rightarrow S^{0}$ be the restriction of $f_{m}$ to the bottom cell of $X(8k+4)_{8k+1}^{\infty}$. Then there exists \textbf{one choice of $f_{k}$ in Step~1.4} such that the following property is satisfied: 
\begin{equation}\label{equation: inductive on beta}
\phi_{k}-\phi_{k-2} \cdot  \chi_{k}\in \langle \phi_{k-1}, 2, \tau_{k}\rangle 
\end{equation}
where $\tau_{k}\in \{0,8\sigma\}$ and $\chi_{k}\in \pi_{16}(S^{0})$. Note that by Lemma~\ref{eta attaching map between columns} that $\phi_0 = \eta$ and we set $\phi_{-1}=0$. 
\end{prop}

This proves that this choice of $f_k$ satisfies the relation in (iv) of Theorem~\ref{thm: inductive fk} and therefore completes the induction.

\end{enumerate}

\begin{rem} \label{rem intuition} \rm
The critical part of Theorem~\ref{thm: inductive fk} is the existence of the map $f_k$. We want to prove it inductively. Namely, we assume that $f_{k-1}$ exists and want to show that $f_k$ exists. This induction would follow easily if the following map were zero:
\begin{equation}\label{equ intui}
\xymatrix{
X(8k+4)^{8k}_{8k-4} \ar@^{->}[r] & X(8k-4)_{8k-7}^\infty.
}\end{equation} 
However, this is not true. Intuitively, the $(8k-2)$-cell in $X(8k+4)^{8k}_{8k-4}$ maps nontrivially to the $(8k-4)$-cell in $X(8k-4)_{8k-7}^\infty$ by $\eta^2$. More precisely, one can show that the above map (\ref{equ intui}) factors through $S^{8k-2}$ as an $\textup{H}\mathbb{F}_2$-quotient, and the latter map in the following composite
$$
\xymatrix{
X(8k+4)^{8k}_{8k-4} \ar@{->>}[r] & S^{8k-2} \ar[r] & X(8k-4)_{8k-7}^\infty.
}
$$
is detected by $\eta^2[8k-4]$ in the Atiyah--Hirzebruch spectral sequence of ${X(8k-4)_{8k-7}^\infty}$. Therefore, we have to show the composite
\begin{equation}\label{equ intui 2}
\xymatrix{
X(8k+4)^{8k}_{8k-4} \ar@{->>}[r] & S^{8k-2} \ar[r] & X(8k-4)_{8k-7}^\infty \ar[r]^-{f_{k-1}} & S^0.
}
\end{equation}
is zero. It turns out that we can show the composite of the latter two maps in $(\ref{equ intui 2})$ is zero. This follows from a technical condition that $f_{k-1}$ can be chosen to satisfy:
%We explain the intuition of this proof here. We use $D(m)^{n}$ to denote the $n$-cell of $X(m)$. To prove the vanishing of the composition $$X(8k+4)_{8k}^\infty \hookrightarrow X(8k+4)\rightarrow S^{0},$$ it is natural to define quotient maps $f_{k}:X(8k+4)_{8k-1}^\infty \rightarrow S^{0}$ inductively. Assume $f_{k-1}$ exists and consider the map $$l:X(8k+4)^{8k}_{8k-4}\rightarrow X(8k-4)_{8k-7}^\infty.$$ Then $f_{k}$ can be defined if the composition $f_{k-1}\circ l$ equals zero. Note the the $l$ sends $D(8k+4)^{8k-2}$ to $D(8k-4)^{8k-4}$-cell by $\eta^{2}$ and this is the only nonzero component of $l$. Therefore, suppose $f_{k-1}$ is chosen to satisfies the following :
\begin{itemize}
\item $f_{k-1}|_{S^{8k-4}}$ factors through $X(8k-12)^{8k-12}_{8k-15}$.
\end{itemize}
Here note that $S^{8k-4}$ is an $\textup{H}\mathbb{F}_2$-subcomplex of $X(8k-4)_{8k-7}^\infty$. In fact, this is due to the composite 
$$\xymatrix{
S^{8k-2}\ar[r]^{\eta^{2}} & S^{8k-4}\ar[r] & X(8k-12)^{8k-12}_{8k-15}= S^{8k-12}\vee \Sigma^{8k-15}C2
}$$
corresponds to an element in the group $(\pi_{8} + \pi_{11}C2 )\cdot \eta^2 =0$. 

Now to complete the induction, we need to show that $f_{k}$ can be chosen to satisfy:
\begin{itemize}
\item $f_{k}|_{S^{8k+4}}$ factors through $X(8k-4)^{8k-4}_{8k-7}$.
\end{itemize}

Firstly, in $X(8k+4)_{8k-4}^\infty$, the $(8k+4)$-cell is only attached to the cells in dimensions $8k-4,\ 8k-3$ and $8k$, all of which map trivially to $X(8k-4)_{8k-7}^\infty$. As a result, we can choose $f_{k}$ such that the restriction $f_{k}|_{S^{8k+4}}$ factors through $X(8k-4)_{8k-7}^\infty$. 

Secondly, by some local arguments that involve attaching maps in $X(8k+4-m)$ for $m=0,\cdots,7$, we can show that $f_{k}$ can be chosen such that $f_{k}|_{S^{8k+4}}$ factors through $X(8k-4)^{8k-4}_{8k-7}$.

This allows us to complete the induction and to prove Theorem~\ref{thm: inductive fk}.  See Figure~\ref{fig:IntuitionForStep1} for an illustration of the discussion above.  

We'd like to comment that our actual argument is a little different from our discussion above. We actually analyze $X(8k-3)^{8k-4}_{8k-7}$ instead of $X(8k-4)^{8k-4}_{8k-7}$. This is used to deduce the inductive relation $(\ref{equation: inductive on beta})$, based on which we identify the first lock. 
\end{rem}

\begin{figure}
\begin{center}
\includegraphics[scale = 0.32,trim={1.8cm 5.5cm 0cm 5.2cm}, clip]{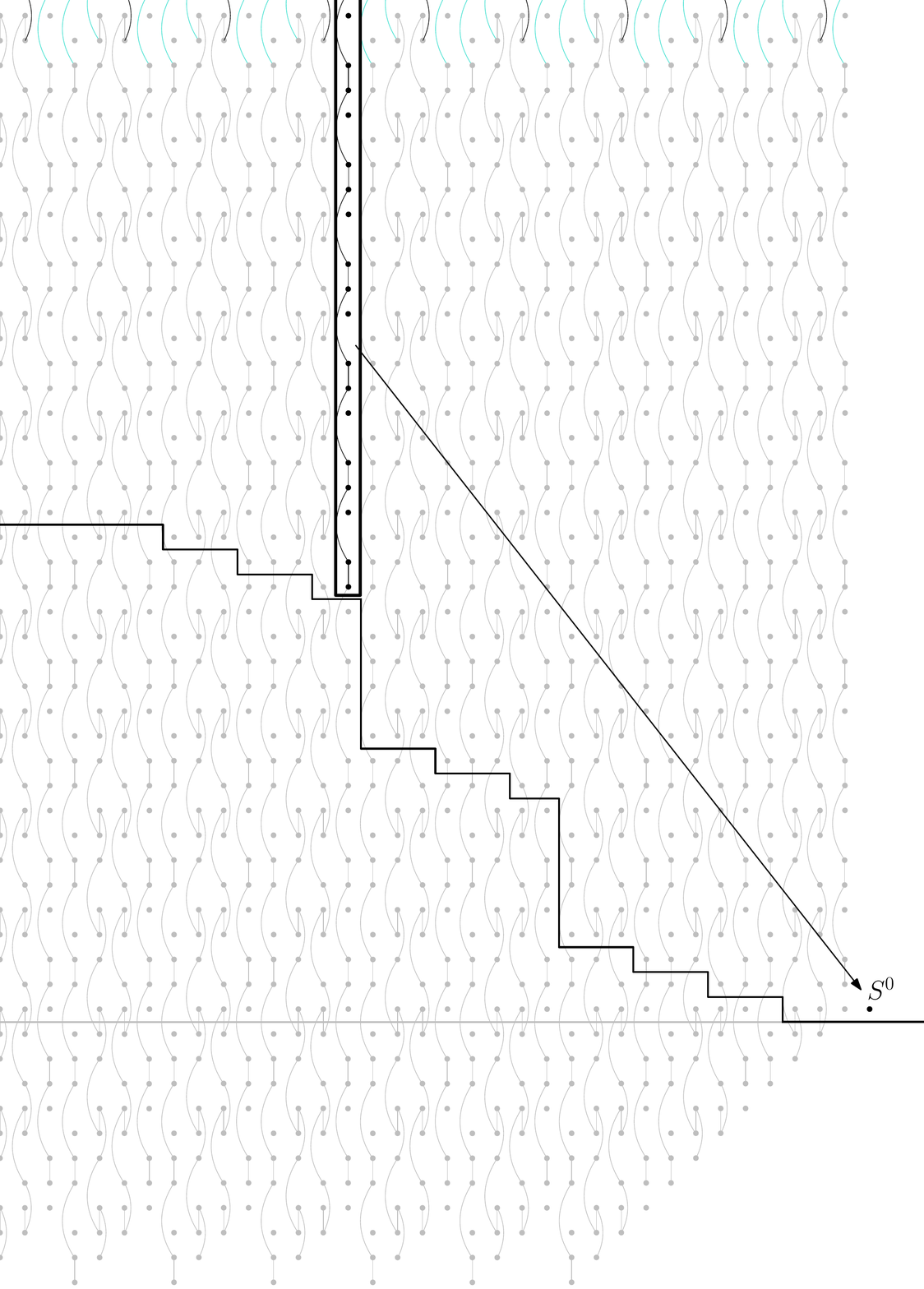}\hspace{0.1in}
\includegraphics[scale = 0.32,trim={1.8cm 5.5cm 0cm 5.2cm}, clip]{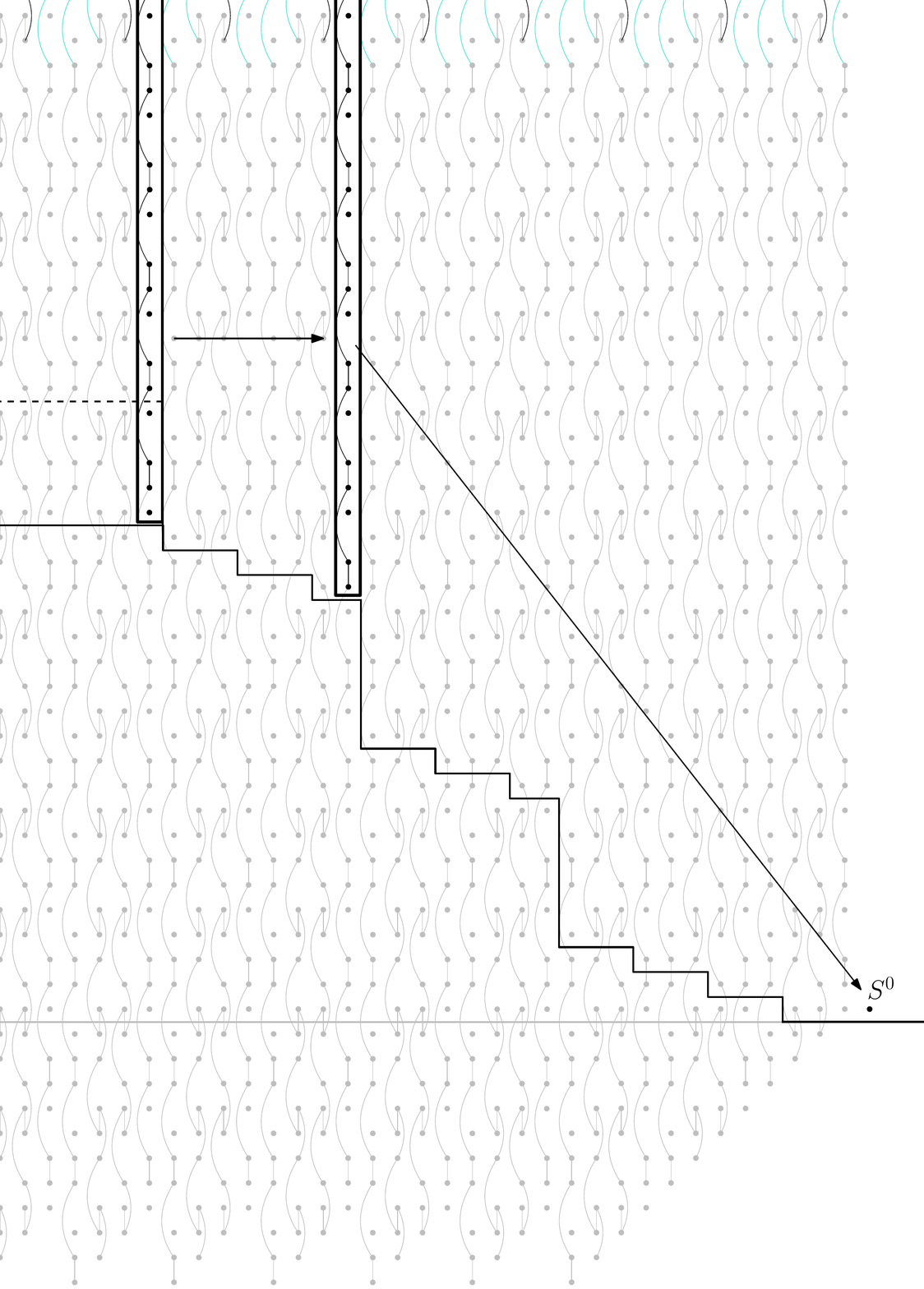} \vspace{0.2in} \\
\includegraphics[scale = 0.32,trim={1.8cm 5.5cm 0cm 5.2cm}, clip]{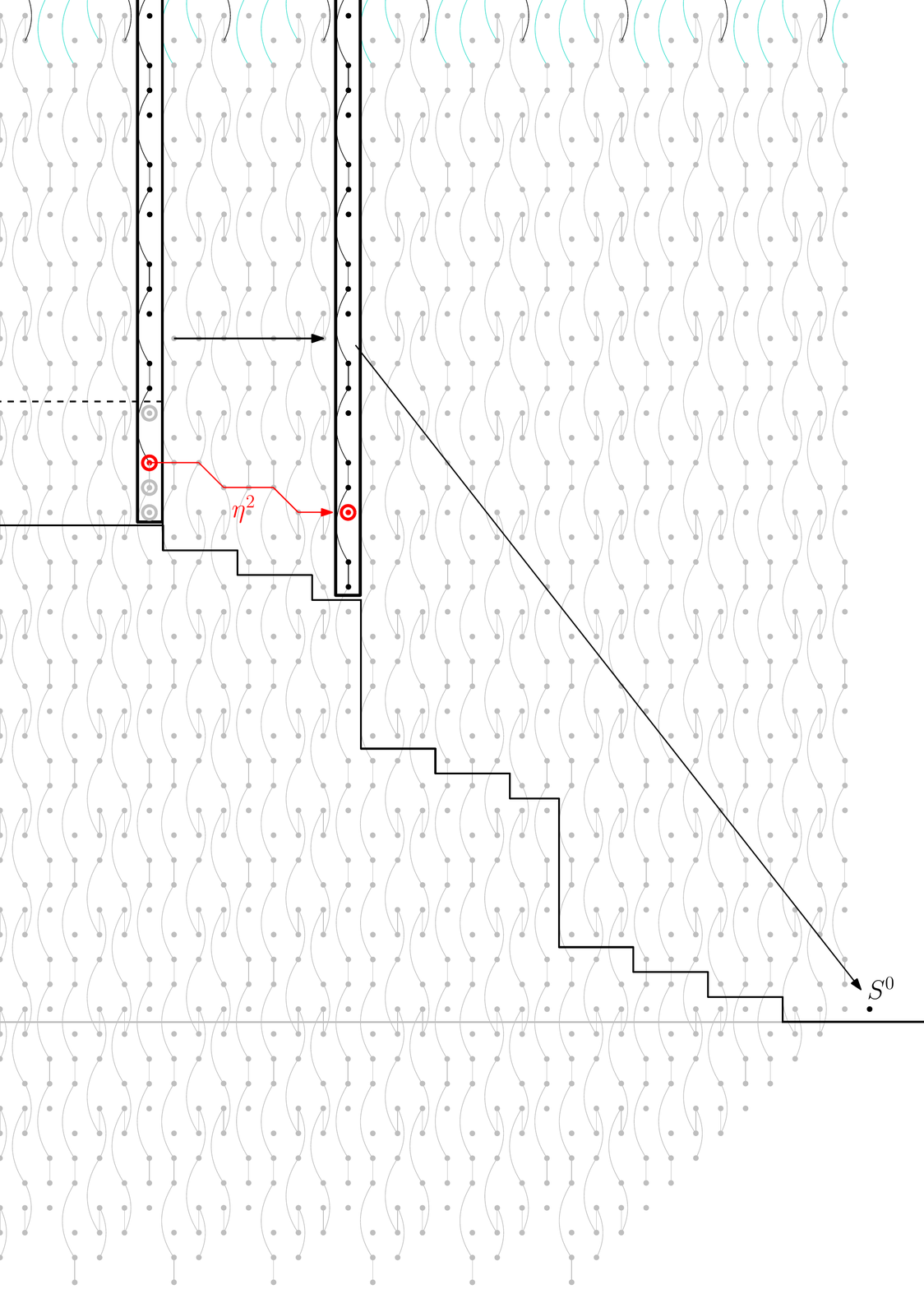}\hspace{0.1in}
\includegraphics[scale = 0.32,trim={1.8cm 5.5cm 0cm 5.2cm}, clip]{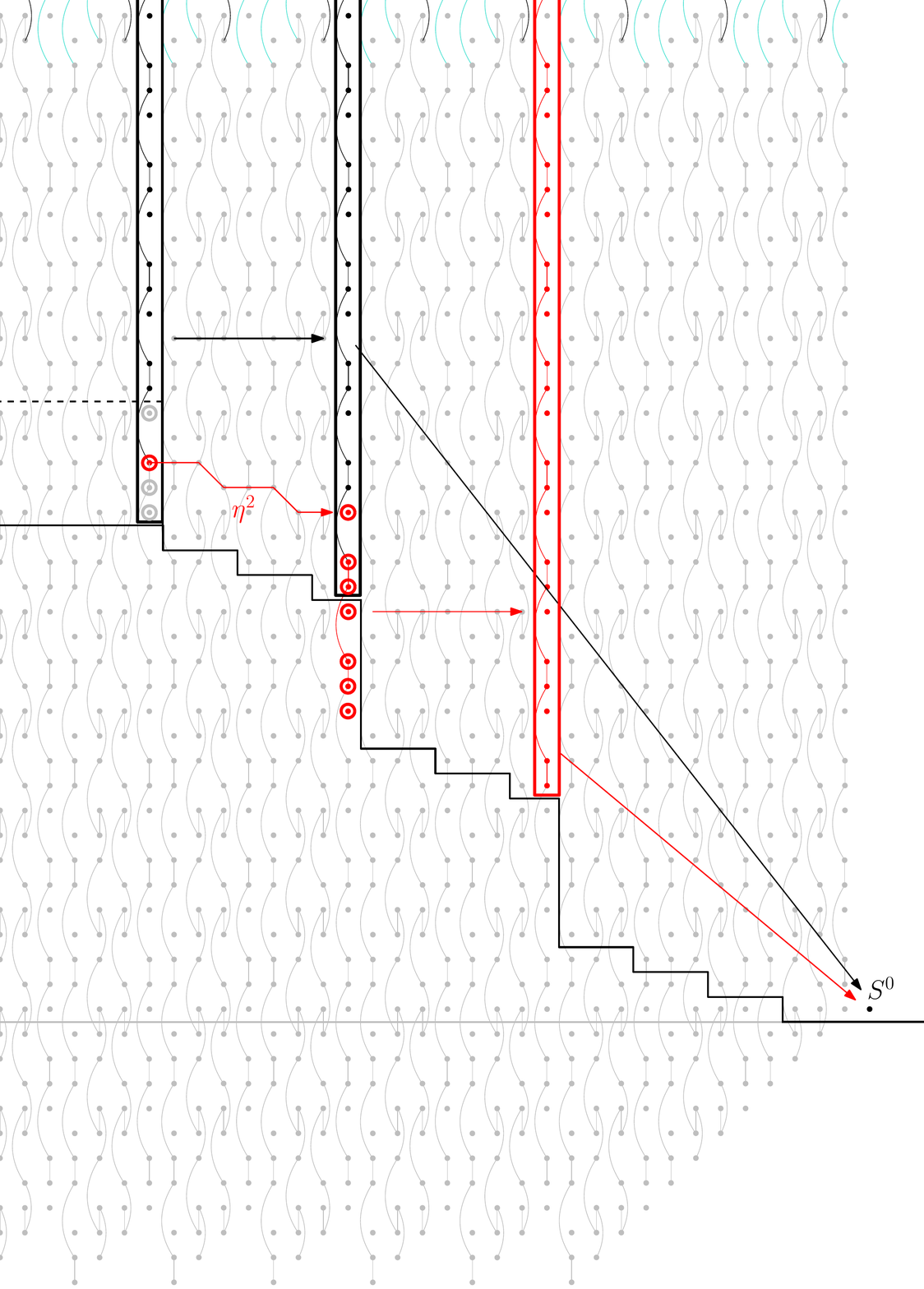} \vspace{0.2in}\\
\includegraphics[scale = 0.32,trim={1.8cm 5.5cm 0cm 5.2cm}, clip]{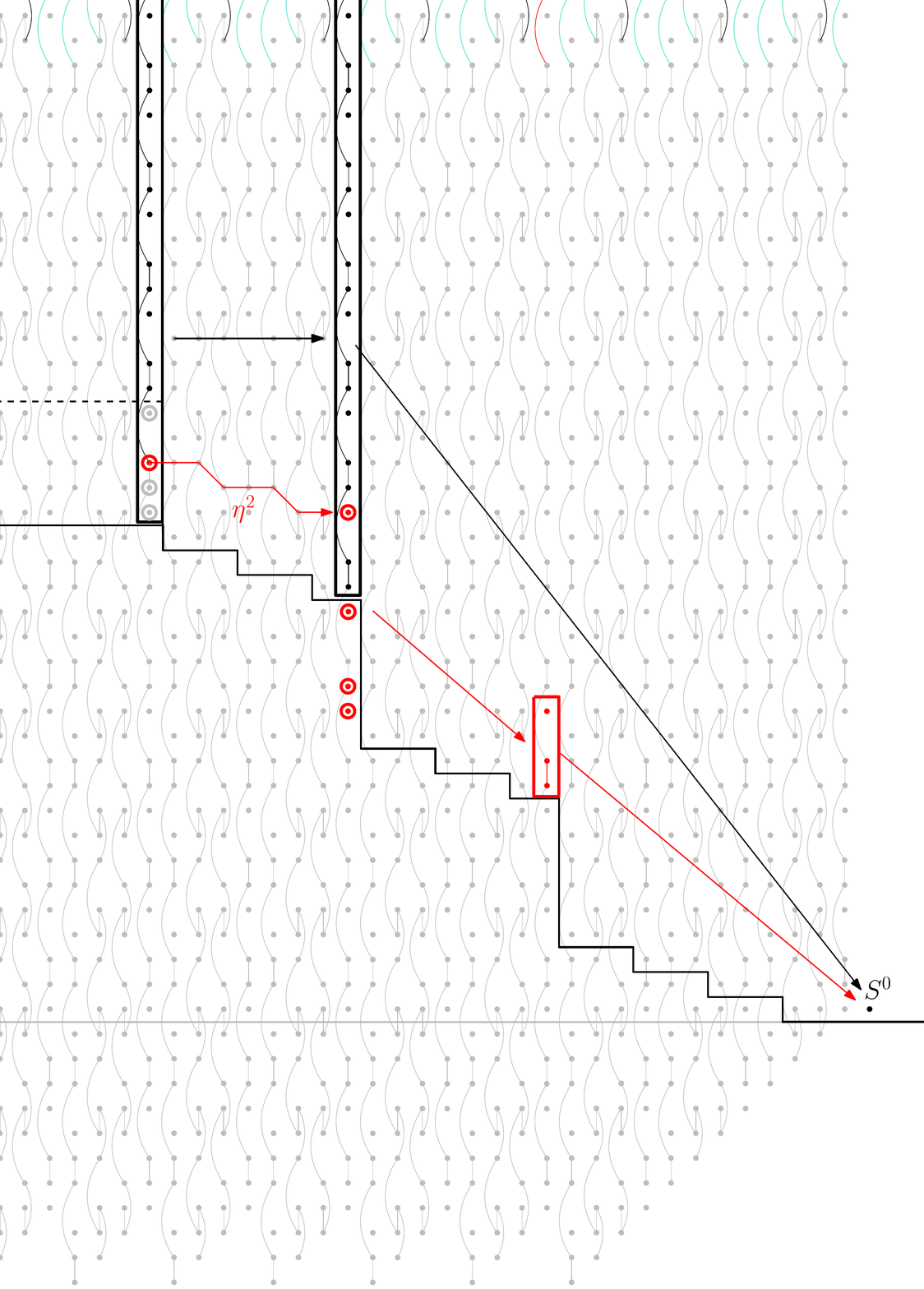}\hspace{0.1in}
\includegraphics[scale = 0.32,trim={1.8cm 5.5cm 0cm 5.2cm}, clip]{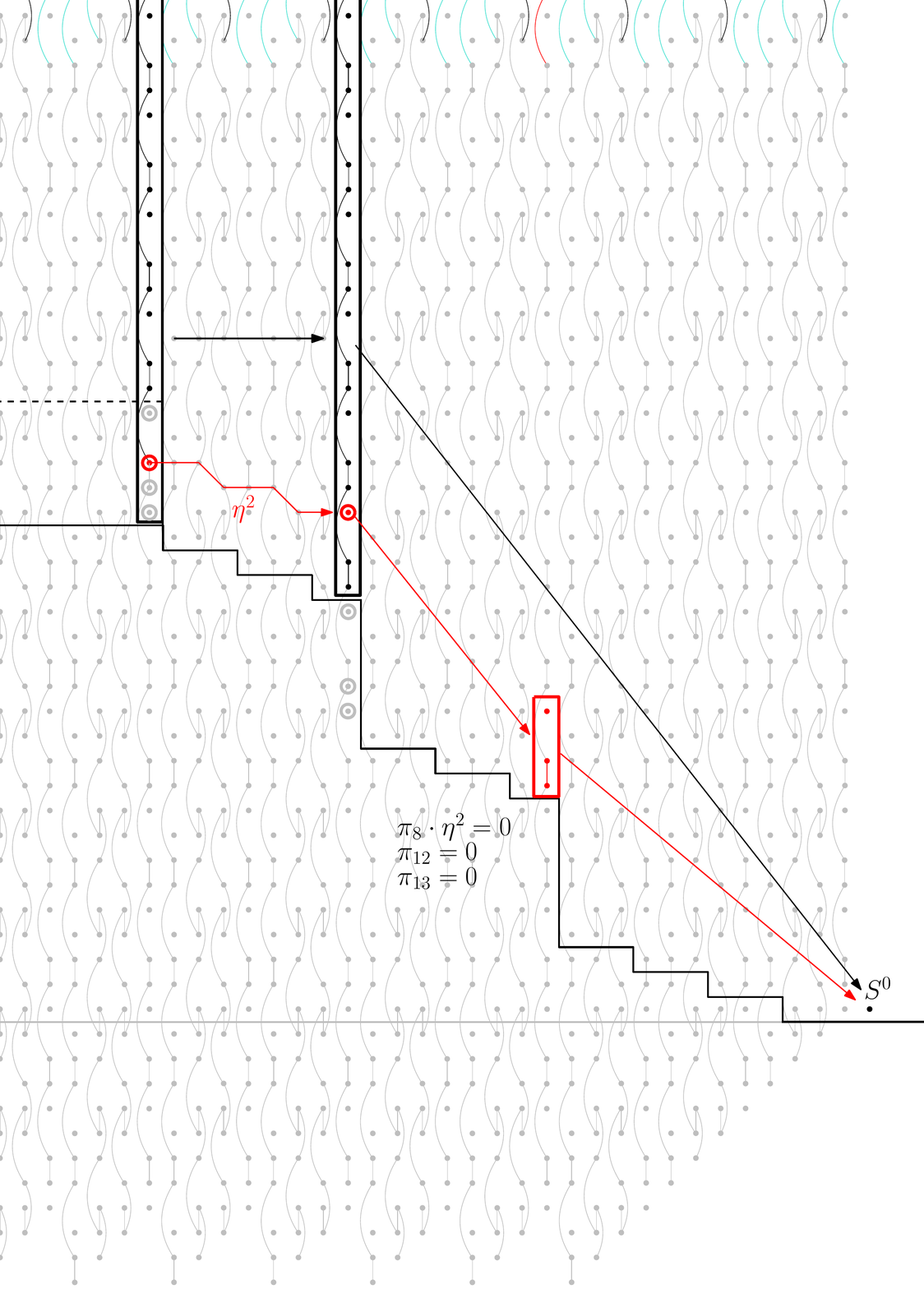}
\caption{Intuition for Step 1.}
\hfill
\label{fig:IntuitionForStep1}
\end{center}
\end{figure}

In the remaining subsections of this section, we will prove Propositions~\ref{proposition: construct a}-\ref{f g a b satisfy requirements} one by one.

\subsection{Proof of Proposition~\ref{proposition: construct a}}

\begin{figure}
\begin{center}
\makebox[\textwidth]{\includegraphics[trim={2.5cm 5.5cm 0.3cm 7.8cm}, clip, page = 1, scale = 0.8]{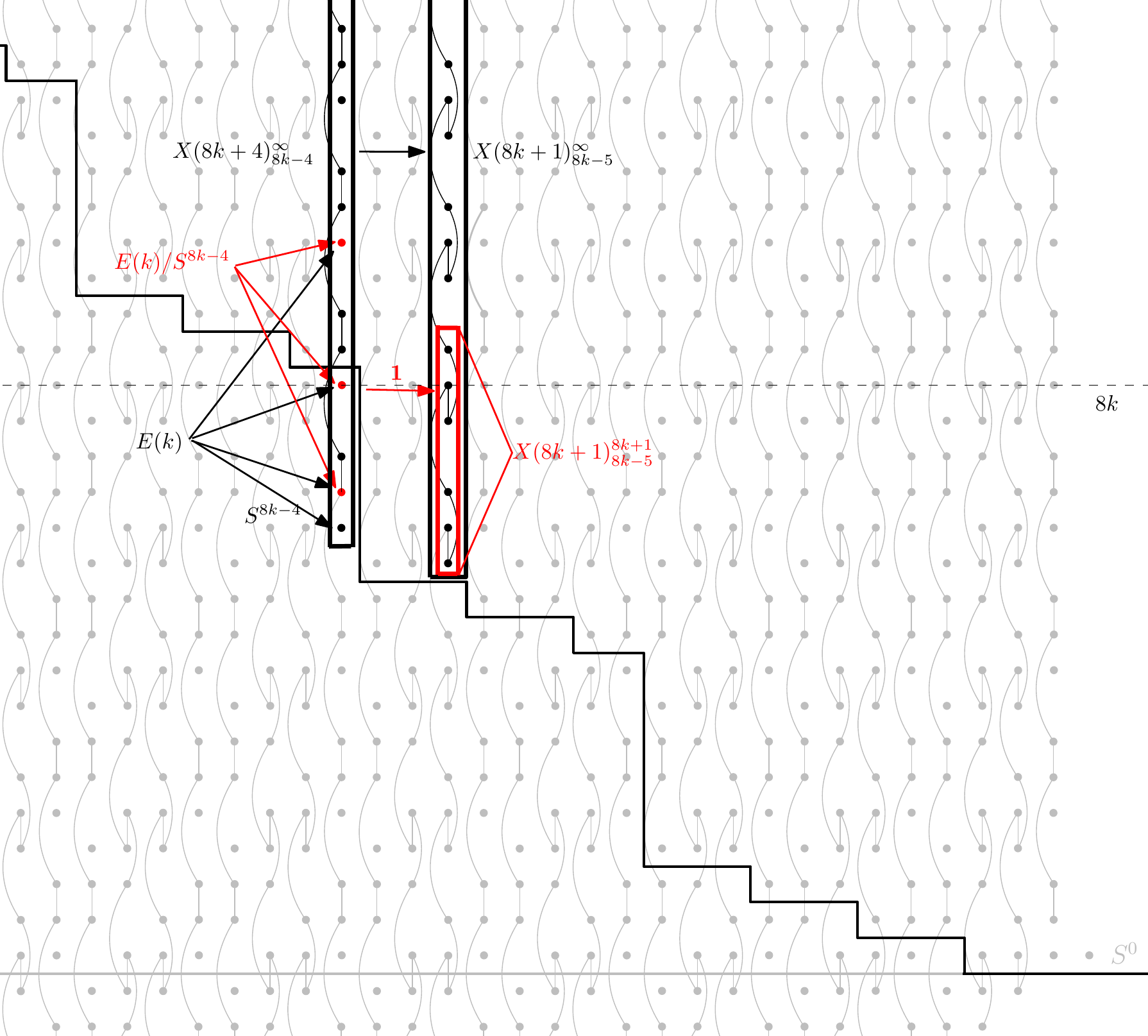}}
\end{center}
\begin{center}
\caption{Step 1.1.1 picture.}
\hfill
\label{fig:Step1point1point1Pic}
\end{center}
\end{figure}

The proof of Proposition~\ref{proposition: construct a} consists of many steps. The goal is to construct a map
$$c_k: S^{8k+4} \longrightarrow X(8k-3)^{8k-4}_{8k-7},$$
such that it is compatible with the map 
$$E(k) \hookrightarrow X(8k+4)^\infty_{8k-4} \rightharpoonup X(8k-3)^\infty_{8k-7}.$$
Since the top cell of $E(k)$ is in dimension $8k+4$, we have the maps
$$E(k) \hookrightarrow X(8k+4)^{8k+4}_{8k-4} \rightharpoonup X(8k-3)^{8k+1}_{8k-7}.$$
So roughly speaking, we want to show that the bottom 3 cells of $E(k)$ maps trivially to $X(8k-3)^{8k+1}_{8k-7}$, and the image of $E(k)$ does not involve the cells in $X(8k-3)^{8k+1}_{8k-3}$. Our strategy is to carefully study the cell structures of the intermediate columns of finite complexes, and to get rid of certain cells gradually.

\underline{\textbf{Step 1.1.1}}: In this step, we focus on column $8k+1$. We use the $\eta$-attaching maps in column $8k+1$ between cells in dimensions $8k-5$ and $8k-3$, $8k+3$ and $8k+5$, to get rid of the cell in dimension $8k-4$ of $E(k)$, and to lower the upper bound of the image to dimension $8k+1$ in column $8k+1$. More precisely, we prove the following lemma. 
\begin{lem} \label{lemma kill first cell in E}
	There exits the following commutative diagram:
\begin{equation}\label{kill first cell in E}
    \xymatrix{
        E(k) \ar@{^{(}->}[r] \ar@{->>}[d] & X(8k+4)_{8k-4}^\infty\ar[r]& X(8k+1)_{8k-5}^\infty \\
    E(k)/S^{8k-4}\ar[rr]^{\mathbf{1}} && X(8k+1)_{8k-5}^{8k+1}\ar@{^{(}->}[u]  }
\end{equation}
\end{lem}

\begin{proof}
	Firstly, we have the following commutative diagram:
\begin{equation} \label{kill first cell in E step 1}
\xymatrix{
S^{8k+4} \ar@{=}[r] & S^{8k+4} \ar[r]^{\eta} & S^{8k+3} \ar@{^{(}->}[r] & X(8k+1)^{8k+5}_{8k+3} \\
E(k) \ar@{->>}[u] \ar@{^{(}->}[r] & X(8k+4)^{8k+4}_{8k-4} \ar@{->>}[u] \ar@^{->}[r] & X(8k+1)^{8k+3}_{8k-5} \ar@{->>}[u] \ar@{^{(}->}[r] & X(8k+1)^{8k+5}_{8k-5} \ar@{->>}[u] \\
E(k)^{8k}_{8k-4} \ar@{^{(}->}[r]  \ar@{^{(}->}[u] & X(8k+4)^{8k+2}_{8k-4} \ar@^{->}[r] \ar@{^{(}->}[u] & X(8k+1)^{8k+1}_{8k-5} \ar@{^{(}->}[u] \ar@{=}[r] & X(8k+1)^{8k+1}_{8k-5} \ar@{^{(}->}[u]
}
\end{equation}
By Lemma~\ref{eta attaching map between columns}, we have that the map in middle of the top row of diagram (\ref{kill first cell in E step 1}) is $\eta$. By Corollary~\ref{2 and eta attaching maps within columns}, we have an $\eta$-attaching map in $X(8k+1)^{8k+5}_{8k+3}$ between the cells in dimensions $8k+3$ and $8k+5$. This corresponds to an Atiyah--Hirzebruch differential
$$1[8k+5] \rightarrow \eta[8k+3].$$
Therefore, the composition of the maps in the top row of diagram (\ref{kill first cell in E step 1}) is zero. In particular, pre-composing with the map
$$\xymatrix{ E(k) \ar@{->>}[r] & S^{8k+4}}$$
is also zero. By the cofiber sequence of the right most column, we know that the map from $E(k)$ to $X(8k+1)^{8k+5}_{8k-5}$ maps through $X(8k+1)^{8k+1}_{8k-5}$.

Secondly, we have the following commutative diagram:
\begin{equation} \label{kill first cell in E step 2}
\xymatrix{
E(k)/S^{8k-4} \ar@{^{(}->}[r] & X(8k+4)^{8k+4}_{8k-3} \ar@^{->}[r] & X(8k+1)^{8k+1}_{8k-4} \ar@{->>}[r] & X(8k+1)^{8k+1}_{8k-1} \\
E(k) \ar@{->>}[u] \ar@{^{(}->}[r] & X(8k+4)^{8k+4}_{8k-4} \ar@{->>}[u] \ar@^{->}[r] & X(8k+1)^{8k+1}_{8k-5} \ar@{->>}[u] \ar@{=}[r] & X(8k+1)^{8k+1}_{8k-5} \ar@{->>}[u] \\
S^{8k-4} \ar@{=}[r]  \ar@{^{(}->}[u] & S^{8k-4} \ar[r]^{\eta} \ar@{^{(}->}[u] & S^{8k-5} \ar@{^{(}->}[u] \ar@{^{(}->}[r] & X(8k+1)^{8k-3}_{8k-5} \ar@{^{(}->}[u]
}
\end{equation}
By Lemma~\ref{eta attaching map between columns}, we have that the map in middle of the bottom row of diagram (\ref{kill first cell in E step 2}) is $\eta$.	By Corollary~\ref{2 and eta attaching maps within columns}, we have an $\eta$-attaching map in $X(8k+1)^{8k-3}_{8k-5}$ between the cells in dimensions $8k-5$ and $8k-3$. This corresponds to an Atiyah--Hirzebruch differential
$$1[8k-3] \rightarrow \eta[8k-5].$$
Therefore, the composition of the maps in the bottom row of diagram (\ref{kill first cell in E step 2}) is zero. In particular, post-composing with the map
$$\xymatrix{ X(8k+1)^{8k-3}_{8k-5} \ar@{^{(}->}[r] & X(8k+1)^{8k+1}_{8k-5}}$$
is also zero. By the cofiber sequence of the left most column, we know that the map from $E(k)$ to $X(8k+1)^{8k+1}_{8k-5}$ factors through $E(k)/S^{8k-4}$.

This gives the required map
$$\mathbf{1}: E(k)/S^{8k-4} \longrightarrow X(8k+1)^{8k+1}_{8k-5}.$$
\end{proof}

\begin{rem}
	We will use arguments similar to the ones in the proof of Lemma~\ref{lemma kill first cell in E} many times in the rest of this paper. Instead of presenting all details in terms of commutative diagrams, we will simply refer them as ``similar arguments as in the proof of Lemma~\ref{kill first cell in E}" or ``cell diagram chasing arguments" due to certain attaching maps.
\end{rem}

\underline{\textbf{Step 1.1.2}}: In this step, we focus on column $8k-2$. We show that in $E(k)/S^{8k-4}$, the cells in dimensions $8k$ and $8k-3$ maps through $S^{8k-6}$ in column $8k-2$. More precisely, we have the following lemma.

\begin{lem}
	There exits the following commutative diagram:
\begin{equation}\label{column 8k2}
    \xymatrix{
       & E(k)/S^{8k-4} \ar[r]^{\mathbf{1}} & X(8k+1)^{8k+1}_{8k-5}\ar@^{->}[r]& X(8k-2)^{8k}_{8k-6} \\
    S^{8k} \vee S^{8k-3} \ar@{=}[r] & E(k)^{8k}_{8k-3}\ar[r]^{\mathbf{2}}\ar@{^{(}->}[u]&S^{8k-6}\ar@{^{(}->}[r] & X(8k-2)_{8k-6}^{8k-4}\ar@{^{(}->}[u] \\}
\end{equation}
\end{lem}

\begin{proof}
	By Proposition~\ref{eta square attaching maps within columns}, there is no $\eta^2$-attaching map in $E(k)^{8k}_{8k-3}$. This shows that
	$$E(k)^{8k}_{8k-3} \simeq S^{8k} \vee S^{8k-3}.$$
	We may therefore consider the cells in dimensions $8k$ and $8k-3$ separately. 
	
	For $S^{8k-3}$, it maps naturally through the $(8k-1)$-skeleton in column $8k-2$. By Proposition~\ref{eta square attaching maps within columns}, there is an $\eta^2$-attaching map in $X(8k-2)^{8k}_{8k-5}$ between the cells in dimensions $8k-4$ and $8k-1$. A similar argument as in the proof of Lemma~\ref{kill first cell in E} shows that $S^{8k-3}$ maps through $S^{8k-6}$ in column $8k-2$.
	
	For $S^{8k}$, firstly note that by Corollary~\ref{2 and eta attaching maps within columns}, there is an $\eta$-attaching map in ${X(8k+1)^{8k+1}_{8k-5}}$ between the cells in dimensions $8k-1$ and $8k+1$. A similar argument as in the proof of Lemma~\ref{kill first cell in E} shows that $S^{8k}$ maps through the $(8k-3)$-skeleton in column $8k+1$. Then it maps naturally through the $(8k-4)$-skeleton in column $8k-2$. To see that it actually maps through $S^{8k-6}$, we only need to show the following composite is zero.
	$$\xymatrix{S^{8k} \ar[r] & X(8k-2)_{8k-6}^{8k-4} \ar@{->>}[r] & X(8k-2)_{8k-5}^{8k-4} = S^{8k-4} \vee S^{8k-5} }$$
	This is in fact true, since $\pi_4 = \pi_5 = 0$.
	
	Combining both parts, this gives the required map
$$\mathbf{2}: S^{8k} \vee S^{8k-3} = E(k)^{8k}_{8k-3} \longrightarrow S^{8k-6}.$$
\end{proof}

We enlarge the above Diagram (\ref{column 8k2}) to the following Diagram (\ref{enlarge diagram}). We will establish the maps $\mathbf{3}, \mathbf{4}$ and $\mathbf{5}$ in \textbf{Steps~1.1.3, 1.1.4 and 1.1.5}:

\begin{equation} \label{enlarge diagram}
    \xymatrix{ S^{8k+1}\vee S^{8k-2} \ar@{-->}[rrrd]^{\mathbf{4}} \\    
     S^{8k+4} \ar@{-->}[rrrdd]_(.6){\mathbf{5}} \ar@{-->}[rrrd]^{\mathbf{3}} \ar[u]^{\partial} \ar[rr] & & X(8k-2)^{8k}_{8k-2} \ar@^{->}[r] & X(8k-3)^{8k}_{8k-3} \\ 
     E(k)/S^{8k-4} \ar@{->>}[u] \ar[r]^{\mathbf{1}} & X(8k+1)^{8k+1}_{8k-5} \ar@^{->}[r] & X(8k-2)^{8k}_{8k-6} \ar@^{->}[r] \ar@{->>}[u] & X(8k-3)^{8k}_{8k-7} \ar@{->>}[u]\\ 
    S^{8k}\vee S^{8k-3}  \ar[r]^{\mathbf{2}} \ar@{^{(}->}[u] & S^{8k-6} \ar@{^{(}->}[r] & X(8k-2)_{8k-6}^{8k-4} \ar@{^{(}->}[u] \ar@^{->}[r] & X(8k-3)^{8k-4}_{8k-7} \ar@{^{(}->}[u]\\}
\end{equation}

\underline{\textbf{Step 1.1.3}}: In this step, we establish the map $\mathbf{3}$, making the triangle under $\mathbf{3}$ in Diagram (\ref{enlarge diagram}) commute.

By Lemma~\ref{eta attaching map between columns}, we have that the map 
$$\xymatrix{S^{8k-6}\ar@{^{(}->}[r] & X(8k-2)^{8k}_{8k-6}\ar@^{->}[r]& X(8k-3)^{8k}_{8k-7}}$$
is $\eta$ mapping into the bottom cell of $X(8k-3)^{8k}_{8k-7}$. Since 
$$\eta \cdot \pi_3 = 0, \ \eta \cdot \pi_6 = 0,$$
the composition of maps in the bottom row of Diagram (\ref{enlarge diagram}) is zero. In particular, post-composing with the map
$$\xymatrix{ X(8k-3)^{8k-4}_{8k-7} \ar@{^{(}->}[r] & X(8k-3)^{8k}_{8k-7}}$$
is also zero. By the cofiber sequence of the left most column, we know that the map from $E(k)/S^{8k-4}$ to $X(8k-3)^{8k}_{8k-7}$ factors through $S^{8k+4}$, which gives the desired map $\mathbf{3}$, making the triangle under $\mathbf{3}$ commute.

Note that we haven't shown the triangle above $\mathbf{3}$ commutes. We will show it later in \textbf{Step 1.1.5}.\\

\underline{\textbf{Step 1.1.4}}: In this step, we establish the map $\mathbf{4}$, making the parallelogram below $\mathbf{4}$ in Diagram (\ref{enlarge diagram}) commute.

By the cofiber sequence in the left most column, it suffices to show the following composite is zero.
$$
\xymatrix{
E(k)/S^{8k-4} \ar@{->>}[r] & S^{8k+4} \ar@{^{(}->}[r]^-{\mathbf{3}} & X(8k-3)^{8k}_{8k-7} \ar@{->>}[r] & X(8k-3)^{8k}_{8k-3}
}
$$
Since both the triangle under $\mathbf{3}$ and the upper rectangent in Diagram (\ref{enlarge diagram}) commute, it is equivalent to show that the following composite is zero.
$$
\xymatrix{
E(k)/S^{8k-4} \ar@{->>}[r] & S^{8k+4} \ar[r] & X(8k-2)^{8k}_{8k-2} \ar@^{->}[r] & X(8k-3)^{8k}_{8k-3}
}
$$
This is in fact true, since the composition of the latter two maps are already zero.
\begin{lem} \label{lemma step 4 enlarge diagram}
	The following composite in Diagram (\ref{enlarge diagram}) is zero.
	$$
\xymatrix{
S^{8k+4} \ar[r] & X(8k-2)^{8k}_{8k-2} \ar@^{->}[r] & X(8k-3)^{8k}_{8k-3}
}
$$
\end{lem}

\begin{proof}
	We first show that the left map factors through the bottom cell $S^{8k-2}$ of the codomain. In fact, the composite
	$$
\xymatrix{
S^{8k+4} \ar[r] & X(8k-2)^{8k}_{8k-2} \ar@{->>}[r] & X(8k-2)^{8k}_{8k-1} = \Sigma^{8k-1} C2
}
$$
corresponds to an element in $\pi_5 C2$. Since $\pi_{4}=\pi_{5}=0$, the group $\pi_5 C2 = 0$. Therefore, it must factor through the bottom cell $S^{8k-2}$. We have the following commutative diagram.
\begin{equation} \label{equation step 4 factor through}
\xymatrix{
S^{8k+4} \ar[r] \ar[rd] & X(8k-2)^{8k}_{8k-2} \ar@^{->}[r] & X(8k-3)^{8k}_{8k-3}\\
& S^{8k-2} \ar@{^{(}->}[u] \ar[r]^{\eta} & S^{8k-3} \ar@{^{(}->}[u]
}
\end{equation}
By Lemma~\ref{eta attaching map between columns}, the map in the bottom row of Diagram (\ref{equation step 4 factor through}) is $\eta$. Since 
$$\eta \cdot \pi_6 = 0,$$
this completes the proof. 
\end{proof}

\underline{\textbf{Step 1.1.5}}: In this step, we establish the map $\mathbf{5}$, making all parts of Diagram (\ref{enlarge diagram}) commute. 

It suffices to show the following lemma.

\begin{lem} \label{lemma step 5 enlarge diagram}
The following composite is zero.
$$\xymatrix{S^{8k+4} \ar[r]^-{\partial} & S^{8k+1}\vee S^{8k-2} \ar[r]^{\mathbf{4}} & X(8k-3)^{8k}_{8k-3}}$$	
\end{lem}

In fact, by Lemma~\ref{lemma step 5 enlarge diagram} and \textbf{Step 4}, the following composite is zero.
$$\xymatrix{S^{8k+4} \ar[r]^-{\mathbf{3}} & X(8k-3)^{8k}_{8k-7} \ar@{->>}[r] & X(8k-3)^{8k}_{8k-3}}$$	
Then by the cofiber sequence in the right most column of Diagram (\ref{enlarge diagram}), the map $\mathbf{3}$ must map through $X(8k-3)^{8k-4}_{8k-7}$, establishing the desired map $\mathbf{5}$. 

To see that all parts of Diagram (\ref{enlarge diagram}) commute, first note that by Lemma~\ref{lemma step 5 enlarge diagram} and Lemma~\ref{lemma step 4 enlarge diagram}, both the triangles above the map $\mathbf{3}$ and under the map $\mathbf{4}$ commute. Next, by the construction of the map $\mathbf{5}$, the triangles above it commute. Finally, by \textbf{Step~1.1.3} and the cofiber sequence of the left most column in Diagram (\ref{enlarge diagram}), the triangle under the map $\mathbf{5}$ commutes. Therefore, all parts of Diagram (\ref{enlarge diagram}) commute.

Now, let's prove Lemma~\ref{lemma step 5 enlarge diagram}.

\begin{proof}[Proof of Lemma~\ref{lemma step 5 enlarge diagram}]
	The composite in the statement splits into the following two composites.
	\begin{equation} \label{step 5 splitting equation 1}
	\xymatrix{S^{8k+4} \ar[r] & S^{8k+1} \ar[r]^-{\mathbf{6}} & X(8k-3)^{8k}_{8k-3}}	
	\end{equation}
	\begin{equation} \label{step 5 splitting equation 2}
	\xymatrix{S^{8k+4} \ar[r] & S^{8k-2} \ar[r]^-{\mathbf{7}} & X(8k-3)^{8k}_{8k-3}}	
	\end{equation}
	
For the first composite (\ref{step 5 splitting equation 1}), let's study the second map $\mathbf{6}$. By Proposition~\ref{eta square attaching maps within columns} and Corollary~\ref{2 and eta attaching maps within columns}, $X(8k-3)^{8k}_{8k-3}$ is a 3 cell complex, with cells in dimensions $8k, \ {8k-1}, \ 8k-3$, and with a 2 and $\eta^2$-attaching map. Since $\eta^3 \neq 0$, there is a nonzero differential
$$\eta[8k] \rightarrow \eta^3[8k-3]$$
in the Atiyah--Hirzebruch spectral sequence of $X(8k-3)^{8k}_{8k-3}$. It follows that the second map $\mathbf{6}$ must map through its $(8k-1)$-skeleton: $S^{8k-1} \vee S^{8k-3}$. Since $\pi_4 = 0$, the map $\mathbf{6}$ must further map through $S^{8k-1}$ and the composite (\ref{step 5 splitting equation 1}) can be decomposed as 
$$\xymatrix{S^{8k+4} \ar[r] & S^{8k+1} \ar[r] & S^{8k-1} \ar@{^{(}->}[r] &  X(8k-3)^{8k}_{8k-3}}.$$
Therefore, due to the relation
$$\pi_2 \cdot \pi_3 = 0,$$
the first composite (\ref{step 5 splitting equation 1}) is zero.

For the second composite (\ref{step 5 splitting equation 2}), the second map $\mathbf{7}$ must map through the ${(8k-2)}$-skeleton of $X(8k-3)^{8k}_{8k-3}$, which is $S^{8k-3}$. Then it follows from the relation 
$$\pi_1 \cdot \pi_6 = 0$$
that the second composite (\ref{step 5 splitting equation 2}) is zero. This completes the proof.
\end{proof}

\begin{figure}[h]
\begin{center}
\makebox[\textwidth]{\includegraphics[trim={2.5cm 5.5cm 0.3cm 7.8cm}, clip, page = 1, scale = 0.8]{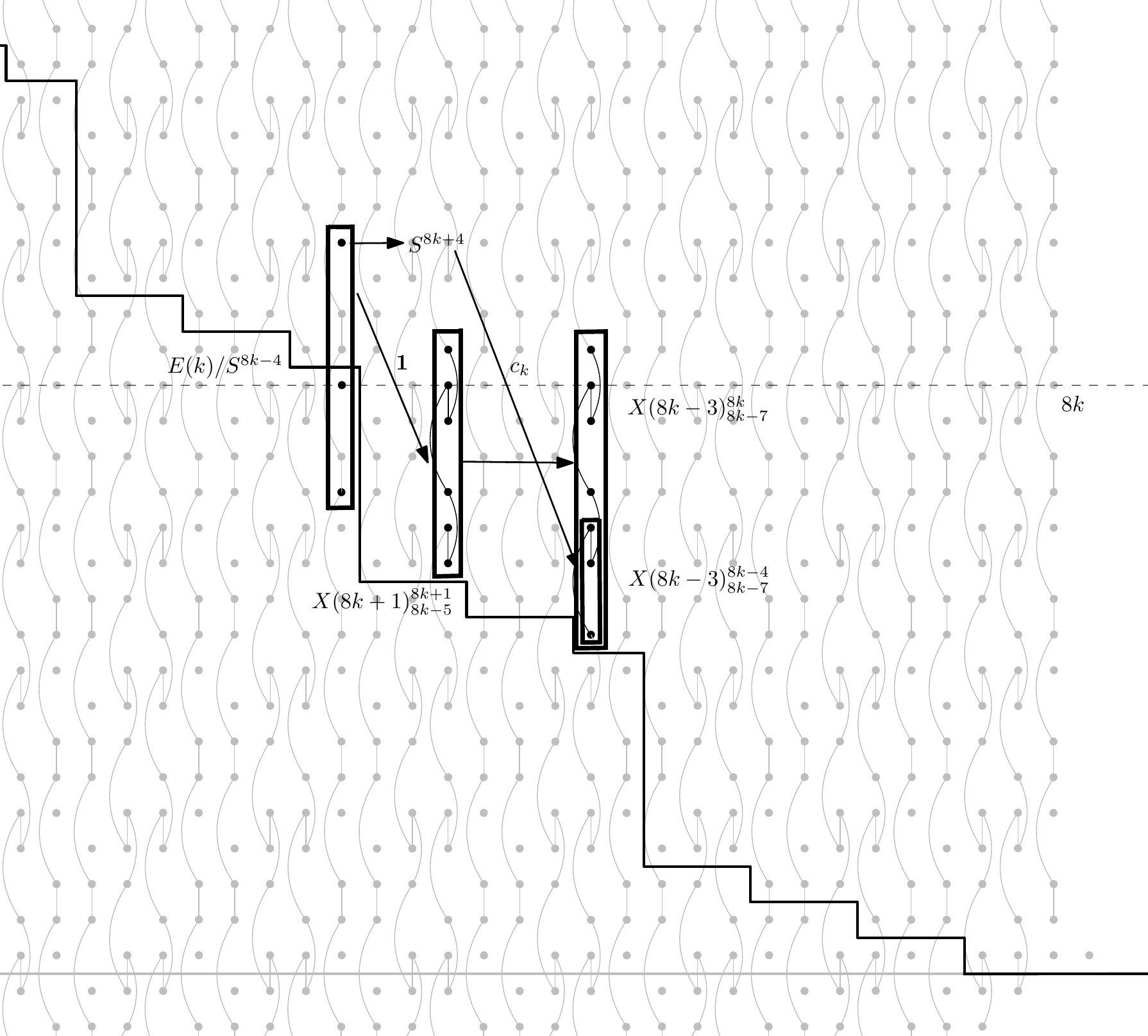}}
\end{center}
\begin{center}
\caption{Step 1.1.5 picture.}
\hfill
\label{fig:Step1point1point5Pic}
\end{center}
\end{figure}

Now we claim that the map $\mathbf{5}$ is our desired map $c_k$ in Proposition~\ref{proposition: construct a}. In fact, part of Diagram (\ref{enlarge diagram}) gives us the following commutative diagram.

\begin{equation}\label{cd3}
\xymatrix{     
   E(k)/S^{8k-4} \ar@{->>}[d] \ar[r]^-{\mathbf{1}}  &    X(8k+1)^{8k+1}_{8k-5} \ar@^{->}[r] & X(8k-3)_{8k-7}^{8k} \\
   S^{8k+4} \ar[rr]^{c_{k}} & & X(8k-3)^{8k-4}_{8k-7} \ar@{^{(}->}[u]}
\end{equation}
Putting Diagrams (\ref{kill first cell in E}) and (\ref{cd3}) together, we have the following commutative diagram.
\begin{displaymath}
    \xymatrix{
      E(k) \ar@{->>}[d] \ar@{^{(}->}[r] & X(8k+4)_{8k-4}^\infty \ar@^{->}[r] & X(8k+1)_{8k-5}^\infty \ar@^{->}[r] & X(8k-3)_{8k-7}^\infty \\ 
   E(k)/S^{8k-4} \ar@{->>}[d] \ar[rr]^{\mathbf{1}} & &    X(8k+1)^{8k+1}_{8k-5} \ar@^{->}[r] \ar@{^{(}->}[u] & X(8k-3)_{8k-7}^{8k}  \ar@{^{(}->}[u] \\
    S^{8k+4} \ar[rrr]^{c_{k}} & & & X(8k-3)^{8k-4}_{8k-7} \ar@{^{(}->}[u]
   } 
\end{displaymath}

Forgetting some terms in this diagram, we obtain Diagram (\ref{diagram: construct a}) in Proposition~\ref{proposition: construct a}.

\subsection{Proof of Proposition~\ref{proposition: construction of u,v}} \label{subsec: hep}
The following Lemma~\ref{homotopy extension} is essentially the homotopy extension property.

\begin{lem}\label{homotopy extension} 
Suppose that we have the following commutative diagram in the stable homotopy category
\begin{equation}\label{diagram before extension}
\xymatrix{
A \ar@{=}[r] \ar[d]_{\mathbf{1}} & A \ar[d]^{\mathbf{2}}\\
B \ar[r] \ar[d] & C \ar[dr] \ar[d] & \\
B/A \ar[drr] & C/A & G  \\
& & F \ar[u]
}
\end{equation}
where $B/A$ and $C/A$ are the cofibers of the maps ${\mathbf{1}}:A\rightarrow B$ and ${\mathbf{2}}:A\rightarrow C$ respectively. Then it can be extended into the following commutative diagram:
\begin{displaymath}
\xymatrix{
A \ar@{=}[r] \ar[d]_{\mathbf{1}} & A \ar[d]_{\mathbf{2}}\\
B \ar[r] \ar[d] & C \ar[dr] \ar[d] & \\
B/A \ar@{-->}[r] \ar[drr] & C/A \ar@{-->}[r] & G \\
& & F \ar[u]
}
\end{displaymath}
\end{lem}

\begin{proof} We can first extend the commutative diagram (\ref{diagram before extension}) to the following commutative diagram:
\begin{equation*}
\xymatrix{
A \ar@{=}[r] \ar_{\mathbf{1}}[d] & A \ar^{\mathbf{2}}[d] \\
B \ar^{\mathbf{3}}[r] \ar_{\mathbf{4}}[d] & C \ar^{\mathbf{5}}[dr] \ar^{\mathbf{6}}[d] & \\
B/A \ar[drr]_<<<<<<{\mathbf{7}} \ar@{-->}[r]^{\mathbf{9}} \ar[d]_{\mathbf{10}} & C/A \ar[d]^<<{\mathbf{11}}& G  \\ 
\Sigma A \ar@{=}[r] & \Sigma A & F \ar_{\mathbf{8}}[u]
}
\end{equation*}
Note that the map $\mathbf{9}: B/A \rightarrow C/A$ is not unique in general. We choose one and stick with our choice. Since the composite 
$$\mathbf{5}\circ \mathbf{2}=\mathbf{5}\circ \mathbf{3}\circ \mathbf{1}=\mathbf{8}\circ \mathbf{7}\circ \mathbf{4}\circ \mathbf{1}:A \longrightarrow G$$
is the zero map, there exists a map 
$$\mathbf{12}:C/A \longrightarrow G, $$ 
making the diagram commute.
\begin{displaymath}
\xymatrix{
C \ar^{\mathbf{5}}[dr] \ar^{\mathbf{6}}[d] & \\
C/A \ar[r]^{\mathbf{12}} & G
}
\end{displaymath}

Now consider the map 
$$
\mathbf{13}=\mathbf{12}\circ \mathbf{9}-\mathbf{8}\circ \mathbf{7}:B/A\longrightarrow G
$$
The map $\mathbf{13}$ is not zero in general. If it were zero, we then have the commutative diagram as requested. 

The fix is to modify the map $\mathbf{12}$. Note that the composite
$$
\mathbf{13} \circ \mathbf{4}=\mathbf{12}\circ \mathbf{9}\circ \mathbf{4}-\mathbf{8}\circ \mathbf{7}\circ \mathbf{4}=\mathbf{12}\circ \mathbf{6}\circ \mathbf{3}-\mathbf{5}\circ \mathbf{3}: B \longrightarrow G
$$
is the zero map. Therefore, by the cofiber sequence
\begin{displaymath}
\xymatrix{
B \ar^{\mathbf{4}}[r]  & B/A \ar[r]^{\mathbf{10}} & \Sigma A,
}
\end{displaymath}
there exists a map 
$$\mathbf{14}: \Sigma A \longrightarrow G$$ 
such that $\mathbf{14} \circ \mathbf{10}=\mathbf{13}$. We define the map 
$$\mathbf{12}':=\mathbf{12}-\mathbf{14}\circ \mathbf{11}: C/A\longrightarrow G.$$
Then the following diagram commutes as requested.
\begin{equation*}
\xymatrix{A\ar@{=}[r]\ar^{\mathbf{1}}[d]&A\ar^{\mathbf{2}}[d]\\
B \ar^{\mathbf{3}}[r]\ar^{\mathbf{4}}[d]& C\ar^{\mathbf{5}}[dr] \ar^{\mathbf{6}}[d]& \\
B/A\ar[drr]_{\mathbf{7}}\ar^{\mathbf{9}}[r] & C/A\ar^{\mathbf{12'}}[r]& G\\
& & F \ar^{\mathbf{8}}[u]
}
\end{equation*}
In fact, we have that
$$\mathbf{12}'\circ \mathbf{6} = \mathbf{12} \circ \mathbf{6}-\mathbf{14}\circ \mathbf{11} \circ \mathbf{6} =  \mathbf{12} \circ \mathbf{6} = \mathbf{5},$$
$$\mathbf{12}'\circ \mathbf{9} =\mathbf{12} \circ \mathbf{9}-\mathbf{14}\circ \mathbf{11} \circ \mathbf{9} = \mathbf{12} \circ \mathbf{9}-\mathbf{14}\circ \mathbf{10} = \mathbf{12} \circ \mathbf{9}-\mathbf{13}= \mathbf{8}\circ \mathbf{7}.$$
\end{proof}
\begin{figure}[h]
\begin{center}
\makebox[\textwidth]{\includegraphics[trim={2.5cm 5.5cm 0.3cm 7.8cm}, clip, page = 1, scale = 0.8]{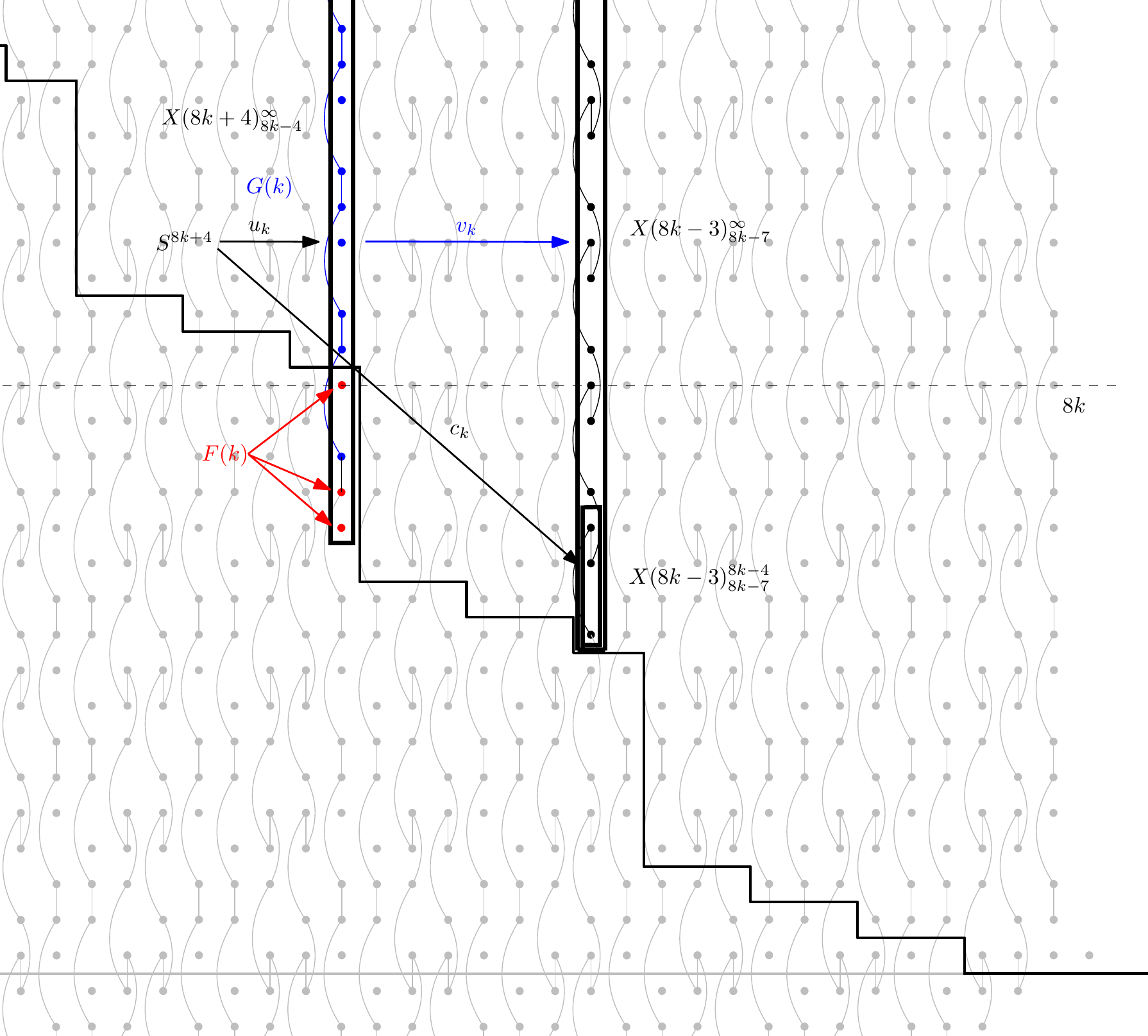}}
\end{center}
\begin{center}
\caption{Step 1.2 picture.}
\hfill
\label{fig:Step1point2Pic}
\end{center}
\end{figure}

From the commutative diagram (\ref{diagram: construct a}) in Proposition~\ref{proposition: construct a} and the definitions of $F(k)$ and $G(k)$, we have the following commutative diagram 
\begin{displaymath}
\xymatrix{F(k)\ar@{=}[r]\ar@{^{(}->}[d]&F(k)\ar@{^{(}->}[d]\\
E(k) \ar@{^{(}->}[r]\ar@{->>}[d]& X(8k+4)_{8k-4}^\infty \ar@^{->}[dr] \ar@{->>}[d] & \\ 
S^{8k+4}
\ar[drr]_{c_{k}} & G(k) & X(8k-3)_{8k-7}^\infty \\
& & X(8k-3)_{8k-7}^{8k-4}\ar@{^{(}->}[u]
}
\end{displaymath}

By Lemma~\ref{homotopy extension}, we can extend it to the following commutative diagram
\begin{equation*}
\xymatrix{F(k)\ar@{=}[r]\ar@{^{(}->}[d]&F(k)\ar@{^{(}->}[d]\\
E(k) \ar@{^{(}->}[r]\ar@{->>}[d]& X(8k+4)_{8k-4}^\infty\ar@^{->}[dr] \ar@{->>}[d]& \\ 
S^{8k+4}\ar[r]^{u_{k}}
\ar[drr]_{c_{k}} & G(k) \ar[r]^-{v_{k}} & X(8k-3)_{8k-7}^\infty \\
& & X(8k-3)_{8k-7}^{8k-4}\ar@{^{(}->}[u]
}
\end{equation*}

Removing the terms $F(k)$, we have the commutative diagram (\ref{diagram: construction of u,v}) in Proposition~\ref{proposition: construction of u,v}. It is clear that The map $u_{k}$ induces an isomorphism on $H_{8k+4}(-;\mathbb{F}_{2})$. In other words, $(S^{8k+4}, u_{k})$ is an $\textup{H}\mathbb{F}_2$-subcomplex of $G(k)$. The completes the proof of Proposition~\ref{proposition: construction of u,v}.

\subsection{Proof of Proposition~\ref{proposition: construct f}} 
In this subsection, we prove Proposition~\ref{proposition: construct f} that for $k\geq 1$, the following composite is zero. 
\begin{equation*}%\label{bottom in G killed}
\xymatrix{
S^{8k-2} \ar@{^{(}->}[r] & G(k) \ar[r]^-{v_{k}} & X(8k-3)_{8k-7}^\infty \ar@^{->}[r]& X(8k-4)_{8k-7}^\infty \ar[r]^-{f_{k-1}} & S^{0} 
}
\end{equation*}

We start with the commutative diagram (\ref{gf factors throught W}) for the case $k-1$ in (iii) of Theorem~\ref{thm: inductive fk}. We enlarge the commutative diagram (\ref{gf factors throught W}) for the case $k-1$ in the following way
\begin{equation} \label{prop53 enlarge}
\xymatrix{
X(8k-3)_{8k-7}^{\infty} \ar@^{->}[r]& X(8k-4)_{8k-7}^\infty \ar[rr]^-{f_{k-1}} & & S^{0}\\
 X(8k-3)_{8k-7}^{8k-4} \ar@{^{(}->}[u] \ar@^{->}[r] & X(8k-4)^{8k-4}_{8k-7} \ar@{^{(}->}[u] & & \\
& S^{8k-4} \ar@{^{(}->}[u] \ar[rr]^{a_{k-1}} & & X(8k-12)^{8k-12}_{8k-15} \ar[uu]^{b_{k-1}}
}
\end{equation}
We next state a lemma about the map $v_{k}$, whose proof we postpone until the end of this subsection. This Lemma~\ref{lem:step4vkfactorthrough} will also be used in Subsection 5.6.

\begin{lem}\label{lem:step4vkfactorthrough}
There exists a map 
$$w_{k}:G(k)^{8k+1}\longrightarrow X(8k-3)_{8k-7}^{8k-4}$$ 
that fits into the following commutative diagram
\begin{equation} \label{vk property}
\xymatrix{G(k)\ar[r]^-{v_{k}}& X(8k-3)^{\infty}_{8k-7}\\
G(k)^{8k+1}\ar@{^{(}->}[u]\ar[r]^-{w_{k}}& X(8k-3)^{8k-4}_{8k-7}\ar@{^{(}->}[u]}.
\end{equation}
\end{lem}

\begin{figure}[h]
\begin{center}
\makebox[\textwidth]{\includegraphics[trim={2.5cm 5.5cm 0.3cm 7.8cm}, clip, page = 1, scale = 0.8]{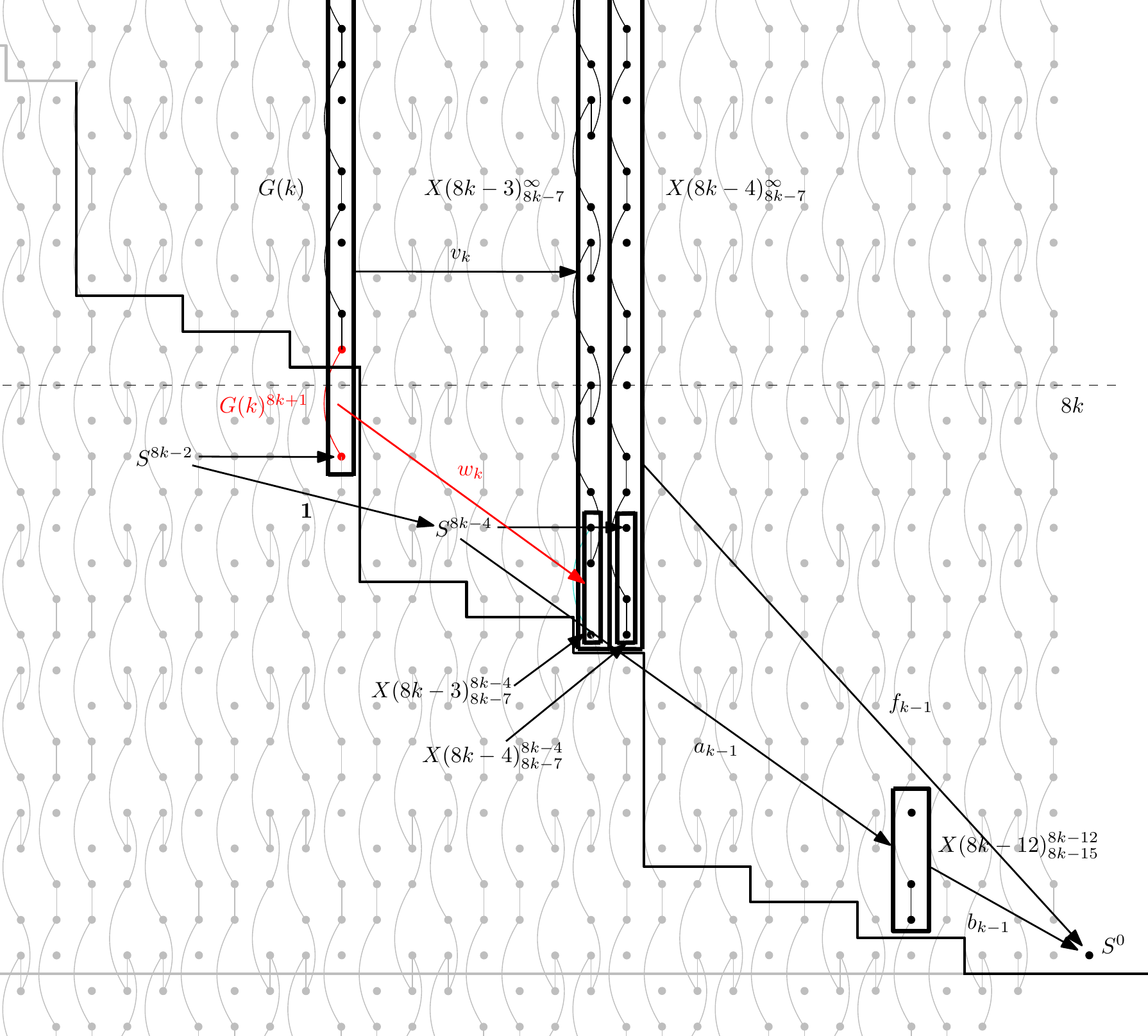}}
\end{center}
\begin{center}
\caption{Step 1.4 picture.}
\hfill
\label{fig:Step1point4Pic}
\end{center}
\end{figure}
Putting these above two diagrams (\ref{prop53 enlarge}) and (\ref{vk property}) together, we obtain the following commutative diagram
\begin{equation} \label{prop53 large diagram}
\xymatrix{
S^{8k-2}\ar@{^{(}->}[r]\ar@{^{(}->}[dr] \ar@{=}[dd] &G(k)\ar[r]^-{v_{k}}& X(8k-3)^{\infty}_{8k-7}\ar@^{->}[r]& X(8k-4)_{8k-7}^\infty \ar[rr]^-{f_{k-1}} & & S^{0}\\&
G(k)^{8k+1}\ar@{^{(}->}[u]\ar[r]^-{w_{k}}& X(8k-3)_{8k-7}^{8k-4} \ar@{^{(}->}[u]\ar@^{->}[r] & X(8k-4)^{8k-4}_{8k-7} \ar@{^{(}->}[u] & & \\S^{8k-2} \ar@{-->}[rrr]^{\mathbf{1}}& &
& S^{8k-4} \ar@{^{(}->}[u] \ar[rr]^{a_{k-1}} & & X(8k-12)^{8k-12}_{8k-15} \ar[uu]^{b_{k-1}}
}
\end{equation}

It is clear that Proposition~\ref{proposition: construct f} follows from the following Lemma~\ref{lem:step4lift2to1}, Lemma~\ref{lem:step43a0} and the above commutative diagram.

\begin{lem}\label{lem:step4lift2to1}
The following composite
$$\xymatrix{ S^{8k-2} \ar@{^{(}->}[r] & G(k)^{8k+1} \ar[r]^-{w_{k}} & X(8k-3)_{8k-7}^{8k-4} \ar@^{->}[r] & X(8k-4)^{8k-4}_{8k-7} }$$ 
factors through $S^{8k-4}$, giving the map $\mathbf{1}$ in the diagram (\ref{prop53 large diagram}).
\end{lem}

\begin{lem} \label{lem:step43a0}
	The following composite is zero.
	\begin{displaymath}
\xymatrix{
S^{8k-2} \ar[r]^{\mathbf{1}} & S^{8k-4} \ar[r]^-{a_{k-1}} & X(8k-12)^{8k-12}_{8k-15}
}
\end{displaymath}
\end{lem}

We first prove Lemma~\ref{lem:step4lift2to1} and Lemma~\ref{lem:step43a0}, and then prove Lemma~\ref{lem:step4vkfactorthrough}.

\begin{proof}[Proof of Lemma~\ref{lem:step4lift2to1}]
By Lemma~\ref{the 3 cell complex wk}, the 3 cell complex $X(8k-4)^{8k-4}_{8k-7}$ splits as
$$S^{8k-4} \vee \Sigma^{8k-7} C2.$$
To show that the map 
$$S^{8k-2} \longrightarrow X(8k-4)^{8k-4}_{8k-7} \simeq S^{8k-4} \vee \Sigma^{8k-7} C2 $$ 
maps through $S^{8k-4}$, we need to check the following composite is zero.
\begin{displaymath}
\xymatrix{
S^{8k-2} \ar[r] & X(8k-4)^{8k-4}_{8k-7} \ar@{->>}[r] & \Sigma^{8k-7} C2
}
\end{displaymath}
This composite corresponds to an element in the group
$$\pi_{8k-2}(\Sigma^{8k-7} C2) = \pi_5 C2 = 0.$$
The last equation follows from the fact that $\pi_4 =\pi_5 =0$. This completes the proof.	
\end{proof}

\begin{proof}[Proof of Lemma~\ref{lem:step43a0}]
By Lemma~\ref{the 3 cell complex wk}, the 3 cell complex $X(8k-12)^{8k-12}_{8k-15}$ splits as
$$S^{8k-12} \vee \Sigma^{8k-15} C2.$$	
Therefore, the composite
\begin{displaymath}
\xymatrix{
S^{8k-2} \ar[r]^{\mathbf{1}} & S^{8k-4} \ar[r]^-{a_{k-1}} & X(8k-12)^{8k-12}_{8k-15} = S^{8k-12} \vee \Sigma^{8k-15} C2
}
\end{displaymath}
corresponds to an element in the group
\begin{align*}
\big(\pi_{8k-4}S^{8k-12} \oplus \pi_{8k-4}(\Sigma^{8k-15} C2)\big) \cdot \pi_{8k-2}S^{8k-4} &  =(\pi_8 \oplus \pi_{11} C2) \cdot \pi_2 \\& \subseteq \pi_8 \cdot \pi_2 \oplus \pi_{13}C2 =0.
\end{align*}
The last equation follows from the facts that 
$$\pi_8 \cdot \pi_2 =0, \ \pi_{12} =\pi_{13} =0.$$ 
This completes the proof.
\end{proof}

Now we present the proof of Lemma~\ref{lem:step4vkfactorthrough}.

\begin{proof}[Proof of Lemma~\ref{lem:step4vkfactorthrough}]
From the cofiber sequence 
$$
\xymatrix{
X(8k-3)_{8k-7}^{8k-4} \ar@{^{(}->}[r] & X(8k-3)_{8k-7}^{\infty} \ar@{->>}[r] & X(8k-3)_{8k-3}^{\infty} },$$
we need to show that the composite
\begin{equation} \label{wkvk compo}
\xymatrix{
G(k)^{8k+1} \ar@{^{(}->}[r] & G(k) \ar[r]^-{v_{k}} & X(8k-3)_{8k-7}^{\infty}\ar@{->>}[r] & X(8k-3)_{8k-3}^{\infty} }\end{equation}
is zero. By Proposition~\ref{eta square attaching maps within columns}, $G(k)^{8k+1}$ is a 2 cell complex with an $\eta^2$-attaching map:
$$G(k)^{8k+1}=\Sigma^{8k-2}C\eta^{2}.$$
Our strategy to show the composite (\ref{wkvk compo}) being zero is to first deal with the bottom cell and then the top cell. 

By the cellular approximation theorem, the restriction of the composite (\ref{wkvk compo}) to the bottom cell $S^{8k-2}$ of $G(k)^{8k+1}$ maps through the bottom cell $S^{8k-3}$ of ${X(8k-3)_{8k-3}^{\infty}}$, by either $\eta$ or $0$. The possibility of $\eta$ is ruled out by a cell diagram chasing argument due to the $\eta$-attaching map between the cells in dimensions $8k-3$ and $8k-5$ in $X(8k-3)_{8k-3}^{\infty}$.

Therefore, the composite (\ref{wkvk compo}) factors through the top cell $S^{8k+1}$ of $G(k)^{8k+1}$. We can further require it factor through the top 2 cells of $G(k)^{8k+2}$, namely 
$$G(k)^{8k+2}_{8k+1} = \Sigma^{8k+1} C2.$$
By the cellular approximation theorem, it maps through the $(8k+2)$-skeleton of 
$X(8k-3)_{8k-3}^{\infty}$. Note that there is no cell in dimension $8k+2$ in $X(8k-3)_{8k-3}^{\infty}$, so it maps through the 4 cell complex $X(8k-3)_{8k-3}^{8k+1}$. We have the following commutative diagram.
\begin{equation*}
	\xymatrix{
	S^{8k-2} \ar@{^{(}->}[d] \ar@/^1pc/[rrrd]^-{=0} & & &\\
	G(k)^{8k+1} \ar@{^{(}->}[r] \ar@{->>}[dd] & G(k)^{8k+2} \ar[r] \ar@{->>}[d] & X(8k-3)_{8k-3}^{8k+1} \ar@{^{(}->}[r] & X(8k-3)_{8k-3}^{\infty} \\
	 & \Sigma^{8k+1} C2 \ar@{-->}[ru] & X(8k-3)_{8k-3}^{8k} \ar@{^{(}->}[u] &\\
	S^{8k+1} \ar@{^{(}->}[ru] \ar@/_1pc/@{-->}[rru] & & &
	}
\end{equation*}
To prove this lemma, it suffices to show the following composite is zero.
\begin{equation} \label{lem513 last step}
\xymatrix{
S^{8k+1} \ar@{^{(}->}[r] & \Sigma^{8k+1} C2 \ar[r] & X(8k-3)_{8k-3}^{8k+1}.
}	
\end{equation}
Firstly, post-composing with the quotient map 
$$\xymatrix{X(8k-3)_{8k-3}^{8k+1} \ar@{->>}[r] & S^{8k+1}}$$
must be zero. This is due to the fact that it maps through the mod 2 Moore spectrum. Therefore, the composite (\ref{lem513 last step}) must map through the $8k$-skeleton of ${X(8k-3)_{8k-3}^{8k+1}}$, namely the 3 cell complex ${X(8k-3)_{8k-3}^{8k}}$:
\begin{equation} \label{lem513 a map}
	S^{8k+1} \longrightarrow X(8k-3)_{8k-3}^{8k}.
\end{equation}
Now let's consider the Atiyah--Hirzebruch filtration of this map (\ref{lem513 a map}). It cannot be detected in filtration $8k$, since there is a nontrivial differential in the Atiyah--Hirzebruch spectral sequence of ${X(8k-3)_{8k-3}^{8k}}$:
$$\eta[8k] \rightarrow \eta^3[8k-3],$$
which is due to the $\eta^2$-attaching map by Proposition~\ref{eta square attaching maps within columns}. If it is detected in filtration $8k-3$, then it must be zero since $\pi_4=0$. Therefore, if it is nonzero, then it must be detected by $\eta^2[8k-1]$. In this case, post-composing with the inclusion to ${X(8k-3)_{8k-3}^{8k+1}}$ is zero, due to the $\eta$-attaching map between the cells in dimensions $8k-1$ and $8k+1$, and therefore the Atiyah--Hirzebruch differential
$$\eta[8k+1] \rightarrow \eta^2[8k-1].$$
In sum, regardless of the actual Atiyah--Hirzebruch filtration of the map (\ref{lem513 a map}), the following composite is always zero.
$$\xymatrix{S^{8k+1} \ar[r]^-{(\ref{lem513 a map})} & X(8k-3)_{8k-3}^{8k} \ar@{^{(}->}[r] &  X(8k-3)_{8k-3}^{8k+1}.  }$$
This completes the proof of the lemma.
\end{proof}

\subsection{Proof of Proposition~\ref{f g a b satisfy requirements}}
We check that the two diagrams (\ref{f is quotient}) and (\ref{gf factors throught W}) in (i) and (iii) of Theorem~\ref{thm: inductive fk} commute for the 4 maps $(f_{k}, \ g_{k}, \ a_{k}, \ b_{k})$.

For the diagram (\ref{f is quotient}) in (i) of Theorem~\ref{thm: inductive fk} for the case $k$, we put together the following commutative diagrams 
\begin{itemize}
	\item diagram (\ref{diagram: construction of f}) in Step 1.4,
	\item diagram (\ref{f is quotient}) in (i) of Theorem~\ref{thm: inductive fk} for the case $k-1$,
	\item the upper right corner of diagram (\ref{diagram: construction of u,v}) in Proposition~\ref{proposition: construction of u,v}.
\end{itemize}

%\begin{displaymath}
%\xymatrix{
%X(8k+4) \ar[r] \ar@{->>}[d] & X(8k-4) \ar@{->>}[d] \\
%X(8k+4)_{8k-4}^\infty \ar[r]& X(8k-4)_{8k-7}^\infty},
%\end{displaymath}

\begin{displaymath}
\xymatrix{X(8k+4) \ar@^{->}[r] \ar@{->>}[d] & X(8k-3)\ar@^{->}[r] \ar@{->>}[d] & X(8k-4) \ar[rr] \ar@{->>}[d] & & S^{0}\\
X(8k+4)_{8k-4}^\infty \ar@{->>}[d] \ar@^{->}[r] &  X(8k-3)_{8k-7}^{\infty}\ar@^{->}[r]  & X(8k-4)_{8k-7}^\infty \ar[rru]^{f_{k-1}} & \\
G(k) \ar[ru]_{v_{k}} \ar@{->>}[d] && &  & \\
X(8k+4)_{8k+1}^\infty \ar@/_2pc/[uuurrrr]_{f_{k}} & & &
} 
\end{displaymath}
The commutativity of the upper left corner of this diagram is due to the compatibility of each columns. 

For the diagram (\ref{gf factors throught W}) in (iii) of Theorem~\ref{thm: inductive fk} for the case $k$, we put together the following commutative diagrams 
\begin{itemize}
    \item diagram (\ref{diagram: construction of f}) in Step 1.4,
	\item the lower half of diagram (\ref{diagram: construction of u,v}) in Proposition~\ref{proposition: construction of u,v}.
	\end{itemize}

\begin{displaymath}
\xymatrix{ 
S^{8k+4} \ar[r]^-{c_{k}} \ar[d]^{u_{k}}& X(8k-3)_{8k-7}^{8k-4}\ar@^{->}[r]\ar@{->>}[d] & X(8k-4)^{8k-4}_{8k-7} \ar@{^{(}->}[d] \\
G(k)\ar@{->>}[d] \ar[r]^-{v_{k}} &X(8k-3)_{8k-7}^\infty\ar@^{->}[r]  & X(8k-4)_{8k-7}^\infty\ar[d]^{f_{k-1}} \\ X(8k+4)_{8k+1}^{\infty}\ar[rr]^{f_{k}}&&S^{0}
}
\end{displaymath}

By the definitions of $g_k$ in Step 1.2 and $b_k$ in Step~1.4, the composites in the left and right columns give us $g_k$ and $b_k$ respectively.

Therefore, we have the diagram (\ref{gf factors throught W}) in (iii) of Theorem~\ref{thm: inductive fk} for the case $k$. This completes the proof.

\subsection{Proof of Proposition~\ref{prop induct beta}}
In this subsection, we prove Proposition~\ref{prop induct beta}: There exists one choice of $f_{k}$ in Step~1.4 such that 
\begin{equation}\label{prop55 inductive on beta}
\phi_{k}-\phi_{k-2} \cdot  \chi_{k}\in \langle \phi_{k-1}, 2, \tau_{k}\rangle 
\end{equation}
\begin{figure}
\begin{center}
\makebox[\textwidth]{\includegraphics[trim={2cm 5.5cm 0.3cm 7cm}, clip, page = 1, scale = 0.8]{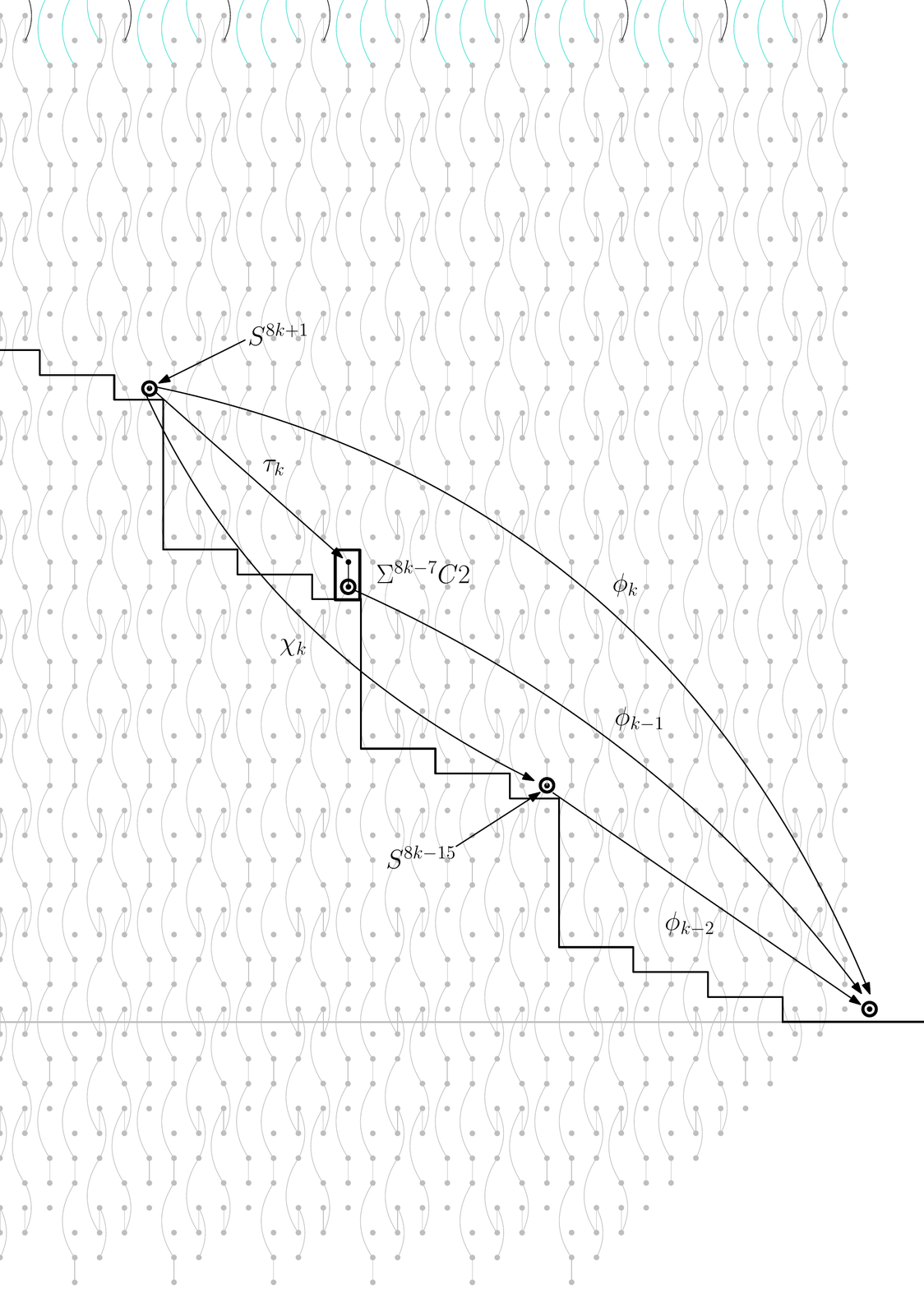}}
\end{center}
\begin{center}
\caption{Step 1.6 picture.}
\hfill
\label{fig:Step1point6Pic}
\end{center}
\end{figure}

where $\phi_{m} \in \pi_{8m+1}$ is the restriction of $f_{m}$ to the bottom cell of $X(8k+4)_{8k+1}^{\infty}$, $\tau_{k}\in \{0,8\sigma\}$ and $\chi_{k}\in \pi_{16}(S^{0})$. Note that by Lemma~\ref{eta attaching map between columns}, $\phi_0 = \eta$ and we set $\phi_{-1}=0$. 

Consider the following composite 
\begin{equation} \label{prop55 general}
\xymatrix{
G(k)^{8k+1} \ar[r]^-{w_{k}} & X(8k-3)^{8k-4}_{8k-7} \ar@^{->}[r] & X(8k-4)^{8k-4}_{8k-7} \ar@{^{(}->}[r] & X(8k-4)_{8k-7}^{\infty}\ar[r]^-{f_{k-1}} & S^{0}
}
\end{equation}
By Lemma~\ref{the 3 cell complex wk}, the 3 cell complex $X(8k-4)^{8k-4}_{8k-7}$ splits:  
$$X(8k-4)^{8k-4}_{8k-7} \simeq \Sigma^{8k-7} C2\vee S^{8k-4}.$$ 
Therefore, the composite (\ref{prop55 general}) can be written as the sum of the following two composites (\ref{step6decomposition 3}) and (\ref{prop55 summand 1}).
\begin{equation}\label{step6decomposition 3}
\xymatrix{G(k)^{8k+1} \ar[r]^-{\mathbf{1}} & \Sigma^{8k-7} C2 \ar@{^{(}->}[r] & X(8k-4)^{\infty}_{8k-7} \ar[r]^-{f_{k-1}} & S^{0}.} 
\end{equation}
\begin{equation} \label{prop55 summand 1}
\xymatrix{G(k)^{8k+1} \ar[r]^-{\mathbf{2}} & S^{8k-4} \ar[r]^-{g_{k}} &  X(8k-4)_{8k-7}^{\infty} \ar[r]^-{f_{k-1}} & S^{0}.}
\end{equation} 
For the composite (\ref{step6decomposition 3}), first note that the map $\mathbf{1}$ equals zero when restrict to bottom cell $S^{8k-2}$ of $G(k)^{8k+1}$. In fact, it corresponds to an element in 
$$\pi_{8k-2}\Sigma^{8k-7}C2 = \pi_5 C2 = 0,$$
which follows from the fact that $\pi_4 = \pi_5 = 0$. 
$$\xymatrix{
S^{8k-2} \ar@{^{(}->}[d] \ar[rd]^{=0} & & & \\
G(k)^{8k+1} \ar@{->>}[d] \ar[r]^-{\mathbf{1}} & \Sigma^{8k-7} C2 \ar@{^{(}->}[r] & X(8k-4)^{\infty}_{8k-7} \ar[r]^-{f_{k-1}} & S^{0}\\
S^{8k+1} \ar@{-->}[ru] & & & 
}$$
Next note that the composite
$$\xymatrix{
\Sigma^{8k-7} C2 = X(8k-4)^{8k-6}_{8k-7} \ar@{^{(}->}[r] & X(8k-4)^{\infty}_{8k-7} \ar[r]^-{f_{k-1}} & S^{0} 
}$$
restricts to $\phi_{k-1}$ on the bottom cell $S^{8k-7}$ of $\Sigma^{8k-7} C2$. Therefore, we have the following commutative diagram:
\begin{equation}  \label{diagram: pi8path}
\xymatrix{ G(k)^{8k+1} \ar@{->>}[d] \ar[rr]^-{(\ref{step6decomposition 3})} & & S^{0}\\ 
S^{8k+1} \ar[rru]_-{\xi_{k}} & &   }
\end{equation}
where $\xi_{k}\in \langle \phi_{k-1},2, \tau_{k}\rangle$ with $\tau_{k}$ an element in $\pi_7$ that is annihilated by multiplication by 2, namely 0 or $8\sigma$.\\

For the composite (\ref{prop55 summand 1}), by the diagram (\ref{gf factors throught W}) for the case $k-1$, we can rewrite it as
\begin{equation} \label{prop55 summand 2}
\xymatrix{
G(k)^{8k+1} \ar[r]^-{\mathbf{2}} & S^{8k-4} \ar[r]^-{a_{k}} & X(8k-12)^{8k-12}_{8k-15} \ar[r]^-{b_{k-1}} & S^{0}.
}	
\end{equation}
Using the splitting 
$$X(8k-12)^{8k-12}_{8k-15} \simeq S^{8k-12} \vee \Sigma^{8k-15} C2,$$ 
we can rewrite the composite (\ref{prop55 summand 2}) as the sum of the following two composites (\ref{prop55 summand 21}) and (\ref{prop55 summand 22}).
\begin{equation} \label{prop55 summand 21}
\xymatrix{
G(k)^{8k+1} \ar[r]^-{\mathbf{2}} & S^{8k-4} \ar[r] & S^{8k-12}\ar[r] & S^{0}
}	
\end{equation}
\begin{equation} \label{prop55 summand 22}
\xymatrix{
G(k)^{8k+1} \ar[r]^-{\mathbf{2}} & S^{8k-4}\ar[r] & \Sigma^{8k-15} C2 \ar@{^{(}->}[r] & X(8k-12)^{\infty}_{8k-15}  \ar[r]^-{f_{k-2}} & S^{0}
}	
\end{equation}
The composite (\ref{prop55 summand 21}) is zero. In fact, since $G(k)^{8k+1} = \Sigma^{8k-2} C \eta^2$ and
$$\pi_{2} \cdot \pi_{8}=0, \ \pi_{13}=0,$$ 
the composition of the first two maps in (\ref{prop55 summand 21}) is already zero. Therefore, the composite (\ref{prop55 summand 2}) can be identified as (\ref{prop55 summand 22}).

For the composite (\ref{prop55 summand 22}), we have the following lemma.

\begin{lem} \label{prop55 lemma}
The following composite is zero: 
\begin{equation}\label{lemma516}
\xymatrix{
G(k)^{8k+1} \ar[r]^-{\mathbf{2}} & S^{8k-4} \ar[r] & \Sigma^{8k-15} C2 \ar@{->>}[r] & S^{8k-14}.
}
\end{equation}
\end{lem}
\begin{proof}
Consider the following diagram.
\begin{displaymath}
\xymatrix{S^{8k-2}\ar@{^{(}->}[d]\ar[rrrd]^{=0}& & &\\
G(k)^{8k+1}\ar@{->>}[d]\ar[r]^{\mathbf{4}}& S^{8k-4}\ar[r]& \Sigma^{8k-15} C2 \ar@{->>}[r] & S^{8k-14}\\
S^{8k+1}\ar@{-->}[rrru]}
\end{displaymath}
Pre-composing the composite (\ref{lemma516}) with the inclusion of the bottom cell $S^{8k-2}$ of $G(k)^{8k+1}$ gives us the zero map. This is due to the fact that $\pi_{12} = 0$. 

The map from $S^{8k+1}$ to $S^{8k-14}$ can be written as a Toda bracket of the form 
$$\langle \alpha, \beta, \eta^2 \rangle \subseteq \pi_{15},$$
where $\beta \in \pi_{2} = \mathbb{Z}/2$ generated by $\eta^{2}$, and $\alpha \in \pi_{10}= \mathbb{Z}/2$ generated by $\{Ph^{2}_{1}\}$. For a precise argument of this fact, we refer to Lemma 5.3 of \cite{WangXu51Stem}.

The indeterminacy of this Toda bracket is
$$\alpha \cdot \pi_5 + \pi_{13} \cdot \eta^2 = 0,$$
since $\pi_5 = 0, \ \pi_{13} = 0$. We claim that this Toda bracket contains zero, therefore it is zero as a set. This completes the proof of the lemma.

In fact, the only potential nonzero element that this Toda bracket contains is
$$\langle \{Ph^{2}_{1}\}, \eta^2, \eta^2 \rangle.$$
The corresponding Massey product
$$\langle Ph^{2}_{1}, h_1^2, h_1^2 \rangle = 0$$
in filtration 9 of the Adams $E_2$-page, which is higher than all nonzero elements in the Adams $E_\infty$-page. Therefore, this potential nonzero element is also zero.
\end{proof}

By Lemma~\ref{prop55 lemma}, the composite (\ref{prop55 summand 22}) maps through the bottom cell $S^{8k-15}$ of $\Sigma^{8k-15} C2$, and we have the following commutative diagram:
\begin{equation}\label{diagram: pi16path}
\xymatrix{
S^{8k-2} \ar@{^{(}->}[d]  & &  S^{8k-14} & \\
G(k)^{8k+1} \ar[rru]^{(\ref{lemma516})=0}  \ar@{-->}[rrd] \ar@{->>}[d] \ar[r]^-{\mathbf{2}} & S^{8k-4}  \ar[r] & \Sigma^{8k-15} C2 \ar@{->>}[u] \ar@{^{(}->}[r] & X(8k-12)^{\infty}_{8k-15} \ar[r]^-{f_{k-2}} & S^{0}\\
S^{8k+1} \ar@{-->}^{\chi_{k}}[rr] & & S^{8k-15} \ar[rru]_-{\phi_{k-2}} \ar@{^{(}->}[u]
}
\end{equation}
Since $\pi_{13} = 0$, the following composite is zero.
$$
\xymatrix{
S^{8k-2} \ar@{^{(}->}[r] & G(k)^{8k+1} \ar[r] & S^{8k-15}.
}
$$
Therefore, the composite (\ref{prop55 summand 22}) further factors through the top cell $S^{8k+1}$ of $G(k)^{8k+1}$. We denote by $\chi_k$ the corresponding element in $\pi_{16}$. 

Removing some of the terms in (\ref{diagram: pi16path}), we obtain the following diagram:
\begin{equation} \label{diagram: pi16pathshort}
\xymatrix{ G(k)^{8k+1} \ar@{->>}[d] \ar[rr]^-{(\ref{prop55 summand 1})} & & S^{0}\\ 
S^{8k+1} \ar[rru]_-{\phi_{k-2} \cdot \chi_{k}} & &   }
\end{equation}
Adding the diagrams (\ref{diagram: pi16pathshort}) and (\ref{diagram: pi8path}) together, we have the following commutative diagram 
\begin{displaymath}
\xymatrix{ G(k)^{8k+1} \ar@{->>}[d] \ar[rr]^-{(\ref{prop55 general})} & & S^{0}\\ 
S^{8k+1} \ar[rru]_-{\xi_k + \phi_{k-2} \cdot \chi_{k}} & &   }
\end{displaymath}
which can be enlarged into the following commutative diagram:
\begin{displaymath}
\xymatrix{
S^{8k-2} \ar@{=}[r] \ar@{^{(}->}[d] & S^{8k-2} \ar@{^{(}->}[d]\\
G(k)^{8k+1} \ar@{->>}[d] \ar@{^{(}->}[r] & G(k) \ar[dr]^{v_{k}} \ar@{->>}[d] & \\
S^{8k+1} \ar[drr]_-{\xi_k + \phi_{k-2} \cdot \chi_{k}} & X(8k+4)_{8k+1}^{\infty} & S^{0}  \\
& & S^{0} \ar@{=}[u]
}
\end{displaymath}
Using the homotopy extension property that we proved, namely Lemma~\ref{homotopy extension}, we have the following commutative diagram.
\begin{displaymath}
\xymatrix{
S^{8k-2} \ar@{=}[r] \ar@{^{(}->}[d] & S^{8k-2} \ar@{^{(}->}[d]\\
G(k)^{8k+1} \ar@{->>}[d] \ar@{^{(}->}[r] & G(k) \ar[dr]^{v_{k}} \ar@{->>}[d] & \\
S^{8k+1}\ar@{^{(}->}[r]^-{l_{k}} \ar[drr]_-{\xi_k + \phi_{k-2} \cdot \chi_{k}} & X(8k+4)_{8k+1}^{\infty}\ar[r]_-{f_{k}} & S^{0}  \\
& & S^{0} \ar@{=}[u]
}
\end{displaymath}
Note that the map $l_{k}$ induces an isomorphism on $H_{8k+1}(-,\mathbb{F}_{2})$ and therefore is an H$\mathbb{F}_2$-subcomplex. In sum, we have constructed a choice of the map $f_{k}$ that satisfies the condition (\ref{equation: inductive on beta}) in Proposition~\ref{prop induct beta}. This completes the proof of Proposition~\ref{prop induct beta}. 

%%(Note that to make the logic cleaner, we repeat some of the argument  in Step 4 (e.g. constructing a null homotopy of $\mathbf{1}$, restricted bottom cell of $G(8k)^{8k+1}$)) again in this step, but will more careful control. It is possible to combine these two steps together, but that will make the logic very confusing.)

%%%%%%%%%%%%%%%%%%%%%%%%%%%%%%%%
%%%%%%%%%%%%%%%%%%%%%%%%%%%%%%%%
\newpage
\section{Step 2: Upper bound detected by $KO$} \label{sec:AddedNewStep2}
In this section, we prove Proposition~\ref{prop: upper bound detected by ko}: 
\begin{prop}[Proposition~\ref{prop: upper bound detected by ko}]For any $k\geq 1$, the composition 
$$X(8k+2)^{8k-4}\xrightarrow{c(8k+2)^{8k-4}}S^{0}\longrightarrow KO$$
is nonzero. 
\end{prop}
Recall that $X(8k+2)^{8k-4}$ is the homotopy orbit of the free $\Pin$-action on 
$$
 S^{-(8k+2)\widetilde{\mathbb{R}}} \wedge S(4k\mathbb{H})_{+}.
$$
Therefore, we have the following isomorphism: 
$$
KO^{0}(X(8k+2)^{8k-4})=KO^{0}_{\Pin}(S^{-(8k+2)\widetilde{\mathbb{R}}} \wedge S(4k\mathbb{H})_{+}).
$$
%This lead us to the study of $\Pin$-equivariant $KO$-theory. (Note that in order to apply Bott periodicity theorem, we are using periodic $KO$ here.) 

%%%%%%
\subsection{Some results about the $\Pin$-equivariant $KO$-theory} 
In this subsection, we list some results about the group $KO^{0}(S^{a\mathbb{H}+b\widetilde{\mathbb{R}}})$ for various $a,b\in \mathbb{Z}$. These results are established in \cite[Section 5]{Schmidt2003} (see also \cite{LinKO}).
\begin{enumerate}[label=(\Roman*)]
\item\label{item: multiplication} There is a commutative and associative multiplication map (given by tensor product of virtual bundles)
$$
KO^{0}_{\Pin}(S^{a\mathbb{H}+b\widetilde{\mathbb{R}}})\otimes KO^{0}_{\Pin}(S^{c\mathbb{H}+d\widetilde{\mathbb{R}}})\rightarrow KO^{0}_{\Pin}(S^{(a+c)\mathbb{H}+(b+d)\widetilde{\mathbb{R}}}).
$$
\item\label{item: representation ring} There is a ring isomorphism 
\begin{eqnarray*}
KO^{0}_{\Pin}(S^{0})&\cong&RO(\Pin)\\
&\cong&\mathbb{Z}[D,A,B]/(D^{2}-1,DA-A,DB-B,B^{2}-4(A-2B)) \\
&&\text{(note that there is a slight typo here in \cite{Schmidt2003})}.
%\footnote{There is a typo in here in \cite{Schmidt2003}.}
\end{eqnarray*}
The generators are defined as follows:
\begin{enumerate}
\item $D=[\mathbb{R}]$.
\item $A=K-(1+D)$, where $K$ is a $2$-dimensional real representation. The representation space of $K$ is $\mathbb{C}=\mathbb{R}\oplus i\mathbb{R}$, with the unit component $S^{1}=\{e^{i\theta}\}$ of $\Pin$ acting via left multiplication and $j$ acting as reflection along the diagonal.
\item $B=[\mathbb{H}]-2(1+D)$. 
\end{enumerate} 

\item\label{item: euler class} There are elements (called Euler classes)
$$\gamma(D)\in KO^{0}_{\Pin}(S^{-\widetilde{\mathbb{R}}}),$$
$$\gamma(\mathbb{H})\in KO^{0}_{\Pin}(S^{-\mathbb{H}}).$$
They satisfy the following property: for any $a<b$ and $c<d$, the map
$$KO^{0}_{\Pin}(S^{b\mathbb{H}+d\widetilde{\mathbb{R}}})\xrightarrow{ \cdot \gamma(D)^{d-c} \gamma(\mathbb{H})^{b-a}} KO^{0}_{\Pin}(S^{a\mathbb{H}+c\widetilde{\mathbb{R}}})
$$
equals the map on $KO^0_{\Pin}(-)$ that is induced by the inclusion 
$$S^{a\mathbb{H}+c\widetilde{\mathbb{R}}}\hookrightarrow S^{b\mathbb{H}+d\widetilde{\mathbb{R}}}.$$

\item\label{item: bott class} There are elements (called Bott classes) 
$$b_{2\mathbb{H}}\in KO^{0}_{\Pin}(S^{2\mathbb{H}}),$$ 
$$b_{8D}\in KO^{0}_{\Pin}(S^{8\widetilde{\mathbb{R}}}),$$ 
such that the following maps are isomorphism for all $a$ and $b$: 
$$
KO^{0}_{\Pin}(S^{a\mathbb{H}+b\widetilde{\mathbb{R}}})\xrightarrow{\cdot b_{2\mathbb{H}}}KO^{0}_{\Pin}(S^{(a+2)\mathbb{H}+b\widetilde{\mathbb{R}}}),
$$
$$
KO^{0}_{\Pin}(S^{a\mathbb{H}+b\widetilde{\mathbb{R}}})\xrightarrow{\cdot b_{8D}}KO^{0}_{\Pin}(S^{a\mathbb{H}+(b+8)\widetilde{\mathbb{R}}}).
$$

\item\label{item: euler class module structure} The relation
$$(D+1)\gamma(D)=2A\gamma(D)=B\gamma(D)=0$$
holds.

\item\label{item: bott times euler} The following relations hold: 
\begin{eqnarray*}
\gamma(D)^{8}b_{8D}&=&8(1-D), \\
\gamma(\mathbb{H})^{2}b_{2\mathbb{H}}&=&A-2B-2D+2.
\end{eqnarray*}

\item\label{item: ko(-2D)} There is an isomorphism 
$$KO^{0}_{\Pin}(S^{-2\widetilde{\mathbb{R}}}) \cong \mathbb{Z}\oplus \bigoplus_{n\geq 1} \mathbb{Z}/2,$$ generated by the elements $\gamma(D)^{2}$ and $A^{n}\gamma(D)^{2}$, $n \geq 1$.
\end{enumerate}

\subsection{Proof of Proposition~\ref{prop: upper bound detected by ko}} Let 
$$c^{k}_{\Pin}:S(4k\mathbb{H})_{+}\rightarrow S^{0}$$
be the base-point preserving map that sends the entire $S(4k\mathbb{H})$ to the point in $S^{0}$ that is not the base-point.  Consider the composition 
$$
c_{\Pin}(8k+2)^{8k-4}: S^{-(8k+2)\widetilde{\mathbb{R}}} \wedge S(4k\mathbb{H})_{+}\xrightarrow{\operatorname{id} \wedge c^{k}_{\Pin}} S^{-(8k+2)\widetilde{\mathbb{R}}} \stackrel{i}{\longrightarrow} S^{0},
$$
where $i$ is induced by the the inclusion 
$$S^0 \hookrightarrow S^{(8k+2)\widetilde{\mathbb{R}}}.$$
\begin{lem}\label{lem: equivariant ko nontrivial}
The map 
$$(c_{\Pin}(8k+2)^{8k-4})^{*}: RO(\Pin)=KO^{0}_{\Pin}(S^{0})\longrightarrow 
KO^{0}_{\Pin}(S^{-(8k+2)\widetilde{\mathbb{R}}} \wedge S(4k\mathbb{H})_{+})$$
sends $1\in RO(\Pin)$ to a nonzero element. 
\end{lem}

\begin{proof} Consider the map $$i^{*}:KO^{0}_{\Pin}(S^{0})\longrightarrow KO^{0}_{\Pin}(S^{-(8k+2)\widetilde{\mathbb{R}}}) $$
that is induced by $i$.  By \ref{item: euler class}, $i^{*}(1)=\gamma(D)^{8k+2}$.  By \ref{item: bott class} and \ref{item: ko(-2D)}, we have an isomorphism
$$
KO^{0}_{\Pin}(S^{-(8k+2)\widetilde{\mathbb{R}}}) \cong \mathbb{Z}\oplus \bigoplus_{n\geq 1} \mathbb{Z}/2, 
$$
generated by the elements $(b_{-8D})^{k}\cdot \gamma(D)^{2}$ and $(b_{-8D})^{k}\cdot A^{n}\gamma(D)^{2}$, $n \geq 1$.  Here, $b_{-8D}$ is the unique element in $KO^{0}_{\Pin}(S^{-8\widetilde{\mathbb{R}}})$ such that $b_{8D}\cdot b_{-8D}=1$.  By \ref{item: bott times euler} and \ref{item: euler class module structure}, we have 
\begin{eqnarray*}
\gamma(D)^{8k+2}&=&\gamma(D)^{8k} \cdot \gamma(D)^{2}\\
&=&\gamma(D)^{8k}\cdot(b_{8D})^{k}\cdot(b_{-8D})^{k}\cdot\gamma(D)^{2}\\
&=&8^{k}\cdot(1-D)^{k}\cdot(b_{-8D})^{k}\cdot\gamma(D)^{2} \,\,\,\,\,\,\,(\text{by \ref{item: bott times euler}})\\
&=&2^{3k}\cdot(1-D)^{k}\cdot\gamma(D)^{2}\cdot (b_{-8D})^{k} \\
&=&2^{3k}\cdot 2^{k}\cdot\gamma(D)^{2}\cdot (b_{-8D})^{k} \,\,\,\,\,\,\,(\text{by \ref{item: euler class module structure}}) \\
&=&2^{4k}(b_{-8D})^{k}\gamma(D)^{2}.
\end{eqnarray*}
To finish the proof, it suffices to show that 
\begin{equation}\label{equ: nonzero in equivariant ko}
(c^{k}_{\Pin}\wedge \operatorname{id})^{*}\left(2^{4k}(b_{-8D})^{k}\gamma(D)^{2}\right)\neq 0.
\end{equation}
We will prove this by contradiction.  Suppose (\ref{equ: nonzero in equivariant ko}) is not true.  Consider the cofiber sequence
$$ S^{-(8k+2)\widetilde{\mathbb{R}}} \wedge S(4k\mathbb{H})_{+} \xrightarrow{\operatorname{id} \wedge c^{k}_{\Pin}} S^{-(8k+2)\widetilde{\mathbb{R}}}\longrightarrow S^{4k\mathbb{H}-(8k+2)\widetilde{\mathbb{R}}}$$ 
that is obtained from $S(4k\mathbb{H})_{+} \longrightarrow S^0 \longrightarrow S^{4k\mathbb{H}}$ by taking $S^{-(8k+2)\widetilde{\mathbb{R}}}\wedge (-)$.  This cofiber sequence induces the sequence
$$
KO^{0}_{\Pin}(S^{4k\mathbb{H}-(8k+2)\widetilde{\mathbb{R}}})\xrightarrow{\gamma(\mathbb{H})^{4k}} KO^{0}_{\Pin}(S^{-(8k+2)\widetilde{\mathbb{R}}})\xrightarrow{(\operatorname{id} \wedge c^{k}_{\Pin})^{*}} KO^{0}_{\Pin}(S^{-(8k+2)\widetilde{\mathbb{R}}} \wedge S(4k\mathbb{H})_{+})
$$
which is exact in the middle.  Since $$(c^{k}_{\Pin}\wedge \operatorname{id})^{*}\left(2^{4k}(b_{-8D})^{k}\gamma(D)^{2}\right)= 0,$$
there exists an element $\alpha\in KO^{0}_{\Pin}(S^{4k\mathbb{H}-(8k+2)\widetilde{\mathbb{R}}})$ such that 
\begin{equation}\label{equ: alpha maps to euler}
2^{4k}(b_{-8D})^{k}\gamma(D)^{2} = \gamma(\mathbb{H})^{4k}\cdot \alpha.
\end{equation}
By \ref{item: bott class} and \ref{item: ko(-2D)}, $\alpha$ can be written as 
$$
(b_{2\mathbb{H}})^{2k}(b_{-8D})^{k}\gamma(D)^{2}\cdot P(A)
$$
for some polynomial $P(A)$.  By \ref{item: bott times euler} and \ref{item: euler class module structure}, equation (\ref{equ: alpha maps to euler}) can be rewritten as 
\begin{eqnarray*}
2^{4k}\cdot (b_{-8D})^{k}\gamma(D)^{2} &=& \left( \gamma(\mathbb{H})^{4k} (b_{2\mathbb{H}})^{2k}\right) \cdot (b_{-8D})^{k}\cdot \gamma(D)^{2}\cdot P(A) \\ 
&=&(A-2B-2D+2)^{2k}\cdot (b_{-8D})^{k}\cdot \gamma(D)^{2}\cdot P(A) \,\,\,\,\,\,\, \text{(by \ref{item: bott times euler})}\\
&=&(A-2B-2D+2)^{2k}\cdot \gamma(D)^{2}\cdot (b_{-8D})^{k}\cdot P(A) \\
&=&(A+4)^{2k}\cdot \gamma(D)^{2}\cdot (b_{-8D})^{k} \cdot P(A) \,\,\,\,\,\,\, \text{(by \ref{item: euler class module structure})}\\
&=&(A+4)^{2k} P(A) \cdot  (b_{-8D})^{k}\gamma(D)^{2}
\end{eqnarray*}
This implies that 
$$2^{4k} \equiv (A+4)^{2k} P(A) \pmod{2A}. $$
By comparing the coefficients of $A^0$ and $A^{2k}$, we see that this is impossible.
\end{proof}

By definition, under the isomorphism 
$$
[S^{-(8k+2)\widetilde{\mathbb{R}}}\wedge S(4k \mathbb{H})_+,S^{0}]_{\Pin} \cong [X(8k+2)^{8k-4},S^{0}],
$$
the element $c(8k+2)^{8k-4}$ corresponds to the element $c_{\Pin}(8k+2)^{8k-4}$.  Therefore, we have the following commutative diagram:
$$
\begin{tikzcd}
KO^{0}(S^{0})\ar[rrrr,"(c(8k+2)^{8k-4})^{*}"]\ar[d] &&&& KO^{0}(X(8k+2)^{8k-4})\ar[d,equal]\\
KO^{0}_{\Pin}(S^{0})\ar[rrrr,"(c_{\Pin}(8k+2)^{8k-4})^{*}"]&&&& KO^{0}_{\Pin}(S^{-(8k+2)\widetilde{\mathbb{R}}}\wedge S(4k \mathbb{H})_+).
\end{tikzcd}
$$
In the commutative diagram above, the left vertical map sends $1$ to $1$.  Therefore, Lemma~\ref{lem: equivariant ko nontrivial} implies that the map 
$$
(c(8k+2)^{8k-4})^{*}:KO^{0}(S^{0})\longrightarrow KO^{0}(X(8k+2)^{8k-4})
$$
is nontrivial.  This finishes the proof of Proposition~\ref{prop: upper bound detected by ko}.

%%%%%%%%%%%%%%%%%%%%%%%%%%%
Recall that the restriction of the map $f_{k}:X(8k+4)^{\infty}_{8k+1}\rightarrow S^{0}$ to the bottom cell of its domain is denoted
$$\phi_{k}:S^{8k+1}\rightarrow S^{0}$$ 
(see Theorem~\ref{thm: inductive fk}).  The following corollary will be used in the next section:

\begin{cor}\label{cor: phi detected by ko} For $k\geq 0$, the map $\phi_{k}$ is detected by $KO$.
\end{cor}
\begin{proof}
For the sake of contradiction, suppose that $\phi_{k}$ is not detected by $KO$.  Then the composition
$$
X(8k+4)^{8k+1}\xrightarrow{c(8k+4)^{8k+1}} S^{0}\longrightarrow KO
$$
is zero.  Since the map $c(8k+7)^{8k+2}:X(8k+7)^{8k+2}\rightarrow S^{0}$ factors through ${c(8k+4)^{8k+1}}$, the composition 
$$
X(8k+7)^{8k+2}\xrightarrow{c(8k+7)^{8k+2}} S^{0}\longrightarrow KO
$$
is zero.  Moreover, since $\pi_{8k+3}(KO)=0$, the composition 
$$
X(8k+7)^{8k+3}\xrightarrow{c(8k+7)^{8k+3}} S^{0}\longrightarrow KO
$$
is also zero. 

By Proposition~\ref{prop: upper bound detected by ko}, the map $c(8k+10)^{8k+4}$ is detected by $KO$.  This maps factors through the map $c(8k+7)^{8k+3}$, which, as we have just shown, is not detected by $KO$.  This is a contradiction.  
\end{proof}

%%%%%%%%%%%%%%%%%%%%%%%%%%%%%%%%%%%%%%%
%%%%%%%%%%%%%%%%%%%%%%%%%%%%%%%%%%%%%
%%%%%%%%%%%%%%%%%%%%%%%%%%%%%%%%%%%%%%%%%
\section{Step 3:  Identifying the map on the first lock as $\{P^{k-1}h_1^{3}\}$} \label{sec:Step3}
In this section, we prove Proposition~\ref{pro: beta-eta-square}: For all $k, m\geq 0$, we have the relations
$$\phi_{k}\cdot \{P^{m}h_{1}^{2}\}=\{P^{m+k}h_{1}^3\}.$$ 
Combining Corollary~\ref{cor: phi detected by ko} and part (iv) of Theorem~\ref{thm: inductive fk}, we have shown that the family 
$$\{\phi_{k}: S^{8k+1}\rightarrow S^{0}\mid k\geq -1 \}$$ 
satisfies the following two properties:
\begin{enumerate}
    \item For $k\geq 0$, $\phi_{k}$ can be detected by $KO$;
    \item For $k\geq 1$, we have that 
    \begin{equation}\label{equation: inductive on beta 2}
        \phi_{k}- \phi_{k-2}\cdot \chi_{k}\in \langle \phi_{k-1}, 2, \tau_{k} \rangle,
    \end{equation}
    for some $\tau_{k}\in \{0,8\sigma\}$ in $\pi_7$ and $\chi_{k}\in \pi_{16}$. Here $\phi_0 = \eta, \ \phi_{-1} = 0$.
\end{enumerate}
Since $\pi_{8k+1} ko = \mathbb{Z}/2$, generated by the Hurewicz image of the element $\{P^kh_1\}$ in $\pi_{8k+1}$ of the sphere spectrum, we make the following definition due to property $(1)$ of the family $\phi_k$ above.

\begin{df} \label{def:varphi}
Define 
$$\varphi_{-1} = 0, \ \varphi_{0} = 0,$$
and for $k \geq 1$,
$$\varphi_{k}=\phi_{k}-\{P^{k}h_{1}\}.$$
\end{df}

It is clear that the Hurewicz image of $\varphi_{k}$ in $\pi_{8k+1}ko$ is zero for all $k$.

Then Proposition~\ref{pro: beta-eta-square} follows from the following lemma for the elements $\varphi_k$ in $\pi_{8k+1}$.

\begin{lem} \label{lem:induct ph1}
	For all $k \geq -1, \ m \geq 0$, the following relations hold:
	$$\varphi_{k} \cdot \{P^m h^2_1\}    = 0.$$
\end{lem}

\begin{proof}[Proof of Proposition~\ref{pro: beta-eta-square}]
By Definition~\ref{def:varphi} and Lemma~\ref{lem:induct ph1}, we have
$$\phi_k \cdot \{P^{m}h_{1}^{2}\} = (\varphi_{k}+\{P^{k}h_{1}\}) \cdot \{P^{m}h_{1}^{2}\} = \{P^{m+k}h_{1}^3\}.$$
\end{proof}

%Now we prove Proposition~\ref{pro: beta-eta-square}.

Now we prove Lemma \ref{lem:induct ph1}.

\begin{proof}[Proof of Lemma \ref{lem:induct ph1}]

We first show that the elements $\tau_k$ are $8\sigma$ for all $k \geq 1$.

Suppose that for some $k$, we have $\tau_k = 0$. Then we would have
\begin{equation} \label{tauk is not 0}
\phi_{k} -   \phi_{k-2} \cdot \chi_{k} \in \langle \phi_{k-1}, 2, 0\rangle =  \phi_{k-1} \cdot  \pi_{8},
\end{equation}
where $\chi_k \in \pi_{16}$. Since no elements in $\pi_8$ and $\pi_{16}$ can be detected by the ring spectrum $KO$, mapping the above relation (\ref{tauk is not 0}) to $\pi_*KO$ gives us $\phi_{k} = 0$ in $\pi_{8k+1}KO$. This contradicts property $(1)$ that $\phi_k$ is detected by $KO$. Therefore, we must have 
$$\tau_k = 8 \sigma$$
for all $k \geq 1$.

Substituting $\phi_{k} = \varphi_{k}+\{P^{k}h_{1}\}$, the relation (\ref{equation: inductive on beta 2}) becomes 
\begin{align*}
   \varphi_{k} + \{P^{k}h_{1}\} &  \in  \varphi_{k-2} \cdot \chi_k +   \{P^{k-2}h_{1}\} \cdot \chi_k \\
   & \ \ \ + \langle \varphi_{k-1} , 2, 8\sigma \rangle + \langle \{P^{k-1}h_{1}\} , 2, 8\sigma \rangle.
\end{align*}
Here we set $\{P^{-1}h_{1}\}=0$ to unify the notation.

We have the Massey product
	$$P^{k}h_1 = \langle P^{k-1}h_1, h_0, h_0^3h_3 \rangle $$
	on the Adams $E_2$-page with zero indeterminacy. Then by Moss's theorem \cite[Theorem 1.2]{Moss}, we have the Toda bracket
	$$\{P^{k}h_1\} \in \langle \{P^{k-1}h_1\}, 2, 8\sigma \rangle.$$ 	
Therefore, we have 
\begin{align*}
   \varphi_{k} & \in \varphi_{k-2} \cdot \chi_k +\{P^{k-2}h_{1}\} \cdot \chi_k + \langle \varphi_{k-1}, 2, 8\sigma \rangle  \\
   & \ \ \ + \{P^{k-1}h_1\} \cdot \pi_{8} +   \pi_{8k-6} \cdot 8\sigma
\end{align*}	
for all $k\geq 1$. 

Using this relation, we complete the proof of Lemma~\ref{lem:induct ph1} by induction on $k$, which states that for all $k \geq -1, \ m \geq 0$
$$\varphi_{k} \cdot \{P^m h^2_1\} = 0.$$

The cases $k= 0, \ -1$ are trivial, since both $\varphi_{-1}$ and $\varphi_{0}$ are zero. 

For $k \geq 1$, suppose the lemma holds for $\varphi_{k-1}$ and $\varphi_{k-2}$. 

Multiplying $\{P^{m}h^{2}_{1}\}$, We have 
\begin{align*}
\varphi_{k} \cdot \{P^{m}h^{2}_{1}\} & \in  \{P^{k+m-2}h^{3}_{1}\} \cdot \chi_k +  \{P^{k+m-1}h^{3}_{1}\} \cdot \pi_{8} \\
& \ \ \ + \langle \varphi_{k-1}, 2, 8\sigma \rangle \cdot \{P^{m}h^{2}_{1}\}. 
\end{align*}
Note that both $\{P^{k+m-2}h^{3}_{1}\}$ and $\{P^{k+m-1}h^{3}_{1}\}$ are divisible by $2$. Since 
$$2 \cdot \pi_{8} = 0, \ 2 \cdot \pi_{16}=0,$$
we have
\begin{align*}
	\varphi_{k} \cdot\{P^m h^2_1\} & \in \langle \varphi_{k-1}, 2, 8\sigma \rangle   \cdot \eta \{P^m h_1\} \\
	& = \varphi_{k-1} \cdot \langle 2, 8\sigma, \eta \rangle \cdot \{P^m h_1\} \\
	& \ni \varphi_{k-1} \cdot \{Ph_1\} \cdot \{P^m h_1\}\\  
	& = \varphi_{k-1} \cdot \{P^{m+1} h_1^2\}\\
	& = 0.
\end{align*}
The indeterminacy 
$$\varphi_{k-1} \cdot \pi_8 \cdot \eta \cdot \{P^mh_1\} + \varphi_{k-1} \cdot 2 \cdot \pi_9 \cdot \{P^mh_1\}$$
is zero, since $2 \cdot \pi_9 = 0$ and that
$$\varphi_{k-1} \cdot \eta \cdot \{P^m h_1\} = \varphi_{k-1} \cdot \{P^m h_1^2\} = 0$$
by induction. Therefore, we have that 
$$\varphi_{k} \cdot \{P^m h_1^2\} = 0$$	
for all $m \geq 0$. This completes the induction and therefore the proof of the lemma.
\end{proof}

%%%%%%%%%%%%%%%%%
%\newpage
%\input{LowerBound}

%%%%%%%%%%%%%%%%%%%%
%%%%%%%%%%%%%%%%%%%%
\newpage
\section{Step 4: A technical lemma for the upper bound} \label{sec:Step4}
In this section, we prove will prove the follow proposition, which is Proposition~\ref{pro: only -1 cell matters} in Section~\ref{sec:OutlineofProof}.
\begin{prop}%\label{pro: only -1 cell matters}
For any $k,m\geq 0$, the map
\begin{equation}\label{eq outline: j'' injective without -1 cell}
j''^{0}(S^{4m+3})\longrightarrow j''^{0}(X(8k+3)_{0}^{4m+3})
\end{equation}
induced by the quotient map $X(8k+3)_{0}^{4m+3}\twoheadrightarrow S^{4m+3}$ is injective.
\end{prop}
The proof makes essential use of two spectra, $ko_{\mathbb{Q}/\mathbb{Z}}$ and $j'$, which we review now.

%%%%%%%%%%%%%%%%%%%%%%%%%%%%%%%%%%
\subsection{The spectra $ko_{\mathbb{Q}/\mathbb{Z}}$ and $j'$}
By Atiyah--Bott--Shapiro \cite[Section~11]{AtiyahBottShapiro}, any spin bundle is $ko$-orientable.  In other words, the spectrum $ko$ is a module over the ring spectrum $M Spin$.  Let $ko_{\mathbb{Q}/\mathbb{Z}}$ be the cofiber of the rationalization map 
$$ko\longrightarrow ko_{\mathbb{Q}}.$$ 
Both $ko_{\mathbb{Q}}$ and $ko_{\mathbb{Q}/\mathbb{Z}}$ are modules over $M Spin$.  Therefore, for any spin bundle $F$ of dimension $n$ over a space $A$, we have the Thom isomorphism 
$$\begin{tikzcd}ko^m(A; \mathbb{Q}/\mathbb{Z}) \ar[r, "\cong"]& ko^{m+n} (\Thom(A, F);\mathbb{Q}/\mathbb{Z}) \end{tikzcd}$$
which is induced by cup product with the Thom class. 

Moreover, if $A'$ is a subspace of $A$, and $F'$ is the restriction of $F$ to $A'$, we also have the relative Thom isomorphism 
\begin{equation}\label{eq:ThomFME}
\begin{tikzcd} \widetilde{ko}^m(A/A'; \mathbb{Q}/\mathbb{Z}) \ar[r, "\cong"]& ko^{m+n}(\Thom(A, F)/\Thom(A', F'); \mathbb{Q}/\mathbb{Z}). 
\end{tikzcd}
\end{equation}

\vspace{0.3in}

\begin{lem}\label{lem thom iso for  koq/z}
Let $V$ be a virtual bundle over a space $B$, and let $E$ be a spin bundle of dimension $n$ over $B$.  Suppose that $B'$ is a subspace of $B$.  Let $V'$ and $E'$ be the restrictions of $V$ and $E$ to $B'$.  We have the Thom isomorphism 
\begin{equation}\label{relative thom iso}
ko^{m}(\left.\Thom(B,V) \middle/ \Thom(B',V')\right.;\mathbb{Q}/\mathbb{Z})
\stackrel{\cong}{\longrightarrow} ko^{m+n}(\Thom(B,V\oplus E)/\Thom(B',V'\oplus E');\mathbb{Q}/\mathbb{Z}).
\end{equation}
The isomorphism above is natural in the sense that if $B''$ is a subspace of $B'$, and $V''$, $E''$ are the restrictions of $V$ and $E$ to $B''$, then the following diagram commutes:  
\begin{equation}\label{thom natural}
\begin{tikzcd}
ko^{m}(\left.\Thom(B,V) \middle/ \Thom(B',V')\right.;\mathbb{Q}/\mathbb{Z}) \ar[r, "\cong"] \ar[d] & ko^{m+n}(\Thom(B,V\oplus E)/\Thom(B',V'\oplus E');\mathbb{Q}/\mathbb{Z}) \ar[d] \\ 
ko^{m}(\Thom(B,V)/\Thom(B'',V'');\mathbb{Q}/\mathbb{Z}) \ar[r, "\cong"] &ko^{m+n}(\Thom(B,V\oplus E)/\Thom(B'',V''\oplus E'');\mathbb{Q}/\mathbb{Z}). 
\end{tikzcd}
\end{equation}
\end{lem}
\begin{proof}
The desired isomorphism follows from the isomorphism (\ref{eq:ThomFME}) by setting 
\begin{eqnarray*}
A &=& D(V), \\
A' &=& D(V') \cup S(V), 
\end{eqnarray*}
and letting $F$ be the pull-back of $E$ to $D(V)$ (here, $D(V)$ and $S(V)$ denote the disc bundle and the sphere bundle of $V$, respectively).  Diagram~(\ref{thom natural}) follows from standard arguments on the point-set level.  
\end{proof}

Next, we introduce a slight variant of the spectrum $j''$: we define $j'$ as the fiber of the map
$$
ko\xrightarrow{\psi^{3}-1}ko.
$$
Note that $j'^{0}(S^{0})=\mathbb{Z}\oplus \mathbb{Z}/2$ while $j''^{0}(S^{0})=\mathbb{Z}$. The map $ko\langle 2\rangle \rightarrow ko$ gives a map $j''\rightarrow j'$ that induces isomorphism on $\pi_{n}(-)$ for any $n\neq -1,0$. This proves the following simple lemma:
\begin{lem}\label{lem: j=j' when no negative cell}
Let $S$ be a finite CW-spectrum with no cell of dimension $\leq 0$. Then $j'^{0}(S)=j''^{0}(S)$.
\end{lem}

These two spectra $j'$ and $ko_{\mathbb{Q}/\mathbb{Z}}$ are related via the following lemma:
\begin{lem} \label{lem: connected j to ko}
Let $j'\langle 1\rangle$ be the 0-connected cover of $j'$.  There is a map $$\iota:j'\langle1\rangle \rightarrow \Sigma^{-1}ko_{\mathbb{Q}/\mathbb{Z}}$$ that induces an injection on $\pi_{4m-1}(-)$ for any positive integer $m$.
\end{lem}
\begin{proof}
Consider following commutative diagram
$$\xymatrix{\Sigma^{-1}ko_{\mathbb{Q}}\ar[d]\ar[r]^{\psi^{3}-1}&\Sigma^{-1}ko_{\mathbb{Q}}\ar[d]\\\Sigma^{-1}ko_{\mathbb{Q}/\mathbb{Z}}\ar[d]&\Sigma^{-1}ko_{\mathbb{Q}/\mathbb{Z}}\ar[d]
\\ko\ar[r]^{\psi^{3}-1}\ar[d]&ko\ar[d]\\
ko_{\mathbb{Q}}\ar[r]^{\psi^{3}-1}&ko_{\mathbb{Q}}.
}$$
In the commutative diagram above, the columns form cofiber sequences.  By the $3\times3$-Lemma \cite[Lemma 2.6]{MayTraces}, we can extend this diagram to the following diagram
$$\xymatrix{\Sigma^{-2}ko_{\mathbb{Q}}\ar[r]\ar[d]&\Sigma^{-1}j'_{\mathbb{Q}}\ar[r]\ar[d]&\Sigma^{-1}ko_{\mathbb{Q}}\ar[d]\ar[r]^{\psi^{3}-1}&\Sigma^{-1}ko_{\mathbb{Q}}\ar[d]\\\Sigma^{-2}ko_{\mathbb{Q}/\mathbb{Z}}\ar[d]\ar[r]&\Sigma^{-1}j'_{\mathbb{Q}/\mathbb{Z}}\ar[r]^-{f}\ar[d]^{g}&\Sigma^{-1}ko_{\mathbb{Q}/\mathbb{Z}}\ar[d]\ar[r]&\Sigma^{-1}ko_{\mathbb{Q}/\mathbb{Z}}\ar[d]
\\\Sigma^{-1}ko\ar[r]\ar[d]&j'\ar[r]\ar[d]^{h}&ko\ar[r]^{\psi^{3}-1}\ar[d]&ko\ar[d]\\\Sigma^{-1}ko_{\mathbb{Q}}\ar[r]&j'_{\mathbb{Q}}\ar[r]&
ko_{\mathbb{Q}}\ar[r]^{\psi^{3}-1}&ko_{\mathbb{Q}},
}$$
where all the rows and columns are cofiber sequences. 

Now, consider the commutative diagram
$$\xymatrix{& \Sigma^{-1}j'_{\mathbb{Q}/\mathbb{Z}}\ar[d]^{g}\\
j'\langle 1\rangle\ar[r]^{i}\ar@{.>}[ru]^{l}&j'\ar[d]^{h}\\
&j'_{\mathbb{Q}}.}$$
Since $j'\langle 1\rangle $ is $0$-connected and $\pi_{i}(j'_{\mathbb{Q}})=0$ for $i\geq 1$, the composition $h\circ i$ equals to zero. Therefore, the composition factors through the fiber of $h$, and there exists a map 
$$l:j'\langle1\rangle \longrightarrow \Sigma^{-1}j'_{\mathbb{Q}/\mathbb{Z}}$$
making the diagram above commute.  The composition
$$
\iota:j'\langle 1\rangle \stackrel{l}{\longrightarrow} \Sigma^{-1}j'_{\mathbb{Q}/\mathbb{Z}}\stackrel{f}{\longrightarrow}\Sigma^{-1}ko_{\mathbb{Q}/\mathbb{Z}}
$$
is our desired map.  

To prove that $\iota$ induces an injection on $\pi_{4m-1}(-)$, first note that $f$ induces an injection on $\pi_{4m-1}(-)$ because $\pi_{4m-1}(\Sigma^{-2}ko_{\mathbb{Q}/\mathbb{Z}})=0$.  Furthermore, since \\ $\pi_{k}(j'_{\mathbb{Q}})=0$ for all $k\geq 0$, the map $g$ induces an isomorphism on $\pi_{4m-1}(-)$ (just like the map $i$).  Therefore, $l$ induces an isomorphism on $\pi_{4m-1}(-)$.  It follows that $\iota$ induces an injection on $\pi_{4m-1}(-)$.
\end{proof}

\subsection{Proof of Proposition~\ref{pro: only -1 cell matters}}
Note that $X(m)^{a}$ is the Thom spectrum
$$
\Thom(-m\lambda|_{BPin(2)^{a+m}}, BPin(2)^{a+m}).$$
Set 
\begin{eqnarray*}
B&=&BPin(2)^{4m+8k+6},\\
B'&=&BPin(2)^{4m+8k+5}, \\
B''&=&BPin(2)^{8k+2},\\
V&=&(-8k-3)\lambda, \\
E&=&(8k+4)\lambda.
\end{eqnarray*}
Since $4\lambda$ is spin, $E$ is spin.  By Lemma~\ref{lem thom iso for  koq/z}, we obtain Thom isomorphisms that fit into the following commutative diagram:
\begin{equation}\label{diagram: koqzthom}    
\xymatrix{ko^{-1}(S^{4m+3};\mathbb{Q}/\mathbb{Z})\ar[d]
         \ar@{}[r]|*=0[@]{\cong}& ko^{8k+3}(S^{4m+8k+7};\mathbb{Q}/\mathbb{Z})\ar[d]\\ko^{-1}(X(8k+3)^{4m+3}_{1};\mathbb{Q}/\mathbb{Z})         \ar@{}[r]|*=0[@]{\cong}& ko^{8k+3}(X(-1)^{8k+4m+7}_{8k+5};\mathbb{Q}/\mathbb{Z})\\
}
\end{equation}

Set $Y=\Thom(\mathbb{H}P^{\infty},V)$ where $V$ is the bundle associated to the adjoint representation of $SU(2)$. Recall that there is a transfer map 
$$
T:Y\rightarrow X(-1)
$$
that induces isomorphism on $H_{4n+3}(-,\mathbb{F}_{2})$ (see Proposition~\ref{transfer}) for any integer $n$. Truncating this map, we obtain a commutative diagram:
\begin{equation}\label{diagram: truncated transfer}
    \xymatrix{S^{4m+8k+7}\ar[rr]^{\cong}&& S^{4m+8k+7}\\
    Y^{8k+4m+7}_{8k+5}\ar@{->>}[u]^{f}\ar[rr]^{T^{8k+4m+7}_{8k+5}}&&X(-1)^{8k+4m+7}_{8k+5}\ar@{->>}[u]^{g}}
\end{equation}
For algebraic reasons, the $ko_{\mathbb{Q}/\mathbb{Z}}$-based Atiyah--Hirzebruch spectral sequence of $Y$ collapses. Therefore, the map $f$ induces injection on $ko^{8k+3}_{\mathbb{Q}/\mathbb{Z}}$. By diagram (\ref{diagram: truncated transfer}), the pinch map $g$ also induces injection on $ko^{8k+3}(-;\mathbb{Q}/\mathbb{Z})$. By (\ref{diagram: koqzthom}), the pinch map $l:X(8k+3)^{4m+3}_{1} \rightarrow S^{4m+3}$ induces an injection
$$
l^{ko}: ko^{-1}(S^{4m+3};\mathbb{Q}/\mathbb{Z})\rightarrow  ko^{-1}(X(8k+3)^{4m+3}_{1};\mathbb{Q}/\mathbb{Z}).
$$
Now we relate $ko_{_{\mathbb{Q}/\mathbb{Z}}}$ and $j'$: the map $$\iota:j'\langle1\rangle \rightarrow \Sigma^{-1}ko_{\mathbb{Q}/\mathbb{Z}}$$ in Lemma~\ref{lem: connected j to ko} provides us with the following diagram:
\begin{equation}\xymatrix{
 j'\langle 1\rangle ^{0}(S^{4m+3})\ar[r]^{\iota_{*}}\ar[d]^{l^{j'\langle 1\rangle }}& ko^{-1}(S^{4m+3};\mathbb{Q}/\mathbb{Z})\ar[d]^{l^{ko}}   \\ 
    j'\langle 1\rangle ^{0}(X(8k+3)^{4m+3}_{1} )\ar[r]& ko^{-1}(X(8k+3)^{4m+3}_{1} ;\mathbb{Q}/\mathbb{Z})}    
\end{equation}
Since both $\iota_*$ and $l^{ko}$ are injective, the map $l^{j'\langle 1\rangle }$ is injective as well.  

Finally, since both $S^{4m+3}$ and $X(8k+3)^{4m+3}_{1}$ have no $0$ and $-1$ cells, $j'^{0}(-)$ and $j'\langle 1\rangle^{0}(-)$ are identical for them.  It follows that the map 
$$l^{j'}:j'^{0}(S^{4m+3})\rightarrow j'^{0}(X(8k+3)_{1}^{4m+3}) $$ 
is injective. By Lemma~\ref{lem: j=j' when no negative cell}, the map
$$l^{j''}:j''^{0}(S^{4m+3})\rightarrow j''^{0}(X(8k+3)_{1}^{4m+3}) $$ 
is also injective, as desired.

%%%%%%%%%%%%%%%%%%
%%%%%%%%%%%%%%%%%%
\newpage
\section{Step 5: Upper Bound} \label{sec:secondlock}

\subsection{Proving differentials using the Chern character} In this subsection, we introduce a useful technique for proving differentials in the $j''$-based Atiyah--Hirzebruch spectral sequence. 
\begin{df}\label{df: ko injective}\rm
A finite $CW$-spectrum $W$ is called \textit{$ko$-injective} if the map 
$$ch(c(-)):ko^{0}(W)\longrightarrow \bigoplus_{*\geq0} H^{2*}(W;\mathbb{Q})$$ 
given by  $\alpha\mapsto ch(c(\alpha))
$
is injective.  Here, $c(\alpha)$ denotes the complexification of $\alpha$.
\end{df}

\begin{thm}\label{prop: detect AHSS differential}
Let $W$ be a finite CW-spectrum that satisfies the following properties: 
\begin{enumerate}
\item $W$ has a single top cell in dimension $4m$;
\item $W$ has no cells in dimension $(4m-1)$;
\item The $(4m-2)$-skeleton $W^{4m-2}$ of $W$ is $ko$-injective;
\item The $2$-skeleton $W^{2}$ of $W$ is homotopy equivalent to $C\eta$.
\end{enumerate}
Furthermore, suppose there is an element $\alpha\in ko^{0}(W)$ that satisfies the equality 
\begin{equation}\label{equation: simple chern char}
ch(c(\alpha))=2^{l}+d, \,\,\, d\in H^{4m}(W;\mathbb{Q})=\mathbb{Q}.    
\end{equation}
Then in the $j''$-based Atiyah--Hirzebruch spectral sequence of $\Sigma^{-1}W$, the following results hold: 
\begin{enumerate}[label=(\Roman*)]
    \item If $\nu(d)\geq \iota(m)$, then the class $2^{l}[-1]$ is a permanent cycle.  Here, $\iota(m)=0$ when $m$ is even and $\iota(m)=1$ when $m$ is odd.
    \item If $\nu(d)< \iota(m)$, then there is a \textbf{nontrivial} differential 
    $$
    2^{l}[-1]\longrightarrow \gamma [4m-1]
    $$
    for some $\gamma\in \pi_{4m-1} j''$. 
\end{enumerate}
\end{thm}

To prove Theorem~\ref{prop: detect AHSS differential}, we first introduce some lemmas.

\begin{lem}\label{lem:psi^3-1alpha_0is0}
Let $\alpha_{k} \in ko^0(W^{4m-2})$ be the pull-back of $\alpha$ under the inclusion map $W^{4m-2}\hookrightarrow W$.  Then $\alpha_{k}\in \ker(\psi^{3}-1)$.
\end{lem}
\begin{proof}
Recall that we have the equality
$$ch_{2r}(\psi^{3}(\phi))=3^{r}ch_{2r}(\phi)$$
for all $\phi\in k^0(W)$.  Since $ch(c(\alpha_0)) = 2^l$,
$$ch(c((\psi^{3}-1)\alpha_{k}))=ch(\psi^{3}c(\alpha_{k}))-ch(c(\alpha_{k}))=2^{l}-2^{l}=0.$$ 
By our assumption, $W^{4m-2}$ is $ko$-injective (property (3)).  Therefore ${\alpha_0\in \ker(\psi^{3}-1)}$, as desired.
\end{proof}

\begin{lem}\label{lem:2^l is permanent cycle in the sub}
In the $j''$-based Atiyah--Hirzebruch spectral sequence for $\Sigma^{-1}W^{4m-2}$, the element $2^{l}[-1]$ is a permanent cycle.
\end{lem}
\begin{proof}
The cofiber sequences 
$$j' \longrightarrow ko \stackrel{\psi^3-1}{\longrightarrow} ko $$ 
and 
$$ S^{0}\hookrightarrow W^{4m-2}\twoheadrightarrow W^{4m-2}_{2}$$
induce the following commutative diagram: 
$$\begin{tikzcd}
\phi_{0} \in j'^0(W^{4m-2}) \ar[r, "\mathbf{1}"] \ar[d, "\mathbf{3}"] &\alpha_{k} \in ko^0(W^{4m-2}) \ar[r, "\psi^3 -1"] \ar[d, "\mathbf{2}"] & ko^0(W^{4m-2})\ar[d] \\ 
j'^0(S^0) = \mathbb{Z} \oplus \mathbb{Z}/2\ar[r, "(id{,} 0)"] \ar[d, "\mathbf{4}"] \ar[rd, "(id{,} id)"] & ko^0(S^0) = \mathbb{Z} \ar[r, , "\psi^3 -1"] &ko^0(S^0) = \mathbb{Z}\\
j'^0(\Sigma^{-1}W_2^{4m-2})\ar[r, "\mathbf{5}"] & j'^0(\Sigma^{-1}W_2^{2}) = j'^0(S^1) = \mathbb{Z}/2 \oplus \mathbb{Z}/2&
\end{tikzcd}$$
Consider the element $\alpha_0 \in ko^0(W^{4m-2})$.  By Lemma~\ref{lem:psi^3-1alpha_0is0}, $(\psi^3 -1)\alpha_0 = 0$.  This implies that there exists an element $\phi_0 \in j'^0(W^{4m-2})$ such that 
$$\mathbf{1}(\phi_0) = \alpha_0. $$
Furthermore, $\mathbf{2}(\alpha_0) = 2^{l}$ because of the commutative diagram  
$$\begin{tikzcd}
ko^0(W^{4m-2}) \ar[r, "c"] \ar[d] &k^0(W^{4m-2}) \ar[r, "ch"] \ar[d] & \bigoplus_{*\geq0}H^{2*}(W^{4m-2}; \mathbb{Q})\ar[d] \\ 
ko^0(S^0) \ar[r, "c"] & k^0(S^0)\ar[r, "ch"] & \bigoplus_{*\geq 0}H^{2*}(S^0; \mathbb{Q}).
\end{tikzcd}$$

Since the the map 
\begin{eqnarray*}
j'^0(S^0) &\longrightarrow& ko^0(S^0)  \\
\mathbb{Z} \oplus \mathbb{Z}/2 &\longrightarrow& \mathbb{Z}
\end{eqnarray*}
is $(id, 0)$, 
$$\mathbf{3}(\phi_0) = (2^{l}, b)$$ 
for some $b \in \mathbb{Z}/2$.  

We claim that $b= 0$.  To see this, consider the composition 
\begin{eqnarray*}
\mathbf{5} \circ \mathbf{4}: j'^0(S^0) &\longrightarrow& j'^0(S^1)  \\
\mathbb{Z} \oplus \mathbb{Z}/2 &\longrightarrow& \mathbb{Z}/2 \oplus \mathbb{Z}/2
\end{eqnarray*}
Since $W^{2}\simeq C\eta$ (property (4)), this map is induced by $\eta: S^1 \to S^0$ and sends $(a, b) \in \mathbb{Z} \oplus \mathbb{Z}/2$ to $(a, b) \in \mathbb{Z}/2 \oplus \mathbb{Z}/2$.  Therefore, under the composition $\mathbf{5} \circ \mathbf{4} \circ \mathbf{3}$, $\phi_0$ is sent to 
$$(0,0) = \mathbf{5} \circ \mathbf{4} \circ \mathbf{3} (\phi_0) = \mathbf{5} \circ \mathbf{4}(2^{l}, b) = (0, b).$$
Therefore $b= 0$. 

Consider the following commutative diagram: 
$$\begin{tikzcd}
j''^0(S^0) = \mathbb{Z} \ar[r, "(1{,}0)"] \ar[d] & j'^0(S^0) = \mathbb{Z} \oplus \mathbb{Z}/2\ar[d, "\mathbf{4}"] \\ 
j''^0(\Sigma^{-1}W_2^{4m-2}) \ar[r, "="] & j'^0(\Sigma^{-1}W_2^{4m-2}).
\end{tikzcd}$$
The bottom horizontal arrow is an equivalence because of Lemma~\ref{lem: j=j' when no negative cell}.  By the previous discussion, $\mathbf{4}(2^{l}, 0) = \mathbf{4} \circ \mathbf{3} (\phi_0) = 0$.  Therefore, the left vertical arrow sends the element $2^{l} \in j''^0(S^0)$ to $0$ as well. This is equivalent to saying that element $2^{l}[-1]$ is permanent cycle in the $j''$-based Atiyah--Hirzebruch spectral sequence for $\Sigma^{-1}W^{4m-2}$.
\end{proof}
\begin{lem}\label{lem:wiskoinjective}
$W$ is $ko$-injective.
\end{lem}
\begin{proof}
Let $\phi$ be an element in $ko^{0}(W)$ with $ch(c(\phi))=0$.  Since $W^{4m-2}$ is $ko$-injective, the pulls-back of $\phi$ under the inclusion $W^{4m-2}\hookrightarrow W$ must be zero.  Therefore, $\phi$ is the pull-back of some element 
$$b\in ko^{0}(S^{4m})=\mathbb{Z}_{(2)}$$ 
under the pinch map $\pi: W\twoheadrightarrow S^{4m}$.  Since 
$$ch(c(b))=2^{\iota(m)}\cdot b = 0,$$ 
$b$ must be 0.  It follows that $\phi = 0$ and $W$ is $ko$-injective, as desired. 
\end{proof}
\begin{prop}\label{lem: condition to be permanent cycle}
The element $2^{l}[-1]$ is a permanent cycle in the $j''$-based Atiyah--Hirzebruch spectral sequence of $\Sigma^{-1}W$ if and only if $\nu(d)\geq \iota (m)$.
\end{prop}
\begin{proof} If $\nu(d)\geq \iota (m)$, then we can find an element $b\in ko^{0}(S^{4m})$ such that 
$$ch(c(b))=d\in H^{4m}(S^{4m}).$$ 
Given this element $b$, we have the equality
$$ch(c(\alpha-\pi^{*}(b)))=2^{l},$$ 
where $\pi^*: ko^0(S^{4m}) \to ko^0(W)$ is induced from the pinch map $\pi: W \twoheadrightarrow S^{4m}$. 
Using Lemma~\ref{lem:wiskoinjective}, we can prove that $2^{l}[-1]$ is a permanent cycle by the exact same argument as the proof of Lemma~\ref{lem:2^l is permanent cycle in the sub}.

Now, suppose that $\nu(d)<\iota (m)$. Consider the commutative diagram 

$$\begin{tikzcd}
j''^0(S^0) = \mathbb{Z} \ar[r, "(1{,}0)"] \ar[d] & j'^0(S^0) = \mathbb{Z} \oplus \mathbb{Z}/2\ar[d] \\ 
j''^0(\Sigma^{-1}W_2^{4m}) \ar[r, "="] & j'^0(\Sigma^{-1}W_2^{4m}).
\end{tikzcd}$$

\begin{comment}
$$\begin{tikzcd}
j''^0(S^0) =\mathbb{Z} \ar[r] \ar[d, "(id{,}0)"] & j''^0(\Sigma^{-1}W_2) \ar[d] \\
j'^0(S^0) = \mathbb{Z} \oplus \mathbb{Z}/2 \ar[r] &j'^0(\Sigma^{-1}W_2)
\end{tikzcd}$$
\end{comment}
To prove that $2^{l}[-1]$ is not a permanent cycle, it suffices to show that the element $(2^{l},0) \in j'^0(S^0)$ is not sent to $0$ under the right vertical map.  

For the sake of contradiction, suppose that $(2^{l},0) \in j'^0(S^0)$ is sent to ${0 \in j'^0(\Sigma^{-1}W_2^{4m})}$.  Consider the following diagram:
$$\begin{tikzcd}
j'^0(W) \ar[d, "\mathbf{2}"] \ar[r, "\mathbf{3}"] & ko^0(W) \ar[r, "\psi^3-1"] \ar[d,"\mathbf{4}"] & ko^0(W) \ar[d] \\
(2^{l}, 0)\in j'^0(S^0) \ar[r,"\mathbf{5}"] \ar[d, "\mathbf{1}"] & ko^0(S^0) \ar[r, "\psi^3-1 = 0"] & ko^0(S^0) \\ 
j'^0(\Sigma^{-1}W_2^{4m}).
\end{tikzcd}$$
Since $\mathbf{1}(2^{l}, 0) = 0$, there exists an element $\tau \in j'^0(W)$ such that $\mathbf{2}(\tau) = (2^{l}, 0)$ by the exactness of the left column.  

Let $\xi= \mathbf{3}(\tau)$.  Since the diagram is commutative, 
$$\mathbf{4}(\xi) = \mathbf{5}(2^{l},0) = 2^{l}.$$
It follows that $ch(c(\xi)) = 2^{l}$.  

Consider the element $\alpha-\xi\in ko^{0}(W)$. We have $$ch(c(\alpha-\xi))=d\in H^{4m}(W)$$
Since $W^{4m-2}$ is $ko$-injective, the element $\alpha-\xi$ equals $\pi^{*}(b)$ for some 
$${b\in ko^{0}(S^{4m})=\mathbb{Z}_{(2)}}.$$ 
By comparing the Chern character, we obtain $b=\frac{d}{2^{\iota(m)}}$. This is impossible because $\frac{d}{2^{\iota(m)}}\notin \mathbb{Z}_{(2)}$.
\end{proof}

\begin{proof}[Proof of Theorem~\ref{prop: detect AHSS differential}]
The claim follows directly from Lemma~\ref{lem:2^l is permanent cycle in the sub} and Proposition~\ref{lem: condition to be permanent cycle}.
\end{proof}

\subsection{Proof of Proposition~\ref{pro: order two between 2 locks}}

For $k\geq 1$, we define $t_k$ to be the composite
\begin{equation} \label{equ tk}	
\xymatrix{
X(8k+3)^{8k-1}_{8k-5} \ar@^{->}[r] & X(8k-4)^{\infty}_{8k-7} \ar[r]^-{f_{k-1}} & S^{0}.
}\end{equation}

Then diagram~(\ref{diag: tau is quotient}) follows directly from diagram~(\ref{f is quotient}).

By Lemma \ref{the 4 cell complex bw 12 locks}, we have a splitting  
$$X(8k+3)^{8k-1}_{8k-5} \simeq \Sigma^{8k-5}C\nu \vee \Sigma^{8k-3}C2.$$ 
Under this splitting, we can write 
$$t_{k}=t'_{k}\vee t''_{k},$$ 
where $t'_k$ and $t''_k$ are the following two composites (\ref{equ tk 1}) and (\ref{equ tk 2}).

\begin{equation} \label{equ tk 1}
\xymatrix{
\Sigma^{8k-5}C\nu \ar@{^{(}->}[r] &  X(8k+3)^{8k-1}_{8k-5} \ar@^{->}[r] & X(8k-4)^{\infty}_{8k-7} \ar[r]^-{f_{k-1}} & S^{0}.
}
\end{equation}

\begin{equation} \label{equ tk 2}
\xymatrix{
\Sigma^{8k-3}C2 \ar@{^{(}->}[r] & X(8k+3)^{8k-1}_{8k-5} \ar@^{->}[r] & X(8k-4)^{\infty}_{8k-7} \ar[r]^-{f_{k-1}} & S^{0}.
}\end{equation}

We will show the following claims on $t'_{k}$ and $t''_{k}$.  These claims directly imply Properties (ii) through (iv).
 
\begin{itemize}
\item \textbf{Claim 1}: $t''_k = 0$.
\item \textbf{Claim 2}: $t'_k$ is of order 2 in $j'$. In other words, the following composite is zero.
\begin{equation} \label{equ tk1 order 2}
\xymatrix{
\Sigma^{8k-5}C\nu \ar[r]^-{2\cdot \textup{id}} & \Sigma^{8k-5}C\nu \ar[r]^-{t'_{k}} & S^{0} \ar[r] & j'.
}
\end{equation}

\item \textbf{Claim 3}: The restriction of $t'_{k}$ to the bottom cell $S^{8k-5}$ is 
$$\{P^{k-1}h^{3}_{1}\} = \{P^{k-1}h_{1}\}  \cdot \eta^2$$ 
in $\pi_{8k-5}$. 
\end{itemize}

It is clear that by Corollary~\ref{cor: ph13 is Mohowald invariant} in Step 2 in Subsection 2.4 that \textbf{Claim 3} is true. In the rest of this subsection, we first prove \textbf{Claim 1}, and then prove \textbf{Claim 2}.

For \textbf{Claim 1}, note that $t''_{k}$ equals to the composite
$$
\xymatrix{
\Sigma^{8k-3}C2 \ar@{^{(}->}[r] & X(8k+3)^{8k-1}_{8k-5} \ar@^{->}[r] & X(8k-3)^{\infty}_{8k-7} \ar@^{->}[r] &  X(8k-4)^{\infty}_{8k-7} \ar[r]^-{f_{k-1}} & S^{0}
}
$$
By exactly the same cell diagram chasing argument as the one in Step 1.1.2, we see that the restriction of the composite 
$$
\xymatrix{
\Sigma^{8k-3}C2 \ar@{^{(}->}[r] & X(8k+3)^{8k-1}_{8k-5}\ar@^{->}[r] & X(8k-3)^{\infty}_{8k-7}
}
$$
to the bottom cell $S^{8k-3}$ is zero. Therefore, we can rewrite $t''_{k}$ as the composite 
$$
\xymatrix{
\Sigma^{8k-3}C2 \ar@{->>}[r] & S^{8k-2}\ar[r]^-{\mathbf{1}} & X(8k-3)^{\infty}_{8k-7} \ar@^{->}[r] &  X(8k-4)^{\infty}_{8k-7}\ar[r]^-{f_{k-1}}& S^{0}
}
$$
for some map $\mathbf{1}$. By cellular approximation theorem, the map $\mathbf{1}$ maps through $X(8k-3)^{8k-3}_{8k-7}$:
$$
\xymatrix{
S^{8k-2}\ar[r]^-{\mathbf{2}} & X(8k-3)^{8k-3}_{8k-7} \ar@{^{(}->}[r] & X(8k-3)^{\infty}_{8k-7}. 
}
$$
Moreover, due to the $\eta$-attaching map in $X(8k-3)^{8k-3}_{8k-7}$ between the cells in dimensions $8k-5$ and $8k-3$, the composite
$$\xymatrix{
S^{8k-4}\ar[r]^-{\mathbf{2}} & X(8k-3)^{8k-3}_{8k-7} \ar@{->>}[r] & S^{8k-3}
}$$
must be zero. Therefore, the map $\mathbf{1}$ maps through ${X(8k-3)^{8k-4}_{8k-7}}$, and we can rewrite $t''_{k}$ as the composite
$$\xymatrix{
\Sigma^{8k-3}C2\ar@{->>}[r] & S^{8k-2}\ar[r]^-{\mathbf{3}} & X(8k-4)^{8k-4}_{8k-7}\ar@{^{(}->}[r] & X(8k-4)^{\infty}_{8k-7}\ar[r]^-{f_{k-1}}&S^{0}
}$$
for some map $\mathbf{3}$. By Theorem~\ref{thm: inductive fk}, there is an 
H$\mathbb{F}_{2}$-subcomplex 
$$\xymatrix{g_{k-1}:S^{8k-4} \ar@{^{(}->}[r] & X(8k-4)^{8k-4}_{8k-7} \ar@{^{(}->}[r] & X(8k+4)^{\infty}_{8k+1}.}$$ 
By Lemma~\ref{the 3 cell complex wk}, the 3 cell complex $X(8k-4)^{8k-4}_{8k-7}$ splits:  
$$X(8k-4)^{8k-4}_{8k-7} \simeq \Sigma^{8k-7} C2\vee S^{8k-4}.$$ 
Since $\pi_{4} = \pi_{5} = 0$, we have 
$$\pi_{8k-2} X(8k-4)^{8k-6}_{8k-7} = \pi_5 C2 = 0,$$
and the map $\mathbf{3}$ maps through the H$\mathbb{F}_{2}$-subcomplex $S^{8k-4}$. In other words, we can rewrite the composite
$$\xymatrix{
S^{8k-2}\ar[r]^-{\mathbf{3}} & X(8k-4)^{8k-4}_{8k-7} \ar@{^{(}->}[r] & X(8k-4)^{\infty}_{8k-7}
}$$
as the composite 
$$\xymatrix{
S^{8k-2}\ar[r]^-{\mathbf{4}} & S^{8k-4}\ar[r]^-{g_{k-1}}& X(8k-4)^{\infty}_{8k-7},
}$$
for some map $\mathbf{4}$ in $\pi_2$. Therefore we can rewrite $t''_{k}$ as the composite
$$\xymatrix{
\Sigma^{8k-3}C2 \ar@{->>}[r] & S^{8k-2}\ar[r]^-{\mathbf{4}} &S^{8k-4}\ar[r]^-{g_{k-1}} & X(8k-4)^{\infty}_{8k-7}\ar[r]^-{f_{k-1}} & S^{0}.
}$$
As in the proof of Proposition~\ref{proposition: construct f}, the composite
$$\xymatrix{
S^{8k-2}\ar[r]^-{\mathbf{4}}& S^{8k-4}\ar[r]^-{g_{k-1}} & X(8k-4)^{\infty}_{8k-7}\ar[r]^-{f_{k-1}} & S^{0}
}$$
is zero. Therefore, we have $t''_{k} = 0$. This completes the proof of \textbf{Claim 1}.\\

%\textbf{Claim 2}: The restriction of $\tau'_{k}$ to the bottom cell is $\{P^{k-1}h^{3}_{1}\}$. 

%By the definition of $\tau'_{k}$ , its restriction to the bottom cell equals 
%$$
%S^{8k-5}\xrightarrow{\eta^{2}}S^{8k-7}\xrightarrow{\phi_{k-1}}S^{0}.
%$$
%(Recall $\phi_{k-1}$ is the restriction of $f_{k-1}$ to the bottom cell of $X(8k-4)^{\infty}_{8k-7}$.) Then the claim follows directly from Proposition~\ref{pro: beta-eta-square}.

For \textbf{Claim 2}, note that the composite $2 \cdot t_k'$ maps through ${X(8k+2)_{8k-5}^{\infty}}$. Due to the $2$-attaching map in $X(8k+2)_{8k-5}^{\infty}$ between the cells in dimensions $8k-5$ and $8k-4$, the composite
$$\xymatrix{
S^{8k-5} \ar@{^{(}->}[r] & \Sigma^{8k-5}C\nu \ar[r]^-{2 \cdot \textup{id}} & \Sigma^{8k-5}C\nu \ar@{^{(}->}[r] & X(8k+3)_{8k-5}^{8k-1} \ar@^{->}[r] & X(8k+2)_{8k-5}^{\infty}
}$$
is zero. Therefore, we can rewrite $2\cdot t'_{k}$ as the composite
$$\xymatrix{
\Sigma^{8k-5}C\nu \ar@{->>}[r] & S^{8k-1} \ar[r]^-{\mathbf{5}} & X(8k+2)_{8k-5}^{\infty}\ar@^{->}[r] & X(8k-4)_{8k-7}^{\infty} \ar[r]^-{f_{k-1}} & S^{0},
}$$
where $\mathbf{5}$ is a map that induces a trivial homomorphism on ${H_{8k-1}(-;\mathbb{F}_{2})}$. By the cellular approximation theorem and the $2$-attaching map in $X(8k+2)_{8k-5}^{\infty}$ between cells of dimension $8k$ and $8k-1$, the map $\mathbf{5}$ maps through ${X(8k+2)_{8k-5}^{8k-2}}$: 
$$\xymatrix{
S^{8k-1}\ar[r]^-{\mathbf{6}} & X(8k+2)_{8k-5}^{8k-2} \ar@{^{(}->}[r] & X(8k+2)_{\infty}^{8k-2}.
}$$
Therefore, we can rewrite $2\cdot t'_{k}$ as the composite
$$\xymatrix{
\hspace{-0.3in}\Sigma^{8k-5}C\nu \ar@{->>}[r] & S^{8k-1} \ar[r]^-{\mathbf{6}} & X(8k+2)_{8k-5}^{8k-2} \ar@^{->}[r] & X(8k-4)_{8k-7}^{8k-4} \ar@{^{(}->}[r] & X(8k-4)^{\infty}_{8k-7} \ar[r]^-{f_{k-1}} & S^{0}.
}$$
By Lemma~\ref{the 3 cell complex wk}, the 3 cell complex $X(8k-4)^{8k-4}_{8k-7}$ splits:  
$$X(8k-4)^{8k-4}_{8k-7} \simeq \Sigma^{8k-7} C2\vee S^{8k-4}.$$ 
So we can write $2\cdot t'_{k}$ as the sum of the following two composites (\ref{equ claim 2 1}) and (\ref{equ claim 2 2}):
\begin{equation} \label{equ claim 2 1}
\xymatrix{
\Sigma^{8k-5}C\nu \ar@{->>}[r] & S^{8k-1} \ar[r]^-{\mathbf{7}} & \Sigma^{8k-7} C2 \ar@{^{(}->}[r] & X(8k-4)^{\infty}_{8k-7} \ar[r]^-{f_{k-1}} & S^{0},
}	
\end{equation}
\begin{equation} \label{equ claim 2 2}
\xymatrix{
\Sigma^{8k-5}C\nu \ar@{->>}[r] & S^{8k-1} \ar[r]^-{\mathbf{8}} & S^{8k-4} \ar@{^{(}->}[r]^-{g_{k-1}} & X(8k-4)^{\infty}_{8k-7} \ar[r]^-{f_{k-1}} & S^{0}.
}	
\end{equation}
For the map $\mathbf{7}$ in the composite (\ref{equ claim 2 1}), it corresponds to an element in the group
$$\pi_{8k-1} \Sigma^{8k-7} C2 = \pi_6 C2 = \mathbb{Z}/2,$$
which is generated by $\nu^2$ on the bottom cell of $\Sigma^{8k-7} C2$. Since $\nu^2$ is not detected by the spectrum $j'$, post-composing (\ref{equ claim 2 1}) with the map $S^0 \rightarrow j'$ is zero.

For the composite (\ref{equ claim 2 2}), note that by Part (iii) of Theorem~\ref{thm: inductive fk}, the composite $g_{k-1} \circ f_{k-1}$ is
$$\xymatrix{S^{8k-4} \ar[r]^-{a_{k-1}} & X(8k-12)^{8k-12}_{8k-15} \ar[r]^-{b_{k-1}} & S^0}.$$
Therefore, the composite (\ref{equ claim 2 2}) can be rewritten as
\begin{equation} \label{equ claim 2 3}
\xymatrix{
\Sigma^{8k-5}C\nu \ar@{->>}[r] & S^{8k-1} \ar[r]^-{\mathbf{8}} & S^{8k-4} \ar[r]^-{a_{k-1}} & X(8k-12)^{8k-12}_{8k-15} \ar[r]^-{b_{k-1}} & S^{0}.
}	
\end{equation}
Using again the splitting 
$$X(8k-12)^{8k-12}_{8k-15} \simeq \Sigma^{8k-15} C2\vee S^{8k-12},$$
the composite (\ref{equ claim 2 3}) can be written as the sum of the following two composites (\ref{equ claim 2 4}) and (\ref{equ claim 2 5}):
\begin{equation} \label{equ claim 2 4}
\xymatrix{
\Sigma^{8k-5}C\nu \ar@{->>}[r] & S^{8k-1} \ar[r]^-{\mathbf{8}} & S^{8k-4} \ar[r]^-{\mathbf{9}} & S^{8k-12} \ar[r] & S^{0},
}	
\end{equation}
\begin{equation} \label{equ claim 2 5}
\xymatrix{
\Sigma^{8k-5}C\nu \ar@{->>}[r] & S^{8k-1} \ar[r]^-{\mathbf{8}} & S^{8k-4} \ar[r]^-{\mathbf{10}} & \Sigma^{8k-15} C2 \ar[r] & S^{0}.
}	
\end{equation}
The composite (\ref{equ claim 2 4}) is zero, since $\mathbf{9} \circ \mathbf{8}$ corresponds to an element in 
$$\pi_8 \cdot \pi_3 = 0.$$
The composite (\ref{equ claim 2 5}) is zero, since $\mathbf{10} \circ \mathbf{8}$ corresponds to an element in 
$$\pi_{11} C2 \cdot \pi_3 = 0.$$
In fact, $\pi_{11} C2 = \mathbb{Z}/2 \oplus \mathbb{Z}/2$, which is generated by $\{Ph_2\}[0]$ and $\{Ph_1\}[1]\cdot \eta$. Both generators are annihilated by $\pi_3$.

Therefore, the composite (\ref{equ claim 2 2}), which equals to the composite (\ref{equ claim 2 3}), is zero.

In sum, we have that $2\cdot t'_{k} = 0$ in $j'$. This finishes the proof of \textbf{Claim 2}.

\subsection{Proof of Proposition~\ref{lem: differential to second lock}}

Recall that there is a map 
$$
j(8k+3):X(8k+3)\xrightarrow{s_{8k+3}}\Sigma^{-8k-3}\mathbb{C}P^{\infty}
$$
 that induces an isomorphism on $(H\mathbb{F}_{2})_{4m-1}(-)$ for any $m$ (see~formula (\ref{cofiber of adjacent columns})).  Truncating this map, we obtain a map
$$(s_{8k+3})^{8k-1}_{-1}:X(8k+3)^{8k-1}_{-1}\rightarrow\Sigma^{-1}Z, $$
where $$Z=\Sigma^{-8k-2}\mathbb{C}P^{8k+1}_{4k+1}=\Thom(\mathbb{C}P^{4k},(4k+1)(L-1)).$$
Here, $L$ denotes the canonical bundle on $\mathbb{C}P^{\infty}$.

The Thom isomorphism gives an identification
$$H^{*}(Z;\mathbb{Q})\cong U_{H}\cdot H^{*}(\mathbb{C}P^{4k};\mathbb{Q})\cong U_{H}\cdot \mathbb{Q}[x]/(x^{4k+1}),$$
where $x=c_{1}(L)$ and $U_{H}$ is the Thom class for homology. 

In order to apply Theorem~\ref{prop: detect AHSS differential} to $Z$, we require the following lemma:
\begin{lem}\label{lem: stunted projective space ko-injective}
For any odd integer $n> 0$ and any $m>n$, the spectrum $
\Sigma^{-2n}\mathbb{C}P^{m}_{n}
$ is $ko$-injective. (See Definition \ref{df: ko injective}.)
\end{lem}

\begin{proof}
%This can be proved by a simple calculation of Atiyah--Hirzebruch spectral sequence. (\color{blue}{To be filled in}).

We show that for the spectrum $\Sigma^{-2n}\mathbb{C}P^{m}_{n}$, where $n>0$ is odd and $m>n$, the map
$$c: ko^0 (\Sigma^{-2n}\mathbb{C}P^{m}_{n}) \longrightarrow ku^0 (\Sigma^{-2n}\mathbb{C}P^{m}_{n})$$
is injective. Since the Chern character map is injective for this spectrum, this would prove the lemma by Definition \ref{df: ko injective}.

The complexification of real vector bundles corresponds to the following map on the spectra level
$$c: ko \longrightarrow ku.$$

For degree reasons, the $ku$-based Atiyah--Hirzebruch spectral sequence for $\Sigma^{-2n}\mathbb{C}P^{m}_{n}$ collapses at the $E_2$-page. In particular, the group $ku^0 (\Sigma^{-2n}\mathbb{C}P^{m}_{n})$ is a direct sum of copies of $\mathbb{Z}$'s.

Since $n>0$ is odd, the bottom two cells of $\Sigma^{-2n}\mathbb{C}P^{m}_{n}$ is $C\eta$. More generally, we can decompose $\Sigma^{-2n}\mathbb{C}P^{m}_{n}$ by its subquotients (with certain attaching maps among them) of the form $\Sigma^{4j} C\eta$ for $j \geq 0$, and with one possible copy of $S^{2m-2n}$ when $m$ is odd. In this case, we have that $2m-2n$ is divisible by 4. Since
$$ko \wedge C\eta \simeq ku,$$
the $ko$-based Atiyah--Hirzebruch spectral sequence for $\Sigma^{-2n}\mathbb{C}P^{m}_{n}$ collapses at the $E_3$-page. This means that we only need to check that the following maps are injective
\begin{equation} \label{equ lemma97 1}
c:ko^0 (\Sigma^{4j} C\eta) \longrightarrow ku^0 (\Sigma^{4j} C\eta),\end{equation}
\begin{equation} \label{equ lemma97 2}
	c:ko^0 (S^{2m-2n}) \longrightarrow ku^0 (S^{2m-2n}),
\end{equation}
where $j\geq 0$ and $2m-2n$ is divisible by 4.

Due to the compatibility of real and complex Bott periodicity, the map 
$${c:ko\longrightarrow ku}$$ 
maps $v_1^4$ to $v_1^4$ in $\pi_8$. So in particular, it induces an isomorphism on $\pi_{8k}$ for all $k \geq 0$. It is also well known that, the generator of ${\pi_4 ko}$ maps to $2v_1^2$ in ${\pi_4 ku}$. So it induces an injective homomorphism on $\pi_{8k+4}$ for all $k \geq 0$. This proves that the map (\ref{equ lemma97 2}) is injective. 

For the map (\ref{equ lemma97 1}), since the Spanier--Whitehead dual of $C\eta$ is ${\Sigma^{-2} C\eta}$, we may rewrite it as
$$\pi_{4j+2} ku = \pi_{4j+2} (ko \wedge C\eta) \longrightarrow \pi_{4j+2} (ku \wedge C\eta) = \pi_{4j+2} (ku \vee \Sigma^2 ku),$$
which is an inclusion of a splitting summand.

Combining the injectivity of the maps (\ref{equ lemma97 1}) and (\ref{equ lemma97 2}) completes the proof of the lemma.

\end{proof}

\begin{lem}\label{lem: bundle over Z with simple Chern character}
There exists an element $\phi\in k^0(Z)$ such that
\begin{equation}\label{equation: bundle over z with simple Chern character}
ch(\phi)=2^{4k-2}+d\cdot U_{H}x^{4k}
\end{equation}

for some $d$ with $\nu(d)=-2$.
\end{lem}
\begin{proof}
There is a Thom isomorphism
$$
k^{0}(Z)\cong U_{K}\cdot k^{0}(\mathbb{C}P^{4k}) \cong U_{K}\cdot\mathbb{Z}_{(2)}[w]/(w^{4k+1}),
$$
where $w=L-1$ and $U_{K}$ is the $K$-theoretic Thom class for the virtual bundle $(4k+1)w$.  We have the relations 
$$ch(w)=e^{x}-1$$
and
\begin{eqnarray*}
ch(U_{K})&=&U_{H}\cdot \chi((4k+1)w) \\
&=&U_{H}\cdot \chi((4k+1)L) \\
&=&U_{H}\cdot \left(\frac{e^{x}-1}{x}\right)^{4k+1}.  
\end{eqnarray*}

Now, suppose 
$$\phi= U_{K} \cdot (a_{0}+a_{1}w+\cdots a_{4k-1} w^{4k-1}),$$ 
where $a_i \in \mathbb{Z}_{(2)}$ for all $0 \leq i \leq 4k-1$.  Our goal is to determine the coefficients $a_i$ so that condition~(\ref{equation: bundle over z with simple Chern character}) holds.  

Applying $ch(-)$ to both sides of the equation and using the formulas above, we get 
$$ch(\phi)= U_{H}\cdot \left(\frac{e^{x}-1}{x}\right)^{4k+1}\cdot \sum\limits_{i=0}^{4k-1} a_{i}(e^{x}-1)^{i}.$$
Now, make the substitution $z:=e^{x}-1$.  Then $x=\ln(z+1)$ and the above equation becomes
\begin{equation}\label{power series}
\left(\frac{\ln(z+1)}{z}\right)^{4k+1} \cdot ch(\phi)=U_{H}\cdot \sum\limits_{i=0}^{4k-1} a_{i}z^{i}\in U_{H}\cdot \mathbb{Q}[z]/(z^{4k+1}).\end{equation}

Condition (\ref{equation: bundle over z with simple Chern character}) requires
$$ch(\phi)=2^{4k-2}+ a\cdot U_{H}\cdot  z^{4k}$$ for some $a$ with $\nu(a)=-2$. 
By comparing the constant terms in (\ref{power series}), we deduce that $ch_{0}(\phi)=a_{0}$ and 
$$\left(\frac{\ln(z+1)}{z}\right)^{4k+1} \cdot \left(2^{4k-2}+d\cdot z^{4k}\right)=\sum\limits_{i=0}^{4k-1} a_{i}z^{i}+O(z^{4k+1}).$$

Let the power series expansion of $\left(\frac{\ln(z+1)}{z}\right)^{4k+1}$ be $1+b_{1}z+b_{2}z^{2}+\cdots$.  By comparing the coefficients of $z^i$ in the equation above, we obtain the relations
\begin{eqnarray*}
a_{0}&=&2^{4k-2}, \\
d&=&a_{0}\cdot b_{4k}, \\
a_{i}&=&2^{4k-2}\cdot b_{4k}, \text{ for }i=1,\ldots , 4k-1.
\end{eqnarray*}
By Lemma~\ref{lem:coeffb4n}, we see that $\nu(d)=-2$. By Lemma~\ref{lem:coeffbm<=4n-1}, we see that $a_{i}\in\mathbb{Z}_{(2)}$ for all $0\leq i\leq 4k-1$. Therefore, $\phi$ belongs to $k^0(X)$.  
\end{proof}
Now, set $\alpha=r(\phi)$.  Then one has 
$$
ch(c(\alpha))=2^{4k-1}+2d\cdot U_{H}x^{4k}.
$$
By Lemma~\ref{lem: stunted projective space ko-injective}, we can apply Theorem~\ref{prop: detect AHSS differential} to $Z$ and conclude the existence of the differential 
$$
2^{4k-1}[-1]\longrightarrow \gamma[8k-1]
$$
 in the $j''$-based Atiyah--Hirzebruch spectral sequence of $\Sigma^{-1}Z$, with $\gamma \neq 0$ in $\pi_{8k-1}j''$. By naturality of Atiyah--Hirzebruch spectral sequence, we can pullback this differential to $X(8k+3)^{8k-1}_{-1}$ using the map $j(8k+3)^{8k-1}_{-1}$. This finishes the proof of Proposition~\ref{lem: differential to second lock}.

%%%%%%%%%%%%%%%%%%%%%%%%%%%%
\newpage
\section{Steps 6 and 7: first lock and second lock} \label{sec:Steps5and6}

In this section, we will prove the claims in Section~\ref{subsec:Step5} and Section~\ref{subsec:Step6}. 
%%%%%%%%%%%%
\subsection{Construction of $Z(k)$}\label{subsec:constructionofZ(k)}
In this subsection, we will construct a spectrum $Z(k)$ for every $k \geq 0$.  This spectrum will be crucial for proving Proposition~\ref{pro: first lock for even} and Proposition~\ref{pro: first lock for odd}.  By Proposition~\ref{prop:TransferMaps}, there is a cofiber sequence
$$\begin{tikzcd}X(8k+4) \ar[r] & X(8k+3) \ar[r, "s_{8k+3}"] & \Sigma^{-(8k+3)} \mathbb{C}P^\infty.\end{tikzcd}$$
By restricting to the subquotient $(-)_{-1}^{8k-2}$, we obtain a cofiber sequence 
$$\begin{tikzcd}X(8k+4)_{-1}^{8k-2} \ar[r] & X(8k+3)_{-1}^{8k-2} \ar[r, "s_{8k+3}"] & \Sigma^{-(8k+3)} \mathbb{C}P_{4k+1}^{8k} \end{tikzcd}.$$

Consider the quotient map 
$$\begin{tikzcd}X(8k+4)_{-1}^{8k-2} \ar[r,twoheadrightarrow]& X(8k+4)_{8k-8}^{8k-2}. \end{tikzcd}$$
By Proposition~\ref{prop:DefinitionY(k)}, there is a 2 cell complex $Y(k)$ with cells in dimensions $8k-4$ and $8k-8$ such that it is an $\textup{H}\mathbb{F}_2$-quotient complex of ${X(8k+4)^{8k-2}_{8k-8}}$.	 There is a commutative diagram 
$$\begin{tikzcd} 
Y(k) \ar[r, "0"] & \ast \\
X(8k+4)_{-1}^{8k-2} \ar[u, twoheadrightarrow] \ar[r] &  X(8k+3)_{-1}^{8k-2} \ar[u, twoheadrightarrow, swap, "0"],
\end{tikzcd}$$
where the left vertical map is the composition 
$$\begin{tikzcd}X(8k+4)_{-1}^{8k-2} \ar[r,twoheadrightarrow]& X(8k+4)_{8k-8}^{8k-2} \ar[r, twoheadrightarrow] & Y(k). \end{tikzcd}$$ 

By the $3 \times 3$-Lemma \cite[Lemma 2.6]{MayTraces}, we can extend this commutative diagram to the following commtuative diagram, where the rows and columns are cofiber sequences: 
$$\begin{tikzcd} 
Y(k) \ar[r, "0"] & \ast \ar[r] & \Sigma Y(k) \\
X(8k+4)_{-1}^{8k-2} \ar[u, twoheadrightarrow] \ar[r] &  X(8k+3)_{-1}^{8k-2} \ar[u, twoheadrightarrow, swap, "0"] \ar[r] & \Sigma^{-(8k+3)} \mathbb{C}P_{4k+1}^{8k}\ar[u, twoheadrightarrow] \\ 
& X(8k+3)_{-1}^{8k-2} \ar[r, "\rho"] \ar[u, hookrightarrow, "\text{id}"] & \Sigma^{-1} Z(k) \ar[u, hookrightarrow].
\end{tikzcd}$$
The complex $Z(k)$ is defined to be the cofiber of the map 
$$\begin{tikzcd}  \Sigma^{-(8k+3)} \mathbb{C}P_{4k+1}^{8k} \ar[r, twoheadrightarrow] &\Sigma Y(k). \end{tikzcd}$$
By Lemma~\ref{cpx}(2), the map $\rho$ induces an isomorphism on $(H\mathbb{F}_2)_{4\ell-1}$ for all $\ell$. 

\begin{lem} \label{lem:Z(k)Property}
The complex $Z(k)$ satisfies the following properties: 
\begin{enumerate}
\item $Z(k)^{8k-8} = \Sigma^{-(8k+2)} \mathbb{C}P_{4k+1}^{8k-3}$; 
\item $Z(k)_{8k-8}^{8k-4} = \left\{\begin{array}{ll}S^{8k-4} \vee S^{8k-8} & k \text{ even,} \\
\Sigma^{8k-8}C\eta^3 & k \text{ odd.} \end{array} \right. $
\end{enumerate}
\end{lem}
\begin{proof}
Property (1) is straightforward from the definition of $Z(k)$. To prove property (2), note that by truncating the transfer map (see (\ref{eq: transfer}))
$$
\text{Tr}:\Thom(\mathbb{H}P^{\infty},V)\longrightarrow X(-1),
$$
we obtain a $H\mathbb{F}_{2}$-sub map
$$\begin{tikzcd}
\mathbf{1}:\Thom(\mathbb{H}P^{\infty},V)^{16k-1}_{16k-5} \ar[r, hookrightarrow]& X(-1)^{16k-1}_{16k-5}.
\end{tikzcd}
$$
Desuspending $\mathbf{1}$ by $\Sigma^{-(8k+4)}(-)$ and applying Proposition~\ref{prop:periodicityX(m)}, we obtain the map 
$$\begin{tikzcd}
\mathbf{2}:\Sigma^{-(8k+4)}\Thom(\mathbb{H}P^{\infty},V)^{16k-1}_{16k-5} \ar[r, hookrightarrow]& X(8k+3)^{8k-5}_{8k-9}.
\end{tikzcd}
$$

By truncating the map $\rho: X(8k+3)_{-1}^{8k-2} \rightarrow \Sigma^{-1} Z(k)$, we obtain a $H\mathbb{F}_{2}$-quotient map
$$\begin{tikzcd}
\mathbf{3}:X(8k+3)^{8k-5}_{8k-9}\ar[r,twoheadrightarrow]& \Sigma^{-1}Z(k)^{8k-4}_{8k-8}.
\end{tikzcd}
$$
The composite 
$$\begin{tikzcd}
\mathbf{3}\circ\mathbf{2}:\Sigma^{-(8k+4)}\Thom(\mathbb{H}P^{\infty},V)^{16k-1}_{16k-5} \ar[r, hookrightarrow]& X(8k+3)^{8k-5}_{8k-9}\ar[r,twoheadrightarrow]& \Sigma^{-1}Z(k)^{8k-4}_{8k-8}
\end{tikzcd}$$
induces an isomorphism on $\mathbb{F}_2$-homology.  Therefore, it is a homotopy equivalence.  The claims now follow from Lemma~\ref{lem:HPTruncationTwist}.
\end{proof}

\begin{rem}\rm \label{rem:SchmidtError}
In the proof of \cite[Theorem 4.9]{Schmidt2003}, Schmidt made a minor error when computing $\pi^{11}\mathbb{C}P^7$.  This error led to Schmidt's proof of Jones Conjecture for $p = 4$.  

Note that Lemma~\ref{lem:Z(k)Property} is a crucial step in our proof of showing that the Jones conjecture is \textit{not} true when $p \equiv 4 \pmod{8}$.  If Schmidt's cohomotopy group computation were true, our statement of Lemma~\ref{lem:Z(k)Property}(2) would be different: $Z(1)_{0}^4 = S^0 \vee S^4$.  This would also lead to an affirmative answer for Jones conjecture for $p =4$ by using our subsequent arguments.  
\end{rem}

\begin{lem}\label{lem:ChernComplexInjective}
For any $m<8k-4$, the $m$-skeleton of $Z(k)$ is $ko$-injective.  
\end{lem}
\begin{proof}
Note that $Z(k)^{m}=\Sigma^{-8k-2}\mathbb{C}P^{l}_{4k+1}$ for some $l\geq 4k+1$.  Therefore, the claim follows from Lemma~\ref{lem: stunted projective space ko-injective}. 
\end{proof}

%%%%%%%%%%%
\subsection{Proof of Proposition~\ref{prop:Step6CommDiagram}}
Consider the map 
$$t_{k}: X(8k+3)^{8k-1}_{8k-5}\longrightarrow S^{0}$$
in Proposition~\ref{pro: order two between 2 locks}.  By properties (ii) and (iii) in Proposition~\ref{pro: order two between 2 locks}, there is a factorization of the map 
$t_k|_{X(8k+3)_{8k-5}^{8k-2}}: X(8k+3)_{8k-5}^{8k-2} \longrightarrow S^0$ as follows:
$$\begin{tikzcd} 
X(8k+3)_{8k-5}^{8k-2} \ar[rrr, "t_k|_{X(8k+3)_{8k-5}^{8k-2}}"] \ar[dd,twoheadrightarrow, swap] &&& S^0 \\ \\
S^{8k-5}\ar[rrruu, "t_k'|_{S^{8k-5}} = \{P^{k-1}h_1^3\}", swap] & 
\end{tikzcd}$$
Here, the vertical map is the restriction of the quotient map 
$$\begin{tikzcd} X(8k+3)^{8k-1}_{8k-5} \ar[r, twoheadrightarrow]& \Sigma^{8k-5}C\nu \end{tikzcd}$$
to the $(8k-2)$-skeleton. 

When restricted to the $(8k-2)$-skeleton, diagram~(\ref{diag: tau is quotient}) becomes the diagram 
$$\begin{tikzcd}
X(8k+3)^{8k-2} \ar[d, twoheadrightarrow] \ar[rrrrdd, "c(8k+3)^{8k-2}"]\\
X(8k+3)^{8k-2}_{-1} \ar[d, twoheadrightarrow] \\ 
X(8k+3)_{8k-5}^{8k-2} \ar[rrrr,"t_k|_{X(8k+3)_{8k-5}^{8k-2}}"] &&&& S^0 
\end{tikzcd}$$
This diagram, combined with the factorization above, produces the following diagram: 
$$\begin{tikzcd}
X(8k+3)^{8k-2} \ar[d, twoheadrightarrow] \ar[rrrrdd, "c(8k+3)^{8k-2}"]\\
X(8k+3)^{8k-2}_{-1} \ar[d, twoheadrightarrow] \\ 
X(8k+3)_{8k-5}^{8k-2} \ar[rr,twoheadrightarrow] &&S^{8k-5} \ar[rr, swap, "\{P^{k-1}h_1^3\}"]&& S^0. 
\end{tikzcd}$$

Given this commutative diagram, Proposition~\ref{prop:Step6CommDiagram} follows from Lemma~\ref{lem:Prop2.18Lemma}. 
\begin{lem}\label{lem:Prop2.18Lemma}
The following diagram commutes: 
$$\begin{tikzcd} 
X(8k+3)^{8k-2}_{-1} \ar[d, twoheadrightarrow] \ar[rr,"\rho"] && \Sigma^{-1}Z(k) \ar[rr,twoheadrightarrow] && S^{8k-5} \ar[d, "\{P^{k-1}h_1^3\}"]\\ 
X(8k+3)_{8k-5}^{8k-2} \ar[rr,twoheadrightarrow] &&S^{8k-5} \ar[rr, "\{P^{k-1}h_1^3\}"]&& S^0. 
\end{tikzcd}$$
\end{lem}
\begin{proof}
Let $\mathbf{1}$ denote the composition 
$$\begin{tikzcd} X(8k+3)^{8k-2}_{-1} \ar[r, twoheadrightarrow] & X(8k+3)_{8k-5}^{8k-2} \ar[r, twoheadrightarrow] &S^{8k-5},  \end{tikzcd}$$
and let $\mathbf{2}$ denote the composition
$$\begin{tikzcd} X(8k+3)^{8k-2}_{-1} \ar[r,"\rho"] & \Sigma^{-1}Z(k) \ar[r,twoheadrightarrow] &S^{8k-5}. \end{tikzcd}$$
We want to show that the map $\mathbf{1}-\mathbf{2}$ becomes 0 after post-composing with the map $\{P^{k-1}h_1^3\}$.  

It is straightforward to see that when restricted to the subcomplex $X(8k+3)_{-1}^{8k-5}$, $(\mathbf{1}-\mathbf{2})|_{X(8k+3)_{-1}^{8k-5}} = 0$.  This is because both $\mathbf{1}$ and $\mathbf{2}$ become the quotient map 
$$\begin{tikzcd} X(8k+3)_{-1}^{8k-5} \ar[r, twoheadrightarrow] &S^{8k-5}. \end{tikzcd}$$
This implies that the map $\mathbf{1}-\mathbf{2}$ factors through the fiber of the inclusion map $X(8k+3)_{-1}^{8k-5} \hookrightarrow X(8k+3)_{-1}^{8k-2}$, which is $\Sigma^{8k-3} C2$: 
$$\begin{tikzcd} 
X(8k+3)_{-1}^{8k-5} \ar[r,hookrightarrow] & X(8k+3)_{-1}^{8k-2} \ar[rd,swap, "\mathbf{1} - \mathbf{2}"]\ar[r, twoheadrightarrow]&X(8k+3)_{8k-3}^{8k-2} = \Sigma^{8k-3}C2 \ar[d] \\
&& S^{8k-5}. 
\end{tikzcd}$$

Given any map $\Sigma^{8k-3}C2 \to S^{8k-5}$, the composition map 
$$\begin{tikzcd}\Sigma^{8k-3}C2 \ar[r]& S^{8k-5} \ar[r, "\eta^2"]&S^{8k-7} \end{tikzcd}$$
is 0 because $\pi_2 \cdot \eta^2 = 0$ and $\pi_5 = 0$.  Since $\eta^2 | \{P^{k-1}h_1^3\}$, the composition 
$$\begin{tikzcd} \Sigma^{8k-3}C2 \ar[r] & S^{8k-5} \ar[rr,"\{P^{k-1}h_1^3\}" ]&& S^0 \end{tikzcd}$$
is zero.  This implies that $\{P^{k-1}h_1^3\}\circ (\mathbf{1} - \mathbf{2} ) = 0$, as desired. 
\end{proof}

%%%%%%%%%%
\subsection{Bundles with simple Chern character}
In this subsection, we will construct virtual bundles over $Z(k)$ with simple Chern characters.  This will allow us to use Theorem~\ref{prop: detect AHSS differential} to establish differentials in the Atiyah--Hirzebruch spectral sequence. 

Recall from Section~\ref{sec:secondlock} the spectrum $Z$, which is defined as the Thom spectrum 
$$\Thom(\mathbb{C}P^{4k};(4k+1)(L-1)) =  \Sigma^{-(8k+2)} \CP_{4k+1}^{8k+1}.$$
By definition, $Z(k)$ is the fiber of a certain $\textup{H}\mathbb{F}_2$-quotient map 
$$\begin{tikzcd}Z^{8k-4} \ar[r, twoheadrightarrow, "\psi_k"] &S^{8k-6}.\end{tikzcd} $$
We denote the generator of $H^{2i}(Z^{8k-4};\mathbb{Z})$ by $x^i$.

\begin{lem}\label{lem: bundle over CP}
There exists an element $\gamma\in k^{0}(Z^{8k-4})$ such that 
\begin{equation}\label{equation: gamma chern class}
ch(\gamma)=2^{4k-5-\nu(k)}+c_{8k-8}x^{4k-4}+c_{8k-6}x^{4k-3}+c_{8k-4}x^{4k-2},
\end{equation}
with $v(c_{8k-8})=-1$ and $v(c_{8k-4})\geq 0$.
\end{lem}
\begin{proof} There is a Thom isomorphism
$$
k^{0}(Z^{8k-4})\cong U_{K}\cdot k^{0}(\mathbb{C}P^{4k-2}) \cong U_{K}\cdot\mathbb{Z}_{(2)}[w]/(w^{4k-1}),
$$
where $w=L-1$ and $U_{K}$ is the $K$-theoretic Thom class for the virtual bundle $(4k+1)w$.  We have the relations 
$$ch(w)=e^{x}-1$$
and
\begin{eqnarray*}
ch(U_{K})&=&U_{H}\cdot \chi((4k+1)w) \\
&=&U_{H}\cdot \chi((4k+1)L) \\
&=&U_{H}\cdot \left(\frac{e^{x}-1}{x}\right)^{4k+1}.  
\end{eqnarray*}

Suppose
$$\gamma=U_{K}\cdot \left(\sum_{i=0}^{4k-5}a_{i}w^{i}\right).$$
After taking Chern characters on both sides, we get  
$$ch(\gamma)=\left(\frac{e^{x}-1}{x}\right)^{4k+1}\cdot \sum_{i=1}^{4k-5}a_{i}(e^{x}-1)^{i}.$$
Just like before, we make the substitution $z=e^{x}-1$.  With this substitution, equation~(\ref{equation: gamma chern class}) is equivalent to the following equation:
\begin{equation}\label{equation: chern character in z}
\left(\frac{z}{\ln(z+1)}\right)^{4k+1}\cdot \sum_{i=0}^{4k-5}a_{i}z^{i}=2^{4k-5-\nu(k)}+o(z^{4k-4}).
\end{equation}
This equation is equivalent to the equation 
\begin{equation}\label{equation: chern character in z(2)}
\sum_{i=0}^{4k-5}a_{i}z^{i}=\left(\frac{\ln(z+1)}{z}\right)^{4k+1}\cdot (2^{4k-5-\nu(k)}+o(z^{4k-4})).
\end{equation}

By comparing coefficients on both sides of equation~(\ref{equation: chern character in z(2)}), we obtain the relations
$$a_{i}=2^{4k-5-\nu(k)}\cdot b_{i}$$
for all $0 \leq i \leq 4k-5$.  By Lemma~\ref{lem:b<=4n-5}, $\nu(b_{i}) \geq \nu(k) - (4k-5)$ for all $0 \leq i \leq 4k-5$.  Therefore, the coefficients $a_{i}\in \mathbb{Z}_{(2)}$ and we have found a $\gamma$ that satisfies equation (\ref{equation: gamma chern class}). 

To show that the rest of the coefficients in $ch(\gamma)$ satisfies the conditions of the lemma, note that by the definition of the coefficients $b_i$,
$$\left(\frac{z}{\ln(z+1)}\right)^{4k+1}\cdot \left(\sum_{i=0}^{\infty}2^{4k-5-\nu(k)}b_{i}z^{i}\right)=2^{4k-5-\nu(k)}.$$
Subtracting equation (\ref{equation: chern character in z}) from this equation and using the relation $z^{4k-1} = 0$, we obtain the following equation:
\begin{eqnarray*}
&&\left(\frac{z}{\ln(z+1)}\right)^{4k+1}\cdot 2^{4k-5-\nu(k)}\cdot (b_{4k-4}z^{4k-4}+b_{4k-3}z^{4k-3}+b_{4k-2}z^{4k-2}) \\ 
&=&ch_{8k-8}(\gamma)+ch_{8k-6}(\gamma)+ch_{8k-4}(\gamma).
\end{eqnarray*}
Substituting $e^{x}-1$ back as $z$, the above equation becomes
\begin{eqnarray*}
&&\left(\frac{e^x-1}{x}\right)^{4k+1}\cdot 2^{4k-5-\nu(k)}\cdot (b_{4k-4}(e^x-1)^{4k-4}+b_{4k-3}(e^x-1)^{4k-3}+b_{4k-2}(e^x-1)^{4k-2}) \\ 
&=&ch_{8k-8}(\gamma)+ch_{8k-6}(\gamma)+ch_{8k-4}(\gamma).
\end{eqnarray*}
After rearranging, we get 
\begin{eqnarray*}
&&\left(\frac{2^{4k-5-\nu(k)}}{x^{4k+1}} \right)\cdot (b_{4k-4}(e^{x}-1)^{8k-3}+b_{4k-3}(e^{x}-1)^{8k-2} +b_{4k-2}z^{4k-2}(e^{x}-1)^{8k-1}) \\
&=&ch_{8k-8}(\gamma)+ch_{8k-6}(\gamma)+ch_{8k-4}(\gamma)\\
&=& c_{8k-8}x^{4k-4} + c_{8k-6} x^{4k-3} + c_{8k-4} x^{4k-2}.
\end{eqnarray*}
Expanding the left hand side and comparing the coefficients of $x^{4k-4}$ and $x^{4k-2}$ on both sides of the equation, we obtain the relations
\begin{eqnarray*}
c_{8k-8}&=&2^{4k-5-\nu(k)}\cdot b_{4k-4}, \\
c_{8k-4}&=&2^{4k-5-\nu(k)}\cdot \left(\frac{(8k-3)(3k-1)}{3}b_{4k-4}+(4k-1)b_{4k-3}+b_{4k-2}\right)\\
&=&- \ 2^{4k-3-\nu(k)}b_{4k-3}+2^{4k-5-\nu(k)}\frac{(8k-3)(3k-1)}{3}b_{4k-4}\\
&&+ \ 2^{4k-5-\nu(k)}(b_{4k-2}-b_{4k-3}).
\end{eqnarray*}
By Lemma~\ref{lem:coeffb4n-4}, 
$$\nu(c_{8k-8}) = 4k-5 - \nu(k) + (\nu(k) - (4k-4)) = -1.$$ 
By Lemma~\ref{lem:coeffb4n-3}, \ref{lem:coeffb4n-4}, and \ref{lem:coeffb4n-2minusb4n-3}, when $n$ is odd, all three terms in the formula for $c_{8k-4}$ are 2-local integers, so $\nu(c_{8k-4}) \geq 0$.  When $n$ is even, the lemmas show that the first term is a 2-local integer while the other two terms are 2-local half-integers (they have 2-adic valuations $-1$), and so $\nu(c_{8k-4}) \geq 0$ again.  This concludes the proof of the lemma. 
\end{proof}

\begin{prop}\label{prop bundle construction 1}
There exists an element $\alpha_{k}\in ko^{0}(Z(k))$ such that 
\begin{enumerate}
\item When $k$ is even, 
\begin{eqnarray}\label{equation: chern character of alpha}
ch(c(\alpha_{k}))&=&2^{4k-4-\nu(k)}.
\end{eqnarray}
\item When $k$ is odd, 
\begin{eqnarray}\label{equation: chern character of alpha}
ch(c(\alpha_{k}))&=&2^{4k-4-\nu(k)}+dx^{4k-2}
\end{eqnarray}
with $\nu(d)=0$.
\end{enumerate}
\end{prop}
\begin{proof}
When $k$ is even, let $\gamma'$ be the pullback of $\gamma$ under the map $Z(k)\rightarrow Z^{8k-4}$ and let $\alpha'=r(\gamma')$ ($r:k^0(Z(k)) \to ko^0(Z(k))$ is the restriction map).  By Lemma~\ref{lem: bundle over CP}, 
$$ch(c(\alpha'))=2^{4k-4-\nu(k)}+2c_{8k-8}x^{4k-4}+2c_{8k-4}x^{4k-2}.$$
Recall from Lemma~\ref{lem:Z(k)Property} that $Z(k)_{8k-8}^{8k-4}=S^{8k-8}\vee S^{8k-4}$ for even $k$. 
Let 
$$\phi_{1},\phi_{2}\in ko^{0}(Z(k)_{8k-8}^{8k-4})=ko^{0}(S^{8k-8})\oplus ko^{0} (S^{8k-4})$$
be the generators for the first and the second summand, respectively.  Since the composition map 
$$ko^0(S^{4m}) \stackrel{c}{\longrightarrow} k^0(S^{4m}) \stackrel{ch}{\longrightarrow} H^*(S^{4m};\mathbb{Q})$$
is multiplication by 1 when $m$ is even and multiplication by 2 when $m$ is odd, we have 
$$ch(c(\phi_{1}))=x^{4k-4}$$
and
$$ch(c(\phi_{2}))=2x^{4k-2}.$$

Now, set 
$$\alpha_{k}=\alpha'-2c_{8k-8}\cdot p_{0}^{*}(\phi_{1})-c_{8k-4}\cdot p^{*}_{0}(\phi_{2})
,$$ 
where 
$$p_{0}: Z(k)\twoheadrightarrow Z(k)_{8k-8}^{8k-4}$$ 
is the quotient map.  Note that this construction is valid because both $2c_{8k-8}$ and $c_{8k-4}$ belong to $\mathbb{Z}_{(2)}$ by Lemma~\ref{lem: bundle over CP}.  It follows that $\alpha_{k}$ satisfies (\ref{equation: chern character of alpha}).

When $k$ is odd, let $\gamma'$ be the pullback of $\gamma$ under the map $Z(k)\rightarrow Z^{8k-4}$ and let $\alpha'=r(\gamma')$.  By Lemma~\ref{lem: bundle over CP}, 
$$ch(c(\alpha'))=2^{4k-4-\nu(k)}+2c_{8k-8}x^{4k-4}+2c_{8k-4}x^{4k-2}.$$
Recall from Lemma~\ref{lem:Z(k)Property} that $Z(k)_{8k-8}^{8k-4}=\Sigma^{8k-8}C\eta^{3}$ for $k$ odd. There is an element 
$$\phi_{3}\in ko^{0}(Z(k)_{8k-8}^{8k-4})=ko^{0}(C\eta^{3})$$ 
such that $$ch(c(\phi_{3}))=x^{4k-4}+e x^{4k-2}$$ for some $e$ with $\nu(e)=0$ (this is because the $e$-invariant of $\eta^{3}$ has 2-adic evaluation $0$). 

Now, set 
$$\alpha_{k}=\alpha'-2c_{8k-8}\cdot p^{*}_{1}(\phi_{3}),$$
where 
$$p_1: Z(k) \twoheadrightarrow Z(k)_{8k-8}^{8k-4}$$
is the quotient map.  Then 
\begin{eqnarray*}
ch(c(\alpha_k)) &=& ch(c(\alpha')) -2 c_{8k-8} \cdot ch(c(p^*_1(\phi_3))) \\ 
&=& 2^{4k-4-\nu(k)}+2c_{8k-8}x^{4k-4}+2c_{8k-4}x^{4k-2} - 2c_{8k-8}\cdot (x^{4k-4}+e x^{4k-2}) \\ 
&=&2^{4k-4-\nu(k)} + (2c_{8k-4} - 2c_{8k-8}\cdot e)\cdot x^{4k-2}. 
\end{eqnarray*}
By Lemma~\ref{lem: bundle over CP}, $d = (2c_{8k-4} - 2c_{8k-8}\cdot e)$ has 2-adic valuation 0.  Therefore, $\alpha_k$ satisfies (\ref{equation: chern character of alpha}), as desired. 
\end{proof}

%%%%%%%%%%%%%%
\subsection{First lock for $k$ odd}

In this subsection, we will prove Proposition~\ref{pro: first lock for odd}, which states that when $k$ is odd, the composition
$$\begin{tikzcd}\Sigma^{-1}Z(k) \ar[r, twoheadrightarrow]& S^{8k-5} \ar[rr,"\{P^{k-1}h_{1}^{3}\}"]&& S^0\end{tikzcd}$$
is zero. 
\begin{proof}[Proof of Proposition~\ref{pro: first lock for odd}]
Let $f: \Sigma^{-1} Z(k)_2^\infty \to S^0$ be the boundary map induced from the cofiber sequence 
$$S^{-1} \hookrightarrow \Sigma^{-1}Z(k) \longrightarrow \Sigma^{-1} Z(k)_2^\infty.$$
In other words, $f$ fits into the sequence
$$S^{-1} \hookrightarrow \Sigma^{-1}Z(k) \longrightarrow \Sigma^{-1} Z(k)_2^\infty \stackrel{f}{\longrightarrow} S^0.$$

We will show that the following diagram is commutative: 
\begin{equation} \label{diagram:2^4n-4Commute}
\begin{tikzcd} 
&S^0& \\
\Sigma^{-1} Z(k) \ar[r] & \Sigma^{-1} Z(k)_2^\infty \ar[r, twoheadrightarrow] \ar[u, "2^{4k-4}f"]& S^{8n-5} \ar[lu, swap, "\{P^{k-1}h_1^3\}"].
\end{tikzcd}
\end{equation}
Our proposition will follow from the commutativity of this diagram.  This is because taking $[-, S^0]$ in the cofiber sequence 
$$\Sigma^{-1}Z(k) \longrightarrow \Sigma^{-1} Z(k)_2^\infty \stackrel{f}{\longrightarrow} S^0$$
produces the sequence
$$[S^0, S^0] \longrightarrow [\Sigma^{-1}Z(k)_2^\infty, S^0] \longrightarrow [\Sigma^{-1}Z(k), S^0]. $$
In this sequence, the element 
$$2^{4k-4} \in [S^0, S^0]$$ 
first maps to 
$$2^{4k-4}f \in [\Sigma^{-1}Z(k)_2^\infty, S^0],$$ 
and then maps to 
$$g \in [\Sigma^{-1}Z(k), S^0]$$ by the commutativity of (\ref{diagram:2^4n-4Commute}).  Since the sequence is exact at $[\Sigma^{-1}Z(k)_2^\infty, S^0]$, we deduce that $g = 0$.  

It remains for us to prove that diagram (\ref{diagram:2^4n-4Commute}) is commutative.  Since $\Sigma^{-1}Z(k)_2^\infty$ has no 0-cells, the Adams filtration for the map $f$ is at least 1.  This implies that the Adams filtration of the map $2^{4k-4} f$ is at least $(4k-4) + 1 = 4k-3$.  Therefore, the map $2^{4k-4}f$ can be lifted through a map $\ell_{4k-3}: \Sigma^{-1}Z(k)_2^\infty \to T_{4k-3}$, where $T_i$ ($i \geq 1$) is the $i$th stage of the Adams tower of $S^0$.
$$\begin{tikzcd}
& T_{4k-3} \ar[d]& \\
& \vdots \ar[d] & \\
& T_2 \ar[r] \ar[d] &T_2 \wedge  \textup{H}\mathbb{F}_2 \\
& T_1 \ar[r] \ar[d] & T_1 \wedge \textup{H}\mathbb{F}_2 \\
\Sigma^{-1} Z(k)_2^\infty \ar[r, "2^{4k-4}f"] \ar[ru, "\ell_1"] \ar[ruu, "\ell_2"] \ar[ruuuu, "\ell_{4k-3}"]& S^0 \ar[r]& \textup{H}\mathbb{F}_2.
\end{tikzcd}$$

The cells of $\Sigma^{-1} Z(k)_2^\infty$ are in dimensions $1$, $3$, $\ldots$, $8k-9$, and $8k-5$.  Since $\pi_i(T_{4k-3}) = 0$ for all $1 \leq i \leq 8k-8$, the $(8k-9)$-skeleton of $\Sigma^{-1} Z(k)_2^\infty$ maps trivially to $S^0$ under the composition map 
$$\begin{tikzcd} (\Sigma^{-1} Z(k)_2^\infty)^{8k-9} \ar[r,hookrightarrow]& \Sigma^{-1} Z(k)_2^\infty \ar[r, "2^{4k-4}f"]& S^0. \end{tikzcd}$$
Therefore, there exists a map $S^{8k-5} \to T_{4k-3}$ such that the following diagram is commutative: 
$$\begin{tikzcd}
S^{8n-5} \ar[r, dashed] & T_{4k-3} \ar[d] \\ 
\Sigma^{-1} Z(k)_2^\infty \ar[r, swap, "2^{4k-4}f"] \ar[u, twoheadrightarrow] \ar[ru, "\ell_{4k-3}"] & S^0.
\end{tikzcd}$$

Let $\mu$ be the composition 
$$S^{8k-5} \longrightarrow T_{4k-3} \longrightarrow S^0.$$
To finish the proof of our proposition, it suffices to show that $\mu = \{P^{k-1}h_1^3\}$.  

Since the Adams filtration of $\mu$ is at least $4k-3$, $\mu$ can be $0$, $\{P^{k-1}h_2\}$, $2\{P^{k-1}h_2\}$, or $4\{P^{k-1}h_2\} = \{P^{k-1}h_1^3\}$.  We will compute the $e$-invariant of $e(\mu)$ and show that $\nu(e(\mu)) = 0$.  This will finish the proof because the 2-adic valuations for the $e$-invariants of the four possibilities above are
\begin{eqnarray*}
\nu(e(0))&\geq& 1, \\
\nu(e(\{P^{k-1}h_2\})) &=& -2,\\
\nu(e(2\{P^{k-1}h_2\})) &=& -1,\\
\nu(e(4\{P^{k-1}h_2\})) &=& 0.  
\end{eqnarray*}

Consider the diagram 
$$\begin{tikzcd}
\Sigma^{-1}Z(k) \ar[r] \ar[d, "\Sigma^{-1}h"] &\Sigma^{-1}Z(k)_2^\infty \ar[r, "f"] \ar[d, twoheadrightarrow]&S^0 \ar[r] \ar[d, "2^{4k-4}"] &Z(k) \ar[d, "h"] \\ 
\Sigma^{-1} C\mu \ar[r] & S^{8k-5} \ar[r, "\mu"] & S^0 \ar[r] & C\mu.
\end{tikzcd}$$
By the definition of the $e$-invariant, there exists an element $\xi \in ko^0(C\mu)$ such that 
$$ch(c(\xi)) = 1 + e(\mu).$$
This implies that when we pullback $\xi$ along the map $h: Z(k) \to C\mu$, the Chern character $ch(c(h^*\xi))$ is equal to 
\begin{eqnarray} \label{eqn:ChernCharxi}
ch(c(h^*\xi)) = 2^{4k-4} + e(\mu)x^{4k-2}. 
\end{eqnarray}

In Proposition~\ref{prop bundle construction 1}, we constructed an element $\alpha_k \in ko^0(Z(k))$ with Chern character
\begin{eqnarray} \label{eqn:ChernCharalpha1}
ch(c(\alpha_{k}))&=&2^{4k-4}+dx^{4k-2} \,\,\, (\nu(d) = 0). 
\end{eqnarray}
Subtracting equation (\ref{eqn:ChernCharalpha1}) from equation (\ref{eqn:ChernCharxi}), we get 
$$ch(c(h^*\xi - \alpha_k)) = (e(\mu) - d) x^{4k-2}. $$
In particular, this shows that when we restrict $h^*\xi - \alpha_k$ to the $(8k-8)$-skeleton $Z(k)^{8k-8}$, 
$$ch\left(c\left(\left.h^*\xi - \alpha_k \right|_{Z(k)^{8k-8}} \right) \right) =0. $$

By Lemma~\ref{lem:ChernComplexInjective}, 
$$\left.h^*\xi - \alpha_k \right|_{Z(k)^{8k-8}} = 0. $$
Therefore,
$$h^*\xi - \alpha_k = p^*(\phi)$$
for some $\phi \in ko^0(S^{8k-4})$.  Here, $p$ is the quotient map $p: Z(k) \twoheadrightarrow S^{8k-4}$.  The Chern character of $p^*\phi$ is 
$$ch(c(p^*\phi)) = a x^{4k-2},$$
where $\nu(a) \geq 1$.  From the relation
$$(e(\mu) - d) x^{4k-2}= ax^{4k-2},$$
we deduce that $e(\mu) = d + a$.  Since $\nu(d) = 0$ and $\nu(a) \geq 1$, $\nu(e(\mu)) = 0$.  This concludes the proof of the proposition. 
\end{proof}

\begin{comment}
\begin{rem}\rm
The intuition of our proof of Proposition~\ref{prop:FirstlockOdd} is as follows: consider the $S^0$-based Atiyah--Hirzebruch spectral sequence of $\Sigma^{-1}Z(k)$. The commutativity of diagram ~(\ref{diagram:2^4n-4Commute}) is equivalent to the fact that the class $2^{4k-4}[-1]$ on the bottom cell supports a differential hitting the class $\{P^{k-1}h_1^3\}[8n-5]$ on the top cell. 
\end{rem}
\end{comment}

%%%%%%%%%%%
\subsection{First lock for $k$ even} 
\begin{proof}[Proof of Proposition~\ref{pro: first lock for even}] 
In Proposition~\ref{prop bundle construction 1}, we showed that there exists an element $\alpha_{k}\in ko^{0}(Z(k))$ such that 
$$ch(c(\alpha_{k}))=2^{4k-4-\nu(k)}.$$ 
By Lemma~\ref{lem:ChernComplexInjective}, we can apply Theorem~\ref{prop: detect AHSS differential} to $Z(k)$.  Theorem~\ref{prop: detect AHSS differential} shows that the element 
$$2^{4k-4-\nu(k)}[-1]$$ 
is a permanent cycle in the $j''$-based Atiyah--Hirzebruch spectral sequence of $\Sigma^{-1}Z(k)$.  

The map $\rho$ constructed in Section~\ref{subsec:constructionofZ(k)} induces a map of spectral sequences from the $j''$-based Atiyah--Hirzebruch spectral sequence of $\Sigma^{-1}Z(k)$ to that of ${X(8k+3)^{8k-5}_{-1}}$.  Therefore, the element 
$$2^{4k-4-\nu(k)}[-1]$$
is also a permanent cycle in the $j''$-based Atiyah--Hirzebruch spectral sequence of $X(8k+3)^{8k-5}_{-1}$ and $X(8k+3)^{8k-5}$.  This finishes the proof of the proposition.  
\end{proof}

\appendix

\section{Coefficients of $\left(\frac{\ln(1+z)}{z}\right)^{4k+1}$}\label{sec:AppendixA}
Let $b_i$ be the coefficient of $z^i$ in the power series expansion of 
$$f(z) = \left(\frac{\ln(1+z)}{z}\right)^{4k+1}=\left(1-\frac{z}{2}+\frac{z^{2}}{3}-\frac{z^{3}}{4}+\cdots\right)^{4k+1}.$$
In this section, we prove several facts about the 2-adic valuations of $b_i$ that we are going to use in the rest of the paper.  

\begin{notation}
For any $r \in \mathbb{Q}$, let $\nu(r)$ be the 2-adic valuation of $r$.  For example, $\nu(4) = 2$, $\nu(3) = 0$, and $\nu\left(\frac{1}{8}\right) = -3$. 
\end{notation}

In the power series expansion of
$$f(z) = \left(\frac{\ln(1+z)}{z}\right)^{4k+1}=\left(1-\frac{z}{2}+\frac{z^{2}}{3}-\frac{z^{3}}{4}+\cdots\right)^{4k+1},$$
the coefficients for $z^m$ is 
$$b_{m}=\sum\limits_{(c_0, c_1, c_2, \ldots)}b_{(c_0, c_{1}, c_2, \ldots)},$$
where the sum ranges through all tuples $(c_0, c_1, c_2, \ldots)$ such that 
\begin{enumerate}
\item $c_i \geq0$ for all $i \geq 0$;
\item $c_0 + c_1 + c_2 + \cdots = 4k+1$;
\item $c_1 + 2c_2 + 3c_3 + \cdots = m$. 
\end{enumerate}
In all the cases that we are interested in, $m$ will always be at most $4k$, so the tuple $(c_0, c_1, c_2, \ldots)$ will always be finite.  Each tuple $(c_0, c_1, c_2, \ldots)$ corresponds to the monomial 
$$(1)^{c_0} \left(-\frac{z}{2} \right)^{c_1} \left(\frac{z^2}{3} \right)^{c_2} \cdots.$$
The number $b_{(c_0, c_{1}, c_2, \ldots)}$ is the coefficient of this monomial, which is 
$$b_{(c_0, c_1, c_2, \ldots)}=(-1)^{c_1 + c_3 + \cdots} \cdot \binom{4k+1}{c_0, c_1, c_2, \ldots} \cdot \frac{1}{2^{c_1}3^{c_2} \cdots}.$$
Here,
$$\binom{4k+1}{c_0, c_1, c_2, \ldots} = \frac{(4k+1)!}{c_0! c_1! c_2! \cdots}.$$ 
In particular, this number is an integer.  

\begin{lem}\label{lem:coeffb4n}
$\nu(b_{4k})=-4k$ for all $k \geq 0$.  
\end{lem}
\begin{proof} 
For any tuple $(c_0, c_1, \ldots )$ with $\sum_{i \geq 0} c_i = 4k+1$ and $\sum_{i \geq 1} ic_i = 4k$, the valuation 
$$\nu\left(\frac{1}{2^{c_1}3^{c_2} \cdots} \right) \geq -(4k-1)$$
except when $(c_0, c_1, \ldots ) = (1, 4k, 0, \ldots )$.  Since 
\begin{eqnarray*}
b_{(1, 4k, 0, \ldots)} &=& (-1)^{4k}\cdot  \binom{4k+1}{1, 4k}\cdot \frac{1}{2^{4k}} \\ 
&=& \frac{(4k+1)}{2^{4k}},
\end{eqnarray*}
the valuation $\nu(b_{4k})$ is equal to $-4k$.  
\end{proof}

\begin{lem}\label{lem:coeffbm<=4n-1}
The inequality $\nu(b_m) \geq -(4k-2)$ holds for all $k \geq 1$ and \\$1\leq m \leq 4k-1$. 
\end{lem}
\begin{proof}

For any positive integer $c$, we have the inequality 
$$\nu\left(\frac{1}{c+1}\right)\geq -c.$$ 
Equality is achieved only when $c=1$.  This implies that
\begin{equation}
\nu(b_{(c_0, c_{1},\cdots)})\geq \nu\left(\frac{1}{2^{c_1}3^{c_2} \cdots} \right) \geq -\sum_{i} i \cdot c_{i} = -m.
\end{equation}
From this, we deduce that $f(b_{m})\geq -(4k-2)$ for all $1\leq m \leq 4k-2$.

For $b_{4k-1}$, given any tuple $(c_0, c_1, c_2, \ldots)$ with $\sum_{i \geq 0} c_i = 4k+1$ and \\$\sum_{i \geq 1} ic_i= 4k-1$, the valuation 
$$\nu\left(\frac{1}{2^{c_1}3^{c_2} \cdots} \right) \geq -(4k-2)$$
except when $(c_0, c_1, c_2, \ldots) = (2, 4k-1, 0, \ldots)$.  Since 
\begin{eqnarray*}
b_{(2, 4k-1, 0, \ldots)} &=& (-1)^{4k-1} \cdot \binom{4k+1}{2, 4k-1}\cdot \frac{1}{2^{4k-1}} \\
&=& - \frac{(4k+1)k}{2^{4k-2}},
\end{eqnarray*}
the 2-adic valuation of the denominator is still at least $-(4k-2)$.  Therefore, $\nu(b_{4k-1}) \geq -(4k-2)$.  
\end{proof}

\begin{lem}\label{lem:coeffb4n-2}
$\nu(b_{4k-2}) = \nu(k) - (4k-3)$ for all $k \geq 1$. 
\end{lem}
\begin{proof}
The coefficient of the monomial $1^3 \left(-\frac{z}{2} \right)^{4k-2}$ in $f(z)$ is 
\begin{eqnarray*}
\binom{4k+1}{3,4k-2}\cdot 1^3\cdot  \left(-\frac{z}{2} \right)^{4k-2} &=& \frac{(4k+1)(4k)(4k-1)}{3!}\cdot \frac{z^{4k-2}}{2^{4k-2}}\\
&=& \text{odd} \cdot \frac{k}{2^{4k-3}} \cdot z^{4k-2}. 
\end{eqnarray*}
The valuation of this number is exactly $\nu(k) - (4k-3)$.  We will prove that the coefficients of all the other monomials in $f(z)$ of degree $z^{4k-2}$ have 2-adic valuations strictly larger than $\nu(k) - (4k-3)$.  

Consider the monomial 
$$\binom{4k+1}{c_0, c_1, c_2, \ldots} \cdot (1)^{c_0} \cdot \left(\frac{z}{2} \right)^{c_1}\cdot  \left(\frac{z^2}{3} \right)^{c_2} \cdot \left(\frac{z^3}{4}\right)^{c_3} \cdots ,$$
where only finitely many of the $c_i$'s are nonzero and $c_1 + 2c_2 + 3c_3 + \cdots = 4k-2$.  To prove our claim above, it suffices to show that the fraction 
$$\frac{\binom{4k+1}{c_0, c_1, c_2, \ldots}\cdot  (1)^{c_0}\cdot \left(\frac{1}{2} \right)^{c_1} \cdot \left(\frac{1}{3} \right)^{c_2}\cdot  \left(\frac{1}{4}\right)^{c_3} \cdots }{\binom{4k+1}{3,4k-2}\cdot 1^3 \cdot \left(\frac{1}{2} \right)^{4k-2}}$$
is an even 2-local integer.  

This fraction is equal to 
\begin{eqnarray*}
&&\frac{2^{4k-2}}{2^{c_1}3^{c_2} 4^{c_3} \cdots} \cdot \frac{\binom{4k+1}{c_0, c_1, c_2, \ldots}}{\binom{4k+1}{3}} \\
&=& \frac{2^{4k-2}}{2^{c_1}3^{c_2} 4^{c_3} \cdots} \cdot \frac{(4k-2)!3!}{c_0! c_1! c_2! \cdots} \\ 
&=& \frac{2^{4k-2}}{2^{c_1}3^{c_2} 4^{c_3} \cdots} \cdot \frac{3!}{c_0(c_0-1)(c_0-2)} \cdot \frac{(4k-2)!}{(c_0-3)! c_1! c_2! \cdots} \\ 
&=& \frac{2^{4k-2}}{2^{c_1}3^{c_2} 4^{c_3} \cdots} \cdot \frac{3!}{c_0(c_0-1)(c_0-2)} \cdot \binom{4k-2}{c_0-3, c_1, c_2, \ldots}.
\end{eqnarray*}
The condition $c_1 + 2c_2 + 3c_3 + \cdots = 4k-2$ essentially guarantees that the product of the first two terms is an even integer when $(c_1, c_2, \ldots)$ differs from $(3, 4k-2, 0, \ldots)$.  There are two exception cases.  They are $(4, 4k-4, 1, 0, \ldots)$ and $(5, 4k-5, 0, 1, 0, \ldots)$.  

For the first exception case, the product is 
$$\frac{2^{4k-2}}{2^{4k-4}3^1} \cdot \frac{3!}{4 \cdot 3 \cdot 2} \cdot \binom{4k-2}{1, 4k-4, 1}.$$
The product of the first two terms is odd, but the last term is $\frac{(4k-2)(4k-3)}{1!1!}$, which is even.

For the second exception case, the product is 
$$\frac{2^{4k-2}}{2^{4k-5} \cdot 4^1} \cdot \frac{3!}{5\cdot 4 \cdot 3} \cdot \binom{4k-2}{2, 4k-5, 1}.$$
The product of the first two terms is odd, but the last term is 
$$\frac{(4k-2)(4k-3)(4k-4)}{2!1!},$$
which is even again.  
Therefore, $\nu(b_{4k-3}) = \nu(k) - (4k-3)$, as desired.  
\end{proof}

\begin{lem}\label{lem:coeffb4n-3}
$\nu(b_{4k-3}) = \nu(k) - (4k-3)$ for all $k \geq 1$. 
\end{lem}
\begin{proof}
The proof is very similar to the proof of Lemma~\ref{lem:coeffb4n-2}.  Given a monomial in $f(z)$ of degree $z^{4k-3}$, the smallest 2-adic valuation of its coefficient is achieved when $(c_0, c_1, c_2, \ldots) = (4, 4k-3, 0, \ldots)$.  This coefficient is 
$$\binom{4k+1}{4} \cdot \frac{1}{2^{4k-3}} = \frac{(4k+1)(4k)(4k-1)(4k-2)}{4!} \cdot \frac{1}{2^{4k-3}}.$$
Its 2-adic valuation is $\nu(k)-(4k-3)$.  

To prove that the 2-adic valuations of the all the other coefficients are strictly bigger than this number, we make a similar computation to the proof of Lemma~\ref{lem:coeffb4n-2} and reduce the problem into showing that the ratio
$$ \frac{2^{4k-3}}{2^{c_1}3^{c_2} 4^{c_3} \cdots} \cdot \frac{1}{\binom{c_0}{4}} \cdot \binom{4k-3}{c_0-4, c_1, c_2, \ldots}$$
is even when $c_1 +2c_2 + 3c_3 + \cdots = 4k-3$ and $(c_0, c_1, c_2, \ldots) \neq (4, 4k-3, 0, \ldots)$.   the product of the first two terms is an even number.  
\end{proof}

\begin{lem}\label{lem:coeffb4n-4}
$\nu(b_{4k-4}) = \nu(k)-(4k-4)$ for all $k \geq 1$.  
\end{lem}
\begin{proof}
The proof for this is again similar to the proof of Lemma~\ref{lem:coeffb4n-2} and Lemma~\ref{lem:coeffb4n-3}.  We claim that the smallest 2-adic valuation is achieved only when $c_0 = 5$, \\$c_1 = 4k-4$, and $c_i = 0$ for all $i \geq 2$.  The corresponding coefficient is 
$$\binom{4k+1}{5}\cdot \frac{1}{2^{4k-4}} = \frac{(4k+1)(4k)(4k-1)(4k-2)(4k-3)}{5!} \cdot \frac{1}{2^{4k-4}} = \text{odd} \cdot \frac{k}{2^{4k-4}}.$$
The 2-adic valuation for this number is $\nu(k) - (4k-4)$.  To prove that all the other coefficients have bigger valuations, we need to show that the ratio
$$ \frac{2^{4k-4}}{2^{c_1}3^{c_2} 4^{c_3} \cdots} \cdot \frac{1}{\binom{c_0}{5}} \cdot \binom{4k-4}{c_0-5, c_1, c_2, \ldots}$$
is even for all the other tuples $(c_0, c_1, \ldots)$ such that  $c_1 +2c_2 + 3c_3 + \cdots = 4k-4$.  The product of the first two terms will always be an even number except when $(c_1, c_2, c_3, \ldots) = (4k-9, 1, 1, 0, \ldots)$.  For this exceptional case, the ratio is 
$$\frac{2^{4k-4}}{2^{4n-9}\cdot 3^1\cdot 4^1} \cdot \frac{1}{\binom{8}{5}} \cdot \binom{4k-4}{3, 4k-9, 1, 1}$$
The product of the first two terms is odd but the last term is 
$$\frac{(4k-4)(4k-5)(4k-6)(4k-7)(4k-8)}{3!1!1!},$$
which is even.  
\end{proof}

\begin{lem}\label{lem:coeffb4n-2minusb4n-3}
We have 
$$\nu(b_{4k-2} - b_{4k-3}) \left\{\begin{array}{ll} = \nu(k) - (4k-4), & \text{{$k \geq 2$ even,}} \\
\geq \nu(k) - (4k-5), & \text{$k \geq 1$ odd} . \end{array}  \right. $$
\end{lem}
\begin{proof}
To prove the lemma, it suffices to consider all the coefficients in $b_{4k-2}$ and $b_{4k-3}$ whose valuation is at most $\nu(k) - (4k-4)$.  For $b_{4k-2}$, they are the following: 
\begin{eqnarray*}
\binom{4k+1}{3, 4k-2}\cdot (1)^3\cdot \left(-\frac{z}{2} \right)^{4k-2} &=& \frac{(4k+1)(4k)(4k-1)}{3!} \cdot \frac{1}{2^{4k-2}} \cdot z^{4k-2} \\
&=& \frac{(4k+1)(4k-1)}{3} \cdot \frac{k}{2^{4k-3}} \cdot z^{4k-2}\\ 
\binom{4k+1}{4, 4k-4, 1}\cdot (1)^4\cdot \left(-\frac{z}{2} \right)^{4k-4}\cdot \left(\frac{z^2}{3} \right)^{1} &=&\frac{(4k+1)(4k)(4k-1)(4k-2)(4k-3)}{4!1!}\cdot \frac{1}{2^{4k-4}} \cdot \frac{1}{3} \cdot z^{4k-2} \\ 
&=& \frac{(4k+1)(4k-1)(2k-1)(4k-3)}{9} \cdot \frac{k}{2^{4k-4}} \cdot z^{4k-2}. 
\end{eqnarray*}
All the other coefficients have 2-adic valuations at least $\nu(k) - (4k-5)$.  For $b_{4k-3}$, only the term 
\begin{eqnarray*}
\binom{4k+1}{4, 4k-3}(1)^4\left(-\frac{z}{2} \right)^{4k-3} &=& -\frac{(4k+1)(4k)(4k-1)(4k-2)}{4!}\cdot \frac{1}{2^{4k-3}} \cdot z^{4k-3} \\ 
&=& - \frac{(4k+1)(4k-1)(2k-1)}{3}\cdot \frac{k}{2^{4k-3}} \cdot z^{4k-3}
\end{eqnarray*}
will matter.  All the other coefficients have 2-adic valuations at least $\nu(k) - (4k-5)$.  

We have 
\begin{eqnarray*}
&&\frac{(4k+1)(4k-1)}{3} \cdot \frac{k}{2^{4k-3}} + \frac{(4k+1)(4k-1)(2k-1)(4k-3)}{9} \cdot \frac{k}{2^{4k-4}} \\
&&- \left(- \frac{(4k+1)(4k-1)(2k-1)}{3}\cdot \frac{k}{2^{4k-3}} \right)  \\ 
&=& \frac{(4k+1)(4k-1)}{3} \cdot \frac{k}{2^{4k-4}} \cdot \left(\frac{1}{2} + \frac{(2k-1)(4k-3)}{3} + \frac{2k-1}{2} \right) \\ 
&=&\frac{(4k+1)(4k-1)}{3} \cdot \frac{k}{2^{4k-4}} \cdot \left(\frac{(2k-1)(4k-3)}{3} + k \right).
\end{eqnarray*}
When $k$ is even, $\frac{(2k-1)(4k-3)}{3} + k$ is odd, and the 2-adic valuation of the last expression is exactly $\nu(k) - (4k-4)$.  When $n$ is odd, $\frac{(2k-1)(4k-3)}{3} + k$ is even, and the 2-adic valuation of the last expression is at least $\nu(k) - (4k-4) + 1 = \nu(k) - (4k-5)$.  This proves the lemma. 
\end{proof}

\begin{lem}\label{lem:b<=4n-5}
For a fixed $k \geq 2$, the inequality $\nu(b_{m}) \geq \nu(k) - (4k-5)$ holds for all $m \leq 4k-5$.  
\end{lem}
\begin{proof}
We claim that the 2-adic valuations of all the coefficients for $b_m$ satisfy $\nu(k) - (4k-5)$.  We will divide the proof into four cases: 
\vspace{0.1in}

\noindent Case 1: there exist $i, j \geq 1$ such that $c_i, c_j \neq 0$ in the tuple $(c_0, c_1, \ldots)$.  Consider the ratio 
\begin{eqnarray*}
&&\frac{\binom{4k+1}{c_0, c_1, c_2, \ldots} \cdot (1)^{c_0} \cdot \left(\frac{1}{2} \right)^{c_1} \cdot \left(\frac{1}{3} \right)^{c_2}\cdot  \left(\frac{1}{4}\right)^{c_3} \cdots}{\frac{k}{2^{4k-5}}} \\ 
&=& \frac{(4k+1)(4k)}{c_i c_j}  \cdot \binom{4k-1}{c_0, c_1, \ldots, c_i-1, \ldots, c_j-1, \ldots} \cdot \frac{1}{1^{c_0}2^{c_1}3^{c_2} \cdots} \cdot \frac{2^{4k-5}}{n} \\ 
&=& \binom{4k-1}{c_0, c_1, \ldots, c_i-1, \ldots, c_j-1, \ldots} \cdot \frac{4k+1}{ c_i c_j \cdot 1^{c_0}2^{c_1}3^{c_2} \cdots} \cdot 2^{4k-3}.
\end{eqnarray*}
Since $c_1 + 2c_2 + 3c_3 + \cdots = m \leq 4k-5$ and $\nu(c_ic_j) \leq c_i + c_j$, 
$$\nu(c_i c_j \cdot 1^{c_0}2^{c_1}3^{c_2} \cdots) \leq 4k-5$$
and the last expression is even.  Therefore, the 2-adic valuation of the coefficient is at least $\nu(k) - (4k-5)$.  
\vspace{0.1in}

\noindent Case 2: There exists only one $i \geq2$ such that $c_i \neq 0$, and that $c_i$ is at least 2.  Consider the ratio 
\begin{eqnarray*}
\frac{\binom{4k+1}{c_0, c_1, c_i}\cdot  \frac{1}{2^{c_1}(i+1)^{c_i}}}{\frac{k}{2^{4k-5}}} &=& \binom{4k-1}{c_0, c_1, c_i -2} \cdot \frac{(4k+1)(4k)}{c_i(c_i-1)} \cdot \frac{1}{2^{c_1}(i+1)^{c_i}}\cdot  \frac{2^{4k-5}}{k}\\ 
&=&\binom{4k-1}{c_0, c_1, c_i -2} \cdot (4k+1) \cdot \frac{2^{4k-3}}{c_i(c_i-1)2^{c_1}(i+1)^{c_i}}\\ 
\end{eqnarray*} 
Since $c_1 + 2c_2 + 3c_3 + \cdots = m \leq 4k-5$ and $\nu(c_i (c_i -1)) \leq c_i$, 
$$\nu(c_i(c_i-1)2^{c_1}(i+1)^{c_i}) \leq 4k-5$$
and the last expression is even.  
\vspace{0.1in}

\noindent Case 3: There exists only one $i \geq 2$ such that $c_i \neq 0 $, and that $c_i$ is 1.  Consider the ratio 
\begin{eqnarray*}
\frac{\binom{4k+1}{c_0, c_1, 1} \cdot \frac{1}{2^{c_1}(i+1)}}{\frac{k}{2^{4k-5}}}&=& \binom{4k-1}{c_0-1, c_1} \cdot \frac{(4k+1)4k}{c_0 \cdot 1} \cdot \frac{1}{2^{c_1}(i+1)} \cdot \frac{2^{4k-5}}{k} \\ 
&=&\binom{4k-1}{c_0-1, c_1} \cdot (4k+1) \cdot \frac{2^{4k-3}}{2^{c_1} (i+1) c_0} \\ 
&=&\binom{4k-1}{c_0-1, c_1} \cdot (4k+1) \cdot  \frac{2^{4k-3 - m+i}}{(i+1)(4k+i-m)} 
\end{eqnarray*} 
where we have used the facts that $c_1 + i = m$ and $c_0 + c_1 = 4k$.  Let $a = i+1$, and $b  = 4k+i -m$.  Then $a \geq 2+1 = 3$ and 
$$b-a = (4k+i-m) - (i+1) = 4k-m -1 \geq 4k- (4k-5) - 1 = 4.$$ 
The term 
$$\frac{2^{4k-3 - m+i}}{(i+1)(4k+i-m)}$$
in the last expression is equal to $\frac{2^{b-3}}{ab}$.  This number is an integer for all positive integers $(a,b)$ where $a \geq 3$ and $b-a \geq 4$. 

\vspace{0.1in}

\noindent Case 4: There exists no $i \geq 2$ such that $c_i \neq 0$.  Consider the ratio 
\begin{eqnarray*}
\frac{\binom{4k+1}{4k+1-m, m}\cdot \frac{1}{2^m}}{\frac{k}{2^{4k-5}}}&=& \binom{4k-1}{4k-1-m, m} \cdot \frac{(4k+1)(4k)}{(4k+1-m)(4k-m)} \cdot \frac{1}{2^m}\cdot \frac{2^{4k-5}}{n}\\ 
&=& \binom{4k-1}{4k-1-m, m} \cdot (4k+1) \cdot \frac{2^{4k-3-m}}{(4k+1-m)(4k-m)}
\end{eqnarray*} 
Since exactly one of $4k+1-m$ and $4k-m$ is even and $4k-m \geq 4k-(4k-5) = 5$, the number
$$\frac{2^{4k-3-m}}{(4k+1-m)(4k-m)}$$
is always an integer.  

\end{proof}

\section{Cell diagrams and the Atiyah--Hirzebruch spectral sequence}\label{sec:AppendixB}

The theory of cell diagrams is a very powerful tool when thinking of finite CW spectra. See \cite{BJM, WangXu, Xu} for example. We use them as illustration purpose in our paper. In this section, we recall the definition of cell diagrams from \cite{BJM} and talk about its connection to the Atiyah--Hirzebruch spectral sequence.

\begin{df} \rm
Let $Z$ be a finite CW spectrum. A cell diagram for $Z$ consists of nodes and edges. The nodes are in 1-1 correspondence with a chosen basis of the mod 2 homology of $Z$, and may be labeled with symbols to indicate the dimension. When two nodes are joined by an edge, then it is possible to form an $\textup{H}\mathbb{F}_2$-subquotient
$$Z'/Z'' = S^n \smile_f e^m,$$
\begin{displaymath}
    \xymatrix{
 *+[o][F-]{m} \ar@{-}[d]^{f}  \\
 *+[o][F-]{n}  }
\end{displaymath}
which is the cofiber of $f$ with certain suspension. Here $f$, the attaching map, is an element in the stable homotopy groups of spheres. For simplicity, we do not draw an edge if the corresponding $f$ is null.

Suppose we have two nodes labeled $n$ and $m$ with $n<m$, and there is no edge joining them. Then there are two possibilities.

The first one is that there is an integer $k$, and a sequence of nodes labeled $n_i, 0\leq i \leq k$, with $n=n_0<n_1<\cdots<n_k=m$, and edges joining the nodes $n_i$ to the nodes $n_{i+1}$. In this case we do not assert that there is an $\textup{H}\mathbb{F}_2$-subquotient of the form above; this does not imply that there is no such $\textup{H}\mathbb{F}_2$-subquotient.

The second one is that there is no such sequence as in the first case. In this case, there exists an $\textup{H}\mathbb{F}_2$-subquotient which a wedge of spheres $S^n\vee S^m$.
\end{df}

\begin{rem} \rm
In \cite{BJM}'s original definition, they use subquotients instead of $\textup{H}\mathbb{F}_2$-subquotients.
\end{rem}

The following example shows the indeterminacy of cell diagrams associated to a given CW spectrum.

\begin{exam}\rm
Let $f$ be the composite of the following two maps:

\begin{displaymath}
    \xymatrix{
 S^2 \ar[r]^{\eta^2} & S^0 \ar[r]^i & C\eta,
 }
\end{displaymath}
where the second map $i$ is the inclusion of the bottom cell. Consider $Cf$: the cofiber of $f$, which is a 3 cell complex with the following cell diagram:

\begin{displaymath}
    \xymatrix{
    *+[o][F-]{3} \\
    *+[o][F-]{2} \ar@{-}@/^1pc/[d]^{\eta} \\
    *+[o][F-]{0} }
\end{displaymath}
It is clear that the top cell of $Cf$ splits off, since $\eta^2$ can be divided by $\eta$. So we do not have to draw any attaching map from the cell in dimension 3 to the one in dimension 0. Note that the cofiber of $\eta^2$ is in fact an $\textup{H}\mathbb{F}_2$-subcomplex of $Cf$. \end{exam}

We give two more interesting examples.

\begin{exam} \label{cp3 and its dual} \rm
Consider the suspension spectrum of $\mathbb{C}P^3$. It is a 3 cell complex with cells in dimensions 2, 4 and 6. It was shown by Adams \cite{Adams} that, the secondary cohomology operation $\Psi$, which is associated to the relation
$$Sq^4 Sq^1 + Sq^2 Sq^1 Sq^2 + Sq^1 Sq^4 = 0,$$
is nonzero on this spectrum. In other words, there exists an attaching map between the cells in dimension 2 and 6, which is detected by $h_0h_2$ in the 3-stem of the Adams $E_\infty$ page. Note that $h_0h_2$ detects two homotopy classes: $2\nu, \ 6\nu$. Their difference is $4\nu = \eta^3$, which is divisible by $\eta$. Therefore, we have its cell diagram as the following:

\begin{displaymath}
    \xymatrix{
    *+[o][F-]{6} \ar@{-}@/_1pc/[dd]_{2\nu} \\
    *+[o][F-]{4} \ar@{-}@/^1pc/[d]^{\eta} \\
    *+[o][F-]{2} }
\end{displaymath}
We can also consider the Spanier--Whitehead dual of the suspension spectrum of $\mathbb{C}P^3$. It is a 3 cell complex with cells in dimensions -2, -4 and -6, with the following cell diagram
\begin{displaymath}
    \xymatrix{
    *+[o][F-]{-2} \ar@{-}@/_1pc/[dd]_{2\nu} \ar@{-}@/^1pc/[d]^{\eta}\\
    *+[o][F-]{-4}  \\
    *+[o][F-]{-6} }
\end{displaymath}
\end{exam}

In a way, the attaching maps drawn in the cell diagram of a CW spectrum correspond to certain differentials in its Atiyah--Hirzebruch spectral sequence. We illustrate this idea through Example~\ref{cp3 and its dual}. For notations regarding the Atiyah--Hirzebruch spectral sequence, we refer to Terminology~\ref{AHSS} and Sections 3 and 6 of \cite{WangXu}.

\begin{exam}\rm
For the suspension spectrum of $\mathbb{C}P^3$, the attaching map $\eta$ corresponds to the $d_2$-differential
$$1[4] \rightarrow \eta[2]$$
and its multiples
$$\alpha[4] \rightarrow \alpha\cdot\eta[2]$$
for any element $\alpha$ in the stable stems, in the Atiyah--Hirzebruch spectral sequence of $\mathbb{C}P^3$. The $2\nu$-attaching map then corresponds to the $d_4$-differential
$$1[6] \rightarrow 2\nu[2]$$
and its multiples. Note that $2[6] \rightarrow 4\nu[2] = \eta^3[2]$, which is already killed by a $d_2$-differential. Therefore $2[6]$ is a permanent cycle.

For its Spanier--Whitehead dual, the attaching map $\eta$ corresponds to the $d_2$-differential
$$1[-2] \rightarrow \eta[-4]$$
and its multiples. For the $2\nu$-attaching map, it does not correspond to a $d_4$-differential
$$1[-2] \not\rightarrow 2\nu[-6],$$
since $1[-2]$ already supports a nonzero $d_2$-differential so it is not present at the $E_4$-page anymore. However, this $d_4$-differential still ``exists", in the sense that some of its multiples still exist. More precisely, suppose that $\beta$ is an element in the stable stems such that $\beta\cdot \eta = 0$. Then $\beta[-2]$ survives to the $E_4$-page and we have a $d_4$-differential
$$\beta[-2] \rightarrow \beta\cdot 2\nu[-6],$$
which might or might not be zero, depending on whether $\beta\cdot 2\nu$ is zero. For example, we have a nonzero $d_4$-differential
$$2[-2] \rightarrow 4\nu[-6] = \eta^3[-6].$$
\end{exam}

%%%%%%%%%%%%%%%%%%%%
\bibliographystyle{alpha}
\bibliography{reference}

\end{document}